\documentclass[11pt]{preprint}
\usepackage[utf8]{inputenc}
\usepackage[full]{textcomp}
\usepackage[osf]{newtxtext} 
\usepackage[cal=boondoxo]{mathalfa}
\usepackage{colortbl}
\usepackage[all]{xy}
\usepackage{comment}

\usepackage{amssymb}
\usepackage{mathtools}
\usepackage{hyperref}
\usepackage{breakurl}
\usepackage{mhenvs}
\usepackage{mhequ} 
\usepackage{mhsymb}
\usepackage{booktabs}

\usepackage{tikz}
\usepackage{tikz-cd}

\usepackage{tcolorbox}
\usepackage{mathrsfs}

\usepackage{longtable}
\usepackage{wrapfig}

\usepackage{graphicx}
\usepackage{xspace}

\usepackage{microtype}
\usepackage{comment}
\usepackage{wasysym}
\usepackage{centernot}
\usepackage{enumitem}
\usepackage{bm}
\usepackage{stackrel}

\makeatletter
\newcommand{\globalcolor}[1]{%
  \color{#1}\global\let\default@color\current@color
}
\makeatother

\usetikzlibrary{calc}
\usetikzlibrary{decorations}
\usetikzlibrary{positioning}
\usetikzlibrary{shapes}
\usetikzlibrary{external}
\usetikzlibrary{decorations.markings}

 \AtBeginEnvironment{tikzcd}{\tikzexternaldisable}
 \AtEndEnvironment{tikzcd}{\tikzexternalenable}

\newif\ifdark
\darkfalse

\ifdark

\definecolor{darkred}{rgb}{0.9,0.2,0.2}
\definecolor{darkblue}{rgb}{0.7,0.3,1}
\definecolor{darkgreen}{rgb}{0.1,0.9,0.1}
\definecolor{franck}{rgb}{0,0.8,1}
\definecolor{pagebackground}{rgb}{.15,.21,.18}
\definecolor{pageforeground}{rgb}{.84,.84,.85}
\pagecolor{pagebackground}
\AtBeginDocument{\globalcolor{pageforeground}}
\definecolor{symbols}{rgb}{0,0.7,1}
\colorlet{connection}{red!80!black}
\colorlet{boxcolor}{blue!50}

\else

\definecolor{darkred}{rgb}{0.7,0.1,0.1}
\definecolor{darkblue}{rgb}{0.4,0.1,0.8}
\definecolor{darkgreen}{rgb}{0.1,0.7,0.1}
\definecolor{franck}{rgb}{0,0,1}
\definecolor{pagebackground}{rgb}{1,1,1}
\definecolor{pageforeground}{rgb}{0,0,0}
\colorlet{symbols}{blue!90!black}
\colorlet{connection}{red!30!black}
\colorlet{boxcolor}{blue!50!black}

\fi

\def\slash{\leavevmode\unskip\kern0.18em/\penalty\exhyphenpenalty\kern0.18em}
\def\dash{\leavevmode\unskip\kern0.18em--\penalty\exhyphenpenalty\kern0.18em}

\newcommand{\eqlaw}{\stackrel{\mbox{\tiny law}}{=}}

\DeclareMathAlphabet{\mathbbm}{U}{bbm}{m}{n}

\DeclareFontFamily{U}{BOONDOX-calo}{\skewchar\font=45 }
\DeclareFontShape{U}{BOONDOX-calo}{m}{n}{
  <-> s*[1.05] BOONDOX-r-calo}{}
\DeclareFontShape{U}{BOONDOX-calo}{b}{n}{
  <-> s*[1.05] BOONDOX-b-calo}{}
\DeclareMathAlphabet{\mcb}{U}{BOONDOX-calo}{m}{n}
\SetMathAlphabet{\mcb}{bold}{U}{BOONDOX-calo}{b}{n}

\setlist{noitemsep,topsep=4pt,leftmargin=1.5em}

\DeclareMathAlphabet{\mathbbm}{U}{bbm}{m}{n}

\DeclareMathAlphabet{\mcb}{U}{BOONDOX-calo}{m}{n}
\SetMathAlphabet{\mcb}{bold}{U}{BOONDOX-calo}{b}{n}
\DeclareFontFamily{U}{mathx}{\hyphenchar\font45}
\DeclareFontShape{U}{mathx}{m}{n}{
      <5> <6> <7> <8> <9> <10>
      <10.95> <12> <14.4> <17.28> <20.74> <24.88>
      mathx10
      }{}
\DeclareSymbolFont{mathx}{U}{mathx}{m}{n}
\DeclareMathSymbol{\bigtimes}{1}{mathx}{"91}

\def\emptyset{{\centernot\ocircle}}

\def\restr{\mathord{\upharpoonright}}

\providecommand{\figures}{false}
{ \ifthenelse{\equal{\figures}{false}} {#1}{\[ {\rm Figure \ missing !} \]} }{}
\def\id{\mathrm{id}}

\let\graftI\curvearrowright
\def\graftID{\mathrel{\textcolor{connection}{\boldsymbol{\curvearrowright}}}}

\def\CH{\mathcal{H}}
\def\CP{\mathcal{P}}
\def\CW{\mathcal{W}}
\def\CG{\mathcal{G}}

\def\CA{\mathcal{A}}
\def\CE{\mathcal{E}}
\def\CC{\mathcal{C}}
\def\CQ{\mathcal{Q}}
\def\CB{\mathcal{B}}
\def\CM{\mathcal{M}}
\def\CT{\mathcal{T}}
\def\RR{\mathfrak{R}}

\tikzstyle{tinydots}=[dash pattern=on \pgflinewidth off \pgflinewidth]
\tikzstyle{superdense}=[dash pattern=on 4pt off 1pt]

\newcommand{\upper}{u}
\newcommand{\low}{\ell}

\newcommand{\mfn}{\mathfrak{n}}

\newcommand{\mfd}{\mathfrak{d}}

\newcommand{\mft}{\mathfrak{t}}
\newcommand{\mfs}{\mathfrak{s}}

\newcommand{\mff}{\mathfrak{f}}

\newcommand{\mfg}{\mathfrak{g}}

\def\Lab{\mathfrak{L}}

\def\Deltam{\Delta^{\!-}}

\def\${|\!|\!|}
\def\DD{\mathscr{D}}

\def\scal#1{{\langle#1\rangle}}
\def\sscal#1{{\langle\mkern-4.5mu\langle#1\rangle\mkern-4.5mu\rangle}}

\def\Trees{\mathfrak{T}}
\def\SS{\mathfrak{S}}
\def\SN{\mathfrak{N}}

\def\Vec{\mathop{\mathrm{Vec}}}

\def\VV{\mathfrak{V}}

\def\hPPi{\rlap{$\hat{\phantom{\PPi}}$}\PPi}

\newenvironment{DIFnomarkup}{}{} 

\def\Ev{\mathrm{Ev}}

\theorembodyfont{\rmfamily}
\newtheorem{example}[lemma]{Example}

\newfont{\indic}{bbmss12}

\def\PPi{\boldsymbol{\Pi}}

\def\Nabla_#1{\nabla_{\!#1}}

\makeatletter
\pgfdeclareshape{crosscircle}
{
  \inheritsavedanchors[from=circle] 
  \inheritanchorborder[from=circle]
  \inheritanchor[from=circle]{north}
  \inheritanchor[from=circle]{north west}
  \inheritanchor[from=circle]{north east}
  \inheritanchor[from=circle]{center}
  \inheritanchor[from=circle]{west}
  \inheritanchor[from=circle]{east}
  \inheritanchor[from=circle]{mid}
  \inheritanchor[from=circle]{mid west}
  \inheritanchor[from=circle]{mid east}
  \inheritanchor[from=circle]{base}
  \inheritanchor[from=circle]{base west}
  \inheritanchor[from=circle]{base east}
  \inheritanchor[from=circle]{south}
  \inheritanchor[from=circle]{south west}
  \inheritanchor[from=circle]{south east}
  \inheritbackgroundpath[from=circle]
  \foregroundpath{
    \centerpoint%
    \pgf@xc=\pgf@x%
    \pgf@yc=\pgf@y%
    \pgfutil@tempdima=\radius%
    \pgfmathsetlength{\pgf@xb}{\pgfkeysvalueof{/pgf/outer xsep}}%
    \pgfmathsetlength{\pgf@yb}{\pgfkeysvalueof{/pgf/outer ysep}}%
    \ifdim\pgf@xb<\pgf@yb%
      \advance\pgfutil@tempdima by-\pgf@yb%
    \else%
      \advance\pgfutil@tempdima by-\pgf@xb%
    \fi%
    \pgfpathmoveto{\pgfpointadd{\pgfqpoint{\pgf@xc}{\pgf@yc}}{\pgfqpoint{-0.707107\pgfutil@tempdima}{-0.707107\pgfutil@tempdima}}}
    \pgfpathlineto{\pgfpointadd{\pgfqpoint{\pgf@xc}{\pgf@yc}}{\pgfqpoint{0.707107\pgfutil@tempdima}{0.707107\pgfutil@tempdima}}}
    \pgfpathmoveto{\pgfpointadd{\pgfqpoint{\pgf@xc}{\pgf@yc}}{\pgfqpoint{-0.707107\pgfutil@tempdima}{0.707107\pgfutil@tempdima}}}
    \pgfpathlineto{\pgfpointadd{\pgfqpoint{\pgf@xc}{\pgf@yc}}{\pgfqpoint{0.707107\pgfutil@tempdima}{-0.707107\pgfutil@tempdima}}}
  }
}
\makeatother

\def\symbol#1{\textcolor{symbols}{#1}}

\def\decorate#1#2{
        \ifnum#2>0
    		\foreach \count in {1,...,#2}{
	       	let
				\p1 = (sourcenode.center),
                \p2 = (sourcenode.east),
				\n1 = {\x2-\x1},
				\n2 = {1mm},
				\n3 = {(1.3+0.6*(\count-1))*\n1},
				\n4 = {0.7*\n1}
			in 
        		node[rectangle,fill=symbols,rotate=30,inner sep=0pt,minimum width=0.2*\n2,minimum height=\n2] at ($(sourcenode.center) + (\n3,\n4)$) {}
				}
		\fi
        \ifnum#1>0
    		\foreach \count in {1,...,#1}{
	       	let
				\p1 = (sourcenode.center),
                \p2 = (sourcenode.east),
				\n1 = {\x2-\x1},
				\n2 = {1mm},
				\n3 = {(1.3+0.6*(\count-1))*\n1},
				\n4 = {0.7*\n1}
			in 
        		node[rectangle,fill=symbols,rotate=-30,inner sep=0pt,minimum width=0.2*\n2,minimum height=\n2] at ($(sourcenode.center) + (-\n3,\n4)$) {}
				}
		\fi
}

\tikzset{
    dectriangle/.style 2 args={
        triangle,
        alias=sourcenode,
        append after command={\decorate{#1}{#2}}
    },
    dectriangle/.default={0}{0},
}

\tikzset{
	cross/.style={path picture={ 
  		\draw[symbols]
			(path picture bounding box.south east) -- (path picture bounding box.north west) (path picture bounding box.south west) -- (path picture bounding box.north east);
		}},
root/.style={circle,fill=green!50!black,inner sep=0pt, minimum size=1.2mm},
        dot/.style={circle,fill=pageforeground,inner sep=0pt, minimum size=1mm},
        dotred/.style={circle,fill=pageforeground!50!pagebackground,inner sep=0pt, minimum size=2mm},
        var/.style={circle,fill=pageforeground!10!pagebackground,draw=pageforeground,inner sep=0pt, minimum size=3mm},
        kernel/.style={semithick,shorten >=2pt,shorten <=2pt},
        kernels/.style={snake=zigzag,shorten >=2pt,shorten <=2pt,segment amplitude=1pt,segment length=4pt,line before snake=2pt,line after snake=5pt,},
        rho/.style={densely dashed,semithick,shorten >=2pt,shorten <=2pt},
           testfcn/.style={dotted,semithick,shorten >=2pt,shorten <=2pt},
        renorm/.style={shape=circle,fill=pagebackground,inner sep=1pt},
        labl/.style={shape=rectangle,fill=pagebackground,inner sep=1pt},
        xic/.style={very thin,circle,draw=symbols,fill=symbols,inner sep=0pt,minimum size=1.2mm},
        g/.style={very thin,rectangle,draw=symbols,fill=symbols!10!pagebackground,inner sep=0pt,minimum width=2.5mm,minimum height=1.2mm},
        xi/.style={very thin,circle,draw=symbols,fill=symbols!10!pagebackground,inner sep=0pt,minimum size=1.2mm},
	xies/.style={very thin,rectangle,fill=green!50!black!25,draw=symbols,inner sep=0pt,minimum size=1.1mm},
	xiesf/.style={very thin,rectangle,fill=green!50!black,draw=symbols,inner sep=0pt,minimum size=1.1mm},
        xix/.style={very thin,crosscircle,fill=symbols!10!pagebackground,draw=symbols,inner sep=0pt,minimum size=1.2mm},
        X/.style={very thin,cross,rectangle,fill=pagebackground,draw=symbols,inner sep=0pt,minimum size=1.2mm},
	xib/.style={thin,circle,fill=symbols!10!pagebackground,draw=symbols,inner sep=0pt,minimum size=1.6mm},
	xie/.style={thin,circle,fill=green!50!black,draw=symbols,inner sep=0pt,minimum size=1.6mm},
	xid/.style={thin,circle,fill=symbols,draw=symbols,inner sep=0pt,minimum size=1.6mm},
	xibx/.style={thin,crosscircle,fill=symbols!10!pagebackground,draw=symbols,inner sep=0pt,minimum size=1.6mm},
	kernels2/.style={very thick,draw=connection,segment length=12pt},
	keps/.style={thin,draw=symbols,->},
	kepspr/.style={thick,draw=connection,->},
	krho/.style={thin,draw=symbols,superdense,->},
	krhopr/.style={thick,draw=connection,superdense,->},
	triangle/.style = { regular polygon, regular polygon sides=3},
	not/.style={thin,circle,draw=connection,fill=connection,inner sep=0pt,minimum size=0.5mm},
	diff/.style = {very thin,draw=symbols,triangle,fill=red!50!black,inner sep=0pt,minimum size=1.6mm},
	diff1/.style = {very thin,dectriangle={1}{0},fill=red!50!black,draw=symbols,inner sep=0pt,minimum size=1.6mm},
	diff2/.style = {very thin,dectriangle={1}{1},fill=red!50!black,draw=symbols,inner sep=0pt,minimum size=1.6mm},
		diffmini/.style = {very thin,rectangle,fill=black,draw=black,inner sep=0pt,minimum size=0.75mm},
	 kernelsmod/.style={very thick,draw=connection,segment length=12pt},
	 rec/.style = {very thin,rectangle,fill=black,draw=black,inner sep=0pt,minimum size=2mm},
	cerc/.style={very thin,circle,draw=black,fill=symbols,inner sep=0pt,minimum size=2mm},
	trup/.style={very thin,regular polygon,regular polygon sides=3,shape border rotate=180,draw=black,fill=red,inner sep=0pt,minimum size=3mm},
	stars/.style={very thin,star,star points=6,star point ratio=0.5, draw=black,fill=red,inner sep=0pt,minimum size=0.7mm},
	>=stealth,
        }

\makeatletter
\def\DeclareSymbol#1#2#3{%
	\expandafter\gdef\csname MH@symb@#1\endcsname{\tikzsetnextfilename{symbol#1}%
	\tikz[baseline=#2,scale=0.15,draw=symbols,line join=round]{#3}}%
	\expandafter\gdef\csname MH@symb@#1s\endcsname{\scalebox{0.75}{\tikzsetnextfilename{symbol#1}%
	\tikz[baseline=#2,scale=0.15,draw=symbols,line join=round]{#3}}}%
	\expandafter\gdef\csname MH@symb@#1ss\endcsname{\scalebox{0.65}{\tikzsetnextfilename{symbol#1}%
	\tikz[baseline=#2,scale=0.15,draw=symbols,line join=round]{#3}}}%
	}
\def\<#1>{\ifthenelse{\boolean{mmode}}{\mathchoice{\csname MH@symb@#1\endcsname}{\csname MH@symb@#1\endcsname}{\csname MH@symb@#1s\endcsname}{\csname MH@symb@#1ss\endcsname}}{\csname MH@symb@#1\endcsname}}
\makeatother

\DeclareSymbol{Xi22}{0.5}{\draw (0,0) node[xi] {} -- (-1,1) node[not] {} -- (0,2) node[xi] {};} 
\DeclareSymbol{Xi2}{-2}{\draw (-1,-0.25) node[xi] {} -- (0,1) node[xi] {};} 
\DeclareSymbol{Xi2b}{-2}{\draw (-1,-0.25) node[xic] {} -- (0,1) node[xic] {};} 
\DeclareSymbol{Xi2g}{-2}{\draw (-1,-0.25) node[xies] {} -- (0,1) node[xi] {};} 
\DeclareSymbol{Xi2g2}{-2}{\draw (-1,-0.25) node[xi] {} -- (0,1) node[xies] {};} 
\DeclareSymbol{cXi2}{-2}{\draw (0,-0.25) node[xi] {} -- (-1,1) node[xic] {};}
\DeclareSymbol{Xi3}{0}{\draw (0,0) node[xi] {} -- (-1,1) node[xi] {} -- (0,2) node[xi] {};}
\DeclareSymbol{XiIIXi}{0}{\draw (0,0) node[xi] {} -- (-1,1); \draw[kernels2] (-1,1) node[not] {} -- (0,2) node[xi] {};}

\DeclareSymbol{Xi4}{2}{\draw (0,0) node[xi] {} -- (-1,1) node[xi] {} -- (0,2) node[xi] {} -- (-1,3) node[xi] {};}
\DeclareSymbol{Xi4_1}{2}{\draw (0,0) node[xic] {} -- (-1,1) node[xic] {} -- (0,2) node[xi] {} -- (-1,3) node[xi] {};}
\DeclareSymbol{Xi4_2}{2}{\draw (0,0) node[xic] {} -- (-1,1) node[xi] {} -- (0,2) node[xi] {} -- (-1,3) node[xic] {};}
\DeclareSymbol{Xi2X}{-2}{\draw (0,-0.25) node[xi] {} -- (-1,1) node[xix] {};}
\DeclareSymbol{XXi2}{-2}{\draw (0,-0.25) node[xix] {} -- (-1,1) node[xi] {};}
\DeclareSymbol{IIXi}{0}{\draw (0,-0.25) node[not] {} -- (-1,1) node[xi] {} -- (0,2) node[xi] {};}
\DeclareSymbol{IXi^2}{-1}{\draw (-1,1) node[xi] {} -- (0,0) node[not] {} -- (1,1) node[xi] {};}
\DeclareSymbol{IIXi^2}{-4}{\draw (0,-1.5) node[not] {} -- (0,0);
\draw[kernels2] (-1,1) node[xi] {} -- (0,0) node[not] {} -- (1,1) node[xi] {};}
\DeclareSymbol{XiX}{-2.8}{\node[xibx] {};}
\DeclareSymbol{tauX}{-2.8}{ \node[X] {};}
\DeclareSymbol{Xi}{-2.8}{\node[xib] {};}

\DeclareSymbol{IXiX}{-1}{\draw (0,-0.25) node[not] {} -- (-1,1) node[xix] {};}
\DeclareSymbol{IXi3}{2}{\draw (0,-0.25) node[not] {} -- (-1,1) node[xi] {} -- (0,2) node[xi] {} -- (-1,3) node[xi] {};}
\DeclareSymbol{IXi}{-2}{\draw (0,-0.25) node[not] {} -- (-1,1) node[xi] {};}
\DeclareSymbol{XiI}{-2}{\draw (0,-0.25) node[xi] {} -- (-1,1) node[not] {};}

\DeclareSymbol{Xi4b}{0}{\draw(0,1.5) node[xi] {} -- (0,0); \draw (-1,1) node[xi] {} -- (0,0) node[xi] {} -- (1,1) node[xi] {};}
\DeclareSymbol{Xi4b'}{0}{\draw(0,1.5) node[xi] {} -- (0,-0.2); \draw (-1,1) node[xi] {} -- (0,-0.2) node[not] {} -- (1,1) node[xi] {};}
\DeclareSymbol{Xi4c}{0}{\draw (0,1) -- (0.8,2.2) node[xi] {};\draw (0,-0.25) node[xi] {} -- (0,1) node[xi] {} -- (-0.8,2.2) node[xi] {};}
\DeclareSymbol{Xi4d}{-4.5}{\draw (0,-1.5) node[not] {} -- (0,0); \draw (-1,1) node[xi] {} -- (0,0) node[xi] {} -- (1,1) node[xi] {};}
\DeclareSymbol{Xi4e}{0}{\draw (0,2) node[xi] {} -- (-1,1) node[xi] {} -- (0,0) node[xi] {} -- (1,1) node[xi] {};}
\DeclareSymbol{Xi4e'}{0}{\draw (0,2) node[xi] {} -- (-1,1) node[xi] {} -- (0,-0.2) node[not] {} -- (1,1) node[xi] {};}

\DeclareSymbol{Xitwo}
{0}{\draw[kernels2] (0,0) node[not] {} -- (-1,1) node[not] {}
-- (-2,2) node[not]{} -- (-3,3) node[xi]  {};
\draw[kernels2] (0,0) -- (1,1) node[xi] {};
\draw[kernels2] (-1,1) -- (0,2) node[xi] {};
\draw[kernels2] (-2,2) -- (-1,3) node[xi] {};}

\DeclareSymbol{IXitwo}
{0}{\draw (-.7,1.2) node[xi] {} -- (0,-0.2) -- (.7,1.2) node[xi] {};}
\DeclareSymbol{I1Xitwo}
{0}{\draw[kernels2] (0,0) node[not] {} -- (-1,1) node[xi] {};
\draw[kernels2] (0,0) -- (1,1) node[xi] {};}

\DeclareSymbol{I1Xitwobis}
{0}{\draw[kernels2] (0,0) node[not] {} -- (-1,1) node[xies] {};
\draw[kernels2] (0,0) -- (1,1) node[xies] {};}

\DeclareSymbol{I1Xitwog}
{0}{\draw[kernels2] (0,0) node[not] {} -- (-1,1) node[xies] {};
\draw[kernels2] (0,0) -- (1,1) node[xi] {};}

\DeclareSymbol{cI1Xitwo}
{0}{\draw[kernels2] (0,0) node[not] {} -- (-1,1) node[xic] {};
\draw[kernels2] (0,0) -- (1,1) node[xi] {};}

\DeclareSymbol{I1IXi3}{0}{\draw (0,0) node[xi] {} -- (-1,1) ; 
\draw[kernels2] (-1,1) node[not] {} -- (0,2) node[xi] {};
\draw[kernels2] (-1,1) node[not] {} -- (-2,2) node[xi] {};}

\DeclareSymbol{I1Xi3c}{-1}{\draw[kernels2](0,1.5) node[xi] {} -- (0,0) node[not] {}; \draw (-1,1) node[xi] {} -- (0,0) ; \draw[kernels2] (0,0) -- (1,1) node[xi] {};}

\DeclareSymbol{I1Xi3cbis}{-1}{\draw[kernels2](0,1.5) node[xies] {} -- (0,0) node[not] {}; \draw (-1,1) node[xies] {} -- (0,0) ; \draw[kernels2] (0,0) -- (1,1) node[xies] {};}

\DeclareSymbol{I1IXi3b}{0}{\draw[kernels2] (0,0) node[not] {} -- (-1,1) ; \draw[kernels2] (0,0)   -- (1,1) node[xi] {} ;
\draw (-1,1) node[xi] {} -- (0,2) node[xi] {};
}

\DeclareSymbol{I1IXi3c}{0}{\draw[kernels2] (0,0) node[not] {} -- (-1,1) ; \draw[kernels2] (0,0)   -- (1,1) node[xi] {} ;
\draw[kernels2] (-1,1) node[not] {} -- (0,2) node[xi] {};
\draw[kernels2] (-1,1) node[not] {} -- (-2,2) node[xi] {};}

\DeclareSymbol{I1IXi3cbis}{0}{\draw[kernels2] (0,0) node[not] {} -- (-1,1) ; \draw[kernels2] (0,0)   -- (1,1) node[xies] {} ;
\draw[kernels2] (-1,1) node[not] {} -- (0,2) node[xies] {};
\draw[kernels2] (-1,1) node[not] {} -- (-2,2) node[xies] {};}

\DeclareSymbol{I1Xi}{0}{\draw[kernels2] (0,0) node[not] {} -- (-1,1)  node[xi] {} ;}

\DeclareSymbol{I1Xi4a}{2}{\draw[kernels2] (0,0) node[not] {} -- (-1,1) ; \draw[kernels2] (0,0) node[not] {} -- (1,1) node[xi] {} ;
\draw (-1,1) node[xi] {} -- (0,2) node[xi] {} -- (-1,3) node[xi] {};}

\DeclareSymbol{cI1Xi4a}{2}{\draw[kernels2] (0,0) node[not] {} -- (-1,1) ; \draw[kernels2] (0,0) node[not] {} -- (1,1) node[xic] {} ;
\draw (-1,1) node[xic] {} -- (0,2) node[xi] {} -- (-1,3) node[xi] {};}

\DeclareSymbol{I1Xi4b}{2}{\draw (0,0) node[xi] {} -- (-1,1) node[xi] {} -- (0,2) ; \draw[kernels2] (0,2) node[not] {} -- (-1,3) node[xi] {};\draw[kernels2] (0,2)  -- (1,3) node[xi] {};
}

\DeclareSymbol{cI1Xi4b}{2}{\draw (0,0) node[xic] {} -- (-1,1) node[xic] {} -- (0,2) ; \draw[kernels2] (0,2) node[not] {} -- (-1,3) node[xi] {};\draw[kernels2] (0,2)  -- (1,3) node[xi] {};
}

\DeclareSymbol{I1Xi4c}{2}{\draw (0,0) node[xi] {} -- (-1,1) node[not] {}; \draw[kernels2] (-1,1) -- (0,2) ; 
\draw[kernels2] (-1,1) -- (-2,2) node[xi] {} ;
\draw (0,2) node[xi] {} -- (-1,3) node[xi] {};}

\DeclareSymbol{cI1Xi4c}{2}{\draw (0,0) node[xic] {} -- (-1,1) node[not] {}; \draw[kernels2] (-1,1) -- (0,2) ; 
\draw[kernels2] (-1,1) -- (-2,2) node[xic] {} ;
\draw (0,2) node[xi] {} -- (-1,3) node[xi] {};}

\DeclareSymbol{I1Xi4ab}{2}{\draw[kernels2] (0,0) node[not] {} -- (-1,1) ; \draw[kernels2] (0,0) node[not] {} -- (1,1) node[xi] {};\draw (-1,1) node[xi] {} -- (0,2) ; \draw[kernels2] (0,2) node[not] {} -- (-1,3) node[xi] {};\draw[kernels2] (0,2)  -- (1,3) node[xi] {}; }

\DeclareSymbol{cI1Xi4ab}{2}{\draw[kernels2] (0,0) node[not] {} -- (-1,1) ; \draw[kernels2] (0,0) node[not] {} -- (1,1) node[xic] {};\draw (-1,1) node[xic] {} -- (0,2) ; \draw[kernels2] (0,2) node[not] {} -- (-1,3) node[xi] {};\draw[kernels2] (0,2)  -- (1,3) node[xi] {}; }

\DeclareSymbol{I1Xi4bc}{2}{\draw (0,0) node[xi] {} -- (-1,1) node[not] {}; \draw[kernels2] (-1,1) -- (0,2) ; 
\draw[kernels2] (-1,1) -- (-2,2) node[xi] {} ; \draw[kernels2] (0,2) node[not] {} -- (-1,3) node[xi] {};\draw[kernels2] (0,2)  -- (1,3) node[xi] {};
}

\DeclareSymbol{cI1Xi4bc}{2}{\draw (0,0) node[xic] {} -- (-1,1) node[not] {}; \draw[kernels2] (-1,1) -- (0,2) ; 
\draw[kernels2] (-1,1) -- (-2,2) node[xic] {} ; \draw[kernels2] (0,2) node[not] {} -- (-1,3) node[xi] {};\draw[kernels2] (0,2)  -- (1,3) node[xi] {};
}

\DeclareSymbol{I1Xi4abcc1}{2}{\draw[kernels2] (0,0) node[not] {} -- (-1,1) node[not] {}
-- (-2,2) node[not]{} -- (-3,3) node[xic]  {};
\draw[kernels2] (0,0) -- (1,1) node[xic] {};
\draw[kernels2] (-1,1) -- (0,2) node[xi] {};
\draw[kernels2] (-2,2) -- (-1,3) node[xi] {};
}

\DeclareSymbol{I1Xi4abcc1b}{2}{\draw[kernels2] (0,0) node[not] {} -- (-1,1) node[not] {}
-- (-2,2) node[not]{} -- (-3,3) node[xi]  {};
\draw[kernels2] (0,0) -- (1,1) node[xic] {};
\draw[kernels2] (-1,1) -- (0,2) node[xic] {};
\draw[kernels2] (-2,2) -- (-1,3) node[xi] {};
}

\DeclareSymbol{I1Xi4abcc2}{2}{\draw[kernels2] (0,0) node[not] {} -- (-1,1) node[not] {}
-- (-2,2) node[not]{} -- (-3,3) node[xic]  {};
\draw[kernels2] (0,0) -- (1,1) node[xi] {};
\draw[kernels2] (-1,1) -- (0,2) node[xi] {};
\draw[kernels2] (-2,2) -- (-1,3) node[xic] {};
}

\DeclareSymbol{I1Xi4ac}{2}{\draw[kernels2] (0,0) node[not] {} -- (-1,1) ; \draw[kernels2] (0,0) node[not] {} -- (1,1) node[xi] {}; 
\draw[kernels2] (-1,1) node[not] {} -- (0,2) ; 
\draw[kernels2] (-1,1) -- (-2,2) node[xi] {} ;
\draw (0,2) node[xi] {} -- (-1,3) node[xi] {};}

\DeclareSymbol{cI1Xi4ac}{2}{\draw[kernels2] (0,0) node[not] {} -- (-1,1) ; \draw[kernels2] (0,0) node[not] {} -- (1,1) node[xic] {}; 
\draw[kernels2] (-1,1) node[not] {} -- (0,2) ; 
\draw[kernels2] (-1,1) -- (-2,2) node[xic] {} ;
\draw (0,2) node[xi] {} -- (-1,3) node[xi] {};}

\DeclareSymbol{I1Xi4acc1}{2}{\draw[kernels2] (0,0) node[not] {} -- (-1,1) ; \draw[kernels2] (0,0) node[not] {} -- (1,1) node[xic] {}; 
\draw[kernels2] (-1,1) node[not] {} -- (0,2) ; 
\draw[kernels2] (-1,1) -- (-2,2) node[xi] {} ;
\draw (0,2) node[xic] {} -- (-1,3) node[xi] {};}

\DeclareSymbol{I1Xi4acc2}{2}{\draw[kernels2] (0,0) node[not] {} -- (-1,1) ; \draw[kernels2] (0,0) node[not] {} -- (1,1) node[xic] {}; 
\draw[kernels2] (-1,1) node[not] {} -- (0,2) ; 
\draw[kernels2] (-1,1) -- (-2,2) node[xi] {} ;
\draw (0,2) node[xi] {} -- (-1,3) node[xic] {};}

\DeclareSymbol{2I1Xi4}{2}{\draw[kernels2] (0,0) node[not] {} -- (-1,1) node[not] {};
\draw[kernels2] (0,0) -- (1,1) node[not] {};
\draw[kernels2] (-1,1) -- (-1.5,2.5) node[xi] {};
\draw[kernels2] (-1,1) -- (-0.5,2.5) node[xi] {};
\draw[kernels2] (1,1) -- (0.5,2.5) node[xi] {};
\draw[kernels2] (1,1) -- (1.5,2.5) node[xi] {};
}

\DeclareSymbol{2I1Xi4dis}{2}{\draw[kernels2] (0,0) node[not] {} -- (-1,1) node[not] {};
\draw[kernels2] (0,0) -- (1,1) node[not] {};
\draw[kernels2] (-1,1) -- (-1.5,2.5) node[xies] {};
\draw[kernels2] (-1,1) -- (-0.5,2.5) node[xies] {};
\draw[kernels2] (1,1) -- (0.5,2.5) node[xies] {};
\draw[kernels2] (1,1) -- (1.5,2.5) node[xies] {};
}

\DeclareSymbol{2I1Xi4c1}{2}{\draw[kernels2] (0,0) node[not] {} -- (-1,1) node[not] {};
\draw[kernels2] (0,0) -- (1,1) node[not] {};
\draw[kernels2] (-1,1) -- (-1.5,2.5) node[xic] {};
\draw[kernels2] (-1,1) -- (-0.5,2.5) node[xi] {};
\draw[kernels2] (1,1) -- (0.5,2.5) node[xic] {};
\draw[kernels2] (1,1) -- (1.5,2.5) node[xi] {};
}

\DeclareSymbol{2I1Xi4c2}{2}{\draw[kernels2] (0,0) node[not] {} -- (-1,1) node[not] {};
\draw[kernels2] (0,0) -- (1,1) node[not] {};
\draw[kernels2] (-1,1) -- (-1.5,2.5) node[xic] {};
\draw[kernels2] (-1,1) -- (-0.5,2.5) node[xic] {};
\draw[kernels2] (1,1) -- (0.5,2.5) node[xi] {};
\draw[kernels2] (1,1) -- (1.5,2.5) node[xi] {};
}

\DeclareSymbol{2I1Xi4b}{2}{\draw[kernels2] (0,0) node[not] {} -- (-1,1) ;
\draw[kernels2] (0,0) -- (1,1);
\draw (-1,1) node[xi] {} -- (-1,2.5) node[xi] {};
\draw (1,1)  node[xi] {} -- (1,2.5) node[xi] {};
}

\DeclareSymbol{2I1Xi4bb}{2}{\draw[kernels2] (0,0) node[not] {} -- (-1,1) ;
\draw[kernels2] (0,0) -- (1,1);
\draw (-1,1) node[xi] {} -- (-1,2.5) node[xiesf] {};
\draw (1,1)  node[xi] {} -- (1,2.5) node[xic] {};
}

\DeclareSymbol{2I1Xi4c}{2}{\draw[kernels2] (0,0) node[not] {} -- (-1,1);
\draw[kernels2] (0,0) -- (1,1) node[not] {};
\draw (-1,1)  node[xi] {} -- (-1,2.5) node[xi] {};
\draw[kernels2] (1,1) -- (0.4,2.5) node[xi] {};
\draw[kernels2] (1,1) -- (1.6,2.5) node[xi] {};
}

\DeclareSymbol{2I1Xi4cc1}{2}{\draw[kernels2] (0,0) node[not] {} -- (-1,1);
\draw[kernels2] (0,0) -- (1,1) node[not] {};
\draw (-1,1)  node[xic] {} -- (-1,2.5) node[xi] {};
\draw[kernels2] (1,1) -- (0.4,2.5) node[xic] {};
\draw[kernels2] (1,1) -- (1.6,2.5) node[xi] {};
}

\DeclareSymbol{2I1Xi4cc2}{2}{\draw[kernels2] (0,0) node[not] {} -- (-1,1);
\draw[kernels2] (0,0) -- (1,1) node[not] {};
\draw (-1,1)  node[xic] {} -- (-1,2.5) node[xic] {};
\draw[kernels2] (1,1) -- (0.4,2.5) node[xi] {};
\draw[kernels2] (1,1) -- (1.6,2.5) node[xi] {};
}

\DeclareSymbol{Xi4ba}{0}{\draw(-0.5,1.5) node[xi] {} -- (0,0); \draw (-1.5,1) node[xi] {} -- (0,0) node[not] {}; \draw[kernels2] (0,0) -- (1.5,1) node[xi] {};
\draw[kernels2] (0,0) -- (0.5,1.5) node[xi] {} ;}

\DeclareSymbol{Xi4badis}{0}{\draw(-0.5,1.5) node[xies] {} -- (0,0); \draw (-1.5,1) node[xies] {} -- (0,0) node[not] {}; \draw[kernels2] (0,0) -- (1.5,1) node[xies] {};
\draw[kernels2] (0,0) -- (0.5,1.5) node[xies] {} ;}

\DeclareSymbol{Xi4ba1}{0}{\draw(-0.5,1.5) node[xi] {} -- (0,0); \draw (-1.5,1) node[xi] {} -- (0,0) node[not] {}; \draw[kernels2] (0,0) -- (1.5,1) node[xic] {};
\draw[kernels2] (0,0) -- (0.5,1.5) node[xic] {} ;}

\DeclareSymbol{Xi4ba1b}{0}{\draw(-0.5,1.5) node[xic] {} -- (0,0); \draw (-1.5,1) node[xic] {} -- (0,0) node[not] {}; \draw[kernels2] (0,0) -- (1.5,1) node[xi] {};
\draw[kernels2] (0,0) -- (0.5,1.5) node[xi] {} ;}

\DeclareSymbol{Xi4ba1bdiff}{0}{\draw(-0.5,1.5) node[xic] {} -- (0,0); \draw (-1.5,1) node[xic] {} -- (0,0) node[not] {}; \draw (0,0) -- (1.5,1) node[xi] {};
\draw (0,0) -- (0.5,1.5) node[xi] {};
\draw(0,0) node[diff] {};}

\DeclareSymbol{Xi4ba1bb}{0}{\draw(-0.5,1.5) node[xic] {} -- (0,0); \draw (-1.5,1) node[xiesf] {} -- (0,0) node[not] {}; \draw[kernels2] (0,0) -- (1.5,1) node[xi] {};
\draw[kernels2] (0,0) -- (0.5,1.5) node[xi] {} ;}

\DeclareSymbol{Xi4ba2}{0}{\draw(-0.5,1.5) node[xi] {} -- (0,0); \draw (-1.5,1) node[xic] {} -- (0,0) node[not] {}; \draw[kernels2] (0,0) -- (1.5,1) node[xi] {};
\draw[kernels2] (0,0) -- (0.5,1.5) node[xic] {} ;}

\DeclareSymbol{Xi4ba2b}{0}{\draw(-0.5,1.5) node[xi] {} -- (0,0); \draw (-1.5,1) node[xic] {} -- (0,0) node[not] {}; \draw[kernels2] (0,0) -- (1.5,1) node[xi] {};
\draw[kernels2] (0,0) -- (0.5,1.5) node[xiesf] {} ;}

\DeclareSymbol{Xi4ca}{0}{\draw (0,1) -- (-1,2.2) node[xi] {};\draw (0,-0.25) node[xi] {} -- (0,1) ; \draw[kernels2] (0,1) node[not] {} -- (1,2.2) node[xi] {};
\draw[kernels2] (0,1) {} -- (0,2.7) node[xi] {};
}

\DeclareSymbol{Xi4cb}{0}{\draw (-1,1) -- (-2,2) node[xi] {};\draw[kernels2] (0,0)  -- (-1,1) node[xi] {} ; \draw[kernels2] (0,0) node[not] {} -- (1,1) node[xi] {} ; 
\draw (-1,1) node[xi] {} -- (0,2) node[xi] {};}

\DeclareSymbol{Xi4cbb}{0}{\draw (-1,1) -- (-2,2) node[xiesf] {};\draw[kernels2] (0,0)  -- (-1,1) node[xi] {} ; \draw[kernels2] (0,0) node[not] {} -- (1,1) node[xi] {} ; 
\draw (-1,1) node[xi] {} -- (0,2) node[xic] {};}

\DeclareSymbol{Xi4cbc1}{0}{\draw (-1,1) -- (-2,2) node[xic] {};\draw[kernels2] (0,0)  -- (-1,1) node[xic] {} ; \draw[kernels2] (0,0) node[not] {} -- (1,1) node[xi] {} ; 
\draw (-1,1) node[xic] {} -- (0,2) node[xi] {};}

\DeclareSymbol{Xi4cbc2}{0}{\draw (-1,1) -- (-2,2) node[xi] {};\draw[kernels2] (0,0)  -- (-1,1) node[xi] {} ; \draw[kernels2] (0,0) node[not] {} -- (1,1) node[xic] {} ; 
\draw (-1,1) node[xic] {} -- (0,2) node[xi] {};}

\DeclareSymbol{Xi4cab}{0}{\draw (-1,1) -- (-2,2) node[xi] {};\draw[kernels2] (0,0)  -- (-1,1); \draw[kernels2] (0,0) node[not] {} -- (1,1) node[xi] {} ; 
\draw[kernels2] (-1,1)  {} -- (0,2) node[xi] {};
\draw[kernels2] (-1,1) node[not] {} -- (-1,2.5) node[xi] {};
}

\DeclareSymbol{Xi4cabdis}{0}{\draw (-1,1) -- (-2,2) node[xies] {};\draw[kernels2] (0,0)  -- (-1,1); \draw[kernels2] (0,0) node[not] {} -- (1,1) node[xies] {} ; 
\draw[kernels2] (-1,1)  {} -- (0,2) node[xies] {};
\draw[kernels2] (-1,1) node[not] {} -- (-1,2.5) node[xies] {};
}

\DeclareSymbol{Xi4cabc1}{0}{\draw (-1,1) -- (-2,2) node[xi] {};\draw[kernels2] (0,0)  -- (-1,1); \draw[kernels2] (0,0) node[not] {} -- (1,1) node[xic] {} ; 
\draw[kernels2] (-1,1)  {} -- (0,2) node[xic] {};
\draw[kernels2] (-1,1) node[not] {} -- (-1,2.5) node[xi] {};
}

\DeclareSymbol{Xi4cabc2}{0}{\draw (-1,1) -- (-2,2) node[xic] {};\draw[kernels2] (0,0)  -- (-1,1); \draw[kernels2] (0,0) node[not] {} -- (1,1) node[xic] {} ; 
\draw[kernels2] (-1,1)  {} -- (0,2) node[xi] {};
\draw[kernels2] (-1,1) node[not] {} -- (-1,2.5) node[xi] {};
}

\DeclareSymbol{Xi4ea}{1.5}{\draw (-1,2.5) node[xi] {} -- (-1,1) node[xi] {} -- (0,0); 
 \draw[kernels2] (0,0)  -- (1,1) node[xi] {};
\draw[kernels2] (0,0) node[not] {} -- (0,1.5) node[xi] {}; }

\DeclareSymbol{Xi4eac1}{1.5}{\draw (-1,2.5) node[xic] {} -- (-1,1) node[xi] {} -- (0,0); 
 \draw[kernels2] (0,0)  -- (1,1) node[xic] {};
\draw[kernels2] (0,0) node[not] {} -- (0,1.5) node[xi] {}; }

\DeclareSymbol{Xi4eac1b}{1.5}{\draw (-1,2.5) node[xic] {} -- (-1,1) node[xi] {} -- (0,0); 
 \draw[kernels2] (0,0)  -- (1,1) node[xiesf] {};
\draw[kernels2] (0,0) node[not] {} -- (0,1.5) node[xi] {}; }

\DeclareSymbol{Xi4eac2}{1.5}{\draw (-1,2.5) node[xic] {} -- (-1,1) node[xic] {} -- (0,0); 
 \draw[kernels2] (0,0)  -- (1,1) node[xi] {};
\draw[kernels2] (0,0) node[not] {} -- (0,1.5) node[xi] {}; }

\DeclareSymbol{Xi4eact1}{1.5}{\draw (-1,2.5) node[xic] {} -- (-1,1) node[xi] {} -- (0,0); 
 \draw (0,0)  -- (1,1) node[xic] {};
\draw[rho] (0,0) node[not] {} -- (0,1.5) node[xi] {}; }

\DeclareSymbol{Xi4eact2}{1.5}{\draw[rho] (-1,2.5) node[xic] {} -- (-1,1) node[xi] {} -- (0,0); 
 \draw (0,0)  -- (1,1) node[xic] {};
\draw (0,0) node[not] {} -- (0,1.5) node[xi] {}; }

\DeclareSymbol{Xi4eabis}{1.5}{\draw (-1,2.5) node[xi] {} -- (-1,1) ; \draw[kernels2] (-1,1) node[xi] {} -- (0,0); 
 \draw (0,0)  -- (1,1) node[xi] {};
\draw[kernels2] (0,0) node[not] {} -- (0,1.5) node[xi] {}; }

\DeclareSymbol{Xi4eabisc1}{1.5}{\draw (-1,2.5) node[xic] {} -- (-1,1) ; \draw[kernels2] (-1,1) node[xi] {} -- (0,0); 
 \draw (0,0)  -- (1,1) node[xi] {};
\draw[kernels2] (0,0) node[not] {} -- (0,1.5) node[xic] {}; }

\DeclareSymbol{Xi4eabisc1b}{1.5}{\draw (-1,2.5) node[xic] {} -- (-1,1) ; \draw[kernels2] (-1,1) node[xi] {} -- (0,0); 
 \draw (0,0)  -- (1,1) node[xi] {};
\draw[kernels2] (0,0) node[not] {} -- (0,1.5) node[xiesf] {}; }

\DeclareSymbol{Xi4eabisc1bis}{1.5}{\draw (-1,2.5) node[xi] {} -- (-1,1) ; \draw[kernels2] (-1,1) node[xi] {} -- (0,0); 
 \draw (0,0)  -- (1,1) node[xi] {};
\draw[kernels2] (0,0) node[not] {} -- (0,1.5) node[xi] {};
\draw (-2,1) node[] {\tiny{$i$}};
\draw (-2,2.5) node[] {\tiny{$\ell$}};
\draw (2,1) node[] {\tiny{$k$}};
\draw (0,2.5) node[] {\tiny{$j$}};
 }

\DeclareSymbol{Xi4eabisc1tris}{1.5}{\draw (-1,2.5) node[xi] {} -- (-1,1) ; \draw[kernels2] (-1,1) node[xi] {} -- (0,0); 
 \draw (0,0)  -- (1,1) node[xi] {};
\draw[kernels2] (0,0) node[not] {} -- (0,1.5) node[xi] {};
\draw (-2,1) node[] {\tiny{i}};
\draw (-2,2.5) node[] {\tiny{j}};
\draw (2,1) node[] {\tiny{j}};
\draw (0,2.5) node[] {\tiny{i}};
 }

\DeclareSymbol{Xi4eabisc1quater}{1.5}{\draw (-1,2.5) node[xic] {} -- (-1,1) ; \draw[kernels2] (-1,1) node[xi] {} -- (0,0); 
 \draw (0,0)  -- (1,1) node[xic] {};
\draw[kernels2] (0,0) node[not] {} -- (0,1.5) node[xi] {};
 }

\DeclareSymbol{Xi4eabisc2}{1.5}{\draw (-1,2.5) node[xic] {} -- (-1,1) ; \draw[kernels2] (-1,1) node[xi] {} -- (0,0); 
 \draw (0,0)  -- (1,1) node[xic] {};
\draw[kernels2] (0,0) node[not] {} -- (0,1.5) node[xi] {}; }

\DeclareSymbol{Xi4eabisc2l}{1.5}{\draw (-1,2.5) node[xiesf] {} -- (-1,1) ; \draw[kernels2] (-1,1) node[xi] {} -- (0,0); 
 \draw (0,0)  -- (1,1) node[xic] {};
\draw[kernels2] (0,0) node[not] {} -- (0,1.5) node[xi] {}; }

\DeclareSymbol{Xi4eabisc2r}{1.5}{\draw (-1,2.5) node[xic] {} -- (-1,1) ; \draw[kernels2] (-1,1) node[xi] {} -- (0,0); 
 \draw (0,0)  -- (1,1) node[xiesf] {};
\draw[kernels2] (0,0) node[not] {} -- (0,1.5) node[xi] {}; }

\DeclareSymbol{Xi4eabisc3}{1.5}{\draw (-1,2.5) node[xic] {} -- (-1,1) ; \draw[kernels2] (-1,1) node[xic] {} -- (0,0); 
 \draw (0,0)  -- (1,1) node[xi] {};
\draw[kernels2] (0,0) node[not] {} -- (0,1.5) node[xi] {}; }

\DeclareSymbol{Xi4eb}{0}{
\draw[kernels2] (0,2) node[xi] {} -- (-1,1) ; \draw[kernels2] (-2,2)  node[xi] {} -- (-1,1) ; \draw (-1,1)  node[not] {} -- (0,0); 
 \draw (0,0) node[xi] {}  -- (1,1) node[xi] {};
}

\DeclareSymbol{Xi4eab}{1.5}{\draw[kernels2] (-1,2.5) node[xi] {} -- (-1,1) ; \draw[kernels2] (-2,2)  node[xi] {} -- (-1,1) ; \draw (-1,1)  node[not] {} -- (0,0); 
 \draw[kernels2] (0,0)  -- (1,1) node[xi] {};
\draw[kernels2] (0,0) node[not] {} -- (0,1.5) node[xi] {}; 
}

\DeclareSymbol{Xi4eabdis}{1.5}{\draw[kernels2] (-1,2.5) node[xies] {} -- (-1,1) ; \draw[kernels2] (-2,2)  node[xies] {} -- (-1,1) ; \draw (-1,1)  node[not] {} -- (0,0); 
 \draw[kernels2] (0,0)  -- (1,1) node[xies] {};
\draw[kernels2] (0,0) node[not] {} -- (0,1.5) node[xies] {}; 
}

\DeclareSymbol{Xi4eabc1}{1.5}{\draw[kernels2] (-1,2.5) node[xic] {} -- (-1,1) ; \draw[kernels2] (-2,2)  node[xi] {} -- (-1,1) ; \draw (-1,1)  node[not] {} -- (0,0); 
 \draw[kernels2] (0,0)  -- (1,1) node[xic] {};
\draw[kernels2] (0,0) node[not] {} -- (0,1.5) node[xi] {}; 
}

\DeclareSymbol{Xi4eabc2}{1.5}{\draw[kernels2] (-1,2.5) node[xi] {} -- (-1,1) ; \draw[kernels2] (-2,2)  node[xi] {} -- (-1,1) ; \draw (-1,1)  node[not] {} -- (0,0); 
 \draw[kernels2] (0,0)  -- (1,1) node[xic] {};
\draw[kernels2] (0,0) node[not] {} -- (0,1.5) node[xic] {}; 
}

\DeclareSymbol{Xi4eabbis}{1.5}{\draw[kernels2] (-1,2.5) node[xi] {} -- (-1,1) ; \draw[kernels2] (-2,2)  node[xi] {} -- (-1,1) ; \draw[kernels2] (-1,1)  node[not] {} -- (0,0); 
 \draw (0,0)  -- (1,1) node[xi] {};
\draw[kernels2] (0,0) node[not] {} -- (0,1.5) node[xi] {}; 
}

\DeclareSymbol{Xi4eabbisc1}{1.5}{\draw[kernels2] (-1,2.5) node[xic] {} -- (-1,1) ; \draw[kernels2] (-2,2)  node[xi] {} -- (-1,1) ; \draw[kernels2] (-1,1)  node[not] {} -- (0,0); 
 \draw (0,0)  -- (1,1) node[xic] {};
\draw[kernels2] (0,0) node[not] {} -- (0,1.5) node[xi] {}; 
}

\DeclareSymbol{Xi4eabbisc1perm}{1.5}{\draw[kernels2] (-1,2.5) node[xi] {} -- (-1,1) ; \draw[kernels2] (-2,2)  node[xic] {} -- (-1,1) ; \draw[kernels2] (-1,1)  node[not] {} -- (0,0); 
 \draw (0,0)  -- (1,1) node[xic] {};
\draw[kernels2] (0,0) node[not] {} -- (0,1.5) node[xi] {}; 
}

\DeclareSymbol{Xi4eabbisc2}{1.5}{\draw[kernels2] (-1,2.5) node[xi] {} -- (-1,1) ; \draw[kernels2] (-2,2)  node[xi] {} -- (-1,1) ; \draw[kernels2] (-1,1)  node[not] {} -- (0,0); 
 \draw (0,0)  -- (1,1) node[xic] {};
\draw[kernels2] (0,0) node[not] {} -- (0,1.5) node[xic] {}; 
}

\DeclareSymbol{Xi2cbis}{0}{\draw[kernels2] (0,1) -- (0.8,2.2) node[xi] {};\draw[kernels2] (0,-0.25) node[not] {} -- (0,1); \draw[kernels2] (0,1) node[not] {} -- (-0.8,2.2) node[xi] {};}

\DeclareSymbol{Xi2cbis1}{0}{\draw (0,1) -- (-0.8,2.2) node[xi] {};\draw[kernels2] (0,-0.25) node[not] {} -- (0,1) node[xi] {}; }

\DeclareSymbol{Xi2Xbis}{-2}{\draw[kernels2] (0,-0.25)  -- (-1,1) ; \draw (-1,1) node[xix] {};
\draw[kernels2] (0,-0.25) node[not] {} -- (1,1) node[xi] {};}

\DeclareSymbol{XXi2bis}{-2}{\draw[kernels2] (0,-0.25) -- (-1,1) node[xi] {};
\draw[kernels2] (0,-0.25) node[X] {} -- (1,1) node[xi] {};}

\DeclareSymbol{I1XiIXi}{0}{\draw[kernels2] (0,-0.25) -- (1,1) node[xi] {};
\draw (0,-0.25) node[not] {} -- (-1,1) node[xi] {};}

\DeclareSymbol{I1XiIXib}{0}{\draw  (0,-0.25) node[xi] {} -- (0,1) node[not] {};
\draw[kernels2] (0,1) -- (0,2.25) ; \draw (0,2.25) node[xi]{}; }

\DeclareSymbol{I1XiIXic}{0}{
\draw[kernels2] (0,0) -- (1,1) node[xi] {} ; 
\draw[kernels2] (0,0) node[not] {}  -- (-1,1) node[not] {} -- (0,2) node[xi] {};
}

\DeclareSymbol{thin}{1.4}{\draw[pagebackground] (-0.3,0) -- (0.3,0); \draw  (0,0) -- (0,2);}
\DeclareSymbol{thin2}{1.4}{\draw[pagebackground] (-0.3,0) -- (0.3,0); \draw[tinydots]  (0,0) -- (0,2);}

\DeclareSymbol{thick}{1.4}{\draw[pagebackground] (-0.3,0) -- (0.3,0); \draw[kernels2]  (0,0) -- (0,2);}
\DeclareSymbol{thick2}{1.4}{\draw[pagebackground] (-0.3,0) -- (0.3,0); \draw[kernels2,tinydots]  (0,0) -- (0,2);}

\DeclareSymbol{Xi4ind}{2}{\draw (0,0) node[xi,label={[label distance=-0.2em]right: \scriptsize  $ i $}]  { } -- (-1,1) node[xi,label={[label distance=-0.2em]left: \scriptsize  $ j $}] {} -- (0,2) node[xi,label={[label distance=-0.2em]right: \scriptsize  $ k $}] {} -- (-1,3) node[xi,label={[label distance=-0.2em]left: \scriptsize  $ \ell $}] {};}

\DeclareSymbol{Xi4c1}{2}{\draw (0,0) node[xic] {} -- (-1,1) node[xi] {} -- (0,2) node[xic] {} -- (-1,3) node[xi] {};} 
\DeclareSymbol{IXi2ex}{0}{\draw (0,-0.25) node[xie] {} -- (-1,1) node[xi] {} ; \draw (0,-0.25)-- (1,1) node[xi] {};}
\DeclareSymbol{IXi2ex1}{0}{\draw (0,-0.25) node[xie] {} -- (-1,1) node[xi] {} -- (0,2) node[xi] {};}

\DeclareSymbol{Xi4b1}{0}{\draw(0,1.5) node[xic] {} -- (0,0); \draw (-1,1) node[xic] {} -- (0,0) node[xi] {} -- (1,1) node[xi] {};}

\DeclareSymbol{Xi4ec1}{0}{\draw (0,2) node[xi] {} -- (-1,1) node[xic] {} -- (0,0) node[xic] {} -- (1,1) node[xi] {};}
\DeclareSymbol{Xi4ec2}{0}{\draw (0,2) node[xic] {} -- (-1,1) node[xi] {} -- (0,0) node[xic] {} -- (1,1) node[xi] {};}
\DeclareSymbol{Xi4ec3}{0}{\draw (0,2) node[xic] {} -- (-1,1) node[xic] {} -- (0,0) node[xi] {} -- (1,1) node[xi] {};}

\DeclareSymbol{I1Xi4ac1}{2}{\draw[kernels2] (0,0) node[not] {} -- (-1,1) ; \draw[kernels2] (0,0) node[not] {} -- (1,1) node[xic] {} ;
\draw (-1,1) node[xi] {} -- (0,2) node[xic] {} -- (-1,3) node[xi] {};}

\DeclareSymbol{I1Xi4ac2}{2}{\draw[kernels2] (0,0) node[not] {} -- (-1,1) ; \draw[kernels2] (0,0) node[not] {} -- (1,1) node[xic] {} ;
\draw (-1,1) node[xi] {} -- (0,2) node[xi] {} -- (-1,3) node[xic] {};}

\DeclareSymbol{I1Xi4bp}{2}{\draw (0,0) node[not] {} -- (-1,1) node[xi] {} -- (0,2) ; \draw[kernels2] (0,2) node[not] {} -- (-1,3) node[xi] {};\draw[kernels2] (0,2)  -- (1,3) node[xi] {};
}

\DeclareSymbol{I1Xi4bc1}{2}{\draw (0,0) node[xic] {} -- (-1,1) node[xi] {} -- (0,2) ; \draw[kernels2] (0,2) node[not] {} -- (-1,3) node[xi] {};\draw[kernels2] (0,2)  -- (1,3) node[xic] {};
}

\DeclareSymbol{I1Xi4bc2}{2}{\draw (0,0) node[xic] {} -- (-1,1) node[xi] {} -- (0,2) ; \draw[kernels2] (0,2) node[not] {} -- (-1,3) node[xic] {};\draw[kernels2] (0,2)  -- (1,3) node[xi] {};
}

\DeclareSymbol{I1Xi4cp}{2}{\draw (0,0) node[not] {} -- (-1,1) node[not] {}; \draw[kernels2] (-1,1) -- (0,2) ; 
\draw[kernels2] (-1,1) -- (-2,2) node[xi] {} ;
\draw (0,2) node[xi] {} -- (-1,3) node[xi] {};}

\DeclareSymbol{I1Xi4cc1}{2}{\draw (0,0) node[xic] {} -- (-1,1) node[not] {}; \draw[kernels2] (-1,1) -- (0,2) ; 
\draw[kernels2] (-1,1) -- (-2,2) node[xi] {} ;
\draw (0,2) node[xic] {} -- (-1,3) node[xi] {};}

\DeclareSymbol{I1Xi4cc2}{2}{\draw (0,0) node[xic] {} -- (-1,1) node[not] {}; \draw[kernels2] (-1,1) -- (0,2) ; 
\draw[kernels2] (-1,1) -- (-2,2) node[xi] {} ;
\draw (0,2) node[xi] {} -- (-1,3) node[xic] {};}

\DeclareSymbol{I1Xi4abc1}{2}{\draw[kernels2] (0,0) node[not] {} -- (-1,1) ; \draw[kernels2] (0,0) node[not] {} -- (1,1) node[xic] {};\draw (-1,1) node[xi] {} -- (0,2) ; \draw[kernels2] (0,2) node[not] {} -- (-1,3) node[xic] {};\draw[kernels2] (0,2)  -- (1,3) node[xi] {}; }

\DeclareSymbol{I1Xi4abc2}{2}{\draw[kernels2] (0,0) node[not] {} -- (-1,1) ; \draw[kernels2] (0,0) node[not] {} -- (1,1) node[xic] {};\draw (-1,1) node[xi] {} -- (0,2) ; \draw[kernels2] (0,2) node[not] {} -- (-1,3) node[xi] {};\draw[kernels2] (0,2)  -- (1,3) node[xic] {}; }

\DeclareSymbol{R1}{0}{\draw (-1,1) node[xi] {} -- (0,0) node[not] {};
\draw[kernels2] (0,1.5) node[xic] {} -- (0,0) -- (1,1) node[xic] {};}
\DeclareSymbol{R2}{0}{\draw (-1,1) node[xic] {} -- (0,0) node[not] {};
\draw[kernels2] (0,1.5)  {} -- (0,0) -- (1,1)  {};
\draw (0,1.5) node[xi] {};
\draw (1,1) node[xic] {};
}
\DeclareSymbol{R3}{1}{\draw[kernels2] (-1,1.5)  {} -- (0,0) node[not] {} -- (1,1.5);
\draw (-1,1.5) node[xi] {};
\draw[kernels2] (0,3) {} -- (1,1.5) -- (2,3)  {};
\draw  (0,3) node[xic] {} ;
\draw (2,3) node[xic] {};}
\DeclareSymbol{R4}{1}{\draw[kernels2] (-1,1.5) node[xic] {} -- (0,0) node[not] {} -- (1,1.5);
\draw[kernels2] (0,3) {} -- (1,1.5) -- (2,3) node[xic] {};
\draw (0,3) node[xi] {};}

\DeclareSymbol{I1Xi4bcp}{2}{\draw (0,0) node[not] {} -- (-1,1) node[not] {}; \draw[kernels2] (-1,1) -- (0,2) ; 
\draw[kernels2] (-1,1) -- (-2,2) node[xi] {} ; \draw[kernels2] (0,2) node[not] {} -- (-1,3) node[xi] {};\draw[kernels2] (0,2)  -- (1,3) node[xi] {};
}

\DeclareSymbol{I1Xi4bcc1}{2}{\draw (0,0) node[xic] {} -- (-1,1) node[not] {}; \draw[kernels2] (-1,1) -- (0,2) ; 
\draw[kernels2] (-1,1) -- (-2,2) node[xi] {} ; \draw[kernels2] (0,2) node[not] {} -- (-1,3) node[xi] {};\draw[kernels2] (0,2)  -- (1,3) node[xic] {};
}

\DeclareSymbol{I1Xi4bcc2}{2}{\draw (0,0) node[xic] {} -- (-1,1) node[not] {}; \draw[kernels2] (-1,1) -- (0,2) ; 
\draw[kernels2] (-1,1) -- (-2,2) node[xi] {} ; \draw[kernels2] (0,2) node[not] {} -- (-1,3) node[xic] {};\draw[kernels2] (0,2)  -- (1,3) node[xi] {};
} 

\DeclareSymbol{2I1Xi4bc1}{2}{\draw[kernels2] (0,0) node[not] {} -- (-1,1) ;
\draw[kernels2] (0,0) -- (1,1);
\draw (-1,1) node[xic] {} -- (-1,2.5) node[xi] {};
\draw (1,1)  node[xic] {} -- (1,2.5) node[xi] {};
}

\DeclareSymbol{2I1Xi4bc2}{2}{\draw[kernels2] (0,0) node[not] {} -- (-1,1) ;
\draw[kernels2] (0,0) -- (1,1);
\draw (-1,1) node[xi] {} -- (-1,2.5) node[xic] {};
\draw (1,1)  node[xic] {} -- (1,2.5) node[xi] {};
}

\DeclareSymbol{diff2I1Xi4bc2}{2}{\draw (0,0) node[diff] {} -- (-1,1) ;
\draw (0,0) -- (1,1);
\draw (-1,1) node[xi] {} -- (-1,2.5) node[xic] {};
\draw (1,1)  node[xic] {} -- (1,2.5) node[xi] {};
}

\DeclareSymbol{2I1Xi4bc3}{2}{\draw[kernels2] (0,0) node[not] {} -- (-1,1) ;
\draw[kernels2] (0,0) -- (1,1);
\draw (-1,1) node[xic] {} -- (-1,2.5) node[xic] {};
\draw (1,1)  node[xi] {} -- (1,2.5) node[xi] {};
}

\DeclareSymbol{Xi41}{0}{\draw (0,1) -- (0.8,2.2) node[xic] {};\draw (0,-0.25) node[xi] {} -- (0,1) node[xi] {} -- (-0.8,2.2) node[xic] {};} 

\DeclareSymbol{Xi42}{0}{\draw (0,1) -- (0.8,2.2) node[xi] {};\draw (0,-0.25) node[xic] {} -- (0,1) node[xi] {} -- (-0.8,2.2) node[xic] {};}

\DeclareSymbol{Xi4ca1}{0}{\draw (0,1) -- (-1,2.2) node[xic] {};\draw (0,-0.25) node[xi] {} -- (0,1) ; \draw[kernels2] (0,1) node[not] {} -- (1,2.2) node[xic] {};
\draw[kernels2] (0,1) {} -- (0,2.7) node[xi] {};
}

\DeclareSymbol{Xi4ca2}{0}{\draw (0,1) -- (-1,2.2) node[xi] {};\draw (0,-0.25) node[xi] {} -- (0,1) ; \draw[kernels2] (0,1) node[not] {} -- (1,2.2) node[xic] {};
\draw[kernels2] (0,1) {} -- (0,2.7) node[xic] {};
}

\DeclareSymbol{Xi4cap}{0}{\draw (0,1) -- (-1,2.2) node[xi] {};\draw (0,-0.25) node[not] {} -- (0,1) ; \draw[kernels2] (0,1) node[not] {} -- (1,2.2) node[xi] {};
\draw[kernels2] (0,1) {} -- (0,2.7) node[xi] {};
}

\DeclareSymbol{Xi3a}{0}{
 \draw (-1,1)  node[xi] {} -- (0,0); 
 \draw (0,0) node[xi] {}  -- (1,1) node[xi] {};
}

\DeclareSymbol{Xi4ebc1}{0}{
\draw[kernels2] (0,2) node[xi] {} -- (-1,1) ; \draw[kernels2] (-2,2)  node[xic] {} -- (-1,1) ; \draw (-1,1)  node[not] {} -- (0,0); 
 \draw (0,0) node[xic] {}  -- (1,1) node[xi] {};
}

\DeclareSymbol{Xi4ebc2}{0}{
\draw[kernels2] (0,2) node[xi] {} -- (-1,1) ; \draw[kernels2] (-2,2)  node[xi] {} -- (-1,1) ; \draw (-1,1)  node[not] {} -- (0,0); 
 \draw (0,0) node[xic] {}  -- (1,1) node[xic] {};
}

\DeclareSymbol{Xi2cbispex}{0}{\draw[kernels2] (0,1) -- (0.8,2.2) node[xi] {};\draw (0,-0.25) node[xie] {} -- (0,1); \draw[kernels2] (0,1) node[not] {} -- (-0.8,2.2) node[xi] {};}

\DeclareSymbol{Xi2cbis1p}{0}{\draw (0,1) -- (-0.8,2.2) node[xi] {};\draw (0,-0.25) node[not] {} -- (0,1) node[xi] {}; }

\DeclareSymbol{Xi2Xp}{-2}{\draw (0,-0.25) node[not] {} -- (-1,1) node[xix] {};} 

\DeclareSymbol{I1XiIXib}{0}{\draw  (0,-0.25) node[xi] {} -- (0,1) node[not] {};
\draw[kernels2] (0,1) -- (0,2.25) ; \draw (0,2.25) node[xi]{}; }

\DeclareSymbol{IXi2b}{0}{\draw  (0,-0.25) node[xi] {} -- (0,1) node[not] {};
\draw (0,1) -- (0,2.25) ; \draw (0,2.25) node[xi]{}; }

\DeclareSymbol{IXi2bex}{0}{\draw  (0,-0.25) node[xi] {} -- (0,1) node[xie] {};
\draw (0,1) -- (0,2.25) ; \draw (0,2.25) node[xi]{}; }

\def\sXi{\symbol{\Xi}} 
 \def\1{\mathbf{\symbol{1}}}

\def\one{\mathbf{1}}
\def\eps{\varepsilon}

\DeclareSymbol{diff}{0}{
\node at (0,0.5) [diff] {};
}

\DeclareSymbol{geo}{0}{
\draw (0,0) node[diff] {};
\draw (0.3,0) node[diff] {};
}

\DeclareSymbol{generic}{0}{
\node at (0,0.6) [xi] {};
}

\DeclareSymbol{g}{0}{
\draw (0,0.6) node[g] {};
}

\DeclareSymbol{Ito}{0}{
\draw (0,0.6) node[xies] {};
}

\DeclareSymbol{Itob}{0}{
\draw (0,0.6) node[xiesf] {};
}

\DeclareSymbol{greycirc}{0}{
\draw (0,0.3) node[xi] {};
}

\DeclareSymbol{not}{0}{
\draw (0,0.6) node[not] {};
\draw[tinydots] (0,0.6) circle (0.8);
}

\DeclareSymbol{genericb}{0}{
\node at (0,0.6) [xic] {};
}

\DeclareSymbol{bluecirc}{0}{
\draw (0,0.3) node[xic] {};
}

\DeclareSymbol{genericxix}{0}{
\draw (0,0.6) node[xix] {};
}

\DeclareSymbol{genericX}{0}{
\draw (0,0.6) node[X] {};
}

\DeclareSymbol{diffIto}{1}{
\draw  (0,2.5) -- (0,0) ;
\draw (0,-0.1) node[diff] {};
\draw (0,2.5) node[xies] {};
}
\DeclareSymbol{Itodiff}{2}{
\draw(0,2.9) -- (0,-0.2);
\draw (0,2.9) node[diff] {};
\draw (0,-0.1) node[xies] {};
}

\DeclareSymbol{diffgeneric}{1}{
\draw  (0,2.5) -- (0,0) ;
\draw (0,-0.1) node[diff] {};
\draw (0,2.5) node[xi] {};
}

\DeclareSymbol{genericdiff}{2}{
\draw(0,2.9) -- (0,-0.2);
\draw (0,2.9) node[diff] {};
\draw (0,-0.1) node[xi] {};
}

\DeclareSymbol{diffdot}{2}{
\draw  (0,3) -- (0,-0.1) ;
\draw (0,3) node[not] {};
\draw (0,-0.1) node[diff] {};
}

\DeclareSymbol{diffdotmini}{0}{
\draw  (0,0) -- (0,1.2) ;
\draw (0,1.2) node[not] {};
\draw (0,0) node[diffmini] {};
}

\DeclareSymbol{dotdiff}{2}{
\draw[kernelsmod]  (0,3) -- (0,-0.1) ;
\draw (0,3) node[diff] {};
\draw (0,-0.1) node[not] {};
}

\DeclareSymbol{dotdiff1}{2}{
\draw[kernelsmod]  (0,3) -- (0,-0.1) ;
\draw (0,3) node[diff1] {};
\draw (0,-0.1) node[not] {};
}

\DeclareSymbol{dotdiff1mini}{0}{
\draw[kernelsmod]  (0,1.2) -- (0,0) ;
\draw (0,1.2) node[diffmini] {};
\draw (0,0) node[not] {};
}

\DeclareSymbol{dotdiff2}{2}{
\draw (0,3) -- (0,-0.1) ;
\draw (0,3) node[diff] {};
\draw (0,-0.1) node[not] {};
}

\DeclareSymbol{dotdiff2mini}{0}{
\draw (0,1.2) -- (0,0) ;
\draw (0,1.2) node[diffmini] {};
\draw (0,0) node[not] {};
}

\DeclareSymbol{dotdiffstraight}{0}{
\draw  (0,3) -- (0,-0.1) ;
\draw (0,3) node[diff] {};
\draw (0,-0.1) node[not] {};
}

\DeclareSymbol{arbre1}{0}{
\draw  (0,0) -- (1.5,1.5) ;
\draw (1.5,1.5) node[not] {};
\draw (0,0) node[not] {};
}

\DeclareSymbol{arbre2}{0}{
\draw  (0,0) -- (1.5,1.5) ;
\draw[kernelsmod] (0,0) -- (-1.5,1.5);
\draw (1.5,1.5) node[not] {};
\draw (0,0) node[not] {};
\draw (-1.5,1.5) node[xi] {};
}

\DeclareSymbol{arbre3}{0}{
\draw  (0,0) -- (1.5,1.5) ;
\draw[kernelsmod] (1.5,1.5) -- (0,3);
\draw (0,0) node[not] {};
\draw (1.5,1.5) node[not] {};
\draw (0,3) node[xi] {};
}

\DeclareSymbol{treeeval}{0}{
\draw (0,0) -- (1,1);
\draw (0,0) node[xi] {};
\draw (1.25,1.25) node[xi] {};
\draw (-0.6,0.6) node[]{\tiny{$i$}};
\draw (0.65,1.85) node[]{\tiny{$j$}};
}

\DeclareSymbol{testeval}{0}{
\draw (0,0) -- (1,1);
\draw (0,0) -- (-1,1);
\draw (0,0) node[xi] {};
\draw (1.25,1.25) node[xi] {};
\draw (-1.25,1.25) node[xi] {};
\draw (-0.6,-0.6) node[]{\tiny{$i$}};
\draw (0.65,1.85) node[]{\tiny{$j$}};
\draw (-1.95,1.85) node[]{\tiny{$k$}};
}

\DeclareSymbol{treeeval2}{0}{
\draw[kernelsmod] (-0.25,-1) -- (1,0.5) ;
\draw[kernelsmod] (1,0.5) -- (-0.25,2);
\draw (1,0.5) node[diff2] {};
\draw (-0.25,-1) node[not] {};
\draw (-0.25,2) node[xi] {};
\draw (-0.6,1.2) node[]{\tiny{1}};
}

\DeclareSymbol{arbreact}{1}{
\draw (0,0) node[not] {};
\draw[kernelsmod] (0,0) -- (1,1);
\draw[kernelsmod] (0,0) -- (-1,1);
\draw (-1,1) node[xic] {};
\draw  (0,2) -- (1,1) ;
\draw (0,2) node[xic] {};
\draw (1,1) node[xi] {};
}

\DeclareSymbol{arbreact1}{0}{
\draw (0,-1.5) -- (0,0);
\draw[kernelsmod] (0,0) -- (1,1);
\draw[kernelsmod] (0,0) -- (-1,1);
\draw  (0,2) -- (1,1) ;
\draw (0,-1.5) node[diff] {};
\draw (0,0) node[not] {};
\draw (-1,1) node[xic] {};
\draw (0,2) node[xic] {};
\draw (1,1) node[xi] {};
}

\DeclareSymbol{arbreact2}{0}{
\draw (0,-0.75) -- (-1,0.5); 
\draw (0,-0.75) -- (1,0.5);
\draw (0,1.5) -- (1,0.5);
\draw (0,1.5) node[xic] {};
\draw (1,0.5) node[xi] {};
\draw (-1,0.5) node[xic] {};
\draw (0,-0.75) node[diff] {};
}

\DeclareSymbol{arbreact3}{0}{
\draw[kernelsmod] (0,-0.75) -- (-1,0.5); 
\draw[kernelsmod] (0,-0.75) -- (1,0.5);
\draw (0,1.75) -- (1,0.5);
\draw (2,1.75) -- (1,0.5);
\draw (0,1.75) node[xic] {};
\draw (1,0.5) node[diff] {};
\draw (-1,0.5) node[xic] {};
\draw (2,1.75) node[xi] {};
\draw (0,-0.75) node[not] {};
}

\DeclareSymbol{pre_im_I1Xitwo}{0}{
\draw[kernels2] (0,-0.3) node[not] {} -- (-0.6,0.7) ;
\draw[kernels2] (0,-0.3) -- (0.6,0.7);
\draw (0,0.9) node[g] {};
}

\DeclareSymbol{pre_im_cI1Xi4ab}{2}{
\draw[kernels2] (0,-1) node[not] {} -- (-0.6,0) ;
\draw[kernels2] (0,-1) -- (0.6,0);
\draw (0,0.2) node[g] {};
\draw (0,0.6) -- (0,1.5);
\draw[kernels2] (0,1.5) node[not] {} -- (-0.6,2.5) ;
\draw[kernels2] (0,1.5) -- (0.6,2.5);
\draw (0,2.7) node[g] {};
}

\DeclareSymbol{pre_im_I1Xi4acc2}{0}{
\draw[kernels2] (-1,-0.5) node[not] {} -- (-1.6,0.5) ;
\draw[kernels2] (-1,-0.5) -- (-0.4,0.5);
\draw[kernels2] (-1,-0.5) -- (0.2,-1.5) node[not] {} ;
\draw (-1,1.1) -- (-1,2);
\draw[kernels2] (0.2,-1.5) -- (0.2,2);
\draw (-1,0.7) node[g] {};
\draw (-0.3,2.2) node[g] {};
}

\DeclareSymbol{pre_im_I1Xi4abcc2}{2}{
\draw[kernels2] (0,-1) node[not] {} -- (-1,0) node[not] {};
\draw[kernels2] (-1,1.2) node[not] {} -- (-1,0);
\draw[kernels2] (-1,1.2) -- (-1.5,2.5);
\draw[kernels2] (-1,1.2) -- (-0.5,2.5);
\draw[kernels2] (-1,0) -- (0.7,2.5);
\draw[kernels2] (0,-1) -- (1.5,2.5);
\draw (-1,2.7) node[g] {};
\draw (1,2.7) node[g] {};
}

\DeclareSymbol{pre_im_2I1Xi4c1}{2}{
\draw[kernels2] (0,-0.5) node[not] {} -- (-1,0.5) node[not] {};
\draw[kernels2] (0,-0.5) -- (1,0.5) node[not] {};
\draw[kernels2] (-1,0.5) node[not] {}-- (-1.7,2);
\draw[kernels2]  (-1,2) -- (1,0.5);
\draw[kernels2] (-1,0.5) -- (1,2);
\draw[kernels2] (1,0.5) -- (1.7,2);
\draw (-1.2,2.2) node[g] {};
\draw (1.2,2.2) node[g] {};
}

\DeclareSymbol{pre_im_Xi4eabisc2}{2}{
\draw[kernels2] (1.2,-0.5) node[not] {} -- (-0.7,0.8) ;
\draw[kernels2] (1.2,-0.5) -- (0.4,0.8);
\draw (0,1.4)  -- (0,2.2);
\draw (1.2,2.2) -- (1.2,-0.6);
\draw (0,1) node[g] {};
\draw (0.6,2.4) node[g] {};
}

\DeclareSymbol{pre_im_Xi4eabisc22}{2}{
\draw (1.2,-0.5) node[not] {} -- (-0.7,0.8) ;
\draw[kernels2] (1.2,-0.5) -- (0.4,0.8);
\draw (0,1.4)  -- (0,2.2);
\draw[kernels2] (1.2,2.2) -- (1.2,-0.6);
\draw (0,1) node[g] {};
\draw (0.6,2.4) node[g] {};
}

\DeclareSymbol{pre_im_Xi4eabisc222}{2}{
\draw[kernels2] (0.4,-0.5) node[not] {} -- (-0.6,1) ;
\draw[kernels2] (1.2,0) -- (0.3,1);
\draw (0,1.1)  -- (0,2.5);
\draw[kernels2] (1.2,2.5) -- (1.2,0) node[not] {} -- (0.4,-0.6);
\draw (0,1.2) node[g] {};
\draw (0.6,2.5) node[g] {};
}

\DeclareSymbol{pre_im_Xi4eabc2}{2}{
\draw (0,-0.5) node[not] {} -- (-1,0.5) node[not] {};
\draw[kernels2] (-1,0.5) -- (-1.5,2);
\draw[kernels2] (0,-0.5)  -- (0.7,2);
\draw[kernels2] (-1,0.5) -- (-0.5,2);
\draw[kernels2] (0,-0.5) -- (1.5,2);
\draw (-1,2.2) node[g] {};
\draw (1,2.2) node[g] {};
}

\DeclareSymbol{pre_im_Xi4eabbisc2}{2}{
\draw[kernels2] (0,-0.5) node[not] {} -- (-1,0.5) node[not] {};
\draw[kernels2] (-1,0.5) -- (-1.5,2);
\draw[kernels2] (0,-0.5)  -- (0.7,2);
\draw[kernels2] (-1,0.5) -- (-0.5,2);
\draw (0,-0.5) -- (1.5,2);
\draw (-1,2.2) node[g] {};
\draw (1,2.2) node[g] {};
}

\DeclareSymbol{pre_im_I1Xi4abcc1}{2}{
\draw[kernels2] (0,-1) node[not] {} -- (-1,0) node[not] {};
\draw[kernels2] (0,1.1) node[not] {} -- (-1,0);
\draw[kernels2] (-1,0) -- (-1.5,2.5);
\draw[kernels2] (0,1.1) node[not] {} -- (-0.5,2.5);
\draw[kernels2] (0,1.1) -- (0.5,2.5);
\draw[kernels2] (0,-1) -- (1.5,2.5);
\draw (-1,2.7) node[g] {};
\draw (1,2.7) node[g] {};
}

\DeclareSymbol{pre_im_Xi4eabc1}{2}{
\draw (0,-0.5) node[not] {} -- (-1,0.5) node[not] {};
\draw[kernels2] (-1,0.5) -- (-1.5,2);
\draw[kernels2] (0,-0.5)  -- (-0.5,2);
\draw[kernels2] (-1,0.5) -- (0.8,2);
\draw[kernels2] (0,-0.5) -- (1.5,2);
\draw (-1,2.2) node[g] {};
\draw (1,2.2) node[g] {};
}

\DeclareSymbol{pre_im_Xi4ba1b}{2}{
\draw[kernels2] (0,0) node[not] {}  -- (1.8,1.5);
\draw[kernels2] (0,0) -- (0.8,1.5);
\draw (0,-0.1) -- (-1.8,1.5);
\draw (0,-0.1) -- (-0.8,1.5);
\draw (-1,1.7) node[g] {};
\draw (1,1.7) node[g] {};
}

\DeclareSymbol{pre_im_Xi4ba2}{2}{
\draw (0,-0.1) node[not] {}  -- (1.8,1.5);
\draw[kernels2] (0,0) -- (0.8,1.5);
\draw (0,-0.1) -- (-1.8,1.5);
\draw[kernels2] (0,0) -- (-0.8,1.5);
\draw (-1,1.7) node[g] {};
\draw (1,1.7) node[g] {};
}

\DeclareSymbol{pre_im_Xi4cabc2}{2}{
\draw[kernels2] (0,-0.5) node[not] {} -- (-1,0.5) node[not] {};
\draw[kernels2] (-1,0.5) -- (-1.5,2);
\draw (-1,0.5)  -- (0.7,2);
\draw[kernels2] (-1,0.5) -- (-0.5,2);
\draw[kernels2] (0,-0.5) -- (1.5,2);
\draw (-1,2.2) node[g] {};
\draw (1,2.2) node[g] {};
}

\DeclareSymbol{pre_im_Xi4cabc1}{2}{
\draw[kernels2] (0,-0.5) node[not] {} -- (-1,0.5) node[not] {};
\draw (-1,0.5) -- (-1.5,2);
\draw[kernels2] (-1,0.5)  -- (0.7,2);
\draw[kernels2] (-1,0.5) -- (-0.5,2);
\draw[kernels2] (0,-0.5) -- (1.5,2);
\draw (-1,2.2) node[g] {};
\draw (1,2.2) node[g] {};
}

\DeclareSymbol{pre_im_Xi4eabbisc1}{2}{
\draw[kernels2] (0,-0.5) node[not] {} -- (-1,0.5) node[not] {};
\draw[kernels2] (-1,0.5) -- (-1.5,2);
\draw[kernels2] (0,-0.5)  -- (-0.5,2);
\draw[kernels2] (-1,0.5) -- (0.8,2);
\draw (0,-0.5) -- (1.5,2);
\draw (-1,2.2) node[g] {};
\draw (1,2.2) node[g] {};
}

\DeclareSymbol{pre_im_1}{0}{
\draw[kernels2] (0,-0.5) node[not] {} -- (-0.6,0.5) ;
\draw[kernels2] (0,-0.5) -- (0.6,0.5);
\draw (0,1.1)  -- (-0.55,2);
\draw (0,1.1)  -- (0.55,2);
\draw (0,0.7) node[g] {};
\draw (0,2.2) node[g] {};
}

\DeclareSymbol{disconnect}{0}{
\draw[kernels2] (0,-0.5) node[not] {} -- (-0.6,0.5) ;
\draw[kernels2] (0,-0.5) -- (0.6,0.5);
\draw (-0.55,1.1)  -- (-0.55,2.3);
\draw (0.55,2.3) -- (0.55,1.5) -- (1.2,1.5) -- (1.2,3.5) -- (0.55,3.5) -- (0.55,2.7);
\draw (0,0.7) node[g] {};
\draw (0,2.5) node[g] {};
}

\DeclareSymbol{pre_im_2}{2}{\draw[kernels2] (0,0) node[not] {} -- (-1,1) node[not] {};
\draw[kernels2] (0,0) -- (1,1) node[not] {};
\draw[kernels2] (-1,1) -- (-1.5,2.5);
\draw[kernels2] (-1,1) -- (-0.5,2.5);
\draw[kernels2] (1,1) -- (0.5,2.5);
\draw[kernels2] (1,1) -- (1.5,2.5);
\draw (-1,2.7) node[g] {};
\draw (1,2.7) node[g] {};
}

\DeclareSymbol{CX_rec}{0}{
\draw [black] (-0.3,1) to (-0.3,-0.3);
\draw [black] (0.3,1) to (0.3,-0.3);
\draw [black] (-0.3,1) to (-0.3,2.3);
\draw [black] (0.3,1) to (0.3,2.3);
\draw (0,1) node[rec] {};
}

\DeclareSymbol{CX_cerc}{0}{
\draw [black] (0,1) to (0,-0.3);
\draw (0,1) node[cerc] {};
}

\DeclareMathAlphabet{\mathpzc}{OT1}{pzc}{m}{it}

\def\sXi{\symbol{\Xi}}

\let\d\partial
\let\eps\varepsilon
\def\tr{\mathop{\mathrm{tr}}\nolimits}
\def\arginf{\mathop{\mathrm{arginf}}}
\def\eqref#1{(\ref{#1})}
\def\DD{\mathscr{D}}
\def\MM{\mathscr{M}}
\def\TT{\mathscr{T}}

\def\Cut{\mathop{\mathrm{Cut}}\nolimits}

\def\Sym{\mathop{\mathrm{Sym}}\nolimits}
\def\Stab{\mathop{\mathrm{Stab}}\nolimits}

\def\Tr{\mathop{\mathrm{Tr}}\nolimits}

\def\Ito{{\text{\rm\tiny It\^o}}}

\def\acyc{{\text{\rm\tiny acyc}}}
\def\canon{{\text{\rm\tiny canon}}}
\def\both{{\text{\rm\tiny both}}}
\def\geo{{\text{\rm\tiny geo}}}
\def\flt{{\text{\rm\tiny flat}}}
\def\nice{{\text{\rm\tiny nice}}}

\def\inc{{\text{\rm\tiny in}}}
\def\out{{\text{\rm\tiny out}}}

\def\typetop{{\top}}
\def\typebot{{\perp}}

\def\swap{S}

\def\types{\{\text{\rm geo},\text{\rm It\^o}\}}

\makeatletter 
\newcommand*{\bigcdot}{}
\DeclareRobustCommand*{\bigcdot}{%
  \mathbin{\mathpalette\bigcdot@{}}%
}
\newcommand*{\bigcdot@scalefactor}{.5}
\newcommand*{\bigcdot@widthfactor}{1.15}
\newcommand*{\bigcdot@}[2]{%
  \sbox0{$#1\vcenter{}$}
  \sbox2{$#1\cdot\m@th$}%
  \hbox to \bigcdot@widthfactor\wd2{%
    \hfil
    \raise\ht0\hbox{%
      \scalebox{\bigcdot@scalefactor}{%
        \lower\ht0\hbox{$#1\bullet\m@th$}%
      }%
    }%
    \hfil
  }%
}
\makeatother
\def\actint{\mathbin{\hat{\bigcdot}}}

\def\act{\bigcdot}
\def\mult{\bigcdot}
\def\sat{{\sf sat}}

\def\BPHZ{\textnormal{\tiny \textsc{bphz}}}
\def\family{{\mathfrak{b}}}

\DeclareMathOperator{\range}{\mathrm{range}}

\def\Moll{\mathrm{Moll}}

\def\Out{\mathrm{Out}}
\def\Vert{\mathrm{Vert}}

\def\In{\mathrm{In}}
\def\Int{\mathrm{Intern}}
\def\type{\mathrm{Type}}
\def\cod{\mathop{\mathrm{cod}}\nolimits}
\def\dom{\mathop{\mathrm{dom}}\nolimits}
\def\Hom{\mathop{\mathrm{Hom}}\nolimits}

\def\gen#1#2{%
\newdimen\height%
\setbox0=\hbox{$#1#2$}%
\height=1.18\ht0%
\setlength{\tabcolsep}{1.5pt}%
\mathchoice{%
\tcbox{%
\renewcommand{\arraystretch}{0.5}%
\begin{tabular}[b]{r!{\color{boxcolor}\vrule}l}\rule{0pt}{\height}$#1$&$#2$\end{tabular}}}%
{\tcbox{%
\renewcommand{\arraystretch}{0.5}%
\begin{tabular}[b]{r!{\color{boxcolor}\vrule}l}\rule{0pt}{\height}$#1$&$#2$\end{tabular}}}%
{\tcbox{%
\renewcommand{\arraystretch}{0.3}\setlength{\tabcolsep}{1pt}%
\begin{tabular}{r!{\color{boxcolor}\vrule}l}$\scriptstyle#1$&$\scriptstyle#2$\end{tabular}}}%
{\tcbox{%
\renewcommand{\arraystretch}{0.3}%
\begin{tabular}{r!{\color{boxcolor}\vrule}l}$\scriptscriptstyle#1$&$\scriptscriptstyle#2$\end{tabular}}}%
}

\def\red{\<genericb>}

\tcbset
{colframe=boxcolor,colback=symbols!7!pagebackground,coltext=pageforeground,
fonttitle=\bfseries,nobeforeafter,center title,size=fbox,boxsep=1.5pt,
top=0mm,bottom=0mm,boxsep=0mm,tcbox raise base}

\def\two{{\<generic>\kern0.05em\<genericb>}}
\def\twoI{{\<Ito>\kern0.05em\<Itob>}}

\def\mail#1{\burlalt{#1}{mailto:#1}}

\begin{document}

\title{Geometric stochastic heat equations}
\author{Y. Bruned$^1$, F. Gabriel$^2$, M. Hairer$^3$, L. Zambotti$^4$}
\institute{University of Edinburgh\and 
\'Ecole Polytechnique F\'ed\'erale de Lausanne\and
Imperial College London\and
LPSM, Sorbonne Université, Université de Paris, CNRS,
Paris\\
Email:\ \begin{minipage}[t]{\linewidth}
\mail{Yvain.Bruned@ed.ac.uk}, \mail{franckr.gabriel@gmail.com},\\ 
\mail{m.hairer@imperial.ac.uk}, \mail{lorenzo.zambotti@upmc.fr}.
\end{minipage}}

\maketitle

\begin{abstract}
We consider a natural class of $\R^d$-valued one-dimensional stochastic PDEs driven by 
space-time white noise that is formally invariant under the action of the diffeomorphism 
group on $\R^d$. This class contains in particular the KPZ equation, the multiplicative 
stochastic heat equation, the additive stochastic heat equation, and rough Burgers-type 
equations. 
We exhibit a one-parameter family of solution theories with the following properties:
\begin{enumerate}
\item For all SPDEs in our class for which a solution was previously available, 
every solution in our family coincides with the previously constructed solution, whether
that was obtained using It\^o calculus (additive and multiplicative stochastic heat equation),
rough path theory (rough Burgers-type equations), or the Hopf-Cole transform (KPZ equation).
\item Every solution theory is equivariant under the action of the diffeomorphism group, i.e.\ 
identities obtained by formal calculations treating the noise as a smooth function are valid. 
\item Every solution theory satisfies an analogue of It\^o's isometry.
\item The counterterms leading to our solution theories vanish at points where the 
equation agrees to leading order with the additive stochastic heat equation.
\end{enumerate}
In particular, points 2 and 3 show that, surprisingly, our solution theories 
enjoy properties analogous to those holding for both the Stratonovich and It\^o interpretations 
of SDEs \textit{simultaneously}. For the natural noisy perturbation of the harmonic map 
flow with values in an arbitrary Riemannian manifold, we show that all these solution theories 
coincide. In particular, this allows us to conjecturally identify the process associated to the
Markov extension of the 
Dirichlet form corresponding to the $L^2$-gradient flow for the Brownian loop measure.\\[.4em]
\end{abstract}

\setcounter{tocdepth}{2}
\tableofcontents

\section{Introduction}

One of the main goals of the present article is to build ``the'' most natural stochastic process
taking values in the space of loops in a compact Riemannian manifold $\CM$\label{manifold page ref}, i.e.\ a random map
$u \colon S^1 \times \R_+ \to \CM$. For a candidate to qualify
for such an admittedly subjective designation, one would like it to be as ``simple'' as possible, all the while 
having as many nice properties as possible. In this article, we interpret this as follows.
\begin{enumerate}
\item The process $u$ should be specified by only using the Riemannian structure on $\CM$ and no
additional data. 
\item The process $u$ should have a purely local specification in the sense that it satisfies the space-time
Markov property.
\item It should be the unique process with these properties among a `large' class of processes.
\end{enumerate}
A natural way of constructing a candidate would be as follows. 
Let $\mu$ be the measure on $L_\CM = \CC(S^1,\CM)$ given by the law of a Brownian loop, i.e.\ the 
Markov process with generator
the Laplace-Beltrami operator on $\CM$, conditioned to return to its starting point at some fixed
time (say $1$). We fix its law at time $0$ to be the probability measure on $\CM$ with 
density proportional to $p_1(x,x)$, which guarantees that $\mu$ is invariant under 
rotations of the circle $S^1$.
One can then consider the Dirichlet form $\CE$ given 
for suitably smooth functions $\Phi\colon L_\CM \to \R$ by
\begin{equ}
\CE(\Phi,\Phi) = \int_{L_\CM} \int_{S^1} \scal{\nabla_x \Phi(u),\nabla_x \Phi(u)}_{u(x)}\,dx\,\mu(du)\;,
\end{equ}
where $\nabla_x \Phi(u) \in T_{u(x)}\CM$ denotes the functional gradient of $\Phi$
and $\scal{\cdot,\cdot}_u$ denotes the scalar product in $T_u\CM$.
Unfortunately, while it is possible to show that $\CE$ is regular and closable 
\cite{LoopsRoeckner}, so that we can construct a (possibly non-conservative) 
Markov process from it, the uniqueness of Markov extensions for $\CE$ is a hard open problem.
Accordingly, point~3 above fails to be verifiable with this approach.


Instead, we will directly make sense of the corresponding stochastic partial differential equation.
At the formal (highly non-rigorous!) level, $\CE$ is expected to represent the Langevin equation
for the measure $\mu$ which, again at a purely heuristic level (but see \cite{Maeda,Driver} for
rigorous interpretations of this identity), is given by
\begin{equ}[e:heuristic]
\mu(du) \propto e^{-H(u)}du\;,\quad H(u) = {1\over 2} \int_{S^1} g_{u(x)}(\d_x u,\d_x u)\,dx\;,
\end{equ}
where ``$du$'' denotes the non-existent ``Lebesgue measure'' on the space of all loops with values in $\CM$.

\begin{remark}
There are various results and heuristic calculations suggesting that the ``correct'' 
formal expression for the Brownian loop
measure should be given by \eqref{e:heuristic}, but with $H$ corrected by a suitable multiple of the 
integral of the scalar curvature along the trajectory $u$. We will revisit this point in more detail in
Remark~\ref{rem:twoparam} and Section~\ref{sec:conjecture}. To some extent, our results explain the
appearance of this term, as well as why there is no consensus in the literature as to the ``correct'' 
constant multiplying it.
\end{remark}

At this point we note that the $L^2$-gradient flow for $H$ (in the intrinsic $L^2$ metric
determined by $g$) is given by the classical curve-shortening
flow (a particular instance of Eells-Sampson's harmonic map flow \cite{Eells}) which, in local coordinates,
can be written as
\begin{equ}
\d_t u^\alpha = \d_x^2 u^\alpha + \Gamma^\alpha_{\beta\gamma}(u)\,\d_x u^\beta\d_x u^\gamma\;,
\end{equ}
where $\Gamma$ \label{gamma page ref} denotes the Christoffel symbols for the Levi-Civita connection and we use Einstein's convention
of summation over repeated indices.
Given that we obtained it as an $L^2$-gradient flow, the natural way of adding noise to 
this equation is to add white noise with a covariance structure that reflects the Riemannian structure 
of $\CM$. In other words, we would like to consider the $\CM$-valued SPDE formally given by
\begin{equ}[e:mainSPDE]
\d_t u^\alpha = \d_x^2 u^\alpha + \Gamma^\alpha_{\beta\gamma}(u)\,\d_x u^\beta\d_x u^\gamma
+ \sigma_i^\alpha(u)\,\xi_i\;,
\end{equ}
where the $\xi_i$ are i.i.d.\ space-time white noises and the $\sigma_i$ \label{sigma page ref} are a collection
of vector fields on $\CM$ that are related to the (inverse) metric tensor by the identity\label{g page ref}
\begin{equ}[e:defMetric]
\sigma_i^\alpha(u)\,\sigma_i^\beta(u) = g^{\alpha\beta}(u)\;.
\end{equ}
(Actually, the invariant measure for \eqref{e:mainSPDE} would, at least formally, be given by
\eqref{e:heuristic} without the factor ${1\over 2}$, which then corresponds to the Brownian loop
measure running at half of its natural speed. We stick to \eqref{e:mainSPDE} as written in order to avoid
having this extra factor $2$ appearing throughout the article.)
In \eqref{e:defMetric} and throughout the paper, we adopt the following convention: every time we have a product of two factors containing an index $i$, a summation over $i$ is implied. For instance:
\[
\sigma_i^\alpha\,\sigma_i^\beta:=\sum_i \sigma_i^\alpha\,\sigma_i^\beta, \quad
\sigma_i^\alpha\,dW_i := \sum_i \sigma_i^\alpha\,dW_i, \quad
\sigma_i^\beta \,\d_\beta\sigma_i^\alpha := \sum_i \sigma_i^\beta \,\d_\beta\sigma_i^\alpha. 
\]

In this article, we will mainly study systems of equations of the type \eqref{e:mainSPDE} without assuming
any relation between the functions $\Gamma^\alpha_{\beta\gamma}$ and $\sigma_i^\alpha$. However, we will
see that in the particular geometric case mentioned above additional cancellations take place (see Lemma \ref{lem:propSigmai} below)
which then allow us to assign in a canonical way one specific Markov process on loop space to every
equation of the type \eqref{e:mainSPDE} in a way that preserves the natural symmetries of this class
of equations and such that the counterterm added to a smoothened version of \eqref{e:mainSPDE}
is `minimal' in a precise sense (see the list of properties on Page~\pageref{properties page ref}). 
It turns out that with this choice of solution theory, we expect ``the'' natural process
we are after (the one that furthermore admits the Brownian loop measure as its invariant measure)
to be given not by \eqref{e:mainSPDE}, but by the equation with an additional term 
${1\over 8}\nabla R(u)$ added to the right hand side, where $R$ denotes the scalar curvature of $\CM$. 
More details are given in Conjecture~\ref{conj:invariant} below, which in fact covers a much larger class 
of equations. The only reason why part of our results is not stated as a theorem is that we are missing
an analogue of the result \cite{Ajay} in the context of discrete regularity structures \cite{ErhardHairerRegularity}, although
such a result is widely expected to hold (see for example \cite{Konstantin,AndersonKhalil} for special cases).

Note that \cite{Funaki} first studied a version of equation \eqref{e:mainSPDE} where the i.i.d.\ noises $\xi_i$ are white in time but coloured in space, and the stochastic integral is in the Stratonovich sense. The version of \eqref{e:mainSPDE} with space-time
white noise had to wait the invention of regularity structures before it could be properly understood and solved.
We also recall that a natural process on loop space leaving invariant the Brownian loop measure on
a Riemannian manifold was constructed by a number of authors in the nineties,
see for example \cite{Driver1,DriverRockner,Driver2,Stroock,Norris}. The process constructed there 
is different from ours since, in the particular case of $\CM = \R^d$, it would
correspond to the Ornstein-Uhlenbeck process from Malliavin calculus, rather than
the stochastic heat equation. In particular, it is not ``local'' in the sense that 
its driving noise has non-trivial spatial correlations. In terms of its interpretation as the Langevin
equation associated to \eqref{e:heuristic}, our construction corresponds to taking gradients
with respect to the intrinsic $L^2$ metric determined by the manifold, while previous results
corresponded to taking gradients in the  metric given by the $H^1$ norm.

\begin{figure}[h]
\begin{center}
\includegraphics[width=4.5cm]{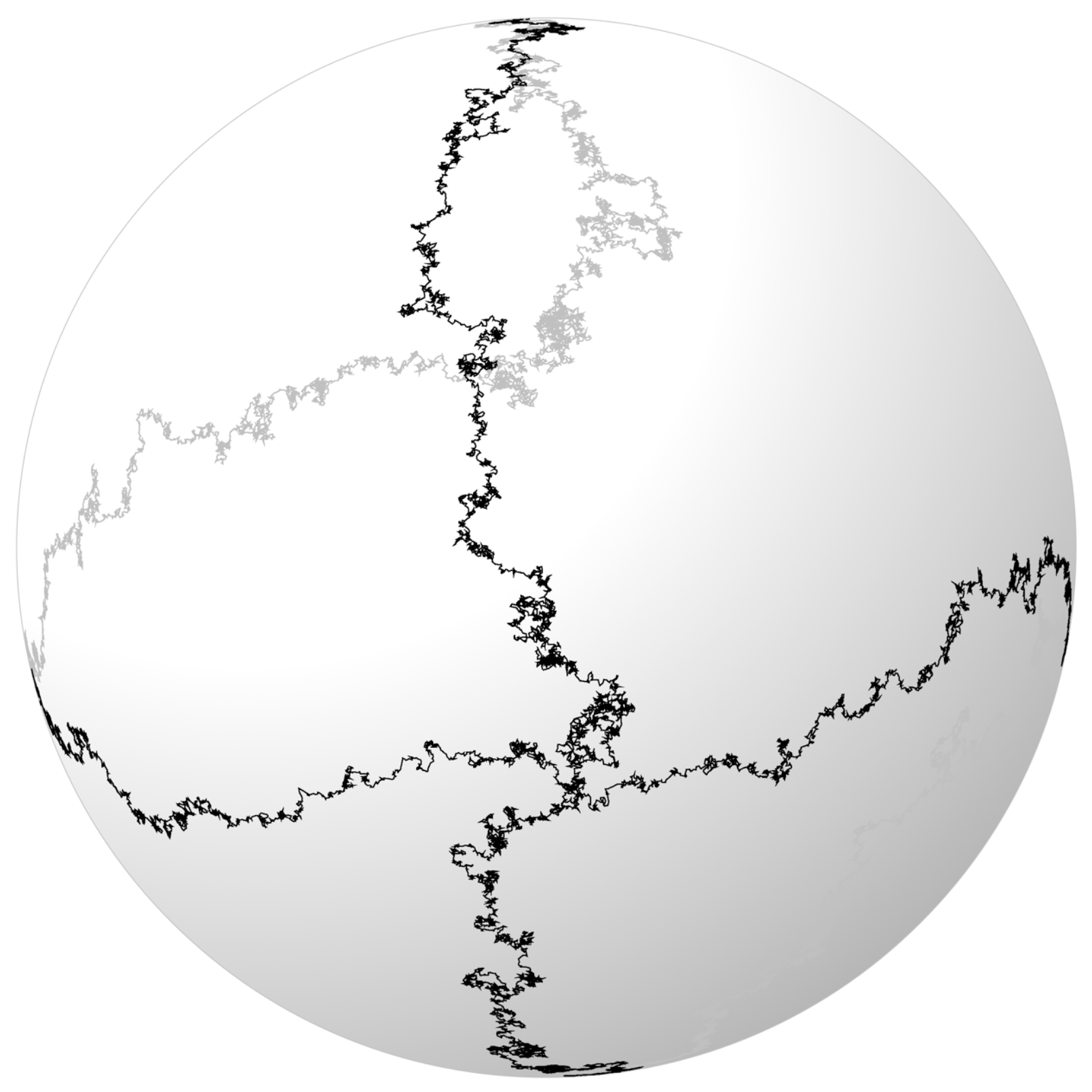}
\hspace{1cm}
\includegraphics[width=4.5cm]{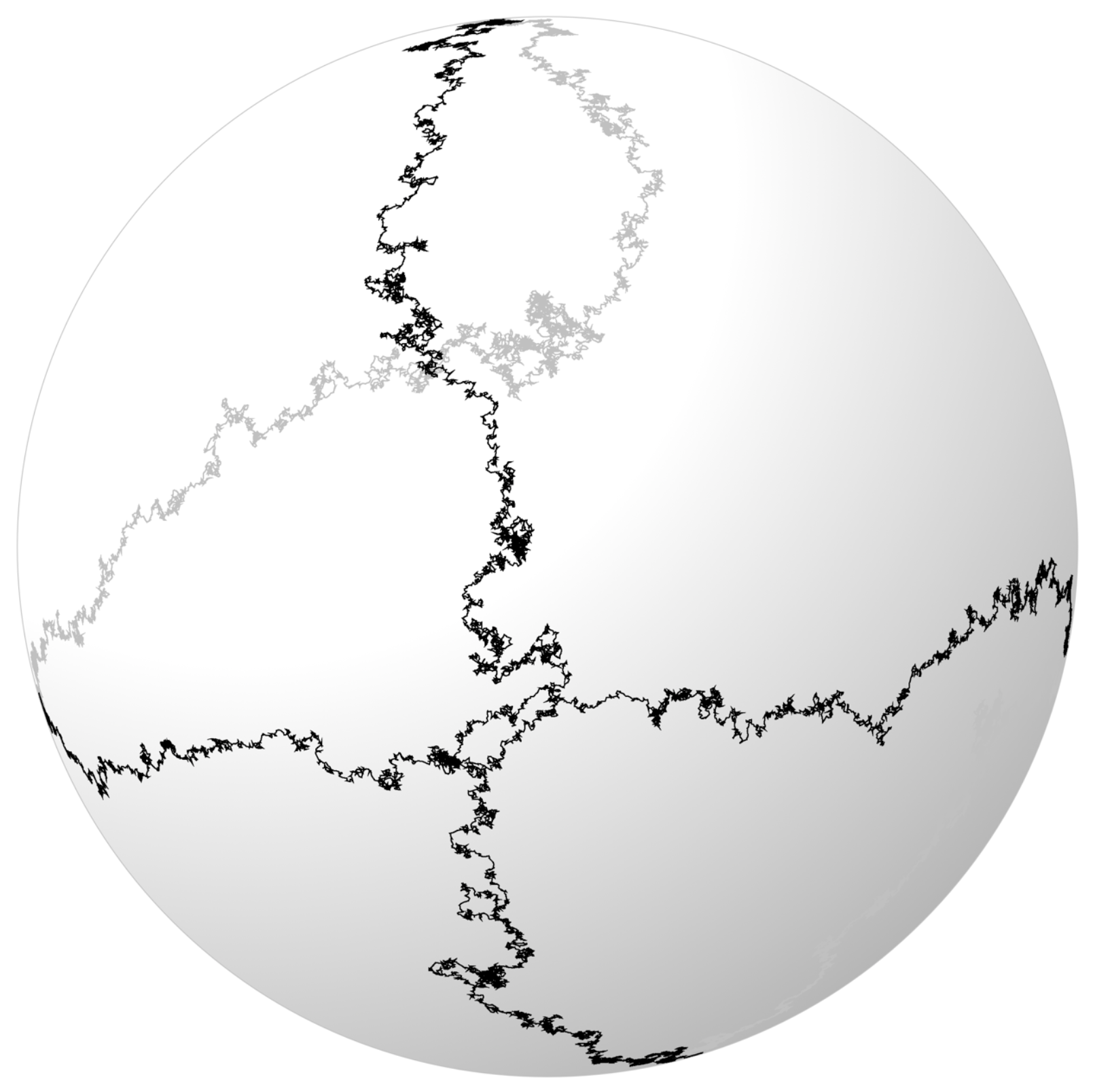}
\end{center}
\caption{Solution to \eqref{e:mainSPDE} on the sphere at two successive times. In this case, all processes in
the canonical family agree since the scalar curvature is constant. We see that the global structure hasn't changed
much, but the local structure is very different between the two times.}
\end{figure}

\subsection{Informal overview of main results}

The recently developed theory of regularity structures \cite{reg} provides a tool to give meaning
to equations of the type \eqref{e:mainSPDE} (see the series of works \cite{BHZ,Ajay,BCCH} for a `black box theorem'
that applies in this case), but instead of giving a single notion of solution it gives a canonical
\textit{family} of notions of solution parametrised by a suitable ``renormalisation group''. 
In the problem under consideration, this group is quite large: it can be identified with $\CS \approx (\R^{54},+)$, even
after taking into account simplifications arising from Gaussianity, the fact that the noises are i.i.d., 
and the $x \leftrightarrow -x$ 
symmetry. The general theory then in principle yields a $54$-dimensional family of candidate solution theories 
parametrised by $\CS$. Furthermore, this parametrisation is not canonical in general.

This should be contrasted with the case of SDEs with smooth coefficients where one has
a one-parameter family of solution theories (interpolating between solutions in the It\^o sense and solutions in
the Stratonovich sense) and there are two distinguished points on that line.
This can be formulated as follows. Write $U_c(\sigma,h)$ for the map sending initial conditions to the law of the 
(It\^o) solution to 
\begin{equ}[e:SDE]
dx^\alpha = h^\alpha(x)\,dt  + \sigma_i^\alpha(x)\,dW_i(t) + c \sigma_i^\beta(x) \,\d_\beta\sigma_i^\alpha(x)\,dt\;,
\end{equ}
with implicit summation over repeated indices and $x$ taking values in $\R^d$. 
The following is then well-known.
\begin{enumerate}
\item The solution theory $U^\Ito \eqdef U_0$ \label{U ito page ref} is the only solution theory (among this one-parameter family)
such that $U_c(\sigma,h) = U_c(\bar \sigma,h)$, whenever $\sigma_i^\alpha \sigma_i^\beta = \bar\sigma_i^\alpha \bar\sigma_i^\beta$.
We will refer to this property by saying that ``$U^\Ito$ satisfies It\^o's isometry''.
\item The solution theory $U^\geo \eqdef U_{1/2}$ \label{U geo page ref} is the only one such that its arguments transform
like vector fields under changes of coordinates.
We will refer to this property by saying that ``$U^\geo$ is equivariant under changes of coordinates''.
\end{enumerate}
Of course the parametrisation of the family $c \mapsto U_c$ of solution theories is completely arbitrary: we 
could have interpreted
\eqref{e:SDE} in the Stratonovich sense, in which case one would have $U^\Ito \eqdef U_{-1/2}$
and $U^\geo \eqdef U_{0}$.

Since the general theory of regularity structures is completely agnostic to the specific structure of our problem, this 
begs the question whether there are solution theories for \eqref{e:mainSPDE} that are equivariant under
 changes of coordinates or that satisfy It\^o's isometry in the above sense.
The goal of this article is to show the following  
\begin{enumerate}\label{properties page ref}
\item Among all natural solution theories for \eqref{e:mainSPDE} there is a $15$-dimensional family,
parametrised by a linear subspace of $\CS$ called $\CS_\geo$, which are all equivariant under changes of coordinates.
\item There is a $19$-dimensional family of solution theories, parametrised by a linear subspace of $\CS$ called $\CS_\Ito$,
which all satisfy It\^o's isometry.
\item There is a two-dimensional family of solution theories that are equivariant under changes of coordinates
and  satisfy It\^o's isometry \textit{simultaneously}. (See Remark~\ref{rem:twoparam} and Proposition~\ref{prop:ItoStratgeo}.)
\item There is a one-dimensional family $(U^\family)_{\family\in\R}$ of solution theories (let us call it the `canonical family', see \eqref{e:mainLimit}) 
which are furthermore obtained as limits to 
equations of the type
\begin{equ}
\d_t u^\alpha = \d_x^2 u^\alpha + \Gamma^\alpha_{\beta\gamma}(u)\,\d_x u^\beta\d_x u^\gamma
+ \sigma_i^\alpha(u)\,\xi_i^{(\eps)} + H_\eps^\alpha(u)\;,
\end{equ}
where the counterterm $H_\eps$ is guaranteed to satisfy $H_\eps(u) = 0$ whenever $u \in \R^d$ is such that 
$\Gamma(u) = 0$ and $D\sigma_i (u) = 0$, and $\xi_i^{(\eps)}:= \rho_\eps * \xi_i$ for an appropriate family of mollifiers, see below.
\item In the particular cases when $\Gamma$ determines the Levi-Civita connection or when the curvature 
tensor determined by $\Gamma$ vanishes, all solution theories in the canonical family do coincide.
\item In all special cases where a canonical notion of solution was previously known to exist, all  
solution theories in the canonical family do coincide with the previously constructed solution.
This includes in particular the KPZ equation (in which case we recover the Hopf-Cole solution)
and the case $\Gamma = 0$ in which case we recover the usual It\^o solutions.
\end{enumerate}

We insist again on the fact that in the classical case of 
finite-dimensional SDEs there does \textit{not} exist any natural notion of solution 
which is equivariant under changes of coordinates
and  satisfies It\^o's isometry simultaneously. The closest we can come to this in the case of SDEs 
would be to define $U(\sigma,h)$ as the solution to
\begin{equs}
dx^\alpha &= h^\alpha(x)\,dt  + \sigma_i^\alpha(x)\,dW_i(t) -{1\over 2}\Gamma^\alpha_{\beta\gamma}(x) \, g^{\beta\gamma}(x)\,dt\;, \\
dx &= h(x)\,dt  + \sigma_i(x)\circ dW_i(t) - {1\over 2} \nabla_{\sigma_i}\sigma_i\,dt\;,
\end{equs} 
where $\Gamma$ are the Christoffel symbols for the (inverse) metric $g^{\beta\gamma} = \sigma_i^\alpha \sigma_i^\beta$.
Indeed, it is a simple exercise to show that these two equations yield the same process. Furthermore, it 
follows immediately from the first expression that the law of this process only depends on $g$ and $h$
(rather than $\sigma$ and $h$), while it follows from the second expression that it is independent of the coordinate
system used to write the equation. However, there does not appear to exist any notion of stochastic integration $\star$
which is defined on a natural class of integrands and such that the process defined above solves 
$dx = h(x)\,dt  + \sigma_i(x)\star dW_i(t)$.

\subsection{Formulation of main result}

We now introduce some notation allowing us to formulate our main results.
We fix some ambient space $\R^d$ (even in the case where $\CM$ is an arbitrary compact manifold
we view it as a submanifold of $\R^d$ by Nash embedding) and we consider arbitrary smooth functions
$\Gamma^\alpha_{\beta\gamma},\sigma_i^\alpha:\R^d\to\R$ as above, the only constraint being
that $\Gamma^\alpha_{\beta\gamma}=\Gamma^\alpha_{\gamma\beta}$. Here and subsequently, 
Greek indices run over $\{1,\ldots,d\}$
while Roman indices run over $\{1,\ldots,m\}$, with $m$ being the number of driving noises.
We also fix a collection of smooth functions $h^\alpha$ and $K^\alpha_\beta$ and, for some arbitrary but fixed
$a \in (0,{1\over 2})$, we denote by $\CC^a_\star$ \label{space solution page ref} a suitable space of parabolic $a$-H\"older
continuous functions $u \colon \R_+\times S^1 \to \R^d$ with possible blow-up at finite time.
(See Section~\ref{sec:BPHZ} below for a precise definition.)
For $\eps > 0$ and $\rho$ a space-time mollifier
(compactly supported, integrating to $1$, and such that $\rho(t,-x)=\rho(t,x)$), we then denote by
$U_\eps^\geo(\Gamma,K,\sigma,h)$ the map from $\CC^a$ to the space of probability measures on $\CC^a_\star$ assigning to $u_0 \in \CC^a$ the law of the maximal solution to 
\begin{equ}[e:genClass]
\d_t u^\alpha = \d_x^2 u^\alpha + \Gamma^\alpha_{\beta\gamma}(u)\,\d_x u^\beta\d_x u^\gamma
+ K^\alpha_\beta(u)\,\d_x u^\beta
+h^\alpha(u) + \sigma_i^\alpha(u)\, \xi_i^{(\eps)}\;,
\end{equ}
where $u:\R_+\times S^1\to\R^d$, $\xi_i^{(\eps)} = \rho_\eps * \xi_i$ and $\rho_\eps(t,x) = \eps^{-3} \rho(t/\eps^2,x/\eps)$.
We henceforth denote by $\CB_\star^a$ \label{law space page ref} the space of continuous maps from $\CC^a(S^1,\R^d)$ into
the space of probability measures on $\CC_\star^a$.

In order to formulate our results, we first note that the class of equations of the 
type \eqref{e:genClass} is invariant under composition by diffeomorphisms in the following way.
We interpret $\Gamma$ as the Christoffel symbols for an arbitrary connection on $\R^d$ and,
for each $i$, the $(\sigma_i^\alpha)_\alpha$ as the components of a vector field on $\R^d$. 
Given a diffeomorphism $\phi$ of $\R^d$, we then act on connections $\Gamma$, vector fields $\sigma$ 
and $(1,1)$-tensors $K$ in the usual way by imposing that
\minilab{e:actions}
\begin{equs} \label{eq:actioncristo}
(\phi \act \Gamma)^\alpha_{\eta\zeta}(\phi(u)) \, \d_\beta \phi^\eta(u) \, \d_\gamma \phi^\zeta(u) &= 
\d_\mu \phi^\alpha(u) \, \Gamma^{\mu}_{\beta\gamma}(u) - \d^2_{\beta\gamma}\phi^\alpha(u)\;, \\
\label{eq:actionvect}
(\phi\act \sigma)^\alpha(\phi(u)) &= \d_\beta \phi^\alpha(u) \, \sigma^\beta(u)\;,\\
 (\phi\act K)^\alpha_\eta(\phi(u)) \, \d_\beta \phi^\eta(u)  &= \d_\mu \phi^\alpha(u) \, K^\mu_\beta(u)\;.\label{eq:actiontens}
\end{equs}
(One can verify that both of these do indeed describe left actions of the group of diffeomorphisms.) Similarly, given a map $U\colon u_0 \mapsto u$, we write $\phi\act U$
for the map that maps $\phi \circ u_0$ to $\phi \circ u$. With these notations, a simple
calculation shows that one has the following equivariance property of $U_\eps^\geo$ 
under the action of the diffeomorphism group:
\begin{equ}
\phi \act U_\eps^\geo(\Gamma,K,\sigma,h)
= U_\eps^\geo(\phi \act \Gamma,\phi \act K,\phi \act \sigma,\phi \act h)\;.
\end{equ}
Recall that the covariant derivative $\nabla_X Y$ \label{covariant derivative page ref} of a vector field $Y$ in the direction of 
another vector field $X$ is the vector field given by 
\begin{equ}[e:covDiff]
(\nabla_X Y)^\alpha (u) = X^\beta(u)\,\d_\beta Y^\alpha(u) + \Gamma^\alpha_{\beta\gamma}(u) \,X^\beta(u) \,Y^\gamma(u)\;.
\end{equ}
It is straightforward to verify that this definition satisfies
\begin{equ}
\phi \act (\nabla_X Y) =  (\phi \act \nabla)_{\phi \act X} (\phi \act Y)\;,
\end{equ}
where $\phi \act \nabla$ denotes the covariant differentiation built as in \eqref{e:covDiff}, but
with $\Gamma$ replaced by $\phi \act \Gamma$.

This allows us to build a number of different vector fields from 
$\Gamma$ and $\sigma$, like for example $\sum_i \nabla_{\sigma_i} \sigma_i$,
which we simply write as $\Nabla_{\<generic>}\<generic>$ (each circle denotes
an instance of $\sigma_i$, with different values of the index corresponding to different
colours and all indices being summed over). With this notation at hand, consider the
following collection of $14$ triple covariant derivatives:\label{VV page ref}
\begin{equs}
\label{eq:covderiv}
\VV =
\big\{ 
\Nabla_{\<generic>}\Nabla_{\<genericb>}&\Nabla_{\<genericb>}\<generic>,
\Nabla_{\<genericb>}\Nabla_{\<genericb>}\Nabla_{\<generic>}\<generic>, 
\Nabla_{\<genericb>}\Nabla_{\Nabla_{\<genericb>}\<generic>}\<generic>, 
\Nabla_{\Nabla_{\<genericb>}\<generic>} \Nabla_{\<genericb>}\<generic>,
 \Nabla_{\Nabla_{\<generic>}\<genericb>} \Nabla_{\<genericb>}\<generic>,
\Nabla_{\Nabla_{\<genericb>}\<genericb>}\Nabla_{\<generic>}\<generic>,
\Nabla_{\Nabla_{\<genericb>}\Nabla_{\<genericb>}\<generic>}\<generic>,\\
&\Nabla_{\Nabla_{\<genericb>}\Nabla_{\<generic>}\<generic>}\<genericb>,
\Nabla_{\Nabla_{\Nabla_{\<genericb>}\<generic>}\<generic>}\<genericb>,
\Nabla_{\Nabla_{\Nabla_{\<genericb>}\<generic>}\<genericb>}\<generic>,
\Nabla_{\<generic>}\Nabla_{\Nabla_{\<genericb>}\<generic>}\<genericb>,
\Nabla_{\Nabla_{\<generic>}\Nabla_{\<genericb>}\<generic>}\<genericb>,
\Nabla_{\Nabla_{\Nabla_{\<generic>}\<generic>}\<genericb>}\<genericb>,
\Nabla_{\<genericb>}\Nabla_{\Nabla_{\<generic>}\<generic>}\<genericb>\ 
\big\}.
\end{equs}
(The only element missing from the list is $\Nabla_{\<genericb>}\Nabla_{\<generic>}\Nabla_{\<genericb>}\<generic>$,
the reason being that it can be written as a linear combination of the $14$ other terms, see \eqref{relationV}
below.)
It will be convenient to view $\VV$ \label{CV page ref} as an `abstract' set of symbols, to write 
$\CV = \scal{\VV}$\label{V page ref} for the vector space
that it generates\footnote{We will stick as much as possible to the convention that the vector space generated by a set
denoted by a Gothic symbol is denoted by the corresponding calligraphic symbol.}, and to write $\Upsilon_{\Gamma,\sigma}\colon \CV \to \CC^\infty(\R^d,\R^d)$ for the map that 
turns each symbol into the corresponding vector field on $\R^d$,
so that for example $\Upsilon_{\Gamma,\sigma}
\Nabla_{\<genericb>}\Nabla_{\Nabla_{\<genericb>}\<generic>}\<generic>
= \sum_{i,j}\Nabla_{\sigma_i}\Nabla_{\Nabla_{\sigma_i}\sigma_j}\sigma_j$.

A special case of \eqref{e:genClass} is given by the case $\Gamma = 0$ and $\sigma$ constant, i.e.\ $D \sigma = 0$, in which case it reduces to the heat equation with additive noise which obviously requires no renormalisation. 
It is then natural to expect that if \eqref{e:genClass} is sufficiently ``close to'' the additive stochastic heat 
equation near some fixed point $u_0 \in \R^d$, then the counterterms required to give meaning to its solutions
vanish at $u_0$. This motivates the introduction of the space $\CV^\nice \subset \CV$ \label{V nice page ref} consisting of those
elements $\tau \in \CV$ such that, for all choices of $d$, $\Gamma$ and $\sigma$, if $u_0$ is a point such that 
$\Gamma(u_0) = 0$ and $D \sigma(u_0) = 0$, then $\bigl(\Upsilon_{\Gamma,\sigma} \tau\bigr)(u_0) = 0$.
We will see in Remark~\ref{rem:codimgeonice} below that $\CV^\nice$ is a $12$-dimensional subspace of $\CV$.

We write $H_{\Gamma,\sigma}$ \label{vector field H page ref}  for the vector field
\begin{equs}
H_{\Gamma,\sigma}^\alpha(u) &= R^\alpha_{\eta\beta\gamma}(u) \, \sigma_i^\beta(u) \left(\nabla_{\sigma_j}\sigma_i-2\nabla_{\sigma_i}\sigma_j\right)^\gamma(u) \, \sigma_j^\eta(u)\\
&= - R^\alpha_{\eta\beta\gamma}(u) \, g^{\beta \zeta}(u) \, (\nabla_\zeta g)^{\gamma\eta}(u)\;, \label{e:fieldCovar}
\end{equs}
see Lemma~\ref{lem:exprtaucstar},
where summation over $i,j$ is implicit, $g$ is given by \eqref{e:defMetric}, and the Riemannian 
curvature tensor $R$ \label{Riemannian curvature tensor page ref} is defined from $\Gamma$ as usual through the identity
\begin{equ}
R^\alpha_{\eta\beta\gamma} X^\beta Y^\gamma Z^\eta = 
\bigl(\nabla_X\nabla_Y Z - \nabla_Y\nabla_X Z - \nabla_{[X,Y]} Z\bigr)^\alpha
\end{equ}
for any three vector fields $X$, $Y$, $Z$ with $\nabla$ defined by \eqref{e:covDiff}. 
The vector field $H$ can be written as a linear combination of elements of $\VV$ and it turns out that this
linear combination belongs to $\CV^\nice$, so that 
it determines a one-dimensional linear subspace $\CV_\star\subset \CV^\nice$ \label{V star page ref} spanned by the element 
$\tau_\star \in \CV^\nice$\label{def taustar} such that
$H_{\Gamma,\sigma} = \Upsilon_{\Gamma,\sigma}\tau_\star$. We also choose an arbitrary complement $\CV_\star^\perp$ \label{V start perp page ref}
so that $\CV^\nice = \CV_\star \oplus \CV_\star^\perp$.
Incidentally, and this is the reason why this subspace is important, we will show in 
Proposition~\ref{prop:ItoStratgeo} below that $\CV_\star$ equals 
all of those elements $\tau \in\CV^\nice$ 
such that $\Upsilon_{\Gamma,\sigma}\tau$ is a vector field that furthermore only depends on the $\sigma_i$ through the 
inverse metric $g$ defined by \eqref{e:defMetric}.

Having introduced all of these auxiliary vector fields,
our main result can be stated as follows (parts of this result were already announced in \cite{proc}).
Throughout this article, we consider $a \in (0,{1\over 2})$ to be a fixed H\"older exponent.
We also write $\Moll$ \label{Moll page ref} for the set of all compactly supported functions $\rho \colon \R^2 \to \R$ integrating
to $1$, such that $\rho(t,-x) = \rho(t,x)$, and such that $\rho(t,x) = 0$ for $t \le 0$ (i.e.\ $\rho$ is
non-anticipative).
Furthermore, we call ``smooth data'' a choice of dimensions $d$ and $m$,
as well as a choice of smooth functions $\Gamma$, $K$, $h$ and $\sigma$ as above.

\begin{theorem}\label{theo:main}
For every mollifier $\rho\in\Moll$ there exists a \textit{unique} choice of constants $\bar c \in \R_+$ and
$c \in \CV_\star^\perp$, as well as a choice of $\hat c \in \R$, such that the following statements are true.
\begin{enumerate}
\item For every $\family \in \R$ and every smooth data, one has
\begin{equs}
\lim_{\eps \to 0} &U_\eps^\geo \Big(\Gamma,K,\sigma,h - {\bar c \over \eps} \Nabla_{\sigma_i}\sigma_i +  H_{\Gamma,\sigma} \Big(\family + \hat c + {\log \eps \over 4\sqrt 3 \pi}\Big)
+ \Upsilon_{\Gamma,\sigma} c \Big) \\
&= U^\family(\Gamma,K,\sigma,h)\;,\label{e:mainLimit}
\end{equs}
in $\CB_\star^a$, for some limit $U^\family$ independent of $\rho$. \label{family solutions page ref}
\item For every $\family \in \R$ and every smooth data,
the process defined by $U^\family(\Gamma,K,\sigma,h)$ is a Feller Markov process on 
$\CC^a(S^1,\R^d)$ with possibly finite explosion time. 
\item For every $\family \in \R$ and every diffeomorphism $\phi$ of $\R^d$, one has the change of variables formula
\begin{equ}[e:equivar]
\phi \act U^\family(\Gamma,K,\sigma,h) = U^\family(\phi \act\Gamma,\phi\act K,\phi \act\sigma,\phi \act h)\;,
\end{equ}
valid for all smooth data.
\item For every $\family \in \R$ and every smooth data, one has
\begin{equ}
U^\family(\Gamma,K,\sigma,h) = U^\family(\Gamma,K,\bar \sigma,h)\;,
\end{equ}
for every $\bar \sigma$ such that $\bar \sigma_i^\alpha \bar \sigma_i^\beta = \sigma_i^\alpha \sigma_i^\beta$.
(Implicit summation over $i$.)
\end{enumerate}
Furthermore, for every choice of $\family$, these solutions satisfy the following.
\begin{enumerate}\setcounter{enumi}{4}
\item In the special case $\Gamma = 0$ and $K=0$, $U^\family(0,0,\sigma,h)$ coincides with the mild solution
to the system of stochastic PDEs
\begin{equ}[e:classicalIto]
du^\alpha = \d_x^2 u^\alpha\,dt + h^\alpha(u)\,dt + \sigma_i^\alpha(u)\,dW_i\;,
\end{equ}
where the $W_i$ are independent $L^2(S^1)$-cylindrical Wiener processes.
\item In the special case $\Gamma = 0$, $d=m$, and $\sigma$ a constant multiple of the identity matrix, 
$U^\family(0,K,\sigma,h)$
coincides with the maximal solution constructed in \cite{CPA:CPA20383}.
\end{enumerate}
\end{theorem}

The case of a Riemannian manifold $\CM$ is so important that we state it as a separate result.
Given a collection $\{\sigma_i\}_{i=1}^m$ of vector fields on a Riemannian manifold $(\CM,g)$, we 
define $\Upsilon_\sigma \colon \CV \to \Gamma^\infty(T\CM)$ as before as the 
linear map from $\CV$ into the set of smooth vector fields on $\CM$ assigning a symbol to the corresponding
triple covariant derivative of the $\sigma$'s.
The slight difference is that this time we view these as vector fields on $\CM$ rather than $\R^n$
and we use the convention that the covariant differentiation is the one given by
the Levi-Civita connection on $\CM$. Writing $\Gamma_\sigma^\infty(T\CM) = \range \Upsilon_\sigma$,
the following result shows that in the natural geometric context the collection $\VV$ has
quite a lot of redundancies.

\begin{lemma}\label{lem:generality}
Let $(\CM,g)$ be a Riemannian manifold and let  
 $\{\sigma_i\}_{i=1}^m$
be a collection of smooth vector fields such that the tensor $\sum_{i=1}^m (\sigma_i \otimes \sigma_i)$
equals the inverse of the metric tensor and such that furthermore $\sum_{i=1}^m \nabla_{\sigma_i} \sigma_i = 0$.

Then, the space $\Gamma_\sigma^\infty(T\CM)$ is of dimension at most $8$
and $\scal{\nabla R} \subset \bigcap_\sigma \Gamma_\sigma^\infty(T\CM)$, where $R$ denotes the
scalar curvature and the intersection ranges over all choices of $\sigma$ with the above properties.
\end{lemma}

\begin{proof}
It suffices to note that since $\sum_{i=1}^m \nabla_{\sigma_i} \sigma_i = 0$, one has 
$\Upsilon_\sigma \tau = 0$ for every one of the $5$ symbols in $\VV$ that contain
$\Nabla_{\<generic>}\<generic>$ as a subsymbol. To eliminate one more degree of freedom,
we note that since $\nabla g = 0$ by definition of the Levi-Civita connection,
one has $\Upsilon_\sigma \tau_\star = 0$ by \eqref{e:fieldCovar}.
The inclusion $\scal{\nabla R} \subset \bigcap_\sigma \Gamma_\sigma^\infty(T\CM)$ follows from
 \eqref{e:secondTerm} below.
\end{proof}

\begin{remark}
One may wonder whether such collections of vector fields $\sigma_i$ do exist. This is always the case
since it suffices to consider a smooth isometric embedding of $\CM$ into $\R^d$ and to choose for $\sigma_i(p)$
the orthogonal projection of the $i$th canonical basis vector onto $T_p \CM$, see
for example \cite{Hsu} or Lemma~\ref{lem:propSigmai} below.
\end{remark}

\begin{remark}\label{rem:ItoGeod1}
Generically, one expects to have $\scal{\nabla R} = \bigcap_\sigma \Gamma_\sigma^\infty(T\CM)$
as a consequence of \eqref{prop:Itogeo} below, combined with the non-degeneracy results of
Section~\ref{sec:injective}. However, this identity fails for example when $d = 1$ since in this case, while
$\nabla R = 0$, $\nabla_{\<generic>}\<generic> = \sigma \sigma' + \Gamma \sigma^2 = {1\over 2} g' + \Gamma g$
depends only on $g$ and therefore belongs to $\bigcap_\sigma \Gamma_\sigma^\infty(T\CM)$.
\end{remark}

\begin{theorem}\label{thm:riemann}
Let $\CM$, $g$ and $\sigma$ be as in Lemma~\ref{lem:generality}, let $h \in \Gamma^\infty(T\CM)$, 
and let $\rho \in \Moll$. 
For any $V_{\rho,\sigma} \in \Gamma_\sigma^\infty(T\CM)$ and any $\eps \in (0,1]$, we denote by
$u_\eps$ the (local) solution to the random PDE
\begin{equ}[e:localSolManifold]
\d_t u_\eps = \nabla_{\d_x u_\eps} \d_x u_\eps + h(u_\eps) + \sum_{i=1}^m \sigma_i(u_\eps)\,\xi_i^\eps
+ V_{\rho,\sigma}(u_\eps)\;.
\end{equ}
Then, $u_\eps$ converges in probability (up to a possible explosion time) to a 
limit $u$ which may in general depend on $\rho$, $\sigma$ and $V_{\rho,\sigma}$. 

However, there exists a unique choice $(\rho,\sigma) \mapsto V^\canon_{\rho,\sigma}$ 
(corresponding to the `canonical family' of the previous theorem) such that 
\begin{enumerate}
\item The limit $u$ is independent of $\rho$ and $\sigma$.
\item For every choice of $\CM$, $g$ and $\sigma$, whenever 
$p \in \CM$ is such that $\bigl(\Nabla_{\sigma_i}\sigma_j\bigr)(p) = 0$ for all
$i$ and $j$, one has $V^\canon_{\rho,\sigma}(p) = 0$.
\item One can find $\rho \mapsto c_\rho \in \CV$ such that $V^\canon_{\rho,\sigma} = \Upsilon_\sigma c_\rho$.
\end{enumerate}
Finally, every choice $(\rho,\sigma) \mapsto V_{\rho,\sigma}$ satisfying condition 1 
is of the form $V_{\rho,\sigma} = V^\canon_{\rho,\sigma} + c \nabla R$.
\end{theorem}

\begin{remark}\label{rem:geom1}
It is natural to call the above limit with the  choice $V_{\rho,\sigma} = V^\canon_{\rho,\sigma}$
``the'' solution to the SPDE
\begin{equ}[e:solLoop]
\d_t u = \nabla_{\d_x u} \d_x u + h(u) + \sum_{i=1}^m \sigma_i(u)\,\xi_i\;.
\end{equ}
We do however conjecture that the invariant measure of this process when $h=0$ is \textit{not} the
Brownian loop measure (at half its natural speed) on $\CM$. Instead, that measure 
should be invariant for this process with $h = {1\over 32}\nabla R$, see Conjecture~\ref{conj:invariant} 
below for more details and a semi-heuristic justification of this claim.
(The reason for the value ${1\over 32}$ rather than the value ${1\over 8}$ appearing there is 
the absence of the factor $\sqrt 2$ in front of the noise term in \eqref{e:solLoop}, combined with the fact that $\nabla R$
is $4$-linear in $\sigma$.)
\end{remark}

\begin{remark}
In the particular case of the KPZ equation, the combination of properties~3 and~5 in Theorem \ref{theo:main} shows that
our specific choice of the finite constant $c$ singles out the Hopf-Cole solution.
Furthermore, since the intrinsic curvature of a one-dimensional Riemannian manifold always
vanishes, one has $H_{\Gamma,\sigma} = 0$ in this case, which explains the cancellation of the two
logarithmically divergent constants in \cite{KPZ}.
\end{remark}

\begin{remark}\label{rem:1dcase}
In the case of the generalised KPZ equation in $\R$, i.e. for $d=1$, given by
\begin{equ}
\partial_t u = \partial_x^{2} u + \Gamma(u) (\partial_x u )^{2} + K(u)\,\partial_x u+ h(u) + \sigma(u) \, \xi\;,
\end{equ}
one has $ H_{\Gamma,\sigma} = 0 $ as in the previous remark, which implies the cancellation of many logarithmically divergent constants and in particular that $U \equiv U^\family$ is independent of $\family$ in this case. 
However, if we restricted our class of equations to this smaller class, then properties~1--4 
of Theorem~\ref{theo:main} would \textit{not} be sufficient to determine $U$.
The reason is that properties~1 and 2 remain satisfied if we change $c$ appearing in \eqref{e:mainLimit} 
by any fixed amount. Since  $\Nabla_{\<generic>}\<generic> \in \CV^\nice$, we could in particular add to 
the right hand side of \eqref{e:mainLimit} any fixed multiple of
the vector field $\sum_i\Nabla_{\sigma_i}\sigma_i = {1\over 2}g' + \Gamma g^2$ already mentioned in Remark~\ref{rem:ItoGeod1} without breaking property~3.
Finally, property~4 is preserved since this vector field only depends on $g$ and not the choice of $\sigma_i$'s.

This shows that in $d=1$ we do not get as many constraints
out of the requirement that our solution theory is simultaneously Itô and geometric. 
This a reflection of the fact that the universe of possible equations is simply too
small to adequately reflect the structure of our renormalisation group. 
One could play a similar game by restricting the number $m$ of noises rather than the dimension of
the target space and see how the constraints evolve according to this number, but we have not done this. 
\end{remark}

\begin{remark}\label{rem:localDiffeo}
We will actually show a slightly stronger version of \eqref{e:equivar}, namely that for any 
two open sets $U, V \subset \R^d$ and any diffeomorphism $\phi \colon U \to V$, 
\eqref{e:equivar} also holds if we restrict ourselves to solutions with initial conditions
taking values in $U$ (respect. $V$), stopped whenever they exit these domains. 
The reason why this is slightly stronger than \eqref{e:equivar} is that there are diffeomorphisms
$U \to V$ that cannot be extended to a global diffeomorphism on $\R^d$ due to topological obstructions. 
\end{remark}

\begin{remark}
The only part which is non-canonical in \eqref{e:mainLimit}
is the value of $\hat c$. However, two different choices 
for this component simply correspond to a reparametrisation of the family $\family \mapsto U^\family$ which is
itself perfectly canonical. 
Furthermore, as an immediate consequence of \eqref{e:mainLimit}, the solution 
theories $U^\family$ for different values
of $\family$ are related by $U^{\bar \family}(\Gamma,K,\sigma,h) = U^{\family}\big(\Gamma,K,\sigma,h + (\bar \family-\family)H_{\Gamma,\sigma}\big)$. 
\end{remark}

\begin{remark}
As already mentioned earlier, part~4 of Theorem \ref{theo:main} is a version of It\^o's isometry in
this context (which is consistent with part~5) while part~3 corresponds to the classical change 
of variables formula.
In this sense, each one of the solution theories $U^\family$ for \eqref{e:mainSPDE} behaves
simultaneously like both an `It\^o solution' and a `Stratonovich solution'. 
\end{remark}

\begin{remark}
It is important that one chooses the mollifier $\rho$ in such a way that 
$\rho(t,x) = \rho(t,-x)$, otherwise one may have to subtract an additional
term of the form $f^{\alpha}_\beta(u)\d_x u^\beta$ for suitable $f$ 
in order to get the same limit and to restore in the limit the $x \leftrightarrow -x$ symmetry
which is then broken by the approximation.
\end{remark}

\begin{remark}\label{rem:twoparam}
If, instead of working in Theorem \ref{theo:main} with counterterms belonging to $\CV^\nice$, we had chosen to work with counterterms in
the ``full'' space $\CV$, the degree of freedom $\family$ appearing in the theorem would turn out
to be two-dimensional instead of one-dimensional. The second ``free'' direction  is then given by 
a counterterm proportional to 
\begin{equ}[e:secondTerm]
\hat H^\alpha(u) \eqdef \bigl(\nabla_\zeta R^\alpha_{\beta\gamma\eta}\bigr)(u) g^{\zeta\gamma}(u) g^{\beta\eta}(u)\;,
\end{equ}
see Proposition~\ref{prop:ItoStratgeo} and Lemma~\ref{lem:exprtaucstar}.

In the case when $\Gamma$ is the Levi-Civita connection given by the (inverse) Riemannian metric $g$
determined by $\sigma$ through \eqref{e:defMetric}, a straightforward calculation shows that
\eqref{e:secondTerm} is equivalent to the simpler identity
\begin{equ}
\hat H^\alpha = {1\over 2} \nabla^{\alpha} R(u)\;,
\end{equ}
where $R$ denotes the scalar curvature. 

This suggests that the corresponding
family of stochastic processes has invariant measures given by the Brownian loop measure, weighted by 
the (exponential of) the integral of the scalar curvature along the loop. Since this is the only 
degree of freedom in the renormalisation group associated to our equation that isn't determined by
symmetry considerations, it also suggests that different
elements from this
family of measures can arise from rather natural-looking approximations to the Brownian loop measure.
This is a well-known
fact that has already been pointed out in the physics literature of the early seventies \cite{Cheng,Um} 
and appears in more recent mathematical works on the topic \cite{Darling,Driver}. 
In our setting, it appears very natural to restrict our renormalisation to $\CV^\nice$, which fixes 
an origin for this degree of freedom, albeit in a rather arbitrary way. Another natural choice 
would be to choose it in such a way that the 
Brownian loop measure is indeed invariant for  \eqref{e:mainSPDE}. It is not clear at this stage
whether these two choices of origin coincide. We will revisit this point in more detail in Section~\ref{sec:conjecture}
where we conjecture that they do \textit{not} and that these two natural choices of origin differ by ${1\over 8}$.
\end{remark}

\subsection{Structure of the article}

Our strategy for the proof for Theorem~\ref{theo:main}, which takes up the remainder of this article, 
goes as follows.
After recalling some of the concepts and notations from the theory of regularity structures in
Section~\ref{sec:rules} and~\ref{sec:renorm}, we apply the results of \cite{BHZ,Ajay,BCCH} to our problem in
Sections~\ref{sec:BPHZthm} and~\ref{sec:reduced}. This shows that there exists a finite-dimensional
space $\CS$ of formal expressions (see Section~\ref{sec:reduced} for a precise definition of $\CS$) 
with a canonical embedding $\CV \subset \CS$, together with natural valuations 
$\Upsilon_{\Gamma,\sigma}\colon \CS \to \CC^\infty(\R^d,\R^d)$,
as well as a choice of renormalisation constants $C_{\eps,\geo}^\BPHZ \in \CS$ \label{constant geo page ref} such that
$U_\eps^\geo(\Gamma,K,\sigma,h+ \Upsilon_{\Gamma,\sigma}C_{\eps,\geo}^\BPHZ)$ converges as $\eps \to 0$
to some limit $\bar U$ independent of $\rho$. 
We similarly show that, for a different approximation $U_\eps^\Ito$ to \eqref{e:mainSPDE}
(instead of only hitting the noise with the mollifier, we hit the whole right hand side 
of the equation with it),
one has an analogous result with a different set of renormalisation constants $C_{\eps,\Ito}^\BPHZ$ \label{constant ito page ref}.
The point here is that we know a priori from \cite{Ajay} that, with this specific choice of renormalisation 
constants. the two limits obtained from these two different approximations do coincide,
call it $U^\BPHZ$.

In Section~\ref{sec:symapprox}, we then introduce two subspaces $\CS_\Ito \subset \CS$ and 
$\CS_\geo \subset \CS$ with the property that for any fixed $\eps$,
$U_\eps^\geo(\Gamma,K,\sigma,h+ \Upsilon_{\Gamma,\sigma}\tau)$ satisfies 
Property~3 of Theorem~\ref{theo:main} for any $\tau \in \CS_\geo$, while 
$U_\eps^\Ito(\Gamma,K,\sigma,h+ \Upsilon_{\Gamma,\sigma}\tau)$ satisfies 
Property~4 of Theorem~\ref{theo:main} for any $\tau \in \CS_\Ito$.
In Sections~\ref{sec:geo} and~\ref{sec:Ito1} we then show
that the renormalisation constants  $C_{\eps,\geo}^\BPHZ$ (resp.\ $C_{\eps,\Ito}^\BPHZ$) `almost' belong to 
$\CS_\geo$ (resp.\ $\CS_\Ito$) in the sense that there exist $\hat C_{\eps,\geo}^\BPHZ \in \CS_\geo$ \label{constant hat geo page ref}
such that $C_{\eps,\geo}^\BPHZ - \hat C_{\eps,\geo}^\BPHZ$ converges to a finite limit in $\CS$, and
similarly for $C_{\eps,\Ito}^\BPHZ$.
This relies in a crucial way on the fact that the corresponding two approximation procedures for our equation 
already exhibit the required symmetries at fixed $\eps > 0$. In order to transfer a convergence result at the
level of stochastic processes to the corresponding result at the level of renormalisation constants,
we also rely on the injectivity of the map
$h \mapsto U^\BPHZ(\Gamma,K,\sigma,h)$, which is shown in Section~\ref{sec:injectivelaw}.

In Section~\ref{sec:both}, we combine both properties to show that 
it is possible to find a solution theory that exhibits both Properties~3 and~4
simultaneously. This relies on the definition of a space $\CS_\both\subset \CS$ \label{S both page ref} of elements
$\tau$ such that, given any two $\sigma$, $\bar\sigma$ such that 
$\bar \sigma_i^\alpha \bar \sigma_i^\beta = \sigma_i^\alpha \sigma_i^\beta$, the
term $(\Upsilon_{\Gamma,\sigma}-\Upsilon_{\Gamma,\bar \sigma})\,\tau$ transforms like a vector field (see Definition~\ref{def_S_both}). 
The crucial remark which underpins our argument is the fact that one has the identity
\begin{equ}[e:mainiden]
\CS_\both = \CS_\Ito + \CS_\geo\;.
\end{equ}
The intersection of these two spaces furthermore consists of the subspace generated by $\tau_\star$
as well as the element generating the counterterm \eqref{e:secondTerm}. We eliminate the latter
by restricting ourselves to the subspace $\CS^\nice \subset \CS$ \label{S nice page ref} consisting of those terms
whose $\Upsilon_{\Gamma,\sigma}$-valuation vanishes at points where both $\Gamma$ and $\d \sigma$ vanish (i.e. 
the space
$\CV^\nice$, defined on page \pageref{V nice page ref}, is equal to $\CV\cap\CS^\nice$). It turns out that 
we can do this thanks to the fact that the BPHZ counterterms $C_{\eps,\geo}^\BPHZ$
and $C_{\eps,\Ito}^\BPHZ$ both belong to $\CS^\nice$ for any choice of mollifier in $\Moll$, see Corollary~\ref{cor:both}.
We then show that \eqref{e:mainiden} still holds if each of the spaces appearing there is
replaced by its intersection with $\CS^\nice$, but their intersection  now
consists solely of multiples of $\tau_\star$.

\begin{remark}\label{rem:SDEs}
As already mentioned, the analogous identity to \eqref{e:mainiden} does not hold in the case of SDEs. Indeed, in that case, 
$\CS$ consists of all linear combinations of the terms $\{\sigma_i^\beta \d_\beta \sigma_i^\alpha\}$
($\alpha$ is a free index and summation over $\beta$ and $i$ is implied). Since this term
is not invariant under changes of coordinates and cannot be written in terms of $g$ alone, one has
$\CS_\Ito = \CS_\geo = 0$. On the other hand, this term does belong to $\CS_\both$
since, whenever $\bar \sigma_i^\alpha \bar \sigma_i^\beta = \sigma_i^\alpha \sigma_i^\beta$ (summation over $i$ implied),
\begin{equ}
\sigma_i^\alpha \,\d_\alpha \sigma_i^\beta
- \bar \sigma_i^\alpha \,\d_\alpha \bar \sigma_i^\beta
= \nabla_{\sigma_i} \sigma_i - \nabla_{\bar \sigma_i} \bar \sigma_i
\end{equ}
for any choice of connection and it therefore does transform like a vector field.
\end{remark}

In Section~\ref{sec:conv} we then combine various existing results with a simple explicit
calculation to show that the divergent part $\hat C_{\eps,\geo}^\BPHZ$ is indeed of the form
shown in part~1 of Theorem~\ref{theo:main}, which in particular relies on the fact that
$\CS_\Ito \cap \CS_\geo \cap \CS^\nice$ coincides with $\CV_\star$.
In order to show that the process constructed in this way coincides with the 
classical It\^o solution in the case $\Gamma = K = 0$, we show that in this particular case
the counterterms for the `It\^o approximation' vanish, i.e.
$U^\family(0,0,h,\sigma) = \lim_{\eps \to 0} U_\eps^\Ito(0,0,h,\sigma)$ without any renormalisation 
needed. This is a consequence of the fact that $\CS_\Ito$ has the property that 
$\Upsilon_{0,\sigma} \tau = 0$ for every $\tau \in \CS_\Ito$, see Proposition~\ref{prop:CB} below.
The fact that $U_\eps^\Ito(0,0,h,\sigma)$ converges to the classical It\^o solution is easy to show.
Section~\ref{sec:geometric} is then devoted to a discussion of the natural geometric situation
mentioned as a motivation at the beginning of the introduction and contains the proof
of Theorem~\ref{thm:riemann}.

Sections~\ref{sec:algebra} and~\ref{sec:spaces} are devoted to the proof of \eqref{e:mainiden}.
While the inclusion $\CS_\Ito + \CS_\geo \subset \CS_\both$ is trivial, its converse is not,
as already noted in Remark~\ref{rem:SDEs}.
In order to show that it holds, we need a better understanding of the structure of the space $\CS$.
This is given by the notions of $T$-algebra and $T_\d$-algebra, which we introduce in Section~\ref{sec:algebra}.
In a nutshell, these notions encapsulate the main features of the spaces 
$\bigoplus_{\upper,\low\geq 0} \CV_\low^\upper[V]$, respectively
$\bigoplus_{\upper,\low\geq 0} \CC^\infty(V,\CV_\low^\upper[V])$, with $V$ a finite-dimensional vector space and
$\CV_\low^\upper[V]:=(V^*)^{\otimes \upper} \otimes V^{\otimes \ell}$: 
they both come with an action of two copies
of the symmetric group (allowing to permute the factors $V$ and $V^*$ in $\CV_\low^\upper[V]$),
a product, as well as a partial trace (pairing the last factor $V^*$ with the last factor $V$);
a $T_\partial$-algebra has moreover a notion of derivative.
Our main abstract result is a non-degeneracy result for the 
morphisms from a large class of $T$-algebras into a space of the type
$\bigoplus_{\upper,\low\geq 0} \CV_\low^\upper[V]$,
see Theorem~\ref{theo:injectiveT}, and for the morphisms from a
large class of $T_\d$-algebras into a space 
$\bigoplus_{\upper,\low\geq 0} \CC^\infty(V,\CV_\low^\upper[V])$,
see Theorem~\ref{theo:mainInjective}.

This provides us with a rigorous underpinning for a diagrammatic calculus
on the possible counterterms appearing in the renormalisation of \eqref{e:mainSPDE}, as well as 
the tools required to give a clean characterisation of the spaces $\CS_\geo$ in Section~\ref{sec:geo}
and $\CS_\Ito$ in Section~\ref{sec:Ito}.
The actual proof of \eqref{e:mainiden} is then given in Section~\ref{sec:counting}. Unfortunately, it is 
not as elegant 
as one may wish and relies on a dimension counting argument that appears to be somewhat \textit{ad hoc}
to the situation at hand. It does however rely strongly on the identifications of  $\CS_\geo$  and  $\CS_\Ito$
which appear to have a more universal flavour.

\subsection*{Acknowledgements}

{\small
The authors are very grateful to the referees for their careful reading of the manuscript. We are particularly indebted to the referee who pointed out a gap in the original argument of Section~\ref{sec:Ito}. Bridging this gap lead to substantial improvements in the clarity of the exposition.

MH gratefully acknowledges financial support from the Leverhulme trust via a Leadership Award, 
the ERC via the consolidator grant 615897:CRITICAL, and the Royal Society via a research professorship. 
YB and MH are grateful to the Newton
Institute for financial support and for the fruitful atmosphere fostered during the 
programme ``Scaling limits, rough paths, quantum field theory''.
We would also like to thank Xue-Mei Li for numerous discussions about stochastic 
analysis on manifolds. 
}

\section{The BPHZ solution}
\label{sec:BPHZ}

In this section, we recall how the results of \cite{BHZ,Ajay,BCCH} can be combined to produce a number of natural
candidate ``solution theories'' for the class of SPDEs \eqref{e:mainSPDE} on $\R^d$.

\subsection{Construction of the regularity structure}
\label{sec:rules}

Recall first that a class of stochastic PDEs is naturally associated to a ``rule''
(in the technical sense of \cite{BHZ}) defining a corresponding regularity structure. 
Such a rule describes how the different noises and convolution operators combine via the
non-linearity to form higher-order stochastic processes that provide a good local description 
of the solution which is stable under suitable limits. 

In our case, there is only one convolution operator (convolution against the heat kernel, or rather
a possibly mollified truncation thereof), while there are $m$ noises, all having the same regularity.
This motivates the introduction of a label set $\Lab = \{\<thin>, \<generic>_1,\ldots,\<generic>_m\}$
with corresponding degrees $\deg(\<thin>) = 2-\kappa$ and $\deg(\<generic>_i) = -{3\over 2}-\kappa$ for
$\kappa > 0$ sufficiently small ($\kappa < 1/100$ will do).
Recall that a ``rule'' is then a map $R$ assigning to each element of $\Lab$ a non-empty collection
of tuples in $\Lab \times \N^2$, where the element of $\N^2$ represents a space-time
multiindex. In our case, we identify $\Lab$ with $\Lab \times \{0\} \subset \Lab \times \N^2$
and we use the shorthand $\<thick> = (\<thin>, (0,1))$ (with the ``$1$'' denoting the space
direction). With these notations, the relevant rule describing the class of equations \eqref{e:mainSPDE}
is given by $R(\<generic>_i) = \{()\}$ and
\begin{equ}
R(\<thin>) = \{(\<thin>^k,\<generic>_i), (\<thick>^\ell,\<thin>^k)\,:\, k \ge 0,\, \ell \in \{0,1,2\},\, i\in \{1,\ldots,m\}\}\;,
\end{equ}
where we used $\<thin>^k$ to denote $k$ repetitions of $\<thin>$ and similarly for $\<thick>^\ell$.

As already remarked in \cite[Sec.~5.4]{BHZ}, the rule $R$ is normal, subcritical and complete,
so that it determines a regularity structure $(\CT,\CG)$, together with a 
``renormalisation group'' $\RR$ of extraction \slash contraction
operations acting continuously on its space of admissible models. 
The space $\CT$ is a graded vector space $\CT = \bigoplus_\alpha \CT_\alpha$ with
each $\CT_\alpha$ finite-dimensional. Elements of $\CT$ are formal linear combinations of 
labelled trees.
Here, a labelled tree $T_{\mff}^{\mfn}$ consists of
\begin{itemize}
\item A combinatorial rooted tree $T$ with vertex set $V_{T}$, edge set $E_{T}$ and root $ \rho_T $.
\item An edge decoration $\mff \colon E_{T} \rightarrow \Lab \times \N^2$. 
We call the first component of $\mff(e)$ the `type' of an edge $e$.
\item A vertex decoration $\mfn \colon V_{T} \rightarrow \N^{2}$.
\end{itemize}
We denote by $N_T \subset E_T$ the set of ``noises'', which are the edges of type $\<generic>_i$ for
some $i$.

Furthermore, we restrict ourselves to trees conforming to the rule $R$ in the following way.
Given a vertex $v \in V_T$, we call the (unique) edge $e_v$ adjacent to $v$ and pointing towards the root $\rho_T$ 
the `outgoing' edge and all other edges adjacent to $v$ `incoming' edges. 
The collection of incoming edges then determines a tuple $\CN(v)$ of $\Lab \times \N^2$ by collecting
all of their decorations $\mff$. With this notation, we restrict ourselves to trees such that
$\CN(v) \in R(\mft_v)$, where $\mft_v$ is the type of the outgoing edge $e_v$, with the
convention that $\mft_{\rho_T} = \<thin>$. We furthermore, impose that $\mfn(v) = 0$ 
if $\mft_v \in \{\<generic>_i\}$.

We henceforth denote by $\Trees$ the collection of all labelled trees conforming to $R$.
The notion of degree naturally extends to $\Trees$ by setting 
\begin{equ}
\deg T_{\mff}^{\mfn} = \sum_{v \in V_T} |\mfn(v)| + \sum_{e \in E_T} \deg(\mff(e))\;,
\end{equ}
where $|(k_0,k_1)| = 2k_0 + k_1$ and $\deg(\mft,k) = \deg \mft - |k|$.
With these notations, the space $\CT$ is nothing but the vector space generated by
$\Trees$, with the grading given by the degree.

A particular role in the general theory of regularity structures is played by the trees
of strictly negative degree.
Denoting by $\<genericX>$ a vertex with label $(0,1)$, by $\<generic>$ a vertex with label $0$ and an incoming edge
of type $\<generic>_i$ for some $i$ and finally by $\<genericxix>$ 
a vertex with label $(0,1)$ and an incoming edge
of type $\<generic>_i$ for some $i$, the complete list of elements $\Trees_- \subset \Trees$ 
of strictly negative degree is as follows, provided that $\kappa>0$ is sufficiently small. 
\begin{equ}
\begin{tabular}{rll} \toprule
Degree & Elements of $\Trees_-$\\
\midrule
$ - \frac{3}{2}^-$ & \<Xi>
\\ \rowcolor{symbols!5!pagebackground}
$-1^-$ &  \<Xi2> \,, \<I1Xitwo> 
\\
$ -\frac{1}{2}^- $ & \<Xi3a> \,,
\<Xi3> \,, \<I1IXi3b> \,, \<I1IXi3>\,, \<I1IXi3c> \,, \<I1Xi3c>\,, \<XiX> \,, \<I1Xi>  
\\  \rowcolor{symbols!5!pagebackground}
$0^-$ &  \<Xi4>\,, \<I1Xi4a> \,,  \<I1Xi4b> \,, \<I1Xi4c> \,, \<I1Xi4ab> \,, \<I1Xi4bc> \,, \<I1Xi4ac> \,,  \<Xitwo> \,, \<2I1Xi4> \,, \<2I1Xi4b> \,,  \<2I1Xi4c> \,, \\ \rowcolor{symbols!5!pagebackground} & \<Xi4b>\,, \<Xi4ba>\,, \<Xi4c>\,,   \<Xi4cb>\,, \<Xi4ca>\,, \<Xi4cab>\,, \<Xi4e> \,, \<Xi4ea> \,, \<Xi4eabis> \,,  \<Xi4eb> \,, \<Xi4eab> \,, \<Xi4eabbis>  
\\
$0^-$ &  \<Xi2X>\,, \<XXi2>\,, \<Xi2Xbis> \,, \<XXi2bis> \,, \<Xi2cbis1> \,, \<Xi2cbis> \,, \<I1XiIXi> \,, \<I1XiIXic> \;, \<XiIIXi>
\\
\bottomrule
\end{tabular}
\end{equ}\label{listPage}
Here, we say that an element has degree $\alpha^-$ if it has degree $\alpha - n\kappa$ for some
integer $n$.
Recall that we have dropped for simplicity the indices from $\<generic>$ and $\<genericxix>$, so that
every element of this list stands for a finite collection of basis vectors, for instance 
\begin{equ}
\<Xi2>=\big\{\<treeeval>: i,j=1,\ldots,m\big\}, \quad \<Xi4eabis>=\big\{\<Xi4eabisc1bis>: i,j,k,\ell=1,\ldots,m\big\}.
\end{equ}
The symbols of this list playing the most important role later on are those belonging to the two
lightly shaded rows,
so we write $\SS_{\<generic>}^{(2)}$ \label{SS2 page ref} and $\SS_{\<generic>}^{(4)}$ \label{SS4 page ref} for these two collections of symbols
and set $\SS_{\<generic>} = \SS_{\<generic>}^{(2)} \cup \SS_{\<generic>}^{(4)}$\label{SS page ref}.
More formally, $\SS_{\<generic>}^{(k)}$ consists of those labelled trees having exactly $k$ noises
that satisfy the ``saturated rule''
\begin{equ}
R^\sat(\<thin>) = \{(\<thin>^k,\<generic>_i), (\<thick>^2,\<thin>^k)\,:\, k \ge 0,\, i\in \{1,\ldots,m\}\}\;,
\end{equ}
and for which the vertex label $\mfn$ vanishes.
We also write $\CS_{\<generic>}$ \label{CS page ref} for the real vector space generated by $\SS_{\<generic>}$.

Recall that the space $\CT$ admits two ``integration maps'' $\CI$ and $\CI'$ obtained
by adding to a given labelled tree a new root vertex with zero label and joining
it to the old root vertex by an edge of type $\<thin>$ or $\<thick>$ respectively.
It also admits a ``product'' obtained by joining both trees at their roots and adding their 
root labels.
Note that while $\CI$ and $\CI'$ are defined on all of $\CT$, the product
is only defined on some domain of $\CT\times \CT$.
(One has for example $\<Xi2cbis1> \cdot \<I1Xi> = \<I1IXi3b>$ but the
product of $\<Xi2cbis1>$ with $\<I1Xitwo>$ is not defined in $\CT$ since our rule does not allow for
more than two thick edges to enter any given vertex.)

\subsection{Realisations of the regularity structure and renormalisation}
\label{sec:renorm}

Fix now a decomposition $P = K+R$ of the heat kernel on the real line into 
a kernel $K$ that is even (in the spatial variable), integrates to zero, and is compactly supported in a 
neighbourhood of the origin and a `remainder' $R$ that is globally smooth.
Fix also a mollifier $\rho \in \Moll$ with $\Moll$ as in the introduction just before Theorem~\ref{theo:main}. 
Given $\eps \ge 0$ and an 
arbitrary $m$-uple $\zeta = (\zeta_i)_{i\le m}$ of  continuous functions, 
we then define linear maps $\CL_\eps(\zeta)\colon \CT \to \CD'(\R^2)$ by  setting
$(\CL_\eps(\zeta)X^k)(z) = z^k$,
where $X^k$ denotes the tree consisting of a single vertex with label $k\in\N^{2}$ and $z^k = (t,x)^k
= t^{k_0} x^{k_1}$, as well as
$\CL_\eps(\zeta)\, \<generic>_i = \zeta_i$. 
This is extended to all of $\CT$ inductively by setting
\begin{equ}[e:mult]
\CL_\eps(\zeta) (\tau \cdot \bar \tau) = \CL_\eps(\zeta)\,\tau \cdot \CL_\eps(\zeta)\bar \tau\;,
\end{equ}
as well as
\begin{equ}
\CL_\eps(\zeta) (\CI\tau) = K_\eps * \CL_\eps(\zeta)\,\tau\;,\quad
\CL_\eps(\zeta) (\CI'\tau) = K'_\eps * \CL_\eps(\zeta)\,\tau\;, 
\end{equ}
where $K_\eps = \rho_\eps * K$ and $K_\eps'$ is its spatial derivative.
Note that if $\eps > 0$, then \eqref{e:mult} is well-defined also if $\zeta$ is a distribution 
since $\CL_\eps(\zeta)\,\tau$ is then smooth
for all $\tau \in \Trees \setminus \{\<generic>_i\}_i$ and the $\<generic>_i$'s are never
multiplied.

Defining the (spatially periodic) stationary stochastic processes $\xi_i^{(\eps)}$ as 
in \eqref{e:genClass}, we then define random linear maps $\PPi_\geo^{(\eps)}, \PPi_\Ito^{(\eps)}\colon \CT \to \CD'(\R^2)$
by setting \label{Ito and geo approximations  page ref}
\begin{equ}[e:defPPieps]
\PPi_\geo^{(\eps)} = \CL_0(\xi^{(\eps)}) \;,\qquad \PPi_\Ito^{(\eps)}= \CL_\eps(\xi)\;.
\end{equ}
Recall also from \cite[Secs~5\ \&\ 6]{BHZ} that the free algebra $\sscal{\Trees_-}$ generated by 
$\Trees_-$ admits a natural Hopf algebra structure as well as a coaction 
$\Deltam\colon \CT \to \sscal{\Trees_-} \otimes \CT$ given by a natural ``extraction-contraction''
procedure. Characters of $\sscal{\Trees_-}$ are identified with elements of 
$\CT_-^*$, where $\CT_-\subset \CT$ is the 
subspace spanned by $\Trees_-$. A simple recursion then shows the following.

\begin{proposition}\label{prop:BPHZchar}
Given $K$, $\rho$ and $\eps$ as above, there exists for $i \in \types$ 
a unique character $C_{\eps,i}^\BPHZ$ of $\sscal{\Trees_-}$
such that, setting \label{BPHZ renormalisation of ito geo model page ref}
\begin{equ}[e:defPiepshat]
\hPPi^{(\eps)}_i \tau = \big(C_{\eps,i}^\BPHZ \otimes \PPi_i^{(\eps)} \big)\Deltam \tau\;,
\end{equ}
one has $\E (\hPPi_i^{(\eps)} \tau)(0) = 0$ for every $\tau \in \Trees_-$.
\end{proposition}

\begin{remark}
The Hopf algebra structure of $\sscal{\Trees_-}$ endows its space $\RR$ of characters 
with a group structure by setting
\begin{equ}
(f \star g)(\tau) = (f \otimes g)\Deltam \tau\;. 
\end{equ}
Furthermore, given any compactly supported $2$-regularising kernel $K$ in the sense
of \cite[Ass.~5.1]{reg}, the map
\begin{equ}[e:actionRG]
(f, \PPi)\mapsto f \star \PPi \eqdef \big(f \otimes \PPi \big)\Deltam \;,
\end{equ}
yields a continuous action of $\RR$ onto the space $\MM_K$ of 
models that are admissible for $K$.
\end{remark}

\begin{remark}\label{rem:noises}
The BPHZ character $C_{\eps,i}^\BPHZ$ is always of the following form. Given 
any $\tau \in \Trees$, we set
$g_{\eps,i}(\tau) = \E \bigl(\PPi_i^{(\eps)}\tau\bigr)(0)$ (with the convention that
$\E \bigl(\PPi_\Ito^{(\eps)}\tau\bigr)(0) = 0$ if $\tau = \sXi_i \bar \tau$ for some
$i \le m$ and some symbol $\bar \tau\in \Trees$)
and we extend this to an
algebra morphism on $\sscal{\Trees}$. Then, there exists a linear
map $\CA^t\colon \Trees_- \to \sscal{\Trees}$ (the ``twisted antipode'') such that 
$C_{\eps,i}^\BPHZ(\tau) = g_{\eps,i}(\CA^t \tau)$, see \cite[Eq.~6.24]{BHZ} where the twisted
antipode is called $\hat \CA$.
Furthermore, the map $\CA^t$ has the property that, viewing the element $\CA^t \tau\in\sscal{\Trees}$ 
as a linear combination of forests with trees in $\Trees$, the noises
appearing in each of these forests are in canonical one-to-one correspondence with the
noises appearing in $\tau$.
\end{remark}

Denote now by $\MM$ the space of all models for the regularity structure $(\CT,\CG)$,
which we turn into a Polish space by endowing it with the sequence of pseudo-metrics 
given in \cite[Sec.~3]{reg}.
We also write $\MM_\eps \subset \MM$ for the models admissible for $K_\eps$ (with $K_0 = K$)
and $\MM_\star \subset [0,1] \times \MM$ for the closed subset consisting of those 
pairs $(\eps, \PPi)$ such that $\PPi \in \MM_\eps$.
We naturally view $\MM_\eps$ as a subset of $\MM_\star$ via the injection $\PPi \mapsto (\eps,\PPi)$.
We know by \cite{reg,BHZ} that for any $\eps > 0$, $\hPPi_\geo^{(\eps)}$ determines a unique
random $\MM_0$-valued random variable, which we again denote by $\hPPi_\geo^{(\eps)}$.
The same argument also works for $(\eps,\hPPi_\Ito^{(\eps)})$, which we view 
as an $\MM_\star$-valued random variable.
The main result of \cite{Ajay} can then be formulated as follows.

\begin{theorem}\label{theo:Ajay}  \label{BPHZ Model  page ref}
Under the above assumptions, there exists an $\MM_0$-valued random variable $\PPi^\BPHZ$
which possibly depends on the choice of decomposition $P = K+R$ but is independent of 
the choice of mollifier $\rho$, and such that for $i \in \types$, $\hPPi_i^{(\eps)}$ converges
as $\eps \to 0$ to $\PPi^\BPHZ$ in probability in $\MM_\star$. 
\end{theorem}
\begin{proof}
The convergence of $\hPPi_\geo^{(\eps)}$ to some limit $\PPi^\BPHZ$ independent of $\rho$ follows from
\cite[Thm.~2.31]{Ajay}.
Let now $\delta > 0$ and write $\hPPi_\Ito^{(\eps,\delta)}$ for the BPHZ
renormalisation of $\CL_\eps(\xi^{(\delta)})$. 
For $\tau$ not containing any factor $\sXi_i$, it follows from the smoothness of
$\PPi_\Ito^{(\eps)}\tau$ (at fixed $\eps > 0$) that
\begin{equ}
\lim_{\delta \to 0} \E \bigl(\PPi_\Ito^{(\eps,\delta)}\tau\bigr)(0)
=
\E \bigl(\PPi_\Ito^{(\eps)}\tau\bigr)(0)\;.
\end{equ}
Furthermore, thanks to the fact that $\rho$ is supported on strictly positive times,
one has $\E \bigl(\PPi_\Ito^{(\eps,\delta)}\tau\bigr)(0) = 0$ for $\tau$ containing a factor $\sXi_i$
and $\delta$ small enough, so that 
$\lim_{\delta \to 0} \hPPi_\Ito^{(\eps,\delta)} = \hPPi_\Ito^{(\eps)}$.
Noting that $\|K_\eps - K\|_{(2-\kappa),m}$ is bounded by some small positive power of $\eps$
and applying again \cite[Thm.~2.31]{Ajay}, we conclude that 
$\lim_{\eps \to 0} \hPPi_\Ito^{(\eps,\delta)} = \hPPi_\geo^{(\delta)}$ and that this convergence 
takes place uniformly in $\delta$, thus concluding the proof.
\end{proof}

This convergence takes place in a sufficiently strong topology that it implies
the corresponding convergence of a suitably renormalised version of our starting problem \eqref{e:mainSPDE},
which is the content of Theorem~\ref{theo:BPHZ} below.
Before we turn to this, we remark the following, which is the reason why we singled out
the set $\SS_{\<generic>}$ above.

\begin{lemma}\label{lem:symx}
For $i \in \types$, the BPHZ characters $C_{\eps,i}^\BPHZ$ satisfy $C_{\eps,i}^\BPHZ(\tau) = 0$ for all 
$\tau \in \Trees_- \setminus \SS_{\<generic>}$.
\end{lemma}

\begin{proof}
For $\deg\tau \in \{-{3\over 2}^-,-{1\over 2}^-\}$, this
follows from the fact that centred Gaussian random variables have vanishing odd moments.
For the remaining symbols of degree $\deg \tau = 0^-$, but containing only two instances of $\<generic>$, 
this follows from the fact that in this case the corresponding processes
$\hPPi_i^{(\eps)} \tau$ are odd in the sense that the identity
\begin{equ}
\bigl(\hPPi_i^{(\eps)} \tau\bigr)(t,-x) \eqlaw -\big(\hPPi_i^{(\eps)} \tau\bigr)(t,x)\;,
\end{equ}
holds in law as stochastic processes.
\end{proof}

\begin{remark}\label{rem:subgroup}
The set of characters $g$ of $\sscal{\Trees_-}$ such that $g(\tau) = 0$ for 
$\tau \in \Trees_- \setminus \SS_{\<generic>}$ forms a subgroup $\RR_{\<generic>}$ of the (reduced) renormalisation
group $\RR$ studied in \cite[Sec.~6.4.3]{BHZ}. Furthermore, this subgroup is canonically
isomorphic to $(\CS_{\<generic>}^*,+)$.
\end{remark}

In particular, we see from the above list that the vertex-labels $\mfn$ do not play any role
as far as the BPHZ characters $C_{\eps,i}^\BPHZ$ are concerned since they vanish on all trees
$T^\mfn_\mff$ such that $\mfn \not\equiv 0$. This suggests the introduction of the
set $\Trees_0 \subset \Trees$ consisting of those trees $T^\mfn_\mff$ with $\mfn \equiv 0$,
which we then simply denote by $T_\mff$. Finally, the following remark will
be useful later on.

\begin{lemma} \label{const:disconnect}
One also has $C_{\eps,i}^\BPHZ(\tau) = 0$ for
$\tau\in \{\<Xi4b>\,,  \<Xi4c>\,,  \<Xi4ba>\,,  \<Xi4cb>\,, \<Xi4ca>\,, \<Xi4cab>\}$.
\end{lemma}

\begin{proof}
For $\tau \in \{\<Xi4b>\,,  \<Xi4c>\}$, this follows from \cite[Sec.~4.3]{wong}, the argument for the remaining
elements being essentially the same. The reason why this is the case is that, for each of these trees
and for each way of partitioning the four instances of $\<generic>$ into two pairs that are then connected by 
two new edges, the resulting graph is connected but not two-connected.

For example, it follows from the general prescription \cite{BHZ} for the BPHZ renormalisation that
one has the identity
\begin{equ}
\scal{C_{\eps,\geo}^\BPHZ,\<Xi4b1>}
= \E \big(\PPi_\geo^{(\eps)}\<Xi4b1>\big)(0)
- \E \big(\PPi_\geo^{(\eps)}\<Xi2>\big)(0)\,\E \big(\PPi_\geo^{(\eps)}\<IXitwo>\big)(0) = 0\;,
\end{equ}
where the last identity follows from Wick's formula and we used $\<generic>$ and $\<genericb>$ to denote
two instances of $\sXi_i$ with different values for $i$.
\end{proof}

\subsection{BPHZ theorem}
\label{sec:BPHZthm}

In order to formulate the ``BPHZ theorem'' of \cite{Ajay,BCCH} in our context, we 
first introduce a valuation $\Upsilon_{\Gamma,\sigma}$ which maps
every linear combination of trees in $\Trees_0$ into a smooth function $\R^d \times \R^d \to \R^d$ in the following way. For every tree $T_\mff \in \Trees_0$, we set \label{Evaluationmap1 page ref}
\begin{equs}\label{recursive_Upsilon}
\ & \big(\Upsilon_{\Gamma,\sigma}T_\mff\big)^{\alpha}(u,q) = 
\\ \ & = \sum_{\beta:{\bf E}\to\{1,\ldots,d\}} \prod_{v \in V_{T}} \left[\Big(\prod_{e \in \mathbb{E}^{+}_{\graftI }(v) } \partial_{u^{\beta_e}}\Big)\Big( \prod_{e \in \mathbb{E}^{+}_{\graftID }(v)} \partial_{q^{\beta_e}}\Big)
\bigl(\bar \Upsilon^{\beta_{e_v}}_{\Gamma,\sigma}(v)\bigr)(u,q)\right]
\end{equs}
where 
\begin{itemize}
\item ${\bf E}$ is the set of edges $e\in E_T$ of type $\<thick>$ or $\<thin>$, or in other words not of type $\<generic>_i$.
\item For $v$ a leaf, we set $\bar \Upsilon^{\beta}_{\Gamma,\sigma}(v)(u,q) = 1$,
for $v$ with an incoming edge of type $\<generic>_i$ for some $i$, we set 
$\bar \Upsilon^{\beta}_{\Gamma,\sigma}(v)(u,q) = \sigma_i^\beta(u)$ (the index $i$ is uniquely determined since
our rules $R$ don't allow to have more than one incoming edge of type $\<generic>$ for the same vertex), while we
set
\begin{equ}[e:substGamma]
\bar \Upsilon^{\beta}_{\Gamma,\sigma}(v)(u,q) = \Gamma^\beta_{\gamma \eta}(u)\, q^\gamma q^\eta\;,
\end{equ}
otherwise.
\item $\mathbb{E}^{+}_{\graftI }(v)$ and $\mathbb{E}^{+}_{\graftID }(v)$ are the sets of edges with 
decorations $\<thin>$ and $\<thick>$ respectively coming into $v\in V_T$.
\item We use the convention $\beta_{e_v} =\alpha$ for $v = \rho_T$, the root of $T$.
\end{itemize}

For instance, writing $\<generic>$ for $\sXi_i$ and $\<genericb>$ for $\sXi_j$, one has
\begin{equ}
\big(\Upsilon_{\Gamma,\sigma}\<Xi4eabisc1quater>\big)^{\alpha}
= 2\d_\eta \Gamma^\alpha_{\beta\gamma}(u)\,\sigma_j^\eta(u)\,\sigma_i^\gamma(u)\, \d_\zeta \sigma_i^\beta(u)\, \sigma_j^\zeta(u)\;. 
\end{equ}
The reason for the factor $2$ appearing here is that \eqref{e:substGamma}
is differentiated (twice) with respect to the $q$-variable.
By linearity, we extend the definition of $\Upsilon_{\Gamma,\sigma}$ to the linear vector space generated by  $\Trees_0$.

\begin{remark}
Note that for all $T_\mff \in \SS_{\<generic>}$,
whenever a vertex isn't a leaf and has no incoming edge of type $\<generic>_i$, then
it has exactly two incoming edges of type \<thick>. As a consequence, $\Upsilon_{\Gamma,\sigma}T_\mff$
only depends on $u$, in which case we drop the $q$-dependence from our notations.
\end{remark}

We also assign to any $\tau \in \Trees_0$ a symmetry factor $S(\tau)$ given by the number
of tree automorphisms of $\tau$ (i.e.\ the number of graph isomorphisms of the corresponding tree $T$ which 
furthermore preserve its root and all of its labels). For example, assuming that all instances of
$\<generic>$ have the same index, we have
\begin{equ}
S(\<Xi2>) = 1\;,\quad S(\<I1Xitwo>) = 2\;,\quad S(\<Xitwo>) = 2\;, \quad S(\<2I1Xi4>) = 8\;,\quad S(\<Xi4b>) = 6\;.
\end{equ}
This endows $\CS_{\<generic>}$ with a Hilbert space structure by postulating that,
for $\tau, \sigma \in \Trees_0$, one has
$\scal{\tau,\sigma} = 0$ for $\tau \neq \sigma$
and $\scal{\tau,\tau} = S(\tau)$ \label{scalar product page ref}. 

\begin{remark}\label{rem:Sdual}
This allows us to identify $\CS_{\<generic>}$ with its dual space
$\CS_{\<generic>}^*$ in the usual way. Combining this with Lemma~\ref{lem:symx} and Remark~\ref{rem:subgroup},
we also identify $\RR_{\<generic>}$ with $(\CS_{\<generic>},+)$. This is in particular the reason
why in Theorem \ref{theo:BPHZ} we can apply $\Upsilon_{\Gamma,\sigma}$ to the functional $C\in\RR_{\<generic>}$.
\end{remark}

For some arbitrary but fixed 
$a \in (0,{1\over 2})$, we introduce the space $\CC^a_\star$ of parabolic $a$-H\"older
continuous functions $u \colon \R_+\times S^1 \to \R^d$ with possible blow-up at finite time.
Formally, we define $\CC^a_\star$ as the set of pairs $(f,t_\star)$ where 
$t_\star \in (0,\infty]$ and 
$f \colon [0,t_\star) \times S^1 \to \R^d$ is a locally $\alpha$-H\"older continuous 
function such that, if $t_\star < \infty$, its $\alpha$-H\"older norm on $[0,t]\times S^1$ 
diverges as $t \uparrow t_\star$. 
This is a Polish space when endowed with the system of pseudo-metrics given by
 $\{d_L \}_{L \ge 1}$,
where $d_L(F,G) = \|S_L (F)-S_L (G)\|_{\CC^a([0,L])}$ with $S_L (f,t_\star)$ equal to  the function $f$, 
stopped when either time or its space-time
parabolic $\CC^a$ norm exceeds $L$. (If $t_\star <\infty$, then this happens before time $t_\star$ by assumption, 
so this  is a well-defined operation.) We also introduce the space $\tilde \CC_\star^a$ of
maps $\CC^a(S^1,\R^d) \to \CC_\star^a$ that are uniformly continuous on bounded
sets, endowed with the topology of uniform convergence
on bounded sets, as well as the space $\CB_\star^a$ \label{CBstar page def} of
continuous maps from $\CC^a(S^1,\R^d)$ to the space of probability measures on $\tilde \CC_\star^a$.

With this notation, the main results of \cite{reg,Ajay,BHZ,BCCH} can be combined into the following statement
where $\rho_\eps * G$ should be read as a somewhat cumbersome way of simply writing $G$ in the case $\eps = 0$.
\begin{theorem}\label{theo:BPHZ}
For every 
choice of smooth $\Gamma$, $\sigma$, $K$, $h$ and every $a \in (0,{1\over 2}-\kappa)$, 
there exists a continuous ``solution map''
$\CA(\Gamma,\sigma,K,h) \colon \MM_\star \to \tilde \CC_\star^a$ 
such that, for every $C \in \RR_{\<generic>}$, every continuous $\zeta$, and setting
\begin{equ}
\MM_\star \ni Z = \bigl(\eps, C \star \CL_\eps(\zeta)\bigr)\;,
\end{equ} 
(with $\star$ as in \eqref{e:actionRG}), $\CA(\Gamma,\sigma,K,h)(Z)(u_0)$ solves the equation
\begin{equs}
\d_t u^\alpha &= \d_x^2 u^\alpha + \rho_\eps * \bigl[\Gamma^\alpha_{\beta\gamma}(u)\,\d_x u^\beta \d_x u^\gamma + K^\alpha_\beta(u)\,\d_x u^\beta \\
&\qquad + h^\alpha(u) +  \sigma^\alpha_i(u)\,\zeta_i + \big(\Upsilon^\alpha_{\Gamma,\sigma}C\big)(u)\bigr]\;,\label{e:generalEquation}
\end{equs}
for $t > 0$ and $u^\alpha(t,\cdot) = u_0^\alpha$ for $t \le 0$. 
Furthermore, $\CA(\Gamma,\sigma,K,h)(Z)$ is jointly continuous in $\Gamma$, $\sigma$, $K$, $h$ and $Z$
provided that the first four functions on $\R^d$ are equipped with the topology of convergence in the 
$\CC^6$ topology on compact sets.
\end{theorem}

\begin{proof}
For $Z \in \MM_0$, this is the content of \cite[Thm.~2.20]{BCCH} (the symmetry factor $S(\tau)$
appearing in \cite[Eq.~2.19]{BCCH} is created by the identification of $\CS_{\<generic>}$ with $\CS_{\<generic>}^*$ 
as already mentioned in Remark~\ref{rem:Sdual}).
The proof for general $Z \in \MM_\star$ is virtually identical. The only difference is that the evolution
now has some small amount of memory, so that one needs to keep track of it over a time interval
of order $\eps^2$ when restarting the fixed point argument. The continuity as $\eps \to 0$ is
taken care of by the time continuity built into the definition of the spaces $\CC_\star^a$.
\end{proof}

\begin{remark}
The third argument $K$ of $\CA$ does not play much of a role from the point of view of the arguments
in this article since it neither contributes to the counterterm $\Upsilon^\alpha_{\Gamma,\sigma}C$ nor
is affected by it. We will therefore from now on only consider the case $K=0$ and drop this argument
in order to keep notations shorter.
\end{remark}

\subsection{Reduced trees}

\label{sec:reduced}

Define now $\SS_{0}$ \label{one noise type  page ref} in the same way as $ \SS_{\<generic>}  $, see Page \pageref{SS4 page ref},
but with $m = 1$, so that there
is only one ``noise type'', and denote by $S_0(\tau)$ the corresponding symmetry factor. 
Given $\tau \in \SS_0$, we write $\CP_\tau^{(2)}$ for the set of all partitions of its noise set
$N_\tau$ consisting of only two-elements blocks.
We then define $\SS$\label{SShat page ref} 
as the set of equivalence classes of pairs $(\tau,P)$ 
with $\tau \in \SS_{0}$ 
and $P \in \CP_\tau^{(2)}$, where $(\tau,P)\sim (\tau',P')$ if there exists
a tree isomorphism mapping $\tau$ to $\tau'$ and $P$ to $P'$. For example, we want to make sure that
we identify partitions of the form $\<Xi4eabbisc1s>$ and $\<Xi4eabbisc1perms>$, where noises belong
to the same block if and only if they have the same colour; moreover there are three inequivalent pairings for $\<Xi4s>$, given by $\<Xi4_1s>$, $\<Xi4c1s>$ and $\<Xi4_2s>$,
while all pairings of $\<Xi4bs>$ are equivalent to $\<Xi4b1s>$.
The full list of elements of $\SS$ obtained in this way is given by
\begin{equs}[e:SS]
{} & \<Xi2s>  \,, \<I1Xitwos> \,,  \<Xi4_1s> \,, \<Xi4c1s>\,, \<Xi4_2s>\,, \<cI1Xi4as> \,,   \<I1Xi4ac1s> \,, \<I1Xi4ac2s> \,, \<cI1Xi4bs> \,, \<I1Xi4bc1s> \,, \<cI1Xi4cs> \,,  \<I1Xi4cc1s> \,, \<I1Xi4cc2s>     \,, \<cI1Xi4abs> \,, \<I1Xi4abc1s> \,,  \<cI1Xi4bcs> \,, \<I1Xi4bcc1s> \,, \<cI1Xi4acs> \,, \<I1Xi4acc1s> \,,   \<I1Xi4acc2s> \,,
\\ & \<I1Xi4abcc2s> \,, \<I1Xi4abcc1s> \,,   \<2I1Xi4c2s> \,,    \<2I1Xi4c1s> \,,  \<2I1Xi4bc3s> \,,  \<2I1Xi4bc1s> \,,  \<2I1Xi4bc2s> \,,  \<2I1Xi4cc2s> \,,  \<2I1Xi4cc1s> \,, \<Xi4b1s> \,, \<Xi4ba1s> \,, \<Xi4ba2s> \,, \<Xi41s> \,, \<Xi42s> \,,   
  \<Xi4cbc1s>\,,  \<Xi4cbc2s>\,, \\ & \<Xi4ca1s> \,, \<Xi4ca2s> \,,    \<Xi4cabc1s>\,, \<Xi4cabc2s>\,, \<Xi4ec3s> \,,  \<Xi4ec1s> \,,  \<Xi4ec2s> \,, \<Xi4eac2s> \,,    \<Xi4eac1s> \,,  \<Xi4eabisc3s> \,, \<Xi4eabisc1s> \,, \<Xi4eabisc2s> \,, \<Xi4ebc2s> \,,  \<Xi4ebc1s> \,,\<Xi4eabc2s> \,,  \<Xi4eabc1s> \,, \<Xi4eabbisc2s> \,, \<Xi4eabbisc1s>\;.
\end{equs}
We denote by $\CS$ 
\label{CS page ref2} the vector space generated by those elements $(\tau,P) \in \SS$. 

Given $\tau \in \SS_0$, we write $\CL_\tau$ for the set
of all maps $f \colon N_\tau \to \{\<generic>_1,\ldots,\<generic>_m\}$
and, for $f \in \CL_\tau$, we write
$f(\tau) \in \SS_{\<generic>}$ for the tree identical to $\tau$, but with the noise type
of each leaf $u \in N_\tau$ changed into $f(u)$. 
Given a partition $P$ of $N_\tau$ and  $f \colon N_\tau \to \{\<generic>_1,\ldots,\<generic>_m\}$,
we also write $f \succ P$ 
if $f$ is constant on each block of $P$.

The main result of this section is then that we can view our renormalisation constants as 
elements of  $\CS$.
To see this, we first define $\iota:\CS\to\CS_{\<generic>}$
\begin{equ} \label{def:iota}
\iota (\tau,P) \eqdef \sum_{f\,:\, f\succ P} f(\tau)\;,
\end{equ}
so that for example
\[
\iota \big(\<Xi4eabisc1quater>\big)=\sum_{i,j} \<Xi4eabisc1tris>\;.
\]

\begin{lemma}\label{lem:iota}
For every $\eps > 0$ one has $\{C^\BPHZ_{\eps,\geo},C^\BPHZ_{\eps,\Ito}\} \subset \iota \CS$.
\end{lemma}

\begin{proof}
Fix $\eps > 0$, $c \in \{C^\BPHZ_{\eps,\geo},C^\BPHZ_{\eps,\Ito}\}$.
As a consequence of the definition of the scalar product on $\CS_{\<generic>}$ given on page \pageref{scalar product page ref}, we have the identity
\begin{equ}[e:Parseval]
c = \sum_{\tau \in \SS_{\<generic>}} \tau{\scal{c,\tau} \over S(\tau)}\;,
\end{equ}
where $S(\tau)$ denotes the symmetry factor defined in the previous section.

We then claim that one can rewrite \eqref{e:Parseval} as
\begin{equ}[e:Parseval2]
c = \sum_{\tau \in \SS_{0}} \sum_{f \in \CL_\tau} f(\tau) {\scal{c,f(\tau)} \over S_0(\tau)}\;.
\end{equ}
Indeed, note that given $\tau_{\<generic>} \in \SS_{\<generic>}$ there exists a \textit{unique}
$\tau \in \SS_0$ for which there exists $f \in \CL_\tau$ such that $f(\tau) = \tau_{\<generic>}$.
The choice of $f$ on the other hand is \textit{not} unique since elements of $\SS_{\<generic>}$ are
really equivalence classes modulo tree isomorphisms.

Writing $G(\tau)$ for the group of tree isomorphisms of $\tau$ and
$G_f(\tau) \subset  G(\tau)$ for those isomorphisms $g$ such that $f \circ g = f$, \eqref{e:Parseval2}
then follows from the fact that $S_0(\tau) = |G(\tau)|$ and $S(f(\tau)) = |G_f(\tau)|$, so that 
$S_0(\tau)/S(f(\tau))$ counts precisely the number of distinct maps $\bar f \in  \CL_\tau$
such that ${\bar f}(\tau) = f(\tau)$ (since $({f\circ g})(\tau) = f(\tau)$ for every $g \in G(\tau)$).

We further note that, given any $P \in \CP_\tau$ and any $f \sim P$, i.e.\ such that the partition generated
by $f$ is \textit{equal} to $P$, the exchangeability of the noises implies that 
$\scal{c,f(\tau)}$ is independent of the particular choice of such an $f$, so we denote it by $\scal{c,(\tau,P)}$
instead. With this notation, Wick's formula then implies the following result, the proof of which
is given below.

\begin{lemma}\label{lem:Wick}
For every $\tau \in \SS_0$, every $\eps > 0$, every $c \in \{C^\BPHZ_{\eps,\geo},C^\BPHZ_{\eps,\Ito}\}$, 
and every $f \in \CL_\tau$, one has the identity
\begin{equ}
\scal{c,f(\tau)} = \sum_{P\in \CP_\tau^{(2)}\,:\, f\succ P} \scal{c,(\tau,P)}\;.
\end{equ}
\end{lemma}
As a consequence of this lemma, we conclude that 
\begin{equ}[e:expansion]
c = \sum_{\tau \in \SS_{0}} \sum_{P\in \CP_\tau^{(2)}} {\scal{c,(\tau,P)} \over S_0(\tau)} \sum_{f\,:\, f\succ P} f(\tau)
=\sum_{\tau \in \SS_{0}} \sum_{P\in \CP_\tau^{(2)}} {\scal{c,(\tau,P)} \over S_0(\tau)} \iota (\tau,P) \;,
\end{equ}
as claimed. This concludes the proof of Lemma \ref{lem:iota}.
\end{proof}

\begin{proof}[of Lemma~\ref{lem:Wick}]
Consider the case $i = \mathrm{geo}$ first and fix $\tau$ and $f$ as in the statement, as well as $\eps > 0$. 
Given a subset $K$ of $N_\tau$, we denote by 
$E_f^K$ the function on $(\R^2)^K$ given by 
\begin{equ}[e:prodE]
E_f^K(z) = \E \Big(\prod_{u \in K} \xi^{(\eps)}_{f(u)}(z_u)\Big)\;.
\end{equ}
Similarly, for any  $P \in \CP^{(2)}(K)$, the set of partitions of $K$ into sets of size two,
we write $E_P^K = E_g^K$ for any $g$ which generates the partition $P$.
Wick's formula can then be stated as
\begin{equ}
E_f^K = \sum_{P \in \CP^{(2)}(K)\,:\, f \succ P} E_P^K\;.
\end{equ}
(Note that this sum is empty if $|K|$ is odd.)

It also follows from Remark~\ref{rem:noises} that there exist finitely many
distributions $\eta_j$ on $(\R^2)^{N_\tau}$ and partitions $O_j$ of $N_\tau$ such that 
\begin{equ}
\scal{c,f(\tau)} = \sum_j \int \Big(\prod_{K \in O_j} E_f^K(z|K)\Big)\,\eta_i(dz)\;,
\end{equ} 
where, for $z \in (\R^2)^{N_\tau}$, we denote by $z|K$ its restriction to $K$, so that 
\begin{equ}
\scal{c,f(\tau)} = \sum_j \sum_{P \in \CP_\tau^{(2)} \atop P \prec O_j \,\&\, P \prec f} \int \Big(\prod_{K \in O_j} E_{P}^K(z|K)\Big)\,\eta_i(dz)\;,
\end{equ} 
where $P \prec O_j$ means that $P$ refines $O_j$. It now suffices to note that if, given 
$P \in \CP_\tau^{(2)}$, we write $f_P$ for any one function with $f_P \sim P$, then we have
\begin{equ}
\scal{c,f(\tau)} = \sum_{P \in \CP_\tau^{(2)}\,:\, P \prec f} \sum_j  \int \Big(\prod_{K \in O_j} E_{f_P}^K(z|K)\Big)\,\eta_i(dz)
= \sum_{P \in \CP_\tau^{(2)}}\scal{c,(\tau,P)}\;,
\end{equ} 
as desired. This is because, if there happens to be some $j$ for which $P$ does not
refine $O_j$, then one of the factors $E_{f_P}^K(z|K)$ for $K \in O_j$ necessarily vanishes.
The proof for $i = \textrm{It\^o}$ is identical, the only difference being that 
$E_f^K$ is now a distribution (a sum of products of Dirac distributions since the $\xi^{(\eps)}$ in
\eqref{e:prodE} should be replaced by $\xi$'s), while 
the $\eta_i$ are smooth $\eps$-dependent functions. 
\end{proof}

\begin{remark}\label{rmq:symfactpairing}
Given a tree $\tau \in \SS_0$ and a pairing $P$ of its noises, we have a natural symmetry factor
$S(\tau,P)$ which counts the number of tree isomorphisms of $\tau$ that keep the pairing $P$ intact.
With this notation, we can then rewrite \eqref{e:expansion} as
\begin{equ}[e:expansionNice]
c = \sum_{(\tau,P)} \iota(\tau,P) {\scal{c,(\tau,P)} \over S(\tau,P)}\;.
\end{equ}
Indeed, if we denote by $N(\tau,P)$ the number of distinct pairings $Q$ of the
noises of $\tau$ such that $(\tau,Q) = (\tau,P)$ modulo tree isomorphism, then one has
$N(\tau,P)/S(\tau) = 1/S(\tau,P)$. This is of course equivalent to $N(\tau,P) = S(\tau)/S(\tau,P)$,
which follows from the fact that $S(\tau,P)$ is the stabiliser of $S(\tau)$ with respect to $P$ if
we view the group of tree isomorphisms as acting on the set of all pairings of the noises of $\tau$
while $N(\tau,P)$ is precisely the size of the orbit of $P$ under that action.
\end{remark}

\begin{remark}\label{rem:scalEVal}
The space $\CS$ is naturally endowed with the scalar product such that elements of $\SS$ are 
orthogonal and such that $\|(\tau,P)\|^2 = S(\tau,P)$. This definition of the scalar product on $\CS$
is consistent with \eqref{e:expansionNice}. 
We also extend the valuation 
$\Upsilon_{\Gamma,\sigma}$ defined in Section~\ref{sec:BPHZthm} to $\CS$ by composition with $\iota$.
\end{remark}

\section{Symmetry properties of the solution map}
\label{sec:symmetry}

The goal of this section is to show that it is possible to construct solution maps having the 
various symmetries laid out in Theorem~\ref{theo:main}.
We do this by analysing two different approximation schemes corresponding to the
models $\PPi_\geo^{(\eps)}$ and $\PPi_\Ito^{(\eps)}$ constructed above. 
We therefore define 
\begin{enumerate}
\item $U_\eps^\geo(\Gamma,\sigma,h) \in \CB_\star^a$ as the law of $\CA(\Gamma,\sigma,h)(\PPi^{(\eps)}_\geo)(u_0)$
\item $U_\eps^\Ito(\Gamma,\sigma,h) \in \CB_\star^a$
as the law of $\CA(\Gamma,\sigma,h)(\PPi^{(\eps)}_\Ito)(u_0)$
\item $U^\BPHZ(\Gamma,\sigma,h) \in \CB_\star^a$
as the law of $\CA(\Gamma,\sigma,h)(\PPi^\BPHZ)(u_0)$
\end{enumerate}
\label{U solutions page ref}
where $\CB_\star^a$ is defined on page~\pageref{CBstar page def}, $\PPi_\geo^{(\eps)}$ and $\PPi_\Ito^{(\eps)}$
are defined in \eqref{e:defPPieps}, and $\PPi^\BPHZ$ is the limit of both $\hPPi_\geo^{(\eps)}$ and $\hPPi_\Ito^{(\eps)}$ given by Theorem \ref{theo:Ajay}, where $\hPPi_\geo^{(\eps)}$ and $\hPPi_\Ito^{(\eps)}$ are defined by \eqref{e:defPiepshat}.

\subsection{Symmetry properties of the approximations}
\label{sec:symapprox}

The reason for considering these two different approximations is that each of them satisfies 
one of the two symmetries, as formulated in the following result.

\begin{proposition}\label{prop:equivariance}
For every diffeomorphism $\phi$ of $\R^d$, the identity
\begin{equ}[e:geo]
\phi \act U_\eps^\geo(\Gamma,\sigma,h) = U_\eps^\geo(\phi \act\Gamma,\phi \act\sigma,\phi \act h)\;,
\end{equ}
holds for all smooth choices of $\Gamma$, $\sigma$ and $h$.
Furthermore, for all smooth $\bar \sigma$ such that 
\begin{equ}[e:sigmabar]
\bar \sigma_i^\alpha \bar \sigma_i^\beta = \sigma_i^\alpha \sigma_i^\beta
\end{equ}
holds, one has
$U_\eps^\Ito(\Gamma,\sigma,h) = U_\eps^\Ito(\Gamma,\bar\sigma,h)$.
\end{proposition}

\begin{proof}
The definition of $U_\eps^\geo$ implies that $U_\eps^\geo(\Gamma,\sigma,h)(u_0)$ is the law of 
the random PDE \eqref{e:genClass}. The
identity \eqref{e:geo} then simply states that the solutions to this random PDE behave as
expected under changes of variables, which follows immediately from the usual rules of calculus
which can be applied pathwise since all objects under consideration are smooth. We
therefore only need to consider the claim for $U_\eps^\Ito$.

Write $\CQ$ \label{CQ} for the set of pairs $(\sigma, \bar \sigma)$ of $d\times m$ real-valued matrices
such that $\sigma \sigma^\top = \bar \sigma \bar \sigma^\top$. It is a simple consequence of the reduced
LQ factorisation that there exists a map $\Sigma\colon \CQ \to SO(m)$ such that, for every $(\sigma, \bar \sigma)\in \CQ$,
one has $\sigma \Sigma(\sigma, \bar \sigma) = \bar \sigma$. By the Kuratowski--Ryll--Nardzewski measurable selection theorem (see for example \cite[Vol.~II, p.~36]{Bogachev}),
we can furthermore choose $\Sigma$ to be Borel measurable.

For $u_0$ fixed, let now $u$ be the maximal solution to \eqref{e:generalEquation} as in 
the proof of Theorem~\ref{theo:BPHZ}. Write $\CH = L^2(S^1,\R^m)$,
and let $W$ by the $\CH$-cylindrical Wiener process such that, for every smooth test function $\phi \colon [-1,\infty)\times S^1 \to \R^m$
(canonically identified with a function $[-1,\infty) \to \CH$),
one has
\begin{equ}
\int_\CR \scal{\phi(t,\cdot),dW(t)} = \xi_i(\phi_i) \;.
\end{equ}
Since we assume that $\rho$ is
non-anticipative, the solution $u$ is smooth and adapted,
so that the distributional products $\sigma_i^\alpha(u)\,\xi_i$ coincide with their interpretations as It\^o
integrals with respect to $W$ as above. 

Fix $L>0$ and let $\tau$ be the stopping time given by the first time at which the space-time $\alpha$-H\"older norm
of $u$ exceeds $L$. For $\Sigma\colon \CQ \to SO(m)$ as before, we then set
\begin{equ}
Q(t)(x) = 
\left\{\begin{array}{cl}
	\Sigma(\sigma(u(t,x)), \sigma(u(t,x)))^\top & \text{if $t \le \tau$,} \\
	\id & \text{otherwise,}
\end{array}\right.
\end{equ} 
so that $t \mapsto Q(t)$ is a progressively measurable process with values in $\CO_m \eqdef \CB_b(S^1,SO(m))$.
(For $t \le 0$, we use the convention $u(t) = u_0$.)
Since pointwise multiplication by an element of $\CO_m$ yields a unitary operator on $\CH$, the process
\begin{equ}
\bar W(t) \eqdef \int_{-1}^t Q(s)\,dW(s)
\end{equ}
is again a cylindrical Wiener process on $\CH$, see \cite{DPZ}, so that in particular the laws of $W$ and of $\bar W$ coincide. 

Denote by $\bar \xi_i$ the corresponding distribution and note that one then has the almost sure distributional identity
\begin{equ}
\bar \sigma^\alpha_i(u) \, \bar \xi_i  = \sigma^\alpha_i(u) \, \xi_i\;,
\end{equ}
(implied summation over $i$), as a straightforward consequence of the fact that, for any bounded $\CH$-valued progressively measurable 
function $A$ one has the identity $\int_0^t \scal{A(s),d\bar W(s)} = \int_0^t \scal{Q^*(s) A(s) ,dW(s)}$.
The claim now follows at once since we are working with the model $\PPi^{(\eps)}_\Ito$ given by
the canonical lift $\CL_\eps(\xi)$, so that the reconstruction applied to any product is simply given by the 
product of the reconstructions. 
\end{proof}

It is then natural to define the following two subspaces of $\CS$
which correspond
to the two invariance \slash equivariance properties appearing in this statement.
(Recall that $\Upsilon_{\Gamma,\sigma}$ is defined on $\CS$ by Remark~\ref{rem:scalEVal}.)

\begin{definition}\label{def:geoIto}
The space $\CS_\geo \subset \CS$ consists of those elements $\tau$ such that,
for all $d, m \ge 1$ and all choices of $\Gamma$ and $\sigma$ as above and all diffeomorphisms $\phi$ 
of $\R^d$ that are homotopic to the identity, 
one has the identity
\begin{equ}[e:idendiff]
\phi \act(\Upsilon_{\Gamma,\sigma}\tau) = \Upsilon_{\phi\act \Gamma,\phi\act \sigma}\tau\;, 
\end{equ}
where the action appearing on the left is that given by \eqref{eq:actionvect}. 
Similarly, the space $\CS_\Ito \subset \CS$ consists of those elements $\tau$ such that,
for all $d$, $m$, $\Gamma$, $\sigma$ and $\bar \sigma$ as above
with $\sigma \sigma^\top = \bar \sigma \bar \sigma^\top$, the identity
$\Upsilon_{\Gamma,\sigma}\tau = \Upsilon_{\Gamma,\bar \sigma}\tau$
holds.
\end{definition}

\begin{remark}\label{rem:restrict}
We do for the moment restrict ourselves to equivariance under diffeomorphisms homotopic to 
the identity. We will however see in Remark~\ref{rem:geo} below that, for $\tau \in \CS_\geo$,
the identity \eqref{e:idendiff} automatically holds for all diffeomorphisms.
\end{remark}

With these notations at hand, we have the following immediate corollary of Proposition~\ref{prop:equivariance}.

\begin{corollary}\label{cor:equivariant}
Assume that, for $i \in \types$, there exist sequences $\tau_i^{(\eps)} \in \CS_i$
such that $\tau_i^{(\eps)} - C_{\eps,i}^\BPHZ$ converges to a finite limit $\alpha_i$ as $\eps \to 0$.
Then the limits
\begin{equ}
U^i \eqdef \lim_{\eps \to 0} \hat U_\eps^i\;,\quad
\hat U_\eps^i(\Gamma,\sigma,h) = U_\eps^i\big(\Gamma,\sigma,h + \Upsilon_{\Gamma,\sigma}\tau_i^{(\eps)}\big)\;,
\end{equ}
exist, satisfy their respective equivariance properties as in Proposition~\ref{prop:equivariance} (except that 
the equivariance under the diffeomorphisms group is a priori restricted to those homotopic to the 
identity), and are equal to
\begin{equ}[e:U=U]
U^i(\Gamma,\sigma,h)=U^\BPHZ\left(\Gamma,\sigma,h+\Upsilon_{\Gamma,\sigma}\alpha_i\right).
\end{equ}
\end{corollary}

\begin{proof}
The fact that the limits exist follows from the continuity with respect to $(h,Z)$ stated in 
Theorem~\ref{theo:BPHZ}, combined with our assumption and Theorem~\ref{theo:Ajay}.
The fact that $U^\geo$ is equivariant under the diffeomorphism group follows immediately from the 
fact that this is true for $\hat U_\eps^\geo$ for every $\eps$, combined with the fact that 
the diffeomorphism group acts continuously on $\CB_\star^a$. The argument for $U^\Ito$ is similar.

Finally, it follows from Theorem~\ref{theo:BPHZ} and the definition of $U^\BPHZ$ that
\[
\begin{split} U^\BPHZ(\Gamma,\sigma,h) & =\lim_{\eps \to 0} U_\eps^i\left(\Gamma,\sigma,h + \Upsilon_{\Gamma,\sigma}C_{\eps,i}^\BPHZ\right) 
\\ & = \lim_{\eps \to 0} \hat U_\eps^i\left(\Gamma,\sigma,h + \Upsilon_{\Gamma,\sigma}(C_{\eps,i}^\BPHZ-\tau_i^{(\eps)})\right) =U^i\left(\Gamma,\sigma,h-\Upsilon_{\Gamma,\sigma}\alpha_i\right)
\end{split}
\]
and \eqref{e:U=U} is proved. 
\end{proof}

Therefore, in order to construct a single solution map satisfying both relevant symmetry properties, we need to show that we can find such sequences in a way such that 
$\tau_\geo^{(\eps)} - C_{\eps,\geo}^\BPHZ$ and $\tau_\Ito^{(\eps)} - C_{\eps,\Ito}^\BPHZ$ both converge to the \textit{same} limit. Given the additional constraint that these sequences need
to lie in the (potentially quite small) subspaces $\CS_\geo$ and $\CS_\Ito$, it is not clear at
all that this is possible. (It is not even clear a priori that one can find sequences as
in the statement of the corollary!) 
One ingredient of our argument is the fact that
the limiting solution map $U^\BPHZ$ is injective in its last argument, which is the content of the
next section.

\subsection{Injectivity of the law}
\label{sec:injectivelaw}

In this section, we show that 
for any fixed $(\Gamma,\sigma)$, the map $h \mapsto U^\BPHZ(\Gamma,\sigma,h)$ is injective.
Note that this is a strictly stronger notion of injectivity than the one for the map 
$\CA(\Gamma,\sigma,h)$ obtained in Theorem~\ref{theo:BPHZ}
that maps the pair $(u_0,\PPi)$ to the corresponding solution. 
Indeed, as already announced in the introduction, the map $(\sigma,h) \mapsto U^\BPHZ(\Gamma,\sigma,h)$
is \textit{not} injective since we get the same law for any two choices of $\sigma$ that define the same $g$
provided that $h$ is adjusted accordingly, while we expect $(\sigma,h) \mapsto \CA(\Gamma,\sigma,h)$ to be injective.

\begin{theorem}\label{theo:injective}
For all dimensions $d$ and $m$ and smooth choices of $\Gamma$ and $\sigma$, the map
\begin{equ}
h \mapsto U^\BPHZ(\Gamma,\sigma,h)
\end{equ}
is injective as a map from $\CC^6$ to $\CB_\star^a$.
\end{theorem}

Our main ingredient in the proof of this result is the following.

\begin{lemma}\label{lem:meanzero}
Let $\PPi^\BPHZ$ denote the BPHZ model as in Theorem \ref{theo:Ajay}, let $u_0 \in \CC^a(S^1)$, and let $V$ be 
the solution in $\CD^\gamma$ to the fixed point problem
\begin{equ}
V^\alpha = Pu_0^\alpha + \CP \one_+ \bigl(\Gamma^\alpha_{\beta,\gamma}(V) \,\DD V^\beta\, \DD V^\gamma
+ h^\alpha(V) + \sigma_i^\alpha(V)\,\sXi_i\bigr)\;.
\end{equ}
Write $u^\alpha = \CR V^\alpha$, which is nothing but the continuous function 
given by the $\one$-component of $V^\alpha$.
Fix furthermore some $r > 0$ and set 
\begin{equ}[e:taudelta]
\tau_r = \inf \bigl\{t > 0\,:\, \|u - P u_0\|_{\CC^a_t} \ge r\bigr\}\;,
\end{equ}
where $\|\cdot\|_{\CC^a_t}$
denotes the space-time H\"older norm on $[0,t] \times S^1$. Then, the random $\R^d$-valued distribution
$\eta_r$ on $\R\times S^1$ given by 
\begin{equ}
\eta_r = \CR \one_{[0,\tau_r]}\bigl(\sigma_i^\alpha(V)\,\sXi_i\bigr)
\end{equ}
is such that, for every (deterministic) test function $\psi$, $\eta_r(\psi)$ coincides
with the It\^o integral
\begin{equ}[e:ItoInt]
\int_0^{\tau_r} \scal{\sigma_i^\alpha(v_s)\psi,dW_i(s)}\;,
\end{equ}
with $(t,x) \mapsto v_t(x)$ given by $\CR V$ (which is guaranteed to be a continuous function and can therefore
be evaluated at fixed time slices) 
and $W_i$ the $L^2$-cylindrical Wiener process associated to the noise $\xi_i = \CR \sXi_i$ as in the
proof of Proposition~\ref{prop:equivariance}.
In particular, these random variables have vanishing expectation.
\end{lemma}

\begin{proof}
Note first that the quantity \eqref{e:ItoInt} is well-defined since
$s \mapsto v_s$ is continuous and adapted.

To show that this is the case, we use the fact that
$\eta_r(\psi) = \lim_{\eps \to 0}  \eta_r^\eps(\psi)$, where $\eta_r^{(\eps)}$ is defined 
as above, but with both the reconstruction operator $\CR$, the modelled distribution $V$,
and the stopping time $\tau_r$ 
obtained from the model $\hPPi^{(\eps)}_\Ito$. It then follows from \cite{BCCH} 
that for any fixed $\eps >0$ one has the identity
\begin{equ}
\eta_r^{(\eps)}(\psi) = \int_0^{\tau_r} \scal{\sigma_i^\alpha(v_s)\psi,dW_i(s)}
+ \int_0^{\tau_r} \scal{\psi, \Upsilon^\alpha_{\Gamma,\sigma} P_{\<generic>} C_\Ito^{\BPHZ,\eps}}\,ds\;,
\end{equ}
where $P_{\<generic>} \colon \CS \to \CS$ is the projection onto the subspace spanned by those symbols
that have a noise edge incident to their root.

On the other hand, we already noted in the proof of Theorem~\ref{theo:Ajay}
that $P_{\<generic>} C_\Ito^{\BPHZ,\eps} = 0$, whence the claim
follows by continuity and the stability of the It\^o integral.
\end{proof}

Before we turn to the proof of Theorem~\ref{theo:injective}, we introduce a larger regularity structure than
the one considered so far. The difference is that we allow for two different edge types,
\<thin> and \<thin2>, and we will consider models that are admissible in the sense that 
edges of type \<thin> represent the kernel $K$ and edges of type \<thin2> represent $K_\eps$.
As in Section~\ref{sec:rules}, we identify \<thin> with $(\<thin>,0)$ and we write 
$\<thick2> = (\<thin2>,(0,1))$.
We then extend our rule $R$ to this larger type set by setting
\begin{equs}
\hat R(\<thin>) &= \{(\<thin>^k,\<generic>_i), (\<thick>^\ell,\<thin>^k)\,:\, k \ge 0,\, \ell \in \{0,1,2\},\, i\in \{1,\ldots,m\}\}\;,\\
\hat R(\<thin2>) &= \{(\<thin>^{k},\<thin2>^{\bar k},\<generic>_i), (\<thick>^\ell,\<thick2>^{\bar \ell},\<thin>^k,\<thin2>^{\bar k})\,:\, k,\bar k \ge 0,\, \ell+\bar \ell \in \{0,1,2\},\, i\in \{1,\ldots,m\}\}\;.
\end{equs}
We write $\hat \Trees$ for the corresponding set of labelled trees and $\hat \CT$ for 
the vector space generated by $\hat \Trees$. This is turned into a regularity structure $\hat \TT$
as in \cite[Sec.~6.4]{BHZ} (the ``reduced regularity structure'' in that lingo) and, similarly to above, 
we write $\hat \MM$ for the space of all models for $\hat \TT$ and $\hat \MM_\eps$ for the corresponding space of
\textit{admissible} models with \<thin> assigned to $K$ and \<thin2> assigned to $K_\eps$.
Since $\hat R$ extends $R$, it follows that we have a canonical inclusion $\CT \subset \hat \CT$,
as well as a canonical projection $\pi_0 \colon \hat \MM_\eps \to \MM_0$ obtained by restricting
a given model to $\CT$.
One also has a canonical projection $\hat \pi \colon \hat \CT \to \CT$ obtained by changing all edges of
type \<thin2> or \<thick2> into edges of type \<thin> or \<thick> respectively. It is immediate from the definition
of $\hat R$ that if $\tau$ is a labelled tree conforming to the rule $\hat R$, then $\hat \pi \tau$ conforms
to the rule $R$, so that $\hat \pi$ does indeed map $\hat \CT$ to $\CT$. Since furthermore
both \<thin> and \<thin2> are assigned to $K$ in the space $\hat \MM_0$, this yields
a canonical inclusion $\iota \colon \MM_0 \hookrightarrow \hat \MM_0$ obtained by right composing with $\hat \pi$.
The inclusion $\iota$ is a right inverse for the projection $\pi_0$.

For $\eps > 0$, we also have a natural extension map $\CE_\eps \colon \MM_0 \to \hat \MM_\eps$, also forming
a right inverse for $\pi_0$, which is defined
as follows. Given a model $\PPi \in \MM_0$, we define $\hat \PPi = \CE_\eps(\PPi)$ in such a way that 
$\hat \PPi$ is admissible and such that, whenever $\tau \in \hat \CT$ is of the form
$\tau = \bar \tau \cdot \tau_1\cdots \tau_k$ for $\bar \tau \in \CT$ 
and $\tau_i \in \hat \CT$ such that their root is only incident to exactly one edge of type
\<thin2> or \<thick2> (the cases $k=0$ and \slash or $\bar \tau = \one$ are allowed), then  
\begin{equ}
\hat \PPi \tau = \PPi \bar \tau \cdot \hat \PPi \tau_1\cdots \hat \PPi \tau_k\;.
\end{equ}
Since $K_\eps$ is smooth, it follows that $\hat \PPi \tau_i$ is smooth for any such symbol,
so this product is simply a product between a distribution and a number of smooth functions. 
Furthermore, this covers all of $\hat \CT$, so that it determines $\hat \PPi$ uniquely.

Since one can easily verify that $\hat R$ is complete (as a consequence of the fact that $R$ is complete),
we also have a renormalisation group $\hat \RR$ associated to $\hat \TT$, with elements of
$\hat \RR$ identified with characters of the free algebra $\sscal{\hat \Trees_-}$
generated by the trees $\tau \in \hat \Trees_-$ of strictly negative degree, as well
as a subgroup $\hat \RR_{\<generic>}$ defined analogously to $\RR_{\<generic>}$. The canonical
inclusion $\Trees_-\subset \hat \Trees_-$ allows to identify $\RR$ with
a subgroup of $\hat \RR$. We also have a natural normal subgroup $\RR^\perp \subset \hat \RR$ consisting
of those group elements which leave $\CT$ invariant. This corresponds precisely to the 
linear functionals that vanish on $\Trees_-$.
Similarly, we define the subgroup
$\RR_{\<generic>}^\perp = \hat \RR_{\<generic>} \cap \RR^\perp$.

Note that the extension operator $\CE_\eps$ is only well-defined for $\eps > 0$ and does \textit{not} 
converge to the operator $\iota$ as $\eps = 0$. However, we have the following result which is the other main
ingredient in the proof of Theorem~\ref{theo:injective}.

\begin{lemma}\label{lem:extend}
Let $\PPi^\BPHZ$ be the (random) BPHZ model on $\TT$ constructed in Theorem~\ref{theo:Ajay}. Then,
there exists a sequence of elements $C_\eps \in \RR_{\<generic>}^\perp$, so that the sequence of models
\begin{equ}
\hPPi_\eps = C_\eps \star \CE_\eps(\PPi^\BPHZ) \;,
\end{equ} 
converges in $\hat \MM$ to the model $\iota \PPi^\BPHZ$.
\end{lemma}

\begin{proof}
We consider the models $\PPi_{\eps,\delta} = \CE_\eps(\hPPi_\geo^{(\delta)})$ and look at
the sequence 
\begin{equ}
\hPPi_{\eps,\delta} = C_{\eps,\delta}^\BPHZ \star \PPi_{\eps,\delta} \;,
\end{equ} 
where $C_{\eps,\delta}^\BPHZ$ is the corresponding BPHZ character which is characterised by the fact
that $\hPPi_{\eps,\delta}$ vanishes on elements of negative degree.

As in Remark~\ref{rem:noises}, we have the explicit formula
\begin{equ}
C_{\eps,\delta}^\BPHZ = \hat \CA^t g_{\eps,\delta}\;,
\end{equ}
where $\hat \CA^t \colon \hat \Trees_- \to \sscal{\hat \Trees}$ is the twisted antipode for the
regularity structure $\hat \TT$ and $g_{\eps,\delta}$ is given by
\begin{equ}
g_{\eps,\delta}(\tau) = \E \bigl(\PPi_{\eps,\delta}\tau\bigr)(0)\;.
\end{equ}
It follows immediately from the 
definition of the twisted antipode \cite[Eq.~6.24]{BHZ} that $\hat \CA^t$ coincides with $\CA^t$ on 
$\Trees_- \subset \hat \Trees_-$. Since we furthermore have $g_{\eps,\delta}(\tau) = 0$
for $\tau \in \Trees_-$ (this is precisely what characterises $\hPPi_\geo^{(\delta)}$,
which furthermore coincides with $\PPi_{\eps,\delta}$ on those elements), it follows that 
$C_{\eps,\delta}^\BPHZ \in \RR_{\<generic>}^\perp$.

By stability of the BPHZ model, $\hPPi_{\eps,\delta}$ converges to $\iota \PPi^\BPHZ$ as $\eps,\delta \to 0$
in whichever way one takes these limits,
so it only remains to show that $g_{\eps,\delta}$ (and therefore also $C_{\eps,\delta}^\BPHZ$) 
has a limit as $\delta \to 0$ for any fixed $\eps$.
For this, we note that the only elements $\tau \in \hat \Trees_-$ for which 
$\PPi_{\eps,\delta}\tau$ is not a stationary process are $\<XiX>_i$, as well as those 
modelled (in the sense that some lines may be replaced by dotted lines) on the first 
four elements of the last line of the table of symbols on page~\pageref{listPage}. These however are
odd (in law) in the spatial variable, so that 
\begin{equ}
 g_{\eps,\delta}(\tau) = 0 = \E \bigl(\PPi_{\eps,\delta}\tau\bigr)(\phi)\;,
\end{equ}
for any test function $\phi$ that is even in the spatial variable. It follows that,
for any such test function integrating to $1$, one has
\begin{equ}
g_{\eps,\delta}(\tau) = \E \bigl(\PPi_{\eps,\delta}\tau\bigr)(\phi)\;.
\end{equ}
Since the model $\PPi_{\eps,\delta}$ converges to the finite limit $\CE_\eps(\PPi^{\BPHZ})$ as $\delta \to 0$,
the claim follows.
\end{proof}

\begin{remark}
Since $C_\eps \in \RR_{\<generic>}^\perp$ and since $(\pi_0 \circ \CE_\eps)(\PPi^\BPHZ) = \PPi^\BPHZ$, one has
$\pi_0 (\hPPi_\eps) = \PPi^\BPHZ$ for every $\eps > 0$.
\end{remark}
We now have all the ingredients in place to give a proof of Theorem~\ref{theo:injective}.

\begin{proof}[of Theorem~\ref{theo:injective}]
Our main ingredient is the fact that we can find a measurable function $F_{\Gamma, h} \colon \CC^a_\star \to \CD'$ with the property that, if we denote by $\bar u$ a random variable with
law $U^\BPHZ(\Gamma,\sigma, \bar h)(u_0)$, then the law of
$F_{\Gamma, h}(\bar u)$ coincides with that of the random variable
\begin{equ}[e:laweta]
\eta_r + \one_{[0,\tau_r]} \bigl( h(\bar u) - \bar h(\bar u)\bigr)\;.
\end{equ}
Assume for the moment that this is indeed the case and, for given $\Gamma$ and $\sigma$,
choose any two vector fields $h \neq \bar h$. In particular there exist $v,w \in \R^d$ and $r > 0$
such that 
\begin{equ}[e:choicevw]
\scal{w, h(v') - \bar h(v')} \ge 1\;,
\end{equ}
for all $v'$ such that $|v'-v| \le r$. 
Write $u$ for a random variable with law $U^\BPHZ(\Gamma,\sigma,h)(v)$, 
$\bar u$  for a random variable with law $U^\BPHZ(\Gamma,\sigma,\bar h)(v)$, and 
let $\tau_r(u)$ be defined as in \eqref{e:taudelta} with $u_0 = v$.

Since $\tau_r > 0$ almost surely, we can find $\delta$ sufficiently small so that 
$\P(\tau_r(u) > \delta) > {1\over 2}$ and $\P(\tau_r(\bar u) > \delta) > {1\over 2}$.
Choose then a smooth real-valued test function $\phi$ integrating to $1$ and 
supported on $[0,\delta] \times S^1$.
By \eqref{e:laweta}, \eqref{e:choicevw}, the fact that $\eta_r$ has vanishing expectation, 
and our choice of $\delta$, it then follows that
\begin{equ}
\scal{F_{\Gamma, h}(u),w\phi} = 0\;,\qquad
\scal{F_{\Gamma, h}(\bar u),w\phi} > {1\over 2}\;,
\end{equ}
so that the laws of $u$ and $\bar u$ are indeed distinct as claimed.

It remains to show that one can construct an $F_{\Gamma,h}$ such that \eqref{e:laweta} holds.
For this, consider the (unique) solutions to the following system of equations
\begin{equs}
V^\alpha &= Pu_0^\alpha + \CP \one_{[0,\tau_r]} \bigl(\Gamma^\alpha_{\beta,\gamma}(V) \,\DD V^\beta\, \DD V^\gamma
+ h^\alpha(V) + \sigma_i^\alpha(V)\,\sXi_i\bigr)\;,\label{e:system}\\
\bar V_\eps^\alpha &= \rho_\eps * Pu_0^\alpha + \CP_\eps \one_{[0,\tau_r]} \bigl(\Gamma^\alpha_{\beta,\gamma}(V) \,\DD V^\beta\, \DD V^\gamma
+ h^\alpha(V) + \sigma_i^\alpha(V)\,\sXi_i\bigr)\;,
\end{equs}
where $\tau_r$ is as above, based on $u = \CR V$.
It follows from the definitions that, for any random model $\hat \PPi \in \hat \MM_\eps$ such that 
$\pi_0 (\hat \PPi) = \PPi^\BPHZ$ (in law), the law of $\CR V^\alpha$ 
coincides with the law of $U^\BPHZ(\Gamma,\sigma, \bar h)(u_0)$ (and $V$ is indeed independent of $\eps$). 
Furthermore, the identity
$\CR \bar V_\eps^\alpha = \rho_\eps * \CR V^\alpha$ holds for every model $\hat \PPi \in \hat \MM_\eps$.
Finally, for any model of the for $\iota \PPi$ for some $\PPi \in \MM_0$, 
$\CR \bar V_0 = \CR V$. In fact, a stronger statement is true, namely
\begin{equ}[e:equalVVbar]
\hat \pi \bar V_0(t,x) = V(t,x)\;,
\end{equ}
and the reconstruction operator $\CR$ is such that $\CR \hat \pi U = \CR U$ for any modelled
distribution $U$.

It then follows from \eqref{e:system} and the fact that we only consider
admissible models that $u = \CR V$ satisfies the distributional identity
\begin{equ}[e:diffequ]
\d_t u^\alpha = \d_x^2 u^\alpha + u_0 \delta_t + \CR \one_{[0,\tau_r]}\bigl(\Gamma^\alpha_{\beta,\gamma}(V) \,\DD V^\beta\, \DD V^\gamma
+ h^\alpha(V)\bigr) + \eta_r^\alpha\;,
\end{equ}
with $\eta_r$ as above, where $\delta_t$ denotes the Dirac distribution centred at the origin in the 
time variable.
Consider now the modelled distributions
\begin{equ}
X_\eps^\alpha = \one_{[0,\tau_r]} \bigl(\Gamma^{\alpha}_{\beta\gamma}(\bar V_\eps)\,\CD \bar V_\eps^\beta\,\CD \bar V_\eps^\gamma \bigr)\;, \qquad
Y^\alpha = \one_{[0,\tau_r]} \bigl(\Gamma^{\alpha}_{\beta\gamma}(V)\,\CD  V^\beta\,\CD V^\gamma\bigr)\;.
\end{equ}
We make use of the following facts.
\begin{enumerate}
\item For any model of the type $\iota \PPi$ with $\PPi \in \MM_0$, one has
$\hat \pi X_0 = Y$ as a consequence of \eqref{e:equalVVbar}. In particular, for any such model
one has $\CR X_0 = \CR Y$.
\item The map mapping the underlying model to $(\CR X_\eps, \CR Y)$ is continuous from
$\hat \MM_\eps$ to $\CD'$. This follows from the results of \cite{BCCH} in virtually the same way as 
the continuity statement of Theorem~\ref{theo:BPHZ}.
\end{enumerate}
Applying this to the specific choice of model given by $\hPPi_\eps$ as in Lemma~\ref{lem:extend},
we conclude that for this choice one has $\CR X_\eps \to \CR Y$ in probability.
On the other hand, it follows from the results of \cite{BCCH} that there exists a function 
$G\colon \R^d \times \R^d \to \R^d$ depending on $\Gamma$ and $\sigma$ but not on $h$ and such that 
one has the identity
\begin{equ}
\CR X_\eps = \one_{[0,\tau_r]} \bigl(\Gamma^{\alpha}_{\beta\gamma}(u_\eps)\,\d_x u_\eps^\beta\,\d_x u_\eps^\gamma + G(u_\eps, u)\bigr)\;.
\end{equ}
(Indeed, it suffices to view $\hat V_\eps = \CP X_\eps$ as an additional component of \eqref{e:system} and to note that
the triangular structure of this system guarantees that the corresponding renormalisation term does not
involve $\hat V_\eps$ itself. The reason why the spatial derivatives of $u$ and $u_\eps$ do not appear
is the same as previously.)

Combining these facts together with \eqref{e:diffequ}, it follows that one has
the almost sure identity
\begin{equ}
\eta_r^\alpha = \d_t u^\alpha - \d_x^2 u^\alpha - u_0 \delta_t - h^\alpha(u)- \lim_{\eps \to 0} \one_{[0,\tau_r]} \bigl(\Gamma^{\alpha}_{\beta\gamma}(u_\eps)\,\d_x u_\eps^\beta\,\d_x u_\eps^\gamma + G(u_\eps, u)\bigr)\;,
\end{equ}
where the convergence of the last term takes place in probability in $\CD'$ and therefore almost surely
along a suitable sequence $\eps \to 0$. The fact that we obtain \eqref{e:laweta} if we replace $h$ by $\bar h$
in \eqref{e:system} is immediate from the construction.  
\end{proof}

\subsection{Equivariance of solutions}
\label{sec:geo}

Recall the definition of $U^\geo_\eps$ above and set
$\bar U(\Gamma,h) = U^\geo_\eps(\Gamma,0,h)$,
which is of course deterministic and independent of $\eps$.
Recall also that $\CS_\geo \subset \CS$ from Definition~\ref{def:geoIto} denotes the space of all `geometric' counterterms.
The aim of this section is to show that the counterterm $C_{\eps,\geo}^\BPHZ$ is ``mostly'' contained
in $\CS_\geo$ in the sense that its projection onto the complement of $\CS_\geo$ converges to a finite limit.

For this, we first fix some complement $\CS_\geo^\perp$ of $\CS_\geo$ in $\CS$ so that 
$\CS = \CS_\geo \oplus \CS_\geo^\perp$. For definiteness, we could take the orthogonal
complement with respect to the scalar product introduced in Section~\ref{sec:reduced}, but this is 
of no particular importance. 
We also fix a mollifier $\rho \in \Moll$ and decompose the corresponding ``geometric'' BPHZ counterterm 
$C_{\eps,\geo}^\BPHZ$
as $C_{\eps,\geo}^\BPHZ = C_\eps^g + C_\eps^c$ with $C_\eps^g \in \CS_\geo$ and $C_\eps^c \in \CS_\geo^\perp$. 
The main result of this section is the following.

\begin{proposition}\label{prop:geo}
There exists $v_\geo \in \CS_\geo^\perp$ such that $\lim_{\eps \to 0} C_\eps^c = v_\geo$.
Furthermore, $v_\geo$ is independent of the choice of mollifier $\rho$.
\end{proposition}

\begin{remark}
Here and below, although the constants are independent of the choice of mollifier, they do
in general depend on the choice of cutoff $K$ of the heat kernel used in constructing our models.
\end{remark}

\begin{proof}
Consider the decomposition $\CS = \CS^{(2)} \oplus \CS^{(4)}$ according to
how many noises appear in a given symbol. Since $C_{\eps,\geo}^\BPHZ\in\iota\CS$ by Lemma \ref{lem:iota}, we can write $C_{\eps,\geo,k}^\BPHZ$ for the component
of $C_{\eps,\geo}^\BPHZ$ in $\CS^{(k)}$ and similarly for $C_{\eps,k}^g$ and $C_{\eps,k}^c$.

We first show that $C_\eps^c$ is bounded. Assume by contradiction that it is not
so that, at least along a some subsequence $\eps \to 0$, 
\begin{equ}
\lim_{\eps \to 0} r_\eps = +\infty \;,\quad r_\eps = r_{\eps,2} + r_{\eps,4}\;,\quad r_{\eps,k} =  |C_{\eps,k}^c|^{1/k}\;. 
\end{equ}
Set furthermore $\alpha_\eps = 1/r_\eps$. 
It is immediate that the pair $(\alpha_\eps^2 C_{\eps,2}^c,\alpha_\eps^4 C_{\eps,4}^c) \in (\CS_\geo^\perp)^2$
remains uniformly bounded as $\eps \to 0$. It also remains uniformly bounded away from the origin since 
one has either $r_{\eps,2} \ge r_{\eps,4}$ in which case $|\alpha_\eps^2 C_{\eps,2}^c| \ge 1/4$ or
$r_{\eps,4} \ge r_{\eps,2}$ in which case $|\alpha_\eps^4 C_{\eps,4}^c| \ge 1/16$. Modulo extracting a further
subsequence, we can therefore assume that the pair converges to a non-degenerate limit $(\hat C_2,\hat C_4) \in (\CS_\geo^\perp)^2$.
Note that we also have 
\begin{equ}
v = \hat C_2 + \hat C_4 \neq 0\;,
\end{equ}
since $\hat C_k \in \CS^{(k)}$ and these spaces are transverse.

It follows from the definition of $\Upsilon_{\Gamma,\sigma}$ that, for $\tau \in \CS^{(k)}$ and $r \in \R$, one has
\begin{equ}
\Upsilon_{\Gamma, r\sigma} \tau = r^k\Upsilon_{\Gamma, \sigma} \tau\;.
\end{equ}
Combining this with Theorem~\ref{theo:BPHZ} shows that, 
for every sequence $\alpha_\eps \to 0$, we have
the convergence in probability in $\CB_\star^a$
\begin{equ}
\lim_{\eps \to 0} U_\eps^\geo(\Gamma,\alpha_\eps \sigma,h + \alpha_\eps^2 \Upsilon_{\Gamma,\sigma} C_{\eps,\geo,2}^\BPHZ + \alpha_\eps^4\Upsilon_{\Gamma,\sigma} C_{\eps,\geo,4}^\BPHZ)
=  \bar U(\Gamma,h)\;.
\end{equ}
With our particular choice of $\alpha_\eps$, this immediately implies that 
(along the subsequence chosen above)
\begin{equ}[e:limitDet]
\lim_{\eps \to 0} U_\eps^\geo(\Gamma,\alpha_\eps \sigma, \Upsilon_{\Gamma,\alpha_\eps \sigma} C_\eps^g)
=  \bar U(\Gamma,-\Upsilon_{\Gamma,\sigma}v)\;.
\end{equ}
Given any diffeomorphism $\phi$ homotopic to the identity, we now conclude from applying \eqref{e:limitDet} 
twice and the fact that the deterministic solution map is equivariant for the diffeomorphism group that
\begin{equs}
\bar U(\phi \act \Gamma,-\phi \act \Upsilon_{\Gamma,\sigma}v)
&=
\phi \act \bar U(\Gamma,-\Upsilon_{\Gamma,\sigma}v) 
= 
\lim_{\eps \to 0} \phi \act U_\eps^\geo(\Gamma,\alpha_\eps \sigma, 
\Upsilon_{\Gamma,\alpha_\eps\sigma} C_\eps^g) \\
&= 
\lim_{\eps \to 0} U_\eps^\geo(\phi\act\Gamma,\alpha_\eps \phi\act\sigma, 
\Upsilon_{\phi \act\Gamma,\alpha_\eps\phi \act\sigma} C_\eps^g) \\
&=
 \bar U(\phi \act\Gamma,-\Upsilon_{\phi \act\Gamma,\phi \act\sigma}v)\;.
\end{equs}
It follows that $\phi \act \Upsilon_{\Gamma,\sigma}v = \Upsilon_{\phi \act\Gamma,\phi \act\sigma}v$
as a consequence of the injectivity of $\bar U$ in its second argument
which, since $d$, $\phi$, $\Gamma$ and $\sigma$ were arbitrary, implies that 
$v \in \CS_\geo$ by definition. This however is in contradiction with the fact that
$v \in \CS_\geo^\perp$ and $|v| \neq 0$.

Having shown that $C_\eps^c$ is bounded, it remains to show that it can only have one
accumulation point. 
Recall that, by the definition of $U^\BPHZ$ and Theorem~\ref{theo:Ajay}, one has
\begin{equ}
U^\BPHZ(\Gamma,\sigma,h) = \lim_{\eps \to 0} U_\eps^\geo(\Gamma,\sigma, h+ \Upsilon_{\Gamma,\sigma} C_{\eps,\geo}^\BPHZ) \;.
\end{equ}
Acting again with an arbitrary diffeomorphism $\phi$ and applying Proposition~\ref{prop:equivariance}, we conclude that
\begin{equs}
\phi&\act U^\BPHZ(\Gamma,\sigma,0)
=
\lim_{\eps \to 0} U_\eps^\geo(\phi\act\Gamma,\phi\act\sigma, \phi\act \Upsilon_{\Gamma,\sigma} C_{\eps,\geo}^\BPHZ) \\
&=
\lim_{\eps \to 0} U_\eps^\geo(\phi\act\Gamma,\phi\act\sigma, \Upsilon_{\phi\act \Gamma,\phi\act \sigma} C_{\eps,\geo}^\BPHZ + \bigl(\phi\act \Upsilon_{\Gamma,\sigma}- \Upsilon_{\phi\act\Gamma,\phi\act\sigma}\bigr) C_\eps^c)\;,
\end{equs}
which implies that the identity
\begin{equ}[e:magic]
\phi \act U^\BPHZ(\Gamma,\sigma,0)
= U^\BPHZ(\phi \act \Gamma,\phi \act \sigma, \bigl(\phi\act \Upsilon_{\Gamma,\sigma}- \Upsilon_{\phi\act\Gamma,\phi\act\sigma}\bigr)v)\;,
\end{equ}
holds for any accumulation point $v$ of $\{C_\eps^c\}_{\eps \le 1}$. Since
$U^\BPHZ$ is injective in its last argument by Theorem~\ref{theo:injective}
and since this argument holds for any choice of $\Gamma$, $\sigma$ and $\phi$ homotopic to the identity, 
we conclude that
any two such accumulation points $v$ and $\bar v$ necessarily satisfy
\begin{equ}
\bigl(\phi\act \Upsilon_{\Gamma,\sigma}- \Upsilon_{\phi\act\Gamma,\phi\act\sigma}\bigr)(v-\bar v) = 0\;,
\end{equ}
for all such $\Gamma$, $\sigma$ and $\phi$. This is precisely the definition of $\CS_\geo$, so that
$v-\bar v \in \CS_\geo$, but since $v, \bar v \in \CS_\geo^\perp$ by construction, 
we conclude that $v = \bar v$ as claimed.

To show that $v$ does not depend on the choice of mollifier either, we use \eqref{e:magic} again in
the same way.
\end{proof}

\subsection{It\^o isometry}
\label{sec:Ito1}

We now show a statement analogous to Proposition~\ref{prop:geo}, but this time regarding the ``It\^o isometry''.
Again, we fix a mollifier $\rho \in \Moll$ (symmetric, compactly supported, and non-anticipative) and decompose 
the corresponding ``It\^o'' BPHZ counterterm $C_{\eps,\Ito}^\BPHZ$
as $C_{\eps,\Ito}^\BPHZ = C_\eps^I + C_\eps^c$ with $C_\eps^I \in \CS_\Ito$ and $C_\eps^c \in \CS_\Ito^\perp$
for some fixed choice of complement $\CS_\Ito^\perp$ of $\CS_\Ito$ in $\CS$ so that 
$\CS = \CS_\Ito \oplus \CS_\Ito^\perp$. 
The main result of this section is then the following.

\begin{proposition}\label{prop:ito}
There exists $v_\Ito \in \CS_\Ito^\perp$ independent of the mollifier $\rho$ 
such that $\lim_{\eps \to 0} C_\eps^c = v_\Ito$.
\end{proposition}

\begin{proof}
The proof is virtually identical to that of Proposition~\ref{prop:geo}, so we only sketch it.
One first shows that $C_\eps^c$ is bounded since otherwise one can again find
$\alpha_\eps \to 0$ and $v \in \CS_\Ito^\perp$ with $v \neq 0$ such that, along some subsequence,
\begin{equ}[e:limit1]
\lim_{\eps \to 0} U_\eps^\Ito(\Gamma,\alpha_\eps \sigma, \Upsilon_{\Gamma,\alpha_\eps \sigma} C_\eps^I)
=  \bar U(\Gamma,-\Upsilon_{\Gamma,\sigma}v)\;.
\end{equ}
Choosing any $\bar \sigma$ such that $\bar \sigma_i^\alpha \bar \sigma_i^\beta = \sigma_i^\alpha \sigma_i^\beta$, we similarly have
\begin{equ}[e:limit2]
\lim_{\eps \to 0} U_\eps^\Ito(\Gamma,\alpha_\eps \bar \sigma, \Upsilon_{\Gamma,\alpha_\eps \bar \sigma} C_\eps^I)
=  \bar U(\Gamma,-\Upsilon_{\Gamma,\bar \sigma}v)\;.
\end{equ}
On the other hand, one has $U_\eps^\Ito(\Gamma,\alpha_\eps \sigma, \Upsilon_{\Gamma,\alpha_\eps \sigma} C_\eps^I)
= U_\eps^\Ito(\Gamma,\alpha_\eps \bar \sigma, \Upsilon_{\Gamma,\alpha_\eps \bar \sigma} C_\eps^I)$ by the
second part of Proposition~\ref{prop:equivariance}
and the definition of $\CS_\Ito$, so that one must have $\Upsilon_{\Gamma,\sigma}v = \Upsilon_{\Gamma,\bar\sigma}v$,
which is in contradiction with the fact that $v \in  \CS_\Ito^\perp$ does not vanish.

The argument that there can be only one accumulation point and that its value is independent of the
mollifier is very similar, except that this time one obtains the identity 
\begin{equ}
U^\BPHZ(\Gamma,\sigma,0)
= U^\BPHZ(\Gamma,\bar \sigma, \bigl(\Upsilon_{\Gamma,\sigma}- \Upsilon_{\Gamma,\bar\sigma}\bigr)v)\;,
\end{equ}
from \eqref{e:limit1} and \eqref{e:limit2}, so we omit it for conciseness.
\end{proof}

We recall that we defined on page \pageref{S nice page ref} the subspace $\CS^\nice \subset \CS$ consisting of those
elements $\tau \in \CS$ such that, whenever $\Gamma$, $\sigma$ and $x$ are such that 
$\Gamma(x) = 0$ and $\d \sigma(x) = 0$, one has $\bigl(\Upsilon_{\Gamma,\sigma}\tau\bigr)(x) = 0$.
This space is characterised as follows.

\begin{proposition}\label{prop:ortho}
The space $\CS^\nice$ consists precisely of those elements $\tau \in \CS$ such that 
\begin{equ}[e:charSnice]
\scal{\<Xi4ba1b>,\tau} = \scal{\<Xi4ba2>,\tau} = \scal{\<Xi4b1>,\tau} = 0\;.
\end{equ}
Furthermore, one has $\<Xi4b1> \perp \CS_\Ito$, ${1\over 2}\<Xi4b1> - \<Xi4ba2> -  {1\over 2} \<Xi4ba1b> \perp \CS_\geo$.
\end{proposition}

\begin{proof}
We start by noting that $\<Xi4b>$ and $\<Xi4ba>$ are the only two symbols
in the list on page~\pageref{listPage} which contain neither a node $\<generic>$ with exactly one
incoming edge (corresponding to a factor $\d \sigma$), nor a node $\<not>$ with no
incoming thin edge (corresponding to a factor $\Gamma$). The three trees appearing in \eqref{e:charSnice}
are the only inequivalent ways of pairing the noises of these two symbols.
It follows immediately that, denoting by $\hat \CS^\nice \subset \CS$ the subspace space such that 
\eqref{e:charSnice} holds, one has $\hat \CS^\nice \subset \CS^\nice$.
The converse inclusion relies on the injectivity results obtained in Theorem~\ref{theo:injectiveT} below
and the description of $\CS$ used in the proof of Proposition~\ref{prop:rangeIto}. 
We postpone its proof to Corollary~\ref{cor:charSnice} below.

In Theorem \ref{theo:mainSum} below it is proven that $\CS_\Ito=\scal{\CB_\Ito}$, with $\CB_\Ito$ given in \eqref{Bito}; from this
it is easy to see that $\<Xi4b1> \perp \CS_\Ito$.
The fact that ${1\over 2}\<Xi4b1> - \<Xi4ba2> -  {1\over 2} \<Xi4ba1b> \perp \CS_\geo$ is shown in Proposition \ref{prop:geoDim} below.
\end{proof}

\begin{remark}\label{rem:codimgeonice}
An immediate consequence of this and of Proposition~\ref{prop:geoDim} is that 
$\CS^\nice \cap \CS_\geo$ is a subspace of $\CS_\geo$ that is of codimension $2$.
\end{remark}

\begin{corollary}\label{cor:both}
The assumption of Corollary~\ref{cor:equivariant} holds for every choice of 
mollifier $\rho \in \Moll$. Furthermore, for $i \in \types$ and  for every choice of $\rho \in \Moll$,
we have $C_{\eps,i}^\BPHZ \in \CS^\nice$,
and we can also choose $\tau_i^{(\eps)} \in \CS^\nice$.
\end{corollary}

\begin{proof} \label{subspace of CS nice}
We first note that both $C_{\eps,\geo}^\BPHZ$ and $C_{\eps,\Ito}^\BPHZ$ do indeed belong to 
$\CS^\nice$ irrespective of the choice of mollifier $\rho$ by Lemma~\ref{const:disconnect}. 

Now, fix $i \in \types$ and write $\pi_i \colon \CS^\nice \to \CS_i^\nice$ for the 
projection associated to the decomposition 
$\CS^\nice=\CS_i^\nice\oplus(\CS_i^\nice)^\perp$, where $\CS_i^\nice = \CS^\nice \cap \CS_i$.
It then suffices to set $\tau_i^{(\eps)} = \pi_i C_{\eps,i}^\BPHZ + \tilde \tau_i$
for any fixed element $\tilde \tau_i \in \CS_i^\nice$.
\end{proof}

\begin{remark}
We could of course have simply set $\tilde \tau_i=0$, but leaving these two elements 
free will allow us to 
adjust them in the proof of Theorem \ref{thm:bullet} below in such a way 
that $\tau_\geo^{(\eps)} - C_{\eps,\geo}^\BPHZ$ and 
$\tau_\Ito^{(\eps)} - C_{\eps,\Ito}^\BPHZ$ converge to the same limit as $\eps \to 0$
as already announced at the end of Section~\ref{sec:symapprox}.
\end{remark}

\subsection{Combining both}
\label{sec:both}

Consider now the solution maps $U^\geo$ and $U^\Ito$ that 
are given by combining Corollaries~\ref{cor:equivariant} and
\ref{cor:both}. We already know from \eqref{e:U=U} and Propositions~\ref{prop:geo} and~\ref{prop:ito} that 
both notions of solution differ from the BPHZ solution by a fixed element in $\CS$, so that there
exists $\tau_0 \in \CS$ (equal to $\alpha_\geo-\alpha_\Ito$, in the notation of \eqref{e:U=U}) such that
\begin{equ}[e:deftau]
U^\geo(\Gamma,\sigma,h) = U^\Ito\big(\Gamma,\sigma,h + \Upsilon_{\Gamma,\sigma} \tau_0\big)\;,
\end{equ}
for every choice of $d$, $\Gamma$, $\sigma$, and $h$.
Since $U^\Ito$ isn't covariant under changes of variables, there is however
no a priori reason for $\Upsilon_{\Gamma,\sigma} \tau_0$ to transform like a vector field. 

Take now a different collection of vector fields $\bar \sigma$ such that \eqref{e:sigmabar} holds.
Then, one has 
\begin{equs}
U^\geo(\Gamma,\bar \sigma,h) &= U^\Ito\big(\Gamma,\bar \sigma,h + \Upsilon_{\Gamma,\bar \sigma} \tau_0\big)
 = U^\Ito\big(\Gamma, \sigma,h + \Upsilon_{\Gamma,\bar \sigma} \tau_0\big) \\
&= U^\geo\big(\Gamma, \sigma,h + (\Upsilon_{\Gamma,\bar \sigma} - \Upsilon_{\Gamma,\sigma})\,\tau_0\big)\;.
\end{equs}
On the other hand, we know that, for any diffeomorphism $\phi$ of $\R^d$, if $\sigma$ and $\bar \sigma$
satisfy \eqref{e:sigmabar}, then one also has
\begin{equ}
(\phi \act \bar \sigma_i)^\alpha (\phi \act \bar \sigma_i)^\beta = 
(\phi \act \sigma_i)^\alpha (\phi \act \sigma_i)^\beta\;.
\end{equ}
As a consequence,
\begin{equs}
\phi \act U^\geo(\Gamma,\bar \sigma,h)
&= U^\geo(\phi\act \Gamma,\phi\act \bar \sigma,\phi\act h) \\
&= U^\geo\big(\phi\act \Gamma, \phi\act \sigma,\phi\act h + (\Upsilon_{\phi\act\Gamma,\phi\act \bar \sigma} - \Upsilon_{\phi\act\Gamma,\phi\act\sigma})\,\tau_0\big)\;,
\end{equs}
as well as
\begin{equs}
\phi \act U^\geo(\Gamma,\bar \sigma,h)
&= \phi\act U^\geo\big(\Gamma, \sigma,h + (\Upsilon_{\Gamma,\bar \sigma} - \Upsilon_{\Gamma,\sigma})\,\tau_0\big) \\
&= U^\geo\big(\phi\act \Gamma, \phi\act \sigma,\phi\act h + \phi\act (\Upsilon_{\Gamma,\bar \sigma} - \Upsilon_{\Gamma,\sigma})\,\tau_0\big)\;.
\end{equs}
Since $U^\geo$ is injective in its last argument, we conclude that $\tau_0$ is such that,
for any choice of $d$, $\Gamma$, any pair $\sigma, \bar \sigma$ such that 
\eqref{e:sigmabar} holds, and any diffeomorphism $\phi$ of $\R^d$ homotopic to the identity, one has
\begin{equ}[e:proptau]
\phi \act (\Upsilon_{\Gamma,\bar \sigma} - \Upsilon_{\Gamma,\sigma})\, \tau_0
= 
(\Upsilon_{\phi \act\Gamma,\phi \act\bar \sigma} - \Upsilon_{\phi \act\Gamma,\phi \act\sigma})\,\tau_0\;.
\end{equ}
These calculations suggest the introduction of a space which combines the two properties of our 
solution theories.

\begin{definition} \label{def_S_both}
Denote by $\CS_\both \subset \CS$ the subspace of elements such that \eqref{e:proptau} holds for
every choice of $d$, $\Gamma$, $\sigma$, $\bar \sigma$ and $\phi$ as above. 
\end{definition}

With $R$ denoting the Riemann curvature tensor 
\begin{equ}[e:R]
R(X,Y)Z = \nabla_X \nabla_Y Z - \nabla_Y \nabla_X Z - \nabla_{[X, Y]} Z\;,
\end{equ}
we define elements $\tau_\star, \tau_c \in \CS$ by\label{def tauc}\label{def taustar2}
\begin{equs}
\tau_\star &= R(\<generic>,\Nabla_{\<genericb>}\<generic> - 2\Nabla_{\<generic>}\<genericb> )\,\<genericb> \;,\label{e:taustar}\\
\tau_c &= 
\Nabla_{\<generic>} \bigl(R(\<genericb>,\<generic>)\,\<genericb>\bigr)
- R(\Nabla_{\<generic>}\<genericb>,\<generic>)\,\<genericb>
- R(\<genericb>,\Nabla_{\<generic>}\<generic>)\,\<genericb>
- R(\<genericb>,\<generic>)\Nabla_{\<generic>}\<genericb>\;\label{e:tauc}.
\end{equs}
Here, given $\tau, \bar \tau \in \Trees_0$, $\nabla_\tau \bar \tau \in \Vec(\Trees_0)$ is 
the element such that 
$\Upsilon_{\Gamma,\sigma}(\nabla_\tau \bar \tau) = \nabla_{\Upsilon_{\Gamma,\sigma} \tau}
\Upsilon_{\Gamma,\sigma} \bar \tau$. For a precise algebraic definition and a suggestive
pictorial representation, see \eqref{e:nabla} and \eqref{e:covar} below.

An explicit calculation shows that  $8\tau_\star$ is equal to 
\begin{equ}[eq:valeurR]
4\,(\<Xi4eac1>+\<Xi4eabisc1>) - 8\,\<Xi4eabisc2> - 2\,\<Xi4eabbisc1>  + 2\,\<Xi4eabc1>+ 2\,\<I1Xi4acc1>-4\,\<I1Xi4acc2>- \<I1Xi4abcc1>+ 2\,\<2I1Xi4cc1> + \<2I1Xi4c1>\;
\end{equ}
and $8\tau_c$ is equal to 
\begin{equ}[eq:valeurC]
4 \<Xi4ba2> - 4 \<Xi4ba1> +4 \<Xi4cabc1> -4 \<Xi4cabc2> +2  \<Xi4eabc2> -2  \<Xi4eabc1> -4 \<Xi4eabbisc2> +4 \<Xi4eabbisc1> +2\<I1Xi4abcc1> - 2 \<I1Xi4abcc2>     
+  \<2I1Xi4c2>   -  \<2I1Xi4c1>\;.
\end{equ}

\begin{lemma}\label{lem:exprtaucstar}
We have the identities
\begin{equs}
\Upsilon^\alpha_{\Gamma,\sigma} \tau_\star = - R^\alpha_{\eta\beta\gamma} \, g^{\beta\zeta} 
\, (\nabla_{\zeta}g)^{\gamma\eta}\;,\qquad
\Upsilon^\alpha_{\Gamma,\sigma} \tau_c = - (\nabla_\zeta R)^\alpha_{\beta\gamma\eta} \, g^{\zeta\gamma} g^{\beta\eta}\;.
\end{equs}
\end{lemma}

\begin{proof}
For the first identity, we start from the expression \eqref{e:taustar}, which yields
\begin{equs}
\Upsilon^\alpha_{\Gamma,\sigma} \tau_\star &= R^\alpha_{\eta\beta\gamma} \left( \sigma_j^\beta (\nabla_{\sigma_i}\sigma_j)^\gamma \sigma_i^\eta -2\sigma_i^\beta  (\nabla_{\sigma_i}\sigma_j)^\gamma \sigma_j^\eta\right) \\
&= R^\alpha_{\eta\beta\gamma} \left(- \sigma_j^\gamma (\nabla_{\sigma_i}\sigma_j)^\eta \sigma_i^\beta- \sigma_j^\eta (\nabla_{\sigma_i}\sigma_j)^\beta \sigma_i^\gamma -2\sigma_i^\beta  (\nabla_{\sigma_i}\sigma_j)^\gamma \sigma_j^\eta\right) \\
&= R^\alpha_{\eta\beta\gamma} \left(- \sigma_j^\gamma (\nabla_{\sigma_i}\sigma_j)^\eta \sigma_i^\beta + \sigma_j^\eta (\nabla_{\sigma_i}\sigma_j)^\gamma \sigma_i^\beta -2\sigma_i^\beta  (\nabla_{\sigma_i}\sigma_j)^\gamma \sigma_j^\eta\right) \\
&= R^\alpha_{\eta\beta\gamma} \left(- \sigma_i^\beta(\nabla_{\sigma_i}\sigma_j)^\eta \sigma_j^\gamma   -\sigma_i^\beta  (\nabla_{\sigma_i}\sigma_j)^\gamma \sigma_j^\eta\right) \\
&= - R^\alpha_{\eta\beta\gamma} \sigma_i^\beta (\nabla_{\sigma_i}g)^{\gamma\eta}  
= - R^\alpha_{\eta\beta\gamma} \sigma_i^\beta \sigma_i^\zeta (\nabla_{\zeta}g)^{\gamma\eta} 
= - R^\alpha_{\eta\beta\gamma} g^{\beta\zeta} (\nabla_{\zeta}g)^{\gamma\eta}\;,
\end{equs}
where we first used the first Bianchi identity, then the antisymmetry of the curvature tensor
and finally the definition of covariant derivative of $g$.
Regarding the second identity, the definition of the covariant derivative of a tensor field
implies that $\tau_c = (\Nabla_{\<generic>} R)(\<genericb>,\<generic>)\,\<genericb> = - (\Nabla_{\<generic>} R)(\<generic>,\<genericb>)\,\<genericb>$, and the claim follows.
\end{proof}

Then, we have the following two identities.

\begin{proposition}\label{prop:ItoStratgeo}
With $\CS_\both$, $\CS_\Ito$ and $\CS_\geo$ defined above, we have
\begin{equs}
\CS_\both &= \CS_\Ito + \CS_\geo \;,\\
\CS_\Ito \cap \CS_\geo &= \scal{\{\tau_\star, \tau_c\}}\;.\label{prop:Itogeo}
\end{equs}
\end{proposition}

The proof of this statement is postponed to Section~\ref{sec:counting} below, see Theorem~\ref{theo:mainSum}.
Before we proceed, we argue that there is a natural way of ``eliminating'' the vector
$\tau_c$ above by restricting to the space $\CS^\nice$ that we have characterised in Proposition~\ref{prop:ortho}. 
We first have the following elementary lemma.

\begin{lemma}\label{lem:ortho}
Let $B$ be a Banach space and let $V \subset B$ be a closed subspace of the form
$V = \bigoplus_{i \in I} V_i$ for finitely many closed subspaces $V_i$. Let furthermore
$W \subset B^*$ be a closed subspace of the dual space such that $W = \bigoplus_{i\in I} W_i$
where each closed subspace $W_i$ satisfies $W_i \perp V_j$ for $i \neq j$.
Then, one has
\begin{equ}
V \cap W^\perp = \bigoplus_{i \in I} (V_i \cap W_i^\perp) = \bigoplus_{i \in I} (V_i \cap W^\perp)\;.
\end{equ}
\end{lemma}

\begin{proof}
The proof is a simple exercise.
\end{proof}

\begin{corollary}\label{cor:spacesNice}
Writing $\CS_i^\nice = \CS_i \cap \CS^\nice$ for $i\in\{{\rm both},{\rm geo},$ {\rm It\^o}$\}$, we
have 
\begin{equ}
\CS_\both^\nice = \CS_\Ito^\nice + \CS_\geo^\nice \;,\qquad
\CS_\Ito^\nice \cap \CS_\geo^\nice = \scal{\{\tau_\star\}}\;.
\end{equ}
\end{corollary}

\begin{proof}
Since $\CS^\nice = W^\perp$ for $W = \scal{\{\<Xi4b1>, \<Xi4ba2>, \<Xi4ba1b>\}}$, it suffices to find
decompositions of $W$ and $V = \CS_\both$ satisfying the assumption of Lemma~\ref{lem:ortho}. 
We decompose $W$ according to
\begin{equ}
W = \R \<Xi4b1> \oplus \R(\<Xi4ba2> + {\textstyle {1\over 2}}\<Xi4ba1b>- {\textstyle {1\over 2}}\<Xi4b1>) \oplus \R(\<Xi4ba1b>) = W_1 \oplus W_2 \oplus W_3\;. 
\end{equ}
We then choose $V_3 = \CS_\Ito \cap \CS_\geo$, $V_2 \subset \CS_\Ito$ any complement of $V_3$ in $\CS_\Ito$
which is orthogonal to $\<Xi4ba1b>$ (this is possible since $\<Xi4ba1b>$ does not
annihilate $\tau_c \in V_3$),
and finally $V_1$ any complement of $V_3$ in $\CS_\geo$ which is orthogonal to $\<Xi4ba1b>$.
Proposition~\ref{prop:ortho} guarantees that these choices do satisfy the assumptions of
Lemma~\ref{lem:ortho}.

We show now that $\CS_\Ito^\nice \cap \CS_\geo^\nice = \scal{\{\tau_\star\}}$, namely that
$\scal{\{\tau_\star, \tau_c\}}\cap\CS^\nice=\scal{\{\tau_\star\}}$, using \eqref{prop:Itogeo}.
First, $\tau_\star$
belongs to $\CS^\nice$ by Proposition \ref{prop:ortho} and to $\CS_\Ito \cap \CS_\geo$ by \eqref{prop:Itogeo}.
In order to show that $\tau_c\notin\CS^\nice$, note that $\scal{\<Xi4ba1b>, \tau_c} \neq 0$ by the explicit expression of $\tau_c$ given in \eqref{eq:valeurC}, while $\<Xi4ba1b>\perp\CS^\nice$ by Proposition \ref{prop:ortho}.
\end{proof}

We now have all the ingredients in place to show that it is possible to construct a solution map
which satisfies both points 3 and 4 of Theorem~\ref{theo:main} simultaneously.

\begin{theorem}\label{thm:bullet}
For both $i \in \types$, there exists a choice of constants $\tilde \tau_i \in \CS_i^\nice$ independent 
of the mollifier $\rho \in \Moll$
such that, defining $\tau_i^{(\eps)}$ as in the proof of Corollary~\ref{cor:both}, one has
\begin{equ}
U^\geo = U^\Ito\;.
\end{equ}
Furthermore, any two choices of $\tilde \tau_i$ having the same properties differ by a multiple
of $\tau_\star$.
\end{theorem}

\begin{proof}
We have already seen in Corollary \ref{cor:both} that both $C_{\eps,\geo}^\BPHZ$ and $C_{\eps,\Ito}^\BPHZ$ belong to 
$\CS^\nice$ for every mollifier $\rho\in\Moll$ and that we can choose $\tau_i^{(\eps)}\in\CS^\nice$.
In particular, we know that the element $\tau_0$ defined in \eqref{e:deftau} belongs to $\CS^\nice$
and is given by
\begin{equs}
\tau_0 &= \lim_{\eps \to 0} \bigl(\tau_\geo^{(\eps)} - C_{\eps,\geo}^\BPHZ - \tau_\Ito^{(\eps)}+C_{\eps,\geo}^\BPHZ\bigr) = \tilde \tau_\geo - v_\geo - \tilde \tau_\Ito + v_\Ito\;,
\end{equs}
for $v_i\in\CS_i^\perp$ as in Propositions~\ref{prop:geo} and~\ref{prop:ito} and $\tilde\tau_i\in\CS_i^\nice$ as in the proof
of Corollary~\ref{cor:both}.
We know furthermore that $\tau \in \CS_\both^\nice$ by \eqref{e:proptau}, so that by Corollary~\ref{cor:spacesNice}
we can choose $\tilde \tau_\geo \in \CS_\geo^\nice$ and $\tilde \tau_\Ito \in \CS_\Ito^\nice$ in such a way
that $\tau_0 = 0$, as desired. 
The second statement is immediate from the second part of Corollary~\ref{cor:spacesNice} and the 
injectivity of the solution maps in their last argument.
\end{proof}

\section{Proof of main results}

We now have almost all of the ingredients in place to prove our main result. Before we
turn to it however, we need one more result on the behaviour of the BPHZ renormalisation constants.

\subsection{Convergence of constants}
\label{sec:conv}

Throughout this section, we write $C_\eps$ for the BPHZ character $C_{\eps,\geo}^\BPHZ$ 
from Proposition~\ref{prop:BPHZchar}.
We interpret this character as an element of $\CS^\nice$ in such a way that the counterterm 
is given by $\Upsilon_{\Gamma,\sigma}C_\eps$. Recall also the definition of 
$\tau_\star$ from \eqref{e:taustar}.
The main result in this section is as follows.

\begin{theorem}\label{theo:constants}
There exists an element $\tau_0 \in \CS^\nice$ and a constant $\bar c > 0$ such that 
\begin{equ}
\lim_{\eps \to 0} \Bigl(C_{\eps} + {\bar c \over \eps} \Nabla_{\<generic>}\<generic> - {\log \eps \over 4\sqrt 3 \pi} \tau_\star \Bigr) = \tau_0\;.
\end{equ}
\end{theorem}

\begin{proof}
By Proposition~\ref{prop:geo}, there exists $v_\geo\in \CS_\geo^\perp\subset\CS^\nice$ such that 
the distance between $C_{\eps} - v_\geo$ and the $13$-dimensional
`geometric' subspace $\CV^\nice=\CS_\geo^\nice$, see Corollary~\ref{cor:spacesNice}
and Proposition~\ref{prop:geoDim} below, converges to $0$ as $\eps \to 0$. 
Using \eqref{eq:valeurR}, we have that
 $\scal{\tau, \Nabla_{\<generic>}\<generic>} = \scal{\tau,\tau_\star} = 0$ for 
\begin{equ}[e:listtau]
\tau \in  \Big\{\<Xi4_1> \,, \<Xi4c1>\,, \<Xi4_2>\,,  \<Xi4ec3> \,,  \<Xi4ec1> \,,  \<Xi4ec2> \,, \<Xi41> \,, \<Xi42> \,,    \<Xi4eac2>, 
\<Xi4eac1> - \<Xi4eabisc1> ,   \<Xi4ca2> \Big\}\;,
\end{equ}
and by Proposition~\ref{prop:geoDim}, these form 
$11$ linearly independent linear functionals on $\CS_\geo^\nice$.
In order to complete the proof, it therefore suffices to show that 
$\scal{C_\eps,\tau}$ converges to a finite limit as $\eps \to 0$ for every $\tau$ 
as in \eqref{e:listtau} and, since $\scal{\<I1Xitwo>,\Nabla_{\<generic>}\<generic>} = \scal{8\<I1Xi4abcc1s>,\tau_\star} = 1$, 
that the limits
\begin{equ}[e:taulimits]
\lim_{\eps \to 0} \Bigl(\scal{C_\eps, \<I1Xitwo>} + {\bar c \over \eps}\Bigr)\;,\qquad
\lim_{\eps \to 0} \Bigl(\scal{C_\eps, \<I1Xi4abcc1s>} - {\log \eps \over 32\sqrt 3 \pi}\Bigr)
\end{equ}
exist and are finite.

By Lemma~\ref{const:disconnect}, $\scal{C_\eps,\tau} =0$ for
$\tau \in \{\<Xi4_1s> \,,  \<Xi4ec3s> \,, \<Xi41s> \,, \<Xi42s> \,,   \<Xi4eac2s>, 
   \<Xi4ca2s>\}$. In \cite{wong}, it was furthermore shown that 
$\scal{C_\eps,\tau} $ converges to finite limits for $\tau \in \{\<Xi4c1s>\,, \<Xi4_2s>\,,   \<Xi4ec1s> \,,  \<Xi4ec2s>\}$, while Lemma~\ref{lemmadiffconst} below shows that 
$\scal{C_\eps, \<Xi4eabisc1s>- \<Xi4eac1s>}$ converges to a finite limit.
 
The fact that the first limit in \eqref{e:taulimits} exists for 
\begin{equ}
\bar c = \int_{\R^2} ((\d_x P\star \rho)(t,x))^2\,dt\,dx
\end{equ}
is a simple calculation exploiting the scale invariance of the heat kernel.
The fact that the second limit in \eqref{e:taulimits} is also finite is the content of Lemma~\ref{lem:logdiv}
below, which concludes the proof.
\end{proof}

\begin{remark}
Note that there is a sign error in \cite{Hao} before Eq.~3.6: the factors $\log \eps$ should read $|\log \eps|$,
which then makes it consistent with \eqref{e:taulimits}.
\end{remark}
   
We now show that the BPHZ renormalisation constant yields a finite limit on $\<Xi4eabisc1s>- \<Xi4eac1s>$
and compute the prefactor of the log-divergence.

\begin{lemma} \label{lemmadiffconst}
For every mollifier $\rho$ as above, there exists a constant $c$ such that 
\begin{equ}
\lim_{\eps \to 0} C_\eps^\BPHZ\Big(\<Xi4eabisc1>- \<Xi4eac1>\Big) = c\;.
\end{equ}
\end{lemma}

\begin{proof}
As in \cite[Sec.~6.3]{Jeremy}, it will be convenient to perform a change of variables and to write
\begin{equ}
K_{\eps,\rho}(z) = \big(\rho * \CS_\eps^{(1)}K\big)(z)\;,\qquad
K_{\eps}(z) = \big(\CS_\eps^{(1)}K\big)(z)\;,
\end{equ}
where $\big(\CS_\eps^{(\alpha)}K\bigr)(t,x) = \eps^\alpha K(\eps^2 t, \eps x)$.
Writing   \tikz [baseline=-0.1cm]{\draw[pagebackground] (0,0) -- (0,0.1);\draw[krho] (0,0) -- (1,0);}
for $K_{\eps,\rho}$,  
\tikz[baseline=-0.1cm]{\draw[pagebackground] (0,0) -- (0,0.1);\draw[krhopr] (0,0) -- (1,0) ;} for
$K_{\eps,\rho}'$,
and  the `plain' versions of these arrows for $K_{\eps}$ and $K_{\eps}'$, we then have
\begin{equs}
C_\eps^\BPHZ\Big(\<Xi4eabisc1>\Big)
= 
\tikzsetnextfilename{logterm1}
\begin{tikzpicture}[baseline=0.4cm,scale=0.5]
\node at (0,0) [root] (0) {}; 
\node at (0,2) [dot] (1) {};
\node at (1,1) [dot] (2) {};

\draw[kepspr] (1) to (0);
\draw[krho] (1) to[bend right=60] (0);
\draw[krho] (2) to[bend right=20] (1);
\draw[kepspr] (2) to[bend left=20] (0);
\end{tikzpicture}
\;,\qquad
C_\eps^\BPHZ\Big(\<Xi4eac1>\Big)
= 
\tikzsetnextfilename{logterm2}
\begin{tikzpicture}[baseline=0.4cm,scale=0.5]
\node at (0,0) [root] (0) {};
\node at (0,2) [dot] (1) {};
\node at (1,1) [dot] (2) {};

\draw[krhopr] (1) to (0);
\draw[keps] (1) to[bend right=60] (0);
\draw[krho] (2) to[bend right=20] (1);
\draw[kepspr] (2) to[bend left=20] (0);
\end{tikzpicture}
\end{equs}
At this point, we note that with the notations of \cite[Sec.~6.3]{Jeremy}, 
$K_{\eps,\rho}$, $K_{\eps,\rho}'$, $K_{\eps}$, and $K_{\eps}$
converge to finite limits in $\CB_{1-\kappa,0}$, $\CB_{2-\kappa,0}$,
$\CB_{1-\kappa,1+\kappa}$ and $\CB_{2-\kappa,2+\kappa}$ respectively.
Furthermore,  
$K_{\eps,\rho} - K_\eps$ and $K_{\eps,\rho}' - K_\eps'$
converge to finite limits in
$\CB_{2-\kappa,1+\kappa}$ and $\CB_{3-\kappa,2+\kappa}$.

Recall also that the product is continuous from
$\CB_{\alpha,\beta}\times \CB_{\bar\alpha,\bar\beta}$
into $\CB_{\alpha+\bar\alpha,\beta+\bar\beta}$ and that the
functional
$K \mapsto \int_{\R^2} K(z)\,dz$ is continuous
on $\CB_{\alpha,\beta}$ if $\alpha> 3$ and $\beta < 3$.
Setting
\begin{equ}
c_\eps^\rho = 
\tikzsetnextfilename{logtermmid}
\begin{tikzpicture}[baseline=0.4cm,scale=0.5]
\node at (0,0) [root] (0) {};
\node at (0,2) [dot] (1) {};
\node at (1,1) [dot] (2) {};

\draw[krhopr] (1) to (0);
\draw[krho] (1) to[bend right=60] (0);
\draw[krho] (2) to[bend right=20] (1);
\draw[kepspr] (2) to[bend left=20] (0);
\end{tikzpicture}\;,
\end{equ}
we conclude from these facts and from \cite[Lem.~6.9]{Jeremy} that 
both $C_\eps^\BPHZ\big(\<Xi4eac1>\big) - c_\eps^\rho$ and 
$C_\eps^\BPHZ\big(\<Xi4eac1>\big) - c_\eps^\rho$ converge to finite limits, which implies the claim.
\end{proof}

\begin{lemma}\label{lem:logdiv}
The second limit in \eqref{e:taulimits} is finite.
\end{lemma}

\begin{proof}
It is shown in \cite[Eq.~6.31]{Jeremy}
that $\scal{C_\eps, \<I1Xi4abcc1s>}$ (which is called $C_2^{(\eps)}$ in that article) differs by a converging
term from 
\begin{equ}
L_\eps = {1\over 4}
\begin{tikzpicture}[baseline=0.4cm,scale=0.5]
\node at (0,0) [root] (0) {}; 
\node at (-1,2) [dot] (1) {};
\node at (1,2) [dot] (2) {};

\draw[krhopr] (1) to[bend right=30] (0);
\draw[krho] (2) to[bend left=30] (0);
\draw[krhopr] (2) to[bend right=30] (1);
\draw[krho] (2) to[bend left=30] (1);
\end{tikzpicture}
= -
{1\over 8}
\begin{tikzpicture}[baseline=0.4cm,scale=0.5]
\node at (0,0) [root] (0) {}; 
\node at (-1,2) [dot] (1) {};
\node at (1,2) [dot] (2) {};

\draw[krhopr] (1) to[bend right=30] (0);
\draw[krhopr] (2) to[bend left=30] (0);
\draw[krho] (2) to[bend right=30] (1);
\draw[krho] (2) to[bend left=30] (1);
\end{tikzpicture}\;,
\end{equ}
where we performed an integration by parts on the top right integration variable.
Using \cite[Lem~6.12]{Jeremy}, we see that this differs
by a converging quantity from 
\begin{equ}[e:boundlog]
-{1\over 16} \int_{\R^2} K^3_{\eps,\rho}(t,x)\,dt\,dx\;.
\end{equ}
Using the identity $\int_\R P^3(t,x)\,dx = 1/ (4\sqrt 3 \pi t)$ and the scaling properties of the
heat kernel, a simple 
asymptotic analysis shows that 
$\int_{\R^2}K^3_{\eps,\rho}(t,x)\,dx$ differs from $\one_{[\eps^2,1]} / (4\sqrt 3 \pi t)$
by a term whose integral converges to a finite limit. Inserting this into \eqref{e:boundlog}
and performing the integration over $t$ yields the desired result.
\end{proof}

\subsection{Proof of the main theorem}

We are now in a position to combine all of our results.

\begin{proof}[of Theorem~\ref{theo:main}]
Recall first the definition of $\tau_0 \in \CS^\nice$ from Theorem~\ref{theo:constants}
and note that the space $\CV^\nice$ mentioned in the introduction is nothing but
the space $\CS_\geo^\nice$.
Since $\Upsilon_{\Gamma,\sigma} \tau_\star = H_{\Gamma,\sigma}$, it follows that claims 1, 3 and 4 
of  Theorem~\ref{theo:main}
hold as long as $c \in \CS_\geo^\nice$ and $\hat c \in \R$ are chosen in such a way 
that 
\begin{equ}
c + \hat c \tau_\star = \lim_{\eps \to 0}\Bigl( \tau_\geo^{(\eps)}
 + {\bar c \over \eps} \Nabla_{\<generic>}\<generic> - {\log \eps \over 4\sqrt 3 \pi} \tau_\star\Bigr)\;,
\end{equ}
where $\tau_\geo^{(\eps)}$ is as in Corollary~\ref{cor:both} and Theorem~\ref{thm:bullet}
and the convergence of the right hand side is guaranteed by Theorem~\ref{theo:constants}.
The uniqueness claim is then the same
as the uniqueness claim modulo $\tau_\star$ that is found in Theorem~\ref{thm:bullet}.
The fact that \eqref{e:equivar} holds for arbitrary diffeomorphisms (and not just those homotopic to
the identity) follows from the fact that the counterterms in $\CS_\geo$ are obtained by taking 
multiple covariant derivatives of the vector fields $\sigma_i$, see Remark~\ref{rem:geo} below.

Claim 2 is an immediate consequence of the fact that the driving noise is white and the 
stability of the limit with respect to perturbations in the initial condition. More details
can be found for example in \cite[Sec.~7.3]{reg}. 

To show that claim~5 holds, note first that, as a consequence of Lemma~\ref{lem:meanzero}, 
the solutions $U^\BPHZ(0,\sigma,h)$ given by the BPHZ model do coincide with the classical
It\^o solutions to \eqref{e:classicalIto}. It therefore remains to show that 
$U^\BPHZ(0,\sigma,h) = U^\Ito(0,\sigma,h)$. This follows immediately from the fact that 
$U^\BPHZ(0,\sigma,h) = U^\Ito(0,\sigma,h + \Upsilon_{0,\sigma}\tau)$ for some
element $\tau \in \CS_\Ito$. Now $\Upsilon_{0,\sigma}\tau$ vanishes unless the symbol $\tau$ contains
only nodes of type $\<generic>$ and no node of type $\<not>$, while $\CS_\Ito$ is orthogonal
to any such symbol by Proposition~\ref{prop:CB} and the second statement of Theorem~\ref{theo:mainSum}.

Finally, the last statement follows as in \cite[Cor.~3.11]{CPA:CPA20383}, noting that since we are in the
situation $\Gamma = 0$ and $\sigma$ constant, $\Upsilon_{\Gamma,\sigma}$ vanishes identically on 
all of $\CS$, so that in this case our notion of solution coincides simply with the limit 
of \eqref{e:genClass} as $\eps \to 0$.
\end{proof}

\subsection{Path integrals and a conjecture}
\label{sec:conjecture}

Let again $\CM$ be a Riemannian manifold with inverse metric tensor $g$, let $A$ be a smooth
vector field on $\CM$, and let $\mu$ be the loop measure on $\CC^a(S^1,\CM)$
determined by the diffusion with generator
\begin{equ}
L f  = {1\over 2}\Delta f + df(A)\;,
\end{equ}
where $\Delta$ is the Laplace-Beltrami operator on $\CM$. 
We assume furthermore that the corresponding diffusion process does not explode in finite time
and is such that $x \mapsto p_1(x,x)$ is integrable on $\CM$.
It is a simple exercise in stochastic calculus to verify that the measure 
$\mu$ is equivalent to the measure $\mu_0$ constructed like $\mu$ but with $A=0$. Its Radon-Nikodym
derivative is given by
\begin{equ}
\exp \Bigl(\int_0^1 \scal{A(u),\circ du(x)}_{g(u)} - {1\over 2}\int_0^1 \bigl(|A(u)|_{g(u)}^2 +\div A(u)\bigr)\,dx \Bigr)\;,
\end{equ}
where $\circ$ denotes Stratonovich integration.
In view of this, it is then natural to conjecture the following (see \cite{CPA:CPA20383} for a statement and proof
in the Euclidean case).

\begin{conjecture}\label{conj:invariant}
In the above situation, let $\sigma_i$ be any collection of smooth vector fields
such that $g = \sigma_i \otimes \sigma_i$ and let $\Gamma^\alpha_{\beta\gamma}$ be the Christoffel
symbols for the Levi-Civita connection on $\CM$. Then, there exists a universal constant $c$ such that the
unique process given by canonical solution of Theorem~\ref{thm:riemann} for the stochastic PDE
\begin{equ}[e:geometric]
\d_t u = \Nabla_{\d_xu}\d_xu + \bigl((d A^\flat)(\d_x u)\bigr)^\sharp - \nabla_A A - {1\over 2} \nabla \div A + c \nabla R(u) + \sqrt 2\sigma_i(u) \, \xi_i\;,
\end{equ}
where $R$ denotes the scalar curvature of $\CM$,
exists for all times and has $\mu$ as its unique invariant measure. Comparing with \eqref{e:genClass}, the coefficients there become here
$h=- \nabla_A A - {1\over 2} \nabla \div A + c \nabla R$, while $K$ is the tensor field such that 
$K\d_x u=\bigl((d A^\flat)(\d_x u)\bigr)^\sharp$.
\end{conjecture}

\begin{remark}
If we can find a sequence of approximations $u^\eps$ to \eqref{e:geometric}
such that the corresponding invariant measures $\mu_\eps$ converge to $\mu$ in a sufficiently
strong sense (all moments in $\CC^a$ bounded), then the claim follows as in \cite[Cor.~1.3]{Konstantin}.
In particular, uniqueness of the invariant measure (restricted to any given homotopy class) 
then follows by combining \cite[Thm~4.8]{Jonathan}
with the Stroock-Varadhan support theorem \cite{SupportThm} which shows that the measure $\mu$ has full support.
\end{remark}

It is not a priori obvious what the value of the constant $c$ appearing in Conjecture~\ref{conj:invariant} 
should be. One would of course be tempted to conjecture that $c = 0$, but there is no good a priori reason
why this would be the case. After all, the only reason why we obtain a canonical notion of solution
is that we restrict a priori our renormalisation constants to the subspace $\CS^\nice \subset \CS$ 
which is chosen in a rather ad hoc manner: it does not arise from any symmetry consideration but from a
remark on the structure of the BPHZ renormalisation constants. 

Parsing the literature, there is consensus on the fact that if one wishes to represent Brownian motion
on $\CM$ by a path integral, then it should indeed be of the form
\begin{equ}
\exp \Bigl(-{1\over 2}\int_0^1 |\d_x u|_{g(u)}^2\,dx + c\int_0^1 R(u)\,dx \Bigr)\;,
\end{equ}
(for which \eqref{e:geometric} with $h=0$ would formally be the natural Langevin dynamic) 
but there is no consensus on the value of $c$. For example, the value $c= {1\over 12}$ would be 
consistent with the Onsager-Machlup functional appearing in the main theorem of \cite{TW} (see also  
\cite{Graham} for example). 
The result \cite{Driver} shows that for two rather 
natural choices of the Riemannian volume form for a piecewise geodesic discretisation of path space,
one obtains either $c=0$ or $c={1\over 6}$. The value $c = {1\over 6}$ also appears in
\cite[Eq.~15]{Cheng} (it may look like it has the opposite sign but if we interpret their statement in our context,
it states that the invariant measure for \eqref{e:geometric} with $c=0$ and $h=0$ is given by the 
 Brownian loop measure weighted by $\exp(-{1\over 6}\int R(u)\,dx)$ which is equivalent to the statement
 that our conjecture holds with $c = {1\over 6}$).
The value $c = {1\over 8}$ does appear in \cite[Eq.~6.5.25]{DeWitt}
(but the same author also makes a case for $c = -{1\over 12}$ in \cite[Eq.~7.12]{DewittRev}),
while $c=-{1\over 8}$ appears in \cite[Eq.~4.9]{Dekker}.
With our particular choice of normalisation, we have the following conjecture.

\begin{conjecture}\label{conj:c}
One has $c = {1\over 8}$ in \eqref{e:geometric}.
\end{conjecture}

The remainder of this section is devoted to an argument in favour of this conjecture.
At this stage, this is of heuristic nature, but we believe that it could in principle be made rigorous
by combining the results of \cite{ErhardHairerRegularity} with a suitable discrete version of the results of \cite{Ajay}.
We take as our starting point the result of 
\cite{Driver} which shows that, approximating a path in $\CM$ with a sequence of $N = 1/\eps$ points, 
the sequence of measures $\mu_n$ on $\CM^N$ with density (with respect to the product measure)
proportional to
\begin{equ}
\exp \Bigl(- {1\over 2}\sum_{k}{d(u_k,u_{k+1})^2 \over \eps} + {1\over 6} \sum_k R(u_k)\,\eps \Bigr)\;,
\end{equ}
where $R$ denotes the scalar curvature of $\CM$, converge to $\mu$ as $\eps \to 0$. Considering the corresponding
Langevin dynamic, we find that
the measure $\mu_n$ is invariant for the system of coupled SDEs on $\CM^N$ given by
\begin{equ}[e:approxSDE]
du_k = {\exp_{u_k}^{-1}(u_{k+1}) + \exp_{u_k}^{-1}(u_{k-1}) \over \eps^2}\,dt + {1\over 6} \nabla R(u_k)\,dt 
+ \sqrt{2 \over \eps} \sigma_i(u)\circ dW_{i,k}\;,
\end{equ}
where the $W_{i,k}$ are i.i.d.\ standard Wiener processes and the $\sigma_i$ are a collection of vector fields
as above satisfying furthermore the property $\nabla_{\sigma_i}\sigma_i = 0$. (We also use the convention that the
index of $u$ is interpreted modulo $N$.)
We now fix a coordinate chart  and write $\delta^\pm_k u^\alpha = u_{k\pm 1}^\alpha - u_k^\alpha$. We
also define a valuation $\hat \Upsilon_{\Gamma}(u,v)$ in the same way as $\Upsilon_{\Gamma,\sigma}$, 
but it acts on symbols with edges of type $\<Ito>$ rather than edges of type $\<generic>_i$ 
and it replaces each instance of $\<Ito>$
with $v$. We furthermore restrict this to ``saturated'' symbols in the sense that vertices without noise edge
have two incoming thick edges, so that each instance of $\<not>$ is replaced by a suitable 
derivative of $2\Gamma$. For example, one has
\begin{equ}
\hat \Upsilon^\alpha_{\Gamma}(u,v) \; \<I1Xitwobis>  = 2\Gamma^\alpha_{\beta\gamma}(u) v^\beta v^\gamma\;.
\end{equ}
(This of course only makes sense for symbols which have $\<Ito>$ as leaves.)
A lengthy but ultimately not terribly interesting calculation shows that, in local coordinates, the first term 
in this equation can be approximated to fourth order in $\delta u$ by 
\begin{equ}[e:approxExp]
{\exp_{u_k}^{-1}(u_{k+1}) + \exp_{u_k}^{-1}(u_{k-1})}
\approx
\bigl( \hat \Upsilon_{\Gamma}(u_k,\delta^+_k u)+ \hat \Upsilon_{\Gamma}(u_k,\delta^-_k u)\bigr) \tau\;,
\end{equ}
where
\begin{equ}[e:counterterm]
\tau = \<Ito> + {1\over 4} \<I1Xitwobis> + {1\over 24} \<I1IXi3cbis> + {1\over 12}\<I1Xi3cbis> +{1\over 96}\<Xi4eabdis>
+ {1\over 48} \<Xi4badis> + {1\over 48} \<Xi4cabdis> + {1\over 192} \<2I1Xi4dis>\;.
\end{equ}
In terms of powercounting, all of these terms should disappear as $\eps \to 0$ except for the first two,
which yield the natural discrete approximation to $\Nabla_{\d_xu}\d_xu$. (The factor ${1\over 4}$ that 
appears in front of \<I1Xitwobis> is multiplied by two since every term is doubled
in the right hand side of \eqref{e:approxExp}.) As a consequence, one would expect that, in local coordinates, 
the BPHZ renormalised solution to \eqref{e:approxSDE} converges as $\eps \to 0$ to the BPHZ renormalised
solution to  \eqref{e:geometric} with $A=0$ and $c = {1\over 6}$. What we are interested in however
is the limit of the \textit{unrenormalised} solution to \eqref{e:approxSDE}, which we also expect to
exist and to coincide with the \textit{canonical} solution to  \eqref{e:geometric} with $A=0$ and 
some value $c$ yet to be determined.

We also know from Corollary~\ref{cor:both} that the BPHZ solution to 
\eqref{e:geometric} in any fixed chart differs from its canonical solution by a counterterm 
$\Upsilon_{\Gamma,\sigma} \tau_0$ with $\tau_0 \in \CS^\nice$, so that in particular
$\scal{\<Xi4ba2>, \tau_0} = \scal{\<Xi4ba1b>, \tau_0} = 0$ by Proposition~\ref{prop:ortho}. 
Recall on the other hand that $\tau_c$ is given by \eqref{eq:valeurC} and
that $2\Upsilon_{\Gamma,\sigma}\tau_c = -\nabla R$ by Remark~\ref{rem:twoparam} 
and Lemma~\ref{lem:exprtaucstar}, so that  
\begin{equ}
\nabla R = \Upsilon_{\Gamma,\sigma}(\<Xi4ba1b>-\<Xi4ba2> + \bar \tau)\;,\qquad 
 \bar \tau \in \CS^\nice\;.
\end{equ}
Focusing on the contributions coming from $\<Xi4ba1b>$ and $\<Xi4ba2>$, we can therefore 
gain enough information to be able to conjecture the value of $c$ without having to know anything 
else about $\tau_0$. As a consequence of this discussion,
one expects the BPHZ counterterm for \eqref{e:approxExp} to be of the form
$\Upsilon_{\Gamma,\sigma} \tau_\eps^{\BPHZ}$ (in the sense that the renormalised solution
solves \eqref{e:approxExp} with an additional term $\big(\Upsilon_{\Gamma,\sigma}\tau_\eps^{\BPHZ}\big)(u_k)$ 
appearing on the right hand side) with 
\begin{equ}
\tau_\eps^{\BPHZ} = \bar c_\eps (\<Xi4ba1b>-\<Xi4ba2>) + \bar \tau_\eps^{\BPHZ}\;,\qquad
\scal{\<Xi4ba2>, \bar \tau_\eps^{\BPHZ}} = \scal{\<Xi4ba1b>, \bar \tau_\eps^{\BPHZ}} = 0\;,
\end{equ}
and such that $\bar c_\eps$ converges to a finite limit $\bar c$ as $\eps \to 0$. If this is the
case, then Conjecture~\ref{conj:c} follows at once with $c = {1\over 6} - \bar c$.

Let us remark at this stage that the reason why we consider the 
expansion \eqref{e:counterterm} to \textit{fourth} 
order is that even though the additional terms disappear in the limit when considering the BPHZ renormalised
solution, they \textit{do} contribute to the counterterms $\tau_\eps^{\BPHZ}$ generated by the 
BPHZ renormalisation procedure and therefore to the identification of $c$.

Denote now by $\tilde u_i$ the (stationary, mean $0$) solution to the linearised equation, namely
\begin{equ}[e:linearised]
d\tilde u_{i,k} = {\delta^+_k \tilde u_i + \delta^-_k \tilde u_i\over \eps^2} + \sqrt{2\over \eps} dW_{i,k}\;.
\end{equ}
The $\tilde u_i$ are independent, Gaussian, and satisfy 
\begin{equ}[e:correlations]
\eps^{-1} \E |\delta^\pm_k \tilde u_i|^2 \to 1\;,\qquad 
\eps^{-1} \E u_{i,k} \delta^\pm_k \tilde u_i \to -{1\over 2}\;.
\end{equ}
(The first limit comes from the fact that the invariant measure of \eqref{e:linearised}
converges to the flat Brownian loop measure. The second limit follows from the first
by exploiting the $k \mapsto k+1$ and $k \mapsto N-k$ symmetries.) 

A formal expansion then suggests that counterterms of the type $\Upsilon_{\Gamma,\sigma} \<Xi4ba2>$ and 
$\Upsilon_{\Gamma,\sigma}\<Xi4ba1b>$ can be generated by the terms 
$\<I1Xitwobis>$, $\<I1Xi3cbis>$ and $\<Xi4badis>$ appearing in \eqref{e:counterterm}.
The counterterm generated by $\<I1Xitwobis>$ is given by
\begin{equ}
-{1\over 4} \cdot 2 \cdot {1\over 2} \cdot 2 \cdot {1\over 4}\cdot \Upsilon_{\Gamma,\sigma} \<Xi4ba2>
= -{1\over 8} \Upsilon_{\Gamma,\sigma} \<Xi4ba2>\;.
\end{equ}
The minus sign comes from the fact that the BPHZ counterterms \textit{cancel out} the expectations of the 
terms of negative homogeneity appearing in a formal 
expansion of the right hand side of the equation, the factor ${1\over 4}$ comes 
from \eqref{e:counterterm}, the factor $2$ from the fact that 
each of these terms is doubled in \eqref{e:approxExp}, the factor ${1\over 2}$ from fact that it is the second order 
Taylor expansion of $\Gamma$ in $\<I1Xitwobis>$ that yields contributions of the desired form, 
the second factor $2$ from the two ways of pairing 
up the two extra factors of $u$ with the
factors $\delta^\pm u$, and the final factor ${1\over 4}$
from squaring the second limit in \eqref{e:correlations}. Similarly, the term $\<I1Xi3cbis>$ generates a counterterm
given by
\begin{equ}
-{1\over 12} \cdot 2 \cdot \Bigl(-{1\over 2}\Bigr)\cdot \Big(\Upsilon_{\Gamma,\sigma} \<Xi4ba1b>
+ 2\Upsilon_{\Gamma,\sigma} \<Xi4ba2>\Big) =  {1\over 12} \Upsilon_{\Gamma,\sigma} \<Xi4ba1b>
+{1\over 6}\Upsilon_{\Gamma,\sigma} \<Xi4ba2>\;,
\end{equ}
while the term $\<Xi4badis>$ generates the counterterm
\begin{equ}
-{1\over 48} \cdot 2 \cdot \Big(\Upsilon_{\Gamma,\sigma} \<Xi4ba1b>
+ 2\Upsilon_{\Gamma,\sigma} \<Xi4ba2>\Big) = -{1\over 24}\Upsilon_{\Gamma,\sigma} \<Xi4ba1b>
- {1\over 12}\Upsilon_{\Gamma,\sigma} \<Xi4ba2>\;.
\end{equ}
Combining these three terms yields ${1\over 24}(\<Xi4ba1b>-\<Xi4ba2>)$, thus suggesting that
one has $\bar c = {1\over 24}$, which then yields $c = {1\over 6} - {1\over 24} = {1\over 8}$ as claimed.

\subsection{Loops on a manifold}
\label{sec:geometric}

The aim of this section is to give a proof of Theorem~\ref{thm:riemann}. Throughout this section, we fix some
$\R^d$, as well as a Riemannian manifold $\CM$ that is smoothly and isometrically embedded into $\R^d$.
We then recall the following fact about such a setup.

\begin{lemma}\label{lem:Gammas}
Let $\pi \colon \R^d \to \CM$ be a smooth function such that 
$\pi \restr \CM = \id$ and such that, for every $p \in \CM$ and $v \in (T_p \CM)^\perp$ 
one has $D_v \pi(p) = 0$. Let furthermore $\bar X$ and $\bar Y$ be smooth vector fields 
on $\R^d$ and write $X$ and $Y$ for their restrictions to $\CM$. Define
\begin{equ}
\big(\bar \nabla_{\bar X} \, \bar Y\big)^\alpha(x) =  \bar X^\beta(x) \, \d_\beta \bar Y^\alpha(x) - \d_{\beta\gamma}^2 \pi^\alpha(x) \,
\bar X^\beta(x) \, \bar Y^\beta(x)\;.
\end{equ} 
Then, for every $x \in \CM$ one has $\bigl(\bar \nabla_{\bar X} \bar Y\bigr)(x)= \bigl(\nabla_X Y\bigr)(x) \in T_x\CM$.
\end{lemma}

\begin{proof}
Since $\R^d$ is flat, Lie derivatives and covariant derivatives coincide. Furthermore,
our assumption guarantees that, for every $x \in \CM$, $D\pi(x)$ is the orthogonal projection onto $T_x\CM$, 
so that, for $x \in \CM$, one has
\begin{equ}[e:iden1]
\bigl(\nabla_X Y\bigr)^\alpha(x) = \d_{\beta}\pi^\alpha(x) \, \bar X^\gamma(x) \, \d_\gamma \bar Y^\beta(x)\;.
\end{equ}
(See for example \cite[Prop.~5.13 (4)]{Petersen}.) On the other hand, since $\bar Y(x) \in T_x \CM$ for $x \in \CM$, one has
the identity
\begin{equ}
\bar Y^\alpha(x) = \d_\beta \pi^\alpha(x) \, \bar Y^\beta(x) \;,\qquad \forall x \in \CM\;.
\end{equ}
Differentiating this identity in the direction $\bar X(x) \in T_x\CM$, we obtain on $\CM$
\begin{equ}
\bar X^\gamma \d_\gamma\bar Y^\alpha(x) = \d^2_{\beta,\gamma} \pi^\alpha(x) \, \bar X^\gamma(x) \, \bar Y^\beta(x)
+ \d_\beta \pi^\alpha(x) \, \bar X^\gamma \, \d_\gamma\bar Y^\beta(x)\;.
\end{equ}
Combining this with \eqref{e:iden1}, the claim follows.
\end{proof}

A natural collection of vector fields $\sigma_i$ is now given by setting
\begin{equ}[e:choiceSigmai]
\sigma_i^\alpha(x) = \d_i \pi^\alpha(x)\;.
\end{equ}
It is clear that, for all $x \in \CM$ one has $\sigma_i(x) \in T_x\CM$ as a consequence of the
fact that $\pi$ maps $x$ to itself and maps any neighbourhood of $x$ to $\CM$. 
With this definition, we have the following.

\begin{lemma}\label{lem:propSigmai}
In the above setting, with $\pi$ as in Lemma~\ref{lem:Gammas} and with $\sigma$ as
in \eqref{e:choiceSigmai} viewed as vector fields on $\CM$, one has 
$\sum_{i=1}^d \sigma_i \otimes \sigma_i = g$ and $\sum_{i=1}^d \nabla_{\sigma_i}\sigma_i = 0$.
\end{lemma}

\begin{proof}
Since, for $x \in \CM$, $D\pi(x)$ equals the \textit{orthogonal} projection $P_x$ onto $T_x \CM$,
one has the identity
\begin{equ}[e:idenPix]
\sum_{i=1}^d \sigma_i^\alpha(x) \, \sigma_i^\beta(x) = P_x^{\alpha\beta} = g^{\alpha \beta}(x)\;,
\end{equ}
where the first identity is nothing but the identity $P_x P_x^* = P_x$ which holds for 
any orthogonal projection, while the second identity is the definition of an isometric
embedding of $\CM$ into $\R^d$. By Lemma~\ref{lem:Gammas}, one then has
\begin{equs}
\sum_{i=1}^d\bigl(\nabla_{\sigma_i}\sigma_i\bigr)^\alpha(x) 
&= \sum_{i=1}^d \bigl(\d_i \pi^\beta \, \d_\beta \d_i \pi^\alpha - \d^2_{\beta \gamma} \pi^\alpha \, \d_i \pi^\beta \, \d_i \pi^\gamma\bigr)(x) \\
& = P_x^{\beta\gamma} \, \d^2_{\beta \gamma} \pi(x) - \d^2_{\beta \gamma} \pi^\alpha(x) (P_x P_x^*)^{\beta \gamma} = 0\;,
\end{equs}
where we made use of \eqref{e:idenPix}.
\end{proof}

We now have all the ingredients in place to give a proof of Theorem~\ref{thm:riemann}.

\begin{proof}[of Theorem~\ref{thm:riemann}]
By Nash's embedding theorem \cite{Nash}, it is possible to find an isometric embedding of $\CM$ in
$\R^d$ for $d$ sufficiently large. We fix such an embedding (hence view $\CM$ as a subset of $\R^d$)
and fix a map $\pi\colon \R^d \to \CM$ as in Lemma~\ref{lem:Gammas}.
For example, we can choose $\pi(y) = \arginf\{|x-y|\,:\, x \in \CM\}$ in a sufficiently small neighbourhood of
$\CM$ and then extend this in an arbitrary smooth way to all of $\R^d$. We also choose $\sigma$ as in 
\eqref{e:choiceSigmai}.

For any $\rho \in \Moll$ and $\eps > 0$, we consider the solution $u_\eps \colon \R_+ \times S^1 \to \R^d$ 
to
\begin{equ}[e:uepsglobal]
\d_t u_\eps^\alpha = \d_x^2 u_\eps^\alpha - \d^2_{\beta \gamma}\pi^\alpha(u_\eps) \, \d_x u_\eps^\beta\, \d_x u_\eps^\gamma 
+ h^\alpha(u_\eps) + \sigma_i^\alpha(u_\eps)\, \xi_i^\eps + V_{\rho,\sigma}^\alpha(u_\eps)\;.
\end{equ}
It follows from Lemma~\ref{lem:Gammas} that, provided that $u_\eps(t,x) \in \CM$, the right hand side
of \eqref{e:uepsglobal} coincides with that of \eqref{e:localSolManifold}.

It is then a standard fact (see for example \cite[Lem.~3.1]{Li} or the original article \cite{Eells}) 
that, provided that the vector field $h$ is tangent to $\CM$, solutions to \eqref{e:uepsglobal} 
that start on $\CM$ stay on $\CM$ for all subsequent times and solve \eqref{e:localSolManifold}.
Furthermore, for any chart $\CU$ of $\CM$,
if we take an initial condition that is entirely contained in $\CU$ then, for as long as this remains the case, the solution 
satisfies \eqref{e:genClass}, with $\Gamma^\alpha_{\beta\gamma}$ given by the Christoffel symbols
of the Levi-Civita connection for the Riemannian metric of $\CM$ in the chart $\CU$.

On the other hand, \eqref{e:uepsglobal} is itself of the form \eqref{e:genClass}, so that 
we can apply Theorem~\ref{theo:main}. In this particular case however, one has
$\Nabla_{\<generic>}\<generic> = 0$ on $\CM$ by Lemma~\ref{lem:propSigmai}. One also has 
$H_{\Gamma,\sigma} = 0$ on $\CM$ since the connection given by \eqref{e:iden1} is the Levi-Civita connection
and, by \eqref{e:fieldCovar}, $H_{\Gamma,\sigma}$ is built from the covariant derivative of the metric, which vanishes
in that case.
These properties may fail to hold on all of $\R^d$ in general, but since we restrict ourselves to solutions that
start and remain on $\CM$ at all times, this is of no concern to us.
This shows that the `canonical family' exhibited in Theorem~\ref{theo:main} is really a `canonical solution',
implying both the convergence and the uniqueness claim made in the statement of Theorem~\ref{thm:riemann}.
The last statement follows from Remark~\ref{rem:twoparam}, in particular the identification of 
$\Upsilon_{\gamma,\sigma} \tau_c$ with the gradient of the scalar curvature in the geometric case.
\end{proof}

\begin{remark}\label{rem:geom}
In this ``geometric'' situation, both divergent terms appearing in \eqref{e:mainLimit} vanish, so that in this case
one also obtains finite limits $\bar U^\rho(\Gamma,\sigma) = \lim_{\eps \to 0} U_\eps^\geo (\Gamma,\sigma)$
without the need for any renormalisation at all! However, these limits do in general depend on the
choice of isometric embedding, since they differ by some vector field of the form
$\Upsilon_{\Gamma,\sigma}\tau$ with $\tau \in \CV$, and these are sensitive to both the choice of 
$\sigma_i$'s and of the mollifier $\rho$.
\end{remark}

\section{General theory of \texorpdfstring{$T$}{T}-algebras}
\label{sec:algebra}

We still need to prove Proposition \ref{prop:ItoStratgeo}. This will be done in Section \ref{sec:spaces}, using the
more general results of this section, where we 
introduce an algebraic structure that encodes the space of ``formal expressions built from 
derivatives of $\sigma_i$ and $\Gamma$ with indices contracted according to Einstein's summation 
convention'' which include the functions $\Upsilon_{\Gamma,\sigma}\tau$ considered above. 
It is then natural to look for a space that exhibits the following features:
\begin{claim}
\item It should be graded by pairs $(\upper,\low)$ of integers denoting the number of ``upper'' and ``lower'' indices
respectively. Furthermore, on each such space, the symmetric group  $\Sym(\upper,\low) = \Sym(\upper) \times \Sym(\low)$
should act by permuting the corresponding indices. An example of element of degree $(1,2)$ would
be $\d_\alpha \Gamma^\alpha_{\beta\gamma}\sigma_1^\beta\d_\zeta\sigma_2^\eta$. It is of degree $(1,2)$ because there
are two free lower indices ($\gamma$ and $\zeta$) and one free upper index ($\eta$).
\item It should come with a product $\mult$ turning it into a (bi)graded algebra (for example
$\d_\alpha \Gamma^\alpha_{\beta\gamma} \sigma_1^\beta \mult \d_\zeta\sigma_2^\eta = \d_\alpha \Gamma^\alpha_{\beta\gamma} \sigma_1^\beta\,\d_\zeta\sigma_2^\eta$), a ``partial trace''
$\tr$ performing a contraction of the last upper and lower indices (so for example 
$\tr \d_\alpha \Gamma^\alpha_{\beta\gamma}\sigma_1^\beta\d_\zeta\sigma_2^\eta = \d_\alpha \Gamma^\alpha_{\beta\gamma}\sigma_1^\beta\d_\eta\sigma_2^\eta$), as well as a derivation $\d$ adding a lower index.
\item The valuation $\Upsilon_{\Gamma,\sigma}$ should be a morphism for the new algebraic structure. 
\end{claim}
The aim of the next subsection is to exhibit a number of intertwining relations that are naturally
satisfied by these operations and to provide a justification for the use of diagrammatic notations for
general expressions of the type mentioned above. For example, the diagrammatic notation yields
\begin{equ}
\d_\alpha \Gamma^\alpha_{\beta\gamma}\sigma_1^\beta\d_\zeta\sigma_2^\eta
= 
\tikzset{external/export next=false}
\begin{tikzpicture}[scale=0.2,baseline=-2]
\coordinate (root) at (0,0);
\coordinate (t1) at (-1.2,1.2);
\coordinate (t2) at (0,2);
\draw[symbols] (root) to[out=-90,in=180] (0.8,-1) to[out=0,in=-90] (1.5,0) to[out=90,in=0] (1,1) to[out=180,in=45] (root) ; 
\draw[kernels2,bend right=40] (t1) -- (root);
\draw[kernels2] (t2) -- (root);
\node[not] (rootnode) at (root) {};
\node[xi] at (t1) {};
\draw[symbols] (2.5,2) -- (2.5,-2);
\node[xic] at (2.5,0) {};
\end{tikzpicture}
\end{equ}
Here, $\<generic>$ denotes $\sigma_1$, $\<genericb>$ denotes $\sigma_2$, $\<not>$ denotes
$\Gamma$, thick edges $\<thick>$ denote those edges that terminate on one of the ``original'' indices
for a given function, and thin edges $\<thin>$ denote those that terminate on one of the indices generated
by a derivation. The ``half-edges'' that are only connected to one vertex denote the free indices,
with ``upper indices'' represented by the outgoing half edges terminating at the bottom, while the 
``lower indices'' are represented by the incoming half edges starting at the top.

\begin{remark}
The reason why we choose free \textit{upper} indices to correspond to edges terminating at the \textit{bottom}
rather than the other way around is that our vector fields (with just one free upper index) will naturally
be in correspondence with the trees appearing in the description of our regularity structure, with the
root being connected to the free index. The clash in terminology comes from conforming 
simultaneously to the conventions for graphical representations of trees (root at the bottom) 
common in the algebraic literature and Einstein conventions for indices (free indices for
vectors at the top).
\end{remark}

\subsection{Definition of a \texorpdfstring{$T$}{T}-algebra}
\label{sec:Talgebra}

For $\upper\in\N$ we write $[\upper]:=\{i\in\N: 1\leq i\leq \upper\}$ \label{upper page ref} and $\Sym(\upper)$ \label{sym page ref} for the symmetric group on $[\upper]$. We also set $\Sym(\upper,\low) = \Sym(\upper) \times \Sym(\low)$ \label{product sym page ref} 
where $\times$ denotes the usual direct product of groups. Given 
$\upper_1,\upper_2\in\N$ and $\alpha_i\in \Sym(\upper_i)$, we define the 
concatenation $\alpha_1\cdot\alpha_2\in\Sym(\upper_1+\upper_2)$ of $\alpha_1$ and $\alpha_2$ as
\[
(\alpha_1\cdot\alpha_2) (i) = 
\left\{\begin{array}{cl}
	\alpha_1(i) & \text{for $i\leq \upper_1$,} \\
	\upper_1+\alpha_2(i-\upper_1)& \text{otherwise.}
\end{array}\right.
\]
This yields natural embeddings 
\begin{equs}
\Sym(\upper_1,\low_1)\times \Sym(\upper_2,\low_2) &\to \Sym(\upper_1+\upper_2,\low_1+\low_2) \\
(\alpha_1,\alpha_2) &\mapsto \alpha_1 \cdot \alpha_2\;,
\end{equs}
but note that $\alpha_1 \cdot \alpha_2 \neq \alpha_2 \cdot \alpha_1$.
In this context, it will be convenient to write $\swap_{\low_1,\low_2}^{\upper_1,\upper_2}$ \label{swapping map page ref} for the element 
of $\Sym(\upper_1+\upper_2,\low_1+\low_2)$ `swapping the two factors'
so that 
\begin{equ}
\swap_{\low_1,\low_2}^{\upper_1,\upper_2} ( \alpha_1 \cdot \alpha_2) = (\alpha_2 \cdot\alpha_1) \swap_{\low_1,\low_2}^{\upper_1,\upper_2}\;.
\end{equ}
We omit an index if its value vanishes (this is unambiguous) and we write 
$\id^\upper_\low$ for the identity on $\Sym(\upper,\low)$. (So we could actually write $\swap^\upper_\low$ 
instead of $\id^\upper_\low$.)
First, we formalise the notion of a $T$-algebra in an abstract way.

\begin{definition}\label{def:Talg}
A $T$-algebra $\CV$ consists of a (bi)graded vector space 
$\CV = \bigoplus \{\CV_\low^\upper\,:\, \upper,\low \ge 0\}$ together with the following additional data.
\begin{claim}
\item For every $\upper,\low \ge 0$, a left action of $\Sym(\upper,\low)$ onto $\CV_\low^\upper$.
\item For any $\upper_i,\low_i \ge 0$, a bilinear product
$\mult \colon \CV_{\low_1}^{\upper_1} \times \CV_{\low_2}^{\upper_2} \to \CV_{\low_1+\low_2}^{\upper_1+\upper_2}$
which is associative and satisfies the coherence properties
\begin{equ}[e:multperm]
B \mult A = \swap^{\upper_1,\upper_2}_{\low_1,\low_2}(A \mult B)\;,\quad
\alpha_1 A \mult \alpha_2 B = (\alpha_1 \cdot \alpha_2)(A\mult B)\;,
\end{equ}
for any $\alpha_i \in \Sym(\upper_i,\low_i)$. Furthermore, this product admits a unit $\one \in \CV_0^0$.
\item For any $\upper,\low \ge 0$, a linear map $\tr \colon \CV_{\low+1}^{\upper+1} \to \CV_{\low}^{\upper}$ \label{trace page ref} such that
\begin{equs}[2]\label{e:trperm}
\alpha \tr A & = \tr \bigl((\alpha \cdot \id_1^1)A\bigr), &\quad \forall \ \alpha &\in \Sym(\upper,\low), \ A \in \CV_{\low+1}^{\upper+1},
\\ \label{e:trcomm}
\tr^2 A & = \tr^2 \big( (\id_{\low}^{\upper} \cdot \swap_{1,1}^{1,1})A\big),&\quad \forall \ A &\in \CV_{\low+2}^{\upper+2}, 
\\ \label{e:trtensor}
\tr(A\mult B) & = A \mult \tr B,&\quad \forall \ A &\in \CV, \ B \in  \CV_{\low+1}^{\upper+1}\;.
\end{equs}
\end{claim}
\end{definition}

\begin{remark}
It was pointed out to us by Ralf Kaufmann
that the definition of a $T$-algebra given here is essentially identical to that of a 
`wheeled PROP' given in \cite{wPROP}. Why this is so should become clearer from the characterisation
of free $T$-algebras given in Proposition~\ref{prop:freeT} below. Note also that there is a
very general notion of a $\CT$-algebra for $\CT$ a monad and our definition of a $T$-algebra
is a special case of $\CT$-algebra for a specific choice of $\CT$ (the same
as the monad $\Gamma^\circlearrowleft$ in \cite[Def.~2.1.7]{wPROP}, but without the inclusion of the `exceptional graphs').
Of course, closely related algebraic structures admitting a `graphical calculus' as in Section~\ref{sec:graphicalCalculus} below
have a long history, probably starting with Penrose's graphical notation \cite{MR0281657,MR1113284}. 
See also \cite[Def.~2.2]{cebron2020traffic} for a recent work 
where a closely related structure naturally appears in non-commutative probability theory.
\end{remark}

\begin{remark}
As usual, we don't actually need $\CV$ to be a vector space, but could consider modules over a ring
instead. 
\end{remark}

\begin{remark}
We will say that $\tau \in \CV$ has degree $(\upper,\low)$ to indicate that $\tau \in \CV_{\low}^{\upper}$.
Note that the identity \eqref{e:trtensor} only holds if $B$ belongs to some $\CV_{\low}^{\upper}$ with both $\low$ and $\upper$ at least one. 
\end{remark}

\begin{remark}\label{rem:Talgebra}
An example of a $T$-algebra is obtained by choosing any vector space $V$ and setting 
$\CV[V] = \bigoplus_{\upper,\low} \CV_\low^\upper[V]$ where
\begin{equ}
\CV_\low^\upper[V] = (V^*)^{\otimes \low} \otimes V^{\otimes \upper} \;, \qquad \upper,\low\geq 0,
\end{equ}
and $V^*$ is the dual space of $V$. The group $\Sym(\upper,\low)$ acts in the natural way
by permuting the factors in the sense that, for $\alpha = (\alpha_\upper,\alpha_\low) \in \Sym(\upper,\low)$,
we set
\begin{equs}
\alpha \bigl(&v_1^*\otimes\ldots \otimes v_\low^* \otimes v_1\otimes\ldots \otimes v_\upper \bigr)\label{e:actionVV} \\
&\qquad = 
v_{\alpha_\low^{-1}(1)}^*\otimes\ldots \otimes v_{\alpha_\low^{-1}(\low)}^* \otimes v_{\alpha_\upper^{-1}(1)}\otimes\ldots \otimes v_{\alpha_\upper^{-1}(\upper)}\;.
\end{equs}
The product is the usual tensor product, except that one sets
\begin{equ}
(A_1 \otimes B_1) \mult (A_2 \otimes B_2) = (A_1\otimes A_2) \otimes (B_1 \otimes B_2)\;,
\end{equ}
with $A_i \in (V^*)^{\otimes \low_i}$ and $B_i \in V^{\otimes \upper_i}$.
The operator $\tr$ is defined by
\begin{equ}
\tr \big((f_1\otimes \ldots \otimes f_{\low+1})\otimes (v_1\otimes \ldots\otimes v_{\upper+1})\big)
= f_{\low+1}(v_{\upper+1}) (f_1\otimes \ldots \otimes f_{\low})\otimes (v_1\otimes \ldots\otimes v_{\upper})\;.
\end{equ}
It is easy to check that this satisfies all the properties mentioned above.
\end{remark}

\begin{definition}\label{def:product}
A $T$-algebra with derivation, also called a $T_\d$-algebra, is a $T$-algebra $\CV$ endowed with a \textit{derivation} $\d$,
namely a collection of linear
maps $\d \colon \CV_\low^\upper \to \CV_{\low+1}^{\upper}$ \label{derivation page ref} with the following coherence properties
for $A \in \CV_\low^\upper$ and $\alpha \in \Sym(\upper,\low)$:
\begin{equ}[e:d2]
\d(\alpha A) = (\id_1^0 \cdot \alpha)\,\d A\;,\qquad
\d^2A = (\swap_{1,1} \cdot \id^\upper_\low) \,\d^2 A\;.
\end{equ}
We furthermore assume that $\d$ behaves ``nicely'' with respect to both $\mult$ and $\tr$ in
the sense that for $A_i\in\CV_{\low_i}^{\upper_i}$
\begin{equ}[e:Leibniz]
\d(A_1\mult A_2) = \d A_1 \mult A_2 + (\swap_{\low_1,1}^{\upper_1} \cdot \id^{\upper_2}_{\low_2}) (A_1 \mult \d A_2)\;,
\end{equ}
as well as
\begin{equ}
\d \tr A = \tr \d A\;, \qquad A \in \CV_{\low+1}^{\upper+1}.
\end{equ}
\end{definition}

\begin{remark}\label{rem:Lie}
Any $T_\d$-algebra $\CV$ comes with natural bilinear ``grafting'' operations $\graftI \colon \CV_k^1 \times \CV_\ell^1 \to \CV_{k+\ell}^1$
 given by
\begin{equ}[e:graft]
A \graftI B = \tr \big(S^2_{1,k+\ell}(\d B \mult A)\big)\;.
\end{equ}
It is an instructive exercise to verify that this turns $\CV_0^1$ 
into a pre-Lie algebra. For this, one first verifies that 
\begin{equ}
A \graftI (B \graftI C) - (A\graftI B)\graftI C = \tr^2(\d^2 C \mult A \mult B)
\end{equ}
and
then notices that this is symmetric in $A,B$ thanks to \eqref{e:trcomm}
and the second identity in \eqref{e:d2}.
In particular, $\CV_0^1$  is a Lie algebra with Lie bracket $[A,B] = A \graftI B - B \graftI A$.
\end{remark}

\begin{remark}
The second identity in \eqref{e:d2} encodes the fact that we are considering ``flat'' spaces
at the algebraic level. Indeed, this is the only identity that breaks if we take for $\CV$ the tensor bundle
over a smooth manifold and for $\d$ the covariant derivative. Removing this condition
however would break the pre-Lie structure of $\CV_0^1$ mentioned above.
\end{remark}

\begin{remark}\label{rem:symmetry}
Combining both identities in \eqref{e:d2} shows that one also has
\begin{equ}
\d^m A = (\alpha \cdot \id_\low^\upper) \d^m A\;,
\end{equ}
for every $m \ge 1$ and every $\alpha \in \Sym(0,m)$. Similarly, \eqref{e:trperm} and \eqref{e:trcomm}
show that 
\begin{equ}
\tr^m A = \tr^m \bigl((\id_\low^\upper\cdot \alpha)  A\bigr)\;,
\end{equ}
for every permutation $\alpha =\delta\times\delta\in \Sym(m,m)$ with $\delta\in \Sym(m)$ and every $A \in \CV^{\upper+m}_{\low+m}$.
\end{remark}

Given two $T$-algebras $\CV$ and $\CW$, a linear map $\phi\colon \CV \to \CW$ is a morphism of 
$T$-algebras if it preserves all of structure given in Definition~\ref{def:Talg} and we write
$\phi \in \Hom(\CV,\CW)$. If $\CV$ and $\CW$ also admit derivations, we write 
$\Hom_\d(\CV,\CW) \subset \Hom(\CV,\CW)$ if $\phi(\d v) = \d \phi(v)$ for all $v \in \CV$.
As usual, an \textit{ideal} of $T$-algebras
is a linear subspace that is a two-sided ideal for the multiplication and that is stable under the
trace and action of the symmetric group while an ideal of  $T_\d$-algebras is furthermore left invariant
by the derivation.

Given 
$\iota \in \Hom_\d(\CV, \CW)$, a linear map $\phi \colon \CV \to \CW$ is an infinitesimal morphism 
with respect to $\iota$ if it satisfies the property
\begin{equ}[infinit-morph]
\phi(A \mult B) = \phi(A) \mult \iota(B) + \iota(A) \mult \phi(B)\;,
\end{equ}
as well as
\begin{equ}
\phi(\alpha A) = \alpha \phi(A)\;,\qquad \phi(\tr A) = \tr \phi(A)\;,\qquad
\phi(\d A) = \d \phi(A)\;,
\end{equ}
for every homogeneous $A$ and every element $\alpha$ of the symmetric group such that these
operations make sense. In particular, an infinitesimal morphism is uniquely determined by its
action on any set of generators for $\CV$. In our case, the morphism $\iota$ will usually be given
by some natural inclusion $\CV \subset \CW$, in which case we omit it from the terminology.

\begin{remark}\label{rem:tensorT}
A typical example of $T$-algebra with derivation relevant for this article is the following. Fix a finite-dimensional
vector space $V$ and let $\CV[V]$ be the $T$-algebra based on $V$ considered in Remark~\ref{rem:Talgebra}.
We then have a $T$-algebra $\CW[V]$ obtained by taking for $\CW_\low^\upper$
the space of all smooth functions $V \to \CV_\low^\upper[V]$. The product, action of the 
symmetric group, and trace $\tr$ are defined 
pointwise, exactly as in Remark~\ref{rem:Talgebra}. The derivation $\d$ is then given by the usual Fr\'echet 
derivative with the canonical identification
\begin{equ}
L(V, \CV_\low^\upper) \approx V^* \otimes \CV_\low^\upper[V] \approx \CV_{\low+1}^{\upper}[V]\;.
\end{equ}
It is a straightforward exercise to verify that it satisfies the required coherence conditions,
with the second part of \eqref{e:d2} a consequence of the fact that the second derivative is symmetric
and \eqref{e:Leibniz} a consequence of the Leibniz rule.
\end{remark}

\begin{remark}
We will repeatedly and without further comment make use of the fact that, given $T$-algebras 
$T_i$ and a commuting diagram
of \textit{linear maps}
\begin{equ}
\begin{tikzcd}[row sep=1.5ex, column sep=11ex]
&   T_{2} \arrow[dr, "\psi_2"] & \\
T_1  \arrow[ur,"\psi_1"]  \arrow[rr, "\psi" description]
&&   T_3
\end{tikzcd}
\end{equ} 
if we know that $\psi$ and $\psi_1$ are morphisms of $T$-algebra and $\psi_1$ is surjective,
then $\psi_2 \in \Hom(T_2,T_3)$. Similarly, if 
$\psi$ and $\psi_2$ are morphisms and $\psi_2$ is injective, then $\psi_1$ is a morphism.
\end{remark}

\subsection{A class of \texorpdfstring{$T$}{T}-algebras}\label{sec:graphicalCalculus}

\begin{definition}
A $T$-graph $g = (V_g,\mfd_g, \phi_g)$ of degree $(\upper,\low)$
consists of a finite vertex set  $V_g$,
a map $\mfd_g \colon V_g \to \N^2$ (we will use the notation $\mfd_g(v) = (o_v, i_v)$), 
as well as a bijection $\phi_g\colon \Out(g) \to \In(g)$ 
where
\begin{equs}[2]
\Out(g) &= [\low] \sqcup \overline{\Out}(g) \;,\qquad&\quad \In(g) &= [\upper] \sqcup \overline{\In}(g)\;, 
\\
\overline{\Out}(g) &= \bigsqcup_{v\in V_g} \{v\}\times [o_{v}]\;,\quad&\quad \overline{\In}(g) &=  \bigsqcup_{v\in V_g} \{v\}\times [i_{v}]\;.\label{e:defTg2}
\end{equs}
We furthermore impose that $\phi_g([\low]) \subset \overline{\In}(g)$ (or equivalently that
$\phi_g^{-1}([\upper]) \subset \overline{\Out}(g)$). We also write $\one$ for the unique $T$-graph 
with empty vertex set, which is necessarily of degree $(0,0)$. We also define
the shorthand  $\Vert(g) = \Out(g) \sqcup \In(g)$.
\end{definition}

As expected, two $T$-graphs $g, \bar  g \in T^\upper_\low$ are identified if there exists a 
bijection $\iota \colon V_g \to V_{\bar  g}$ such that 
$\mfd_{\bar g} \circ \iota = \iota \circ \mfd_g$ and $\phi_{\bar g} \circ \hat \iota = \hat \iota \circ \phi_g$,
where $\hat \iota \colon  \Vert(g) \to \Vert(\bar g)$ denotes
the obvious extension of $\iota$.
We then write $T^\upper_\low$ for the collection of all $T$-graphs of degree $(\upper,\low)$, modulo 
identification. 

\begin{wrapfigure}[5]{r}{2.5cm}
\vspace{-1.8em}
\tikzset{external/export next=false}
\begin{tikzpicture}[scale=0.2,baseline=-2,draw=symbols,line join=round,decoration={
    markings,
    mark=at position 0.5 with {\arrow{>}}}]
\draw [black,postaction={decorate}] (-2,4.35) to[out=-90,in=90] (0.3,0.55);
\draw [very thick,white] (2,4.35) to[out=-90,in=90] (-0.3,0.55);
\draw [black,postaction={decorate}] (2,4.35) to[out=-90,in=90] (-0.3,0.55);
\draw [black,postaction={decorate}] (0.3,-0.55) to (0.3,-2);
\draw [black,postaction={decorate}] (-1.8,7) to (-1.8,5.3);
\draw [black,postaction={decorate}] (2,7) to (2,5.45);
\draw [black,postaction={decorate}] (-0.3,-0.55) .. controls (-0.3,-4.7) and (-6.7,8.7) .. (-2.2,5.2);
\draw (0,0) node[rec] {}; 
\draw (-2,4.9) node[trup] {};
\draw (2,4.9) node[cerc] {};
\end{tikzpicture}
\vspace{-.5em}
\end{wrapfigure}
On the right, we see a graphical representation of a $T$-graph
$g \in T^1_2$ with one vertex of degree $(1,2)$ (red triangle), one of degree $(1,1)$ (blue circle)
and one of degree $(2,2)$. The map $\phi_g$ is represented by directed edges connecting 
$v \in \Out(g)$ to $\phi_g(v) \in \In(g)$. 
Note that order matters, so this $T$-graph is distinct from
the one obtained for example by swapping the endpoints of the two lines entering the 
black square from the top. In subsequent graphical representations, we will always use
the convention that edges enter vertices from the top and leave them from the bottom.
Similarly, edges leaving from $[\ell]$ will be drawn as half-edges entering from the top
while edges targeting $[u]$ are drawn as half-edges exiting at the bottom. 

\begin{remark}\label{rem:graph}
A $T$-graph $g$ defines a directed graph $\hat G_g$ on the vertex set
$\hat V_g \eqdef [\ell] \sqcup [u] \sqcup V_g$ by postulating that
each element $e \in \Out(g)$ yields an edge from $\pi(e)$ to $\pi(\phi_g(e))$. Here, $\pi\colon \Vert(g) \to \hat V_g$ is the natural projection preserving $[u]\sqcup [\ell]$ and mapping elements of the 
form $(v,i) \in V_g\times \N$ to
$v$. Given $e \in \Out(g)$, we will also use the notations
\begin{equ}
\dom e \eqdef \pi(e)\;,\qquad \cod e \eqdef \pi(\phi_g(e))\;.
\end{equ}

\end{remark}

We do have a natural ``trace'' operation $\tr \colon T^{\upper+1}_{\low+1} \to T^\upper_\low$ given
by setting $V_{\tr g} = V_g$, $\mfd_{\tr g} = \mfd_{g}$ and defining $\phi_{\tr g}$ by
\begin{equ}
\phi_{\tr g}(o) = 
\left\{\begin{array}{cl}
	\phi_g(\low + 1) & \text{if $\phi_g(o) = \upper + 1$,} \\
	\phi_g(o) & \text{otherwise.}
\end{array}\right.
\end{equ}
We also have an action of $\Sym(\upper,\low)$ on $T^\upper_\low$: given $\alpha = (\alpha^\inc,\alpha^\out) \in \Sym(\upper,\low)$, we consider its canonical representation by permutation maps on $[\upper] \sqcup [\low] \subset \Vert(g)$, which we extend to all of $\Vert(g)$ by the identity  on the remainder.
With this notation, we obtain a new $T$-graph $\alpha \bigcdot g$ by
setting $\alpha \bigcdot g = (V_g,\mfd_g, \alpha \bigcdot \phi_g)$ where
\begin{equ}
\alpha \bigcdot \phi_g \eqdef \alpha \circ \phi_g \circ \alpha^{-1}\;.
\end{equ}

Given two $T$-graphs $g \in T^\upper_\low$ and $\bar g \in T^{\bar\upper}_{\bar\low}$, we define their product $g \cdot \bar g = (V_g \sqcup V_{\bar g}, \mfd_g \sqcup \mfd_{\bar g}, \phi_{g\cdot \bar g}) \in T^{\upper+\bar\upper}_{\low+\bar\low}$ by setting $\phi_{g\cdot \bar g} = \phi_{g} \sqcup \phi_{\bar g}$, where
we make use of the identifications $\Vert(g \cdot \bar g) \simeq \Vert(g) \sqcup \Vert(\bar g)$. Here, the identification $[\low]\sqcup [\bar \low] \simeq [\low + \bar \low]$
is given by identifying $[\bar \low]$ with $\{\low+1,\ldots,\low + \bar \low\}$ (and similarly for 
$[\upper]\sqcup [\bar \upper] \simeq [\upper + \bar \upper]$).
Note that one then has the identity $S^{\upper,\bar\upper}_{\low,\bar \low}\act(g\cdot \bar g) = \bar g \cdot g$,
so that $\bar g \cdot g \neq g \cdot \bar g$ in general.

Given a typed vertex set $(V,\mfd)$, an important role is played by the group of `internal' 
symmetries $\Sym_V  = \bigtimes_{v \in V} \Sym(i_v,o_v)$ acting on the collection $T_V$ of 
$T$-graphs with vertex set $V$. 
To describe this action, note that given any $T$-graph $g = (V,\mfd,\phi) \in T_V$, 
$\alpha = (\alpha^\inc_v,\alpha^\out_v)_{v \in V}  \in \Sym_V$ acts naturally on 
$\Vert(g)$ by 
\begin{equs}[e:alph]
\alpha(v,i) & = (v,\alpha^\inc_v(i)) \quad  \text{for }(v,i) \in \overline{\In}(g), 
\\ \alpha(v,o) &= (v,\alpha^\out_v(o))\quad  \text{for }(v,o) \in 
\overline{\Out}(g), 
\\ \alpha(k) &=k \quad  \text{for }k \in [\upper]\sqcup[\low].
\end{equs}
Therefore, $\alpha \in \Sym_V$ defines a new $T$-graph $\alpha \actint g \in T_V$ exactly as above,
namely by setting $\alpha \actint g = (V,\mfd,\alpha \circ \phi \circ \alpha^{-1})$.
Note that if we also fix the degree $(\upper,\low)$ of our graphs, the actions of $\Sym_V$ and 
$\Sym(\upper,\low)$ commute, so we actually have an action of $\Sym(\upper,\low) \times \Sym_V$, 
but the two play quite different roles which is why we keep separate notations. Instead of fixing 
the vertex set $V$, we will usually start from a given $T$-graph $g$ and we write $\Sym_g$ as 
a shorthand for $\Sym_{V_g}$.

\begin{remark}\label{rem:actint}
One subtlety that is easy to miss is that it may well happen that $\alpha \actint g = g$
in $T^\upper_\low$, but that the corresponding identification $\iota \colon g \to g$ can \textit{not} 
be taken as the identity on $V_g$. The following example has $\Sym_g \simeq \Z_2$ and the action of the non-trivial
element $\alpha$ is given by 
\begin{equ}
\alpha \actint 
\tikzset{external/export next=false}
\begin{tikzpicture}[scale=0.2,baseline=0,draw=symbols,line join=round]
\node at (0,0) [rec,thin,fill=blue!10,minimum width=4mm] (rect) {};
\draw[black] (1,2) node[cerc,thin,fill=white] {} to[out=-90,in=90] ($(rect.north east)-(3mm,0)$);
\draw[black] (-1,2) node[cerc,thin,fill=red!10] {} to[out=-90,in=90] ($(rect.north west)+(3mm,0)$);
\end{tikzpicture}
=
\tikzset{external/export next=false}
\begin{tikzpicture}[scale=0.2,baseline=0,draw=symbols,line join=round]
\node at (0,0) [rec,thin,fill=blue!10,minimum width=4mm] (rect) {};
\node at (1,2) [cerc,thin,fill=white] (c1) {};
\node at (-1,2) [cerc,thin,fill=red!10] (c2) {};
\draw[black] (c1) to[out=-90,in=90] ($(rect.north west)+(3mm,0)$);
\draw[white,very thick] (c2) to[out=-90,in=90] ($(rect.north east)-(3mm,0)$);
\draw[black] (c2) to[out=-90,in=90] ($(rect.north east)-(3mm,0)$);
\end{tikzpicture}
=
\tikzset{external/export next=false}
\begin{tikzpicture}[scale=0.2,baseline=0,draw=symbols,line join=round]
\node at (0,0) [rec,thin,fill=blue!10,minimum width=4mm] (rect) {};
\draw[black] (1,2) node[cerc,thin,fill=red!10] {} to[out=-90,in=90] ($(rect.north east)-(3mm,0)$);
\draw[black] (-1,2) node[cerc,thin,fill=white] {} to[out=-90,in=90] ($(rect.north west)+(3mm,0)$);
\end{tikzpicture}\;,
\end{equ}
where vertices that are identified are drawn in the same colour.
\end{remark}

Let now $W = \{W^u_\ell\,:\, u,\ell \ge 0\}$ be a collection of vector spaces such that
each $W^u_\ell$ is endowed with an action of $\Sym(u,\ell)$. 
Given $W$ and a $T$-graph $g$, we then set 
\begin{equ}
\widehat{\Tr}_g(W) \eqdef \Big( \bigotimes_{v \in V_g} W_{i_v}^{o_v}\Big)\Big/{\sim}\;,
\end{equ}
where $\sim$ is the smallest linear equivalence relation such that, for 
$w = \bigotimes_{v\in V_g} w_v$ and any identification $\iota \colon g \to g$, one has
\begin{equ}[e:ident]
w \sim  \bigotimes_{v\in V_g} w_{\iota v}\;.
\end{equ}
It will be convenient to keep track of $g$ in our notations and to 
denote generic elements of $\widehat{\Tr}_g(W)$ by $g \otimes \bigotimes_{v \in V_g} w_v$.
We illustrate the equivalence relation \eqref{e:ident} with the following 
two connected $T$-graphs of respective types $(0,2)$ and $(0,0)$: 
\begin{equs}\label{exemple_equivalence}
\tikzset{external/export next=false}
\begin{tikzpicture}[scale=0.2,baseline=0,draw=symbols,line join=round]
\node at (0,0) [rec,thin,fill=blue!10,minimum width=4mm] (rect) {};
\node at (6,0) [rec,thin,fill=blue!10,minimum width=4mm] (rect1) {};
\draw[black] (-2,0) to[out=90,in=90] ($(rect.north west)+(3mm,0)$);
\draw[black] ($(rect1.south)$) to[out=-90,in=-90] (-2,0);
\draw[black] ($(rect.south)$) to[out=-90,in=90] ($(rect1.north west)+(3mm,0)$);
\draw[black] (7,2)  to ($(rect1.north east)-(3mm,0)$);
\draw[black] (1,2)  to ($(rect.north east)-(3mm,0)$);
\end{tikzpicture}
\qquad \qquad
\begin{tikzpicture}[scale=0.2,baseline=0,draw=symbols,line join=round]
\node at (0,0) [rec,thin,fill=blue!10,minimum width=4mm] (rect) {};
\node at (6,0) [rec,thin,fill=blue!10,minimum width=4mm] (rect1) {};
\draw[black] (-2,0) to[out=90,in=90] ($(rect.north west)+(3mm,0)$);
\draw[black] ($(rect1.south)$) to[out=-90,in=-90] (-2,0);
\draw[black] ($(rect.south)$) to[out=-90,in=90] ($(rect1.north west)+(3mm,0)$);
\end{tikzpicture}  
\end{equs}
The  first graph does not admit any non-trivial self-identification
since the two half-edges are distinguished. 
The second graph on the other hand admits one that swaps its two vertices. 

\begin{remark}\label{rem:symmetric}
The presence of the quotienting operation by $\sim$ 
guarantees that if $g \sim \bar g$ are identified, then there is a canonical
isomorphism $\widehat{\Tr}_g(W) \simeq \widehat{\Tr}_{\bar g}(W)$. Without this
equivalence relation such an isomorphism would still exist, but it wouldn't be 
canonical unless $g$ has no non-trivial internal symmetries. Our notation then implies
that for any identification $\iota \colon g \to \bar g$, one has
$g \otimes \bigotimes_{v \in V_g} w_{\iota v} = \bar g \otimes \bigotimes_{\bar v \in V_{\bar g}} w_{\bar v}$.
\end{remark}

We then define a $T$-algebra $\Tr(W)$\label{page:TrW} in the following way. 
In a first step, we define the vector spaces  
\begin{equ}
\widehat{\Tr}(W)^\upper_\low = \bigoplus_{g \in T^\upper_\low} \widehat{\Tr}_g(W)\;,
\end{equ}
where Remark~\ref{rem:symmetric} guarantees that this is well-defined despite the
abuse of notation made by implicitly assuming that we chose a specific representative
for each element of $T^\upper_\low$. 
Given $w = g \otimes \bigotimes_{v \in V_g} w_{v} \in \widehat{\Tr}_g(W)$ and 
$\alpha = (\alpha_v)_{v \in V_g} \in \Sym_g$, we then define 
\begin{equ}[e:defactint]
\alpha \bigcdot w \eqdef (\alpha \actint g) \otimes \bigotimes_{v \in V_{\alpha \actint g}} (\alpha_v w_v) \in \widehat{\Tr}_{\alpha \actint g}(W)\;,
\end{equ}
where we implicitly use the identification of $V_{\alpha \actint g}$ with $V_g$.

\begin{remark}
As already mentioned in Remark~\ref{rem:actint}, one may have $\alpha \actint g \sim g$,
but the corresponding identification may \textit{not} be the identity on $V_g$. In this case, 
the corresponding reordering of the factors is implicit in the notation \eqref{e:defactint}.
\end{remark}

With this notation at hand, we write $I^\upper_\low$ for the smallest linear subspace containing all 
elements of the form
\begin{equ}[e:equivrel]
w - \alpha \bigcdot w\;,
\end{equ}
for $w \in \widehat{\Tr}_g(W)$ with $g \in T^\upper_\low$ and $\alpha \in \Sym_g$. 
We then define
\begin{equ}
\Tr(W) = \bigoplus_{\upper,\low \ge 0}\Tr(W)^\upper_\low\;,\qquad  \Tr(W)^\upper_\low \eqdef \widehat{\Tr}(W)^\upper_\low / I^\upper_\low\;.
\end{equ}
We similarly write $\Tr_g(W)$\label{page:TrWg} for the image of $\widehat{\Tr}_g(W)$ under the canonical projection $\widehat{\Tr}(W) \to \Tr(W)$.

\begin{remark}\label{rem:viewTrg}
Besides the definition of $\Tr_g(W)$ just given, we can also view it 
as the subspace of $\widehat{\Tr}_{g}(W)$ consisting 
of the elements that are invariant under the action of $\Stab_{g}$, the stabiliser of 
$g$ in $\Sym_{g}$. Note that \eqref{e:equivrel} also yields canonical identifications
$\Tr_g(W) \simeq \Tr_{\alpha \actint g}(W)$ for $\alpha \in \Sym_g$.
\end{remark}

\begin{remark}\label{ident:dual}
Setting $(W^*)^\upper_\low \eqdef (W^\upper_\low)^*$, we note that 
$\Sym(\upper,\low)$ naturally acts on it by postulating that $(\alpha w)(\tau)
= w(\alpha^{-1}\tau)$ for $\tau \in W^\upper_\low$ and $w \in (W^\upper_\low)^*$.
By Remark~\ref{rem:viewTrg}, this allows us to identify
$\Tr_g(W)^*$ with $\Tr_g(W^*)$ and, more generally, $\bigl(\Tr(W)^\upper_\low\bigr)^*$
with $\Tr(W^*)^\upper_\low$. This identification is consistent in the sense that if $\hat g = \alpha \actint g$
for some $\alpha \in \Sym_g$, then the identification $\Tr_{\hat g}(W)^* \simeq \Tr_{\hat g}(W^*)$
is consistent with the identification $\Tr_{\hat g}(W) \simeq \Tr_{g}(W)$ given by \eqref{e:equivrel}.
\end{remark}

The proof of the following result is straightforward and without surprises.

\begin{lemma}
The product, trace, and action of the symmetric group 
extend to $\Tr(W)$ and turn it into a $T$-algebra.\qed
\end{lemma}

\begin{remark}\label{rem:identification}
The reason for quotienting by $I^\upper_\low$ defined in this particular way is as follows. 
Consider the ``elementary'' $T$-graph $g \in T^\upper_\low$
with $V_g = \{\bullet\}$ consisting of a single vertex of type $(\upper,\low)$, 
$\phi_g (j) = (\bullet,j)$ for $j \in [\low]$, and 
$\phi_g (\bullet,i) = i$ for $i \in [\upper]$. Then, we have a canonical injection 
$W^\upper_\low \to \widehat{\Tr}(W)^\upper_\low$ given by $\tau \mapsto g \otimes \tau$. 
For this particular $g$, one also
has $\Sym_g \simeq \Sym(\upper,\low)$ and one readily verifies from the definitions that
one has the identity $\alpha \bigcdot g = \alpha^{-1} \actint g$. (This is not true in general, only for 
this specific $g$!)
It follows that, with the same notations as above,
\begin{equ}
\alpha \bigcdot (g \otimes \tau)
\eqdef (\alpha \bigcdot g) \otimes \tau = (\alpha^{-1} \actint g) \otimes \tau \simeq g \otimes (\alpha \bigcdot \tau)\;,
\end{equ}
so that the identification $W^\upper_\low \simeq \Tr_g(W) \subset \Tr(W)^\upper_\low$ extends to 
the action of  $\Sym(\upper,\low)$ on both spaces.
\end{remark}

The following result shows that $\Tr(W)$ is ``free'', which shows that this construction is very natural.
\begin{proposition}\label{prop:freeT}
Given a $T$-algebra $\CV$ and linear maps $\Phi \colon W^\upper_\low \to \CV^\upper_\low$
 intertwining the action of $\Sym(\upper,\low)$, then there exists a unique extension of $\Phi$
to $\Tr(W)$ such that $\Phi \in \Hom(\Tr(W), \CV)$.
\end{proposition}

\begin{remark}\label{rem:linear}
In particular, given $\Phi, \Psi \in \Hom(\Tr(W), \CV)$, we can define
$\Phi + \Psi \in \Hom(\Tr(W), \CV)$ as the unique extension of the collection of
linear maps $\Phi + \Psi \colon W^\upper_\low \to \CV^\upper_\low$.
We also say that $\Phi \in \Hom(\Tr(W), \Tr(\bar W))$ is ``linear'' if it maps
$W^\upper_\low$ into $\bar W^\upper_\low$ for every $\upper,\low$.
Linear morphisms are important because their image is again of the form $\Tr(\hat V)$
for some $\hat V \subset V$, which need not be true in general.
\end{remark}

In order to prepare for the proof of this result, we first present a few preliminary definitions and results.
Given any $m \ge 1$, and any $(\upper,\low)$ with $\upper \wedge \low \ge m$,
we also have a `diagonal' representation $D$ of the symmetric group $\Sym(m)$ on 
$\Tr(W)^\upper_\low$ by
\begin{equ}
D_\sigma = (\id_{\upper-m}\cdot \sigma) \times (\id_{\low-m}\cdot \sigma)\;,\qquad \sigma \in \Sym(m)\;,
\end{equ}
acting on $\Tr(W)^\upper_\low$ via the representation of $\Sym(\upper,\low)$. 
We intentionally omit the indices $\upper$, $\low$ and $m$ from the notations
since these are always uniquely determined from context. We also
write $e^\upper_\low \in T^\upper_\low$ for the elementary graph defined in Remark~\ref{rem:identification}.
With these notations at hand, we have the following crucial preliminary result. 
\begin{lemma}\label{lem:representationg}
Every $T$-graph $g \in T^\upper_\low$ can be written as
\begin{equ}[e:reprg]
g = \tr^m \alpha \big(g_1 \mult\ldots\mult g_{n}\big)\;,\qquad g_i \in \{e^k_p\,:\,k,p \ge 0\}\;,\quad
\alpha \in \Sym(\upper+m,\low+m)\;,
\end{equ}
with the convention that $m=n=0$ if $g = \one$. Furthermore, if one also has
\begin{equ}
g = \tr^{m'} \alpha' \big(g_1' \mult\ldots\mult g_{n'}'\big)\;,
\end{equ}
then $m'=m$, $n'=n$, and there exist elements $\mu \in \Sym(n)$, $\nu \in \Sym(m)$ with the following properties.
\begin{enumerate}
\item One has $g_i' = g_{\mu(i)}$ for every $i$.
\item The element $\alpha'$ is related to $\alpha$ by
$\alpha' = D_\nu \cdot \alpha \cdot \hat \mu$,
where $\hat \mu \in \Sym(\upper+m, \low+m)$ is the unique element such that 
\begin{equ}
\hat \mu (V_{\mu(1)} \cdots V_{\mu(n)}) = V_1\cdots V_n \;,
\end{equ}
for any $T$-algebra $\CV$ and any collection of elements $V_i \in \CV$ such that 
$V_i$ has the same degree as $g_i$.
\end{enumerate}
\end{lemma}

\begin{proof}
The statement is trivial for $g = \one$, so we assume that 
 $g = (V_g,\mfd_g,\phi_g)$ with $n = |V_g| > 0$,  and we set 
$m = |\Int(g)|$, where the set of ``internal edges'' is given by
\begin{equ}
\Int(g) \eqdef  \overline{\Out}(g)\setminus \phi_g^{-1}([u])\;.
\end{equ}
We then consider an ordering $[n] \ni i \mapsto v_i \in V_g$ of its vertices, 
as well as an ordering $[m] \ni i \mapsto e_i \in \Int(g)$ of its internal edges. 

We claim that this additional data
determines a representation of $g$ of the type \eqref{e:reprg}.
Indeed, note first that since we have ordered the vertices, we can naturally set
\begin{equ}
g_i = e^{\upper_i}_{\low_i}\;,\quad (\upper_i,\low_i) \eqdef \mfd_g(v_i)\;.
\end{equ}
We then obtain a natural ordering $\eta_1\colon \overline{\Out}(g) \to [\upper + m]$ by
choosing it to be the unique bijection such that 
$\eta_1(v_i,j) < \eta_1(v_k,\ell)$ if and only if $i < k$ or $i=k$ but $j < \ell$.

On the other hand, we obtain a natural map $\eta_2 \colon [\low + m] \to \overline{\In}(g)$
by setting $\eta_2(j) = \phi_g(j)$ for $j \in [\low]$ and, for $j \in [m]$,
we set $\eta_2(\low + j) = \phi_g(\eta_1^{-1}(\upper + j))$.
Using these maps, as well as the ordering $i \mapsto e_i$, we
build $\alpha = \alpha_1 \times \alpha_2 \in \Sym(\upper+m,\low+m)$ as follows.

For $\alpha_1$, we set 
\begin{equ}
\alpha_1^{-1}(j) = 
\left\{\begin{array}{cl}
	\eta_1(\phi^{-1}(j)) & \text{if $j \le \upper$,} \\
	\eta_1(e_{j-\upper}) & \text{otherwise.}
\end{array}\right.
\end{equ}
Regarding $\alpha_2$, we set it to be the only bijection such that
\begin{equ}
\alpha_2^{-1}(j) \in 
\left\{\begin{array}{cl}
	\eta_2^{-1}(\phi(j)) & \text{if $j \le \low$,} \\
	\eta_2^{-1}(\phi(e_{j-\upper})) & \text{otherwise.}
\end{array}\right.
\end{equ}
It is tedious but straightforward to show that this choice of $\alpha$ guarantees that
\eqref{e:reprg} holds. Furthermore, any two representations of $g$ built in this way
differ from \eqref{e:reprg} by an element of 
$\Sym(m) \times \Sym(n)$ representing
a relabelling of the internal edges and the vertices.

The fact that $g$ cannot be represented in any other way follows from the fact that we can easily
reverse this argument to extract the orderings $i \mapsto v_i$ and $i \mapsto e_i$ from
the representation \eqref{e:reprg}.
\end{proof}

\begin{proof}[of Proposition~\ref{prop:freeT}]
Given
any $T$-graph $g = (V_g,\mfd_g,\phi_g)$, we write it as
\begin{equ}
g = \tr^m \alpha \big(g_1 \mult\ldots\mult g_{n}\big)\;,
\end{equ}
as in Lemma~\ref{lem:representationg} and, given $\tau = \bigotimes_{i=1}^n w_i \in \widehat{\Tr}_g(W)$ 
(here $w_i$ is assigned to the unique vertex of $g_i$), we then set 
\begin{equ}
\Phi(\tau) = \tr^m \alpha \big(\Phi(w_1) \mult\ldots\mult \Phi(w_{n})\big)\;.
\end{equ}
It follows from Lemma~\ref{lem:representationg}, combined with Remark~\ref{rem:symmetry} that 
$\Phi$ is well-defined, so that it suffices to check that it is a morphism. 
The fact that $\Phi(\tr g) = \tr \Phi(g)$ if $\low\wedge\upper \ge 1$ and 
that $\Phi(\sigma g) = \sigma \Phi(g)$ for $\sigma \in \Sym(\upper,\low)$ follows immediately from
the well-posedness of $\Phi$. We also note that $\Phi(\one) = \one$ since $n=0$ in this case.

Regarding the product, we use the fact that for any $T$-algebra 
and any two elements $A$ and $B$ of degrees
at least $(\ell,\ell)$ and $(m,m)$ respectively, we 
can find an element $\alpha$ of the symmetric group
such that $\tr^\ell A \mult \tr^m B = \tr^{\ell+m} \alpha \bigl(A \mult B\bigr)$. Indeed, making repeated use of \eqref{e:multperm},
\eqref{e:trperm} and \eqref{e:trtensor}, we find that there are elements $\alpha_1$, $\alpha_2$ and $\alpha$ such that
\begin{equs}
\tr^\ell A \mult \tr^m B &= \alpha_1 \bigl(\tr^m B\mult \tr^\ell A\bigr)
= \alpha_1 \tr^\ell \bigl(\tr^m B\mult A\bigr)
= \tr^\ell \alpha_2 \bigl(A \mult \tr^m B\bigr) \\
&= \tr^\ell \alpha_2 \tr^m \bigl(A \mult B\bigr)
= \tr^{\ell+m} \alpha \bigl(A \mult B\bigr)\;.
\end{equs}
The required claim then follows at once.
\end{proof}

\subsection{A general injectivity result}
\label{sec:injective}

The main reason for using the formalism of $T$-algebras is to allow us to infer analytical
results from algebraic considerations. In order to ``go back'' to analysis, we need to show that the
specific $T$-algebras relevant for our argument, namely those of the form $\Tr(W)$, can be represented 
in a faithful manner by analytic objects which live in $T$-algebras of the type $\CW[V]$ as described 
in Remark~\ref{rem:tensorT}.
The linchpin of such arguments is the following injectivity result where, given a finite collection $\SN$ of
$T$-graphs, we write $\Tr_\SN(W)$ as a shorthand for $\bigoplus_{g \in \SN} \Tr_g(W)$. Here, we say that 
a $T$-graph $g \in T^\upper_\low$ is \textit{anchored} if $V_g$ is non-empty, $\upper + \low >0$, and
every $v \in V_g$ is connected to some element of $[\low]\sqcup [\upper]$ in $\hat G_g$. 

\begin{theorem}\label{theo:injectiveT}
Let $W$ as above be such that each $W^k_\ell$ is finite-dimensional and let $\SN$ be any finite
collection of connected anchored $T$-graphs. Then, there exists a finite-dimensional vector space $V$
and a morphism of $T$-algebras $\Phi \colon \Tr(W) \to \CV[V]$ which is 
injective on $\Tr_\SN(W)$.

Furthermore, given $\hat W \subset W$ invariant under the action of the symmetric group, 
there exists a second morphism 
$\hat \Phi\colon \Tr(W) \to \CV[V]$ such that $\Phi = \hat \Phi$ on $\Tr(\hat W)\subset \Tr(W)$ and, 
whenever $\tau \in \Tr_\SN(W)$
is such that $\Phi \tau = \hat \Phi \tau$, one has $\tau \in \Tr(\hat W)$.
\end{theorem}

\begin{proof}
Given a $T$-graph $g$, we first define a map
$\type_g \colon \Out(g) \sqcup \In(g) \to  (\N^2\times \N) \sqcup \{\typetop,\typebot\}$
as follows. 
For $e = (v,j) \in \overline{\Out}(g)\sqcup \overline{\In}(g)$, we set $\type_g(e) = (\mfd_g(v),j)$.
We also set $\type_g(e) = \typetop$ for $e \in [\low] \subset \Out(g)$
and $\type_g(e) = \typebot$ for $e \in [\upper] \subset \In(g)$.

We then introduce the following notion. Given $T$-graphs $g$ and $G$, 
a \textit{$T$-graph morphism} from $g$ to $G$ is a map
\begin{equ}
\psi \colon \Out(g) \to \Out(G)\;,
\end{equ}
with the following properties.
(Here, we say that two edges $e,e' \in \Out(g)$ `touch' each other 
if $\{\dom e, \cod e\} \cap \{\dom e' , \cod e'\} \neq \emptyset$, see Remark~\ref{rem:graph} for 
notations.)
\begin{enumerate}
\item If $e, e'\in \Out(g)$ touch each other, then 
$\psi(e)$ and $\psi(e')$ also touch each other.
Furthermore, they touch `in the same way' in the sense that 
for any $f_i \in \{\dom,\cod\}$, $f_1(e) = f_2(e')$ implies that
$f_1(\psi(e)) = f_2(\psi(e'))$.
\item On $\overline{\Out}(g)$, we have the identity
\begin{equ}[e:typematchdom]
\type_g = \type_G \circ \psi\;,
\end{equ}
while on $\phi_g^{-1}(\overline{\In}(g))$ we have the identity
\begin{equ}[e:typematchcod]
\type_g \circ \phi_g = \type_G \circ \phi_G \circ \psi\;.
\end{equ}
\end{enumerate}

We denote by $\CE(g,G)$ the set of all $T$-graph morphisms from $g$ to $G$.
From now on, we choose a finite set $\hat \SN \subset \SN$ with the property that for every
$g \in \SN$ one can find $\hat g \in \hat \SN$ which can be obtained from $g$ 
by the action of an element of $\Sym(\upper,\low) \times \Sym_g$. We also
order the elements of $\hat \SN$ so that $\hat \SN = \{g_1,\ldots,g_n\}$ for some $n \ge 1$ 
and we set $G = g_1\cdot \ldots\cdot g_n$.

We then define the finite-dimensional vector space $V$ by
\begin{equ}[e:defV]
V = \Vec\big(\Out(G)\big)\otimes \Tr_G(W)\;,
\end{equ}
and, for each $(k,\ell)$, we fix a finite index set $A^k_\ell$, as well as
elements $w_{i} \in W^k_\ell$
and $w_{i}^* \in (W^k_\ell)^*$ with $i \in A^k_\ell$ such that 
$ w_i^* (w_j) = \delta_{i,j}$. (This implies that $|A^k_\ell| \le \dim W^k_\ell$,
but we do not enforce equality.) In order to distinguish the different copies of
$\Tr_G(W)$, we will again denote a typical element of $V$ by $e \otimes \tau$ with
$e \in \Out(G)$ and $\tau \in \Tr_G(W)$.

For any $T$-graph $g$, we then denote by $\CF_g$ the set of all functions
$f \colon V_g \to \bigsqcup_{k,\ell}A^k_\ell$ with the property that
$f(v) \in A^{o_v}_{i_v}$ and, for $f \in \CF_g$, we write $w_f \in \Tr_g(W)$
for the element given by $\bigotimes_{v \in V_g} w_{f(v)}$, and similarly
for $w_f^* \in \Tr_g(W^*)$. 
We also note that Remark~\ref{ident:dual} yields the identification
\begin{equ}
V^* \simeq \Vec\big(\Out(G)\big)\otimes \Tr_G(W^*)\;,
\end{equ}
which will  be used repeatedly in the sequel. It will also be convenient to allow
to select subsets $\hat A^k_\ell \subset A^k_\ell$ and to write  $\hat \CF_g$ for
the set of functions $f \colon V_g \to \bigsqcup_{k,\ell}\hat A^k_\ell$ defined otherwise
like $\CF_g$.

Given a $T$-graph $g$, we have an equivalence relation on $V_g$ for which
$v \sim v'$ if and only if they are connected in the graph $\hat G_g$ defined in 
Remark~\ref{rem:graph}. Write $C_g$ for the set of these equivalence classes, so that
elements $H\in C_g$ denote subsets of $V_g$ corresponding to a connected component 
of $g$. 
Given a $T$-graph morphism $\psi \colon g \to G$, a set $H \subset V_g$,
and an element $f \in \CF_g$, we then write 
\begin{equ}
\CF_{G}^{\psi,f}\subset \CF_{G} \eqdef \prod_{H \in C_g}\CF_{H,G}^{\psi,f}\;,
\end{equ}
where, for any $H \subset V_g$, $\CF_{H,G}^{\psi,f} \subset \CF_{G}$ denotes
the set of all functions $\bar f$ such that $\bar f \circ \hat\psi = f$ when
restricted to $H$. Given $F \in \CF_{G}^{\psi,f}$, we denote by
$F_H$ its component in $\CF_{H,G}^{\psi,f}$. 
 
With these notations at hand, we define a  linear map $\Phi\colon \Tr(W) \to \CV[V]$
by postulating that, for every $T$-graph $g$ of degree $(\upper,\low)$ and 
every element $\tau \in \widehat{\Tr}_g(W)$, one has
\begin{equs}
\Phi(\tau) &= \sum_{f \in \hat\CF_g }\sum_{\alpha \in \Sym_g} w_{f}^*(\alpha \bigcdot \tau)\times \label{e:defPhi} \\
&\quad \times \sum_{\psi \in \CE(\alpha \actint g,G)}\sum_{F \in \CF_{G}^{\psi,f}} \Big( \bigotimes_{i \in [\low]} 
\bigl(\psi(i)\otimes w_{F_{H(i)}}^*\bigr) \otimes \bigotimes_{j \in [\upper]} 
\bigl(\psi(\phi_g^{-1}(j))\otimes w_{F_{H(j)}}\bigr)\Big) \\
&\in (V^*)^{\otimes \low} \otimes V^{\otimes \upper}\;,
\end{equs}
where $H(i) \in C_g$ denotes the (unique) connected component that is connected 
to $i$ in $\hat G_g$. 
This is indeed well-defined on $\Tr(W)$ thanks to the symmetrisation over 
$\Sym_g$.

We claim that $\Phi \in \Hom(\Tr(W),\CV[V])$. First, given two $T$-graphs $g_1$ and $g_2$, we note
that $\hat\CF_{g_1\cdot g_2} \simeq \hat\CF_{g_1}\times \hat\CF_{g_2}$ under the equivalence $f_1 \sqcup f_2 \simeq (f_1,f_2)$, 
$\Sym_{g_1\cdot g_2} \simeq \Sym_{g_1} \times \Sym_{g_2}$ (both as groups and as sets),
$\CE(\alpha \actint (g_1\cdot g_2),G) \simeq \CE((\alpha_1 \actint g_1)\cdot (\alpha_2 \actint g_2),G)\simeq \CE(\alpha_1 \actint g_1,G)\times \CE(\alpha_2 \actint g_2,G)$, and
$\CF_{G}^{\psi_1 \sqcup \psi_2,f_1 \sqcup f_2} \simeq \CF_{G}^{\psi_1 ,f_1} \times \CF_{G}^{\psi_2,f_2}$. The required multiplicativity property follows at once.

Given $\beta = (\beta_\upper,\beta_\low) \in \Sym(\upper,\low)$, we have a natural bijection
$\psi \mapsto \beta \psi$ between $\CE(g,G)$ and $\CE(\beta \bigcdot g,G)$ given  
by setting $\beta \psi \restr [\low] = \psi \circ \beta_\low^{-1}$ and $\beta \psi = \psi$ otherwise. 
Combining this with \eqref{e:actionVV} it is not difficult to verify that 
indeed $\Phi(\alpha \bigcdot \tau) = \alpha \bigcdot \Phi(\tau)$ as required.
Finally, since $(e\otimes w_{f}^*)(\bar e\otimes w_g) = \delta_{e,\bar e} \delta_{f,g}$
by definition, one has for $\upper,\low \ge 1$,
\begin{equs}
\tr \Phi(\tau) &=
 \sum_{f \in \hat\CF_g }\sum_{\alpha \in \Sym_g} w_{f}^*(\alpha \bigcdot \tau) \sum_{\psi \in \CE(\alpha \actint g,G)}\sum_{F \in \CF_{G}^{\psi,f}} \one_{\psi(\upper) = \psi(\phi_g^{-1}(\low))}\one_{F_{H(\upper)}= F_{H(\low)}} \\
&\qquad \times\Big( \bigotimes_{i \in [\low-1]} 
\bigl(\psi(i)\otimes w_{F_{H(i)}}^*\bigr) \otimes \bigotimes_{j \in [\upper-1]} 
\bigl(\psi(\phi_g^{-1}(j))\otimes w_{F_{H(j)}}\bigr)\Big) \;.
\end{equs}
We then note that $\hat\CF_g = \hat\CF_{\tr g}$, $\Sym_g = \Sym_{\tr g}$, $\tr(\alpha \actint g) = \alpha \actint \tr g$, and that there is a natural bijection $\iota$ between  those 
elements $\psi \in \CE(g, G)$ such that the last incoming and outgoing edge are both mapped 
to the same edge in $G$ and $\CE(\tr g, G)$. Furthermore, since taking the trace of $g$ 
connects the last outgoing edge to the last incoming edge, $\CF_{G}^{\iota \psi,f}$ is
identified with those $F \in \CF_{G}^{\psi,f}$ such that $F_{H(\upper)} = F_{H(\low)}$. 
It follows at once
that $\tr \Phi(\tau) = \Phi(\tr \tau)$ as required, so that $\Phi$ is indeed a morphism.

We now claim that $\Phi$ is injective on $\Tr_{\SN}(W)$.
For this, we say that a $T$-graph morphism $\psi\colon g \to G$ is  \textit{anchored}
if \eqref{e:typematchdom} and \eqref{e:typematchcod}
extend to all of $\Out(g)$ and furthermore $\psi$ is injective on the external
edges $[\low] \sqcup \phi_{g}^{-1}([\upper])$.

\begin{lemma}\label{lem:injectiveT}
Let $g_{1,2} \in T_\low^\upper$ be two anchored $T$-graphs with $g_1$ connected
and assume that there exist anchored $T$-graph morphisms $\psi_i\colon g_i \to G$ such that 
$\psi_1 = \psi_2$ on $[\low]$ and $\psi_1 \circ \phi_{g_1}^{-1} = \psi_2 \circ \phi_{g_2}^{-1}$
on $[u]$. Then, there exists a unique $T$-graph identification $\iota \colon V_{g_1}\to V_{g_2}$ 
such that furthermore $\psi_1 = \psi_2 \circ \hat \iota$.
\end{lemma}

Assuming that this statement holds for the moment,
let us fix some $\hat g \in T^\upper_\low$ such that $\alpha \bigcdot \hat g\in \hat \SN$
for some $\alpha \in \Sym(\upper,\low)$. Note that there is then a natural
anchored injective $T$-graph morphism $\iota \colon \hat g \to G$. (By 
Lemma~\ref{lem:injectiveT} it is only unique modulo an element of the stabiliser of
$\hat g$ in $\Sym(\upper,\low)$, which may be non-trivial in general.) One example is the first graph in \eqref{exemple_equivalence} where one can permute the two incoming edges. 
Having fixed $\iota$, we then set for any $\hat F \in \CF_{G}$
\begin{equ}
\Phi^\star_{\hat F} \eqdef \bigotimes_{i \in [\low]} 
\bigl(\iota(i)\otimes w_{\hat F}\bigr) \otimes \bigotimes_{j \in [\upper]} 
\bigl(\iota(\phi_{\hat g}^{-1}(j))\otimes w_{\hat F}^*\bigr)\;.
\end{equ}
Note that this is an element of
$V^{\otimes \low} \otimes (V^*)^{\otimes \upper} \simeq ((V^*)^{\otimes \low} \otimes V^{\otimes \upper})^* = (\CV^\upper_\low[V])^* \subset (\CV[V])^*$.
Comparing this to \eqref{e:defPhi} and applying Lemma~\ref{lem:injectiveT}, we conclude that, 
for $\tau \in \Tr_g(W)$, one has $\Phi^\star_{\hat F}(\Phi(\tau)) = 0$ unless
there exists $\beta \in \Sym_g$ such that $\hat g = \beta \actint g$. Since 
$\Tr_g(W) = \Tr_{\beta \actint g}(W)$, we can just as well assume that 
$g = \hat g$ and identify $\tau$ with an element of $\widehat{\Tr}_{\hat g}(W)$
invariant under the action of the stabiliser $\Stab_{\hat g}$ of $\hat g$ in $\Sym_{\hat g}$ (see \eqref{e:equivrel} and Remark~\ref{ident:dual}). 

Since furthermore $\hat g$ is connected, $\CF_G^{\psi,f}$ simply consists of those elements 
$F \in \CF_G$ such that $f = F \circ \psi$. Also, Lemma~\ref{lem:injectiveT} implies that there are no
non-trivial self-identifications so that $\widehat{\Tr}_{\hat g}(W) = \bigotimes_{v \in V_{\hat g}} W^\upper_\low$. 
We then obtain
\begin{equ}[e:adjoint]
\Phi^\star_{\hat F}(\Phi(\tau)) 
= \one_{\hat F\circ \iota \in \hat\CF_{\hat g}} \sum_{\alpha \in \Stab_{\hat g}} w_{\hat F\circ \iota}^*(\alpha \bigcdot \tau)
= \one_{\hat F\circ \iota \in \hat\CF_{\hat g}}|{\Stab_{\hat g}}|\,w_{\hat F\circ \iota}^*(\tau)\;,
\end{equ}
where, to obtain the last equality, we view $\tau$ as an element of $\widehat{\Tr}_{\hat g}(W)$ 
that is invariant under the action of $\Stab_{\hat g}$ as in Remark~\ref{rem:viewTrg}. Since $\hat F\circ \iota$
can be arbitrary, it follows at once that as soon as one sets $|A^k_\ell| = \dim W^k_\ell$ and
$\hat A^k_\ell = A^k_\ell$, 
one can use $\Phi^\star_{\hat F}$ to construct a left inverse to
$\Phi$ on $\Tr_{\SN}(W)$, which proves the first claim.

To prove the second claim, take $\Phi$ as above and consider $\hat \Phi$ built in the same way, but 
with $\hat A^k_\ell \subset A^k_\ell$ chosen in such a way that the elements $w_i$  with 
$i \in \hat A^k_\ell$ form a basis of $\hat W^k_\ell$. (We do of course need to choose the $w_i$ in a suitable
way for this to be possible.) We then have
$\Phi = \hat \Phi$ on $\Tr(\hat W)$ since, for $\tau \in \widehat{\Tr_g}(\hat W)$, one has $w_f^*(\tau) = 0$
unless $f \in \hat F_g$. Conversely, if $\tau \in \widehat{\Tr_g}(W)$ is such that 
$\Phi(\tau) = \hat \Phi(\tau)$, then it follows from \eqref{e:adjoint} that 
$w_{f}^*(\tau) = 0$ for every $f \in \CF_{\hat g}\setminus \hat \CF_{\hat g}$,
so that $\tau \in \widehat{\Tr_g}(\hat W)$ as required.
\end{proof}

\begin{proof}[of Lemma~\ref{lem:injectiveT}]
In a first step, we note that one must have $\range \psi_1 = \range \psi_2$. 
Indeed, let $\CA \subset \Out(G)$ be the subset constructed recursively as follows.
We set $\CA_0 = \psi_i([\low]) \cup (\psi_i \circ \phi_{g_i}^{-1})([u])$ which belongs to
the range of both $\psi_i$ by definition. Then, given $\CA_k$, we let $\CA_{k+1}$ be the 
set of all edges $e \in \Out(G)$ touching some edge $e' \in \CA_k$. Assuming that $\CA_k \subset \range \psi_i$,
the definition of an $T$-graph morphism then implies that the same is true for 
$\CA_{k+1}$, so that $\CA = \bigcup_k \CA_k \subset \range \psi_i$. On the other hand, the assumption that 
the $g_i$ are connected guarantees that $\range \psi_i = \CA$.

In a second step, we show that the $\psi_i$ are actually injective. Assume by contradiction that
$\psi_1$ (say) is not injective, so that there exist edges $e \neq e'$ such that 
$\psi_1(e) = \psi_1(e') = \hat e \in \CA$. Since $\CA$ is connected, we can find a path 
$(\hat e_k)_{k=0}^n$ of edges in $\CA$ such that $\hat e_0 = \hat e$, $\hat e_n \in \Out(G) \setminus \In(G)$
and such that, for every $k < n$, $\hat e_k$ and $\hat e_{k+1}$ touch each other in the same sense
as in the previous paragraph.
The definition of an $T$-graph morphism (in particular the ``local'' injectivity of $\psi_1$ on the set of edges
adjacent to any given vertex implied by property 2) 
then implies that there are unique paths $e_k$ and $e_k'$ with 
$\psi_1(e_k) = \psi_1(e_k') = \hat e_k$ and such that $e_k$ and $e_{k+1}$ touch each other,
and similarly for $e_k'$. Furthermore, these paths must stay disjoint since, if $e_k \neq e_k'$,
then $e_k$ and $e_k'$ cannot touch each other (again by local injectivity),
so that one must also have $e_{k+1} \neq e_{k+1}'$. It follows that $e_n \neq e_n'$, but these must both belong
to  $\Out(g_1) \setminus \In(g_1)$ by the definition of the $T$-graph morphism being anchored, which
contradicts the fact that $\psi_1$ is injective there.

Since both $\psi_i$ are bijections between $\Out(g_i)$ and $\CA$, this yields a unique
bijection $\hat \iota \colon \Out(g_1) \to \Out(g_2)$ such that $\psi_1 = \psi_2 \circ \hat \iota$.
One can verify that one actually has $\hat \iota \in \CE(g_1,g_2)$. The required map
$\iota$ is now obtained by mapping $v \in V_1$ to the unique $\iota(v) \in V_2$ such that
 $\dom e = v$  implies $\dom \hat \iota(e) = \iota(v)$ and similarly for $\cod e$.
One can verify that this is indeed an identification, thus concluding the proof. 
\end{proof}

It will be useful in the sequel to have at our disposal a ``local'' version of Theorem~\ref{theo:injectiveT}.
By Proposition~\ref{prop:freeT} and since $\SN$ is finite, the action of $\Phi$ on
$\Tr_\SN(W)$ is uniquely determined by its action of $W^\upper_\low$ 
with $\upper+\low \le N$ for some $N$. Since furthermore the space $V$ is finite-dimensional, 
the space $\Hom(\Tr(W), \CV[V])$ is endowed with a natural topology, generated for example by the
distance
\begin{equ}
d(\Phi, \bar \Phi) = \sum_{k,\ell \ge 0} 2^{-k-\ell} (1\wedge \|\Phi-\bar \Phi\,|\, L(W^k_\ell, \CV(V)^k_\ell)\|)\;,
\end{equ}
where $\|x|\CX\|$ denotes the norm of $x$, viewed as an element of the Banach space $\CX$.
We then have the following result.

\begin{proposition}\label{prop:injectiveLocal}
In the setting of Theorem~\ref{theo:injectiveT}, let $\CU \subset \Hom(\Tr(W), \CV[V])$
be open. Then, for every $\tau \in \Tr(W) \setminus \Tr(\hat W)$ there exist 
$\Phi, \bar \Phi \in \CU$ so that $\Phi(\sigma) = \bar \Phi(\sigma)$ for all 
$\sigma \in \Tr(\hat W)$ but $\Phi(\tau) \neq \bar \Phi(\tau)$.
\end{proposition}

\begin{proof}
Fix $\tau \in \Tr(W) \setminus \Tr(\hat W)$ and $\Phi_0 \in \CU$, note that one necessarily 
has some finite collection $\SN$
such that $\tau \in \Tr_\SN(W)$, and choose $\Phi_1$ and $\bar \Phi_1$
as in the second part of Theorem~\ref{theo:injectiveT}. We then define 
$\Phi_t \in \Hom(\Tr(W), \CV[V])$ to be the unique morphism such that 
\begin{equ}[e:interpPhi]
\Phi_t \tau = \Phi_0 \tau + t \Phi_1 \tau\;,\qquad \forall \tau \in W\;,
\end{equ}
and similarly for $\bar \Phi_t$. Since the map $t \mapsto \Phi_t$ is continuous,
we have $\Phi_t, \bar \Phi_t \in \CU$ for all $t$ sufficiently small. 

Write now $\tau = \sum_{k \ge 0} \tau_k$ as a finite sum, where 
each $\tau_k$ is homogeneous of order $k$ in the sense that it
is a sum of elements in $\Tr_g(W)$ for $g$ a $T$-graph with $k$ vertices. Then there exists
$k$ such that $\tau_k \in \Tr(W) \setminus \Tr(\hat W)$, so that 
$\Phi_1 \tau_k \neq \bar\Phi_1 \tau_k$. Furthermore, the map $t \mapsto \Phi_t \tau$ is a
polynomial with $\d_t^k \Phi_t \tau |_{t=0} = k! \Phi_1 \tau_k$. It follows that 
the polynomials $t \mapsto \Phi_t \tau$ and $t \mapsto \bar \Phi_t \tau$ are distinct, so there
are arbitrarily small values of $t$ for which they differ. Setting $\Phi = \Phi_t$ and $\bar \Phi = \bar \Phi_t$ 
for such a value of $t$ concludes the proof.
\end{proof}

\begin{remark}\label{rem:DiffValue}
In much the same way, one shows in the setting of Proposition~\ref{prop:injectiveLocal} 
that for any $\tau \in \Tr(W) \setminus \{0\}$ 
and any fixed $a \in \CV[V]$, one can find $\Phi \in \CU$ such that $\Phi(\tau) \neq a$.
\end{remark}

\subsection{Application to \texorpdfstring{$T$}{T}-algebras with derivation}

Most of the $T$-algebras relevant for this article admit a derivation. We first 
show that it is straightforward to give a notion of derivation for collections 
of spaces $W = (W^k_\ell)_{k,\ell\ge 0}$ as
above that then endows $\Tr(W)$ with a derivation. 
Indeed, it suffices to assume that $W$ comes with a collection of linear 
maps $\d \colon W^\upper_\low \to W^\upper_{\low+1}$
satisfying the identities 
\begin{equ}[e:idendd]
(S_{1,1}\cdot \id^\upper_\low)\d^2 w = \d^2 w\;,\quad
\text{and}\quad \d(\alpha w) = (\id_1^0\cdot \alpha)\d w\;, 
\end{equ}
for every $w \in W^\upper_\low$ and every $\alpha \in \Sym(k,\ell)$.

\begin{lemma}\label{lem:Talgstruc}
In the setting just described, the space $\Tr(W)$ admits a unique 
derivation $\d$ compatible with the injection $W \hookrightarrow \Tr(W)$.
\end{lemma}

\begin{proof}
Given a $T$-graph $g = (V_g, \mfd, \phi_g) \in T^\upper_\low$ and a vertex $v \in V_g$, we define
$\d_v g = (V_g, \d_v \mft, \d_v\phi_g) \in (\upper,\low+1)$ by setting
$\d_v\mft(v) = (o_v, i_v+1)$. We then have a map 
$\iota \colon \Out(g) \sqcup \In(g) \to \Out(\d_vg) \sqcup \In(\d_vg)$
given by setting $\iota(j) = j+1$ for $j \in [\ell] \subset \Out(g)$, $\iota(v,j) = (v,j+1)$
for $(v,j) \in \In(g)$ (here $v$ is the specific vertex that we have fixed), and the identity
otherwise. Note that the image of $\iota$ leaves out exactly two elements, namely
$1 \in [\ell+1]$ and $(v,1)$. This allows us to define $\d_v\phi_g$ by setting
$\d_v\phi_g(1) = (v,1)$ and $\d_v \phi_g(x) = \iota (\phi_v(\iota^{-1}(x)))$ otherwise.

Since the vertex sets of $g$ and $\d_v g$ are naturally identified, this induces a map
$\d_v \colon \widehat{\Tr}_g(W) \to \widehat{\Tr}_{\d_vg}(W)$ by setting
\begin{equ}
\d_v \Bigl(g \otimes \bigotimes_{\bar v \in V_g} w_{\bar v}\Bigr)
\eqdef \Bigl(\d_v g \otimes \bigotimes_{\bar v \in V_g} \hat w_{\bar v}\Bigr)\;,\quad
\hat w_{\bar v} = 
\left\{\begin{array}{cl}
	\d w_{\bar v} & \text{if $\bar v = v$,} \\
	w_{\bar v} & \text{otherwise.}
\end{array}\right.
\end{equ}
Finally, given $\tau \in \Tr_g(W)$, we set
$\d \tau \eqdef \sum_{v \in V_g} \d_v \tau$. The verification of the axioms
given in Definition~\ref{def:product} is a straightforward exercise.

Uniqueness of $\d$ is immediate since, by Lemma~\ref{lem:representationg}, every $T$-graph can be built 
from elementary graphs 
of the type given in Remark~\ref{rem:identification} by using the operations of a $T$-algebra. 
\end{proof}

One has the following.

\begin{lemma}\label{lem:derT}
In the setting of Proposition~\ref{prop:freeT}, assume furthermore that $W$ is endowed with a derivation $\d$
in the sense of \eqref{e:idendd}, that $\CV$ is a $T_\d$-algebra, and that 
$\Phi \colon \bigoplus_{k,\ell} W^k_\ell \to \CV$ satisfies $\d \Phi(\tau) = \Phi(\d \tau)$.
Then, the extension of $\Phi$ to $\Tr(W)$ actually belongs to $\Hom_\d(\Tr(W),\CV)$ for the 
derivation on $\Tr(W)$ given by Lemma~\ref{lem:Talgstruc}.
\end{lemma}

\begin{proof}
Using Lemma~\ref{lem:representationg} and the axioms satisfied by $T$-algebras and derivations, it is straightforward to verify that $\Phi(\d\tau) = \d\Phi(\tau)$ for $\tau \in \Tr_g(W)$ with $g$ an
arbitrary $T$-graph. 
\end{proof}

As a consequence we obtain the following result which is crucial to our methodology, especially 
when combined with Theorem~\ref{theo:injectiveT}.

\begin{theorem}\label{theo:mainInjective}
Let $W$ be such that each $W^k_\ell$ is finite-dimensional and let $\SN$ be any finite
collection of connected anchored $T$-graphs.
Assume furthermore that $W$ is endowed with a derivation $\d$
satisfying \eqref{e:idendd}. 
Then, for every $\Phi \in\Hom(\Tr W, \CV[V])$,
there exists $\Psi \in\Hom_\d(\Tr(W), \CW[V])$ such that
$(\Psi \tau)(0) = \Phi \tau$ for every $\tau \in  \Tr_\SN(W)$.

If there exists $n \ge 0$ such that $\Phi \circ \d^{n+1}$ vanishes identically on every $W^k_\ell$, 
then there exists $\Psi \in\Hom_\d(\Tr(W), \CW[V])$ such that
$(\Psi \tau)(0) = \Phi \tau$ for every $\tau \in  \Tr(W)$.
\end{theorem}

\begin{proof}
For all $k,\ell \ge 0$, we choose subspaces $X^k_\ell \subset W^k_\ell$ in such a way that
$X^k_0 = W^k_0$ and, for $\ell > 0$,
\begin{equ}[e:decompW]
W^k_\ell  = X^k_\ell \oplus \d W^k_{\ell-1}\;.
\end{equ}
Let furthermore $N$ denote the largest degree appearing in $\SN$. 
For $\tau \in X^k_\ell$ with $k+\ell \le N$, we then choose $\Psi\tau \in \CC^\infty(V, (V^*)^{\ell}\otimes V^{k})$ as 
the unique polynomial of degree $N-k-\ell$ such that 
\begin{equ}[e:defPsi]
\big(D^m(\Psi \tau)\big)(0) = \Phi(\d^m \tau)\;,\qquad \forall k+\ell +m \le N \;.
\end{equ}
This is possible since $\Phi(\d^m \tau)$ is guaranteed to be symmetric in the first $m$ factors
of $V^*$. We then extend this to all of $W$ by imposing that $\Psi \d \tau = D \Psi \tau$, which is
possible by \eqref{e:decompW}, combined with the fact that $\d \tau = 0$ implies that $D \Psi \tau$ 
vanishes identically by \eqref{e:defPsi}. For $\tau \in X^k_\ell$ with $k+\ell > N$, 
we simply set $\Psi \tau = 0$.

We extend $\Psi$ uniquely to a morphism of $T$-algebras by Proposition~\ref{prop:freeT}
which belongs to $\Hom_\d$ by Lemma~\ref{lem:derT}. The fact that it satisfies the
claimed identity follows from \eqref{e:defPsi}.
The last claim is immediate by performing the same construction with $\Psi \tau$ a polynomial of
degree $n$.
\end{proof}

\subsection{Free \texorpdfstring{$T_\d$}{T d}-algebras and \texorpdfstring{$\CX$}{X}-graphs}

We construct the ``free'' $T_\d$-algebras in the following way. Fix a finite set $\CX$ \label{set of types page ref} of generators
and, for every $\mfs \in \CX$, a degree $(o_\mfs,i_\mfs) \in \N^2$.
We then define a set 
\begin{equ}
\hat \CX = \{(\beta, \d^k \mfs)\,:\, \mfs \in \CX\;,\quad \beta \in \Sym(o_\mfs,k+i_\mfs)\} / {\sim}\;,
\end{equ}
where we postulate that 
\begin{equ}[e:identbeta]
(\beta, \d^k \mfs) \sim (\beta \circ (\gamma \cdot \id_{i_\mfs}^{o_\mfs}), \d^k \mfs)\;,
\end{equ}
for any $\mfs \in \CX$, $k \ge 2$, $\beta \in \Sym(o_\mfs,k+i_\mfs)$ and $\gamma \in \Sym(k)$.
Given this definition, we set
\begin{equ}[e:defWCX]
W\CX^\upper_\low \eqdef \Vec(\hat \CX^\upper_\low)\;,
\end{equ}
where $\hat \CX^\upper_\low$ consists of those elements in $\hat \CX$ of degree $(\upper,\low)$.
(Here the degree of $(\beta, \d^k \mfs)$ is $(o_\mfs,k+i_\mfs)$.)
Elements $\alpha \in \Sym(\upper,\low)$ naturally act on $(\beta, \d^k \mfs) \in W\CX^\upper_\low$ by $(\alpha \circ \beta, \d^k \mfs)$, which passes to the quotient and is therefore well-defined.
We furthermore have a derivation $\d\colon W\CX^\upper_\low \to W\CX^\upper_{\low+1}$ obtained by postulating that
\begin{equ}
\d(\beta, \d^k \mfs) \eqdef (\id^1_0 \cdot \beta, \d^{k+1} \mfs)\;,
\end{equ}
which can be seen to satisfy \eqref{e:idendd}. It follows from Lemma~\ref{lem:Talgstruc} 
that $T_\d(\CX) \eqdef \Tr(W\CX)$\label{page:TdCX} is a $T_\d$-algebra which we call the ``free $T_\d$-algebra with generators $\CX$''.
Since one has a natural inclusion $\CX \subset \hat \CX \subset T_\d(\CX)$ given by $\mfs \mapsto (\id, \mfs)$, 
this terminology is justified by the following.

\begin{proposition}
For every $T_\d$-algebra $\CW$, every map $\Phi\colon \CX \to \CW$ 
extends unique\-ly to an element of $\Hom_\d(T_\d(\CX),\CW)$. 
\end{proposition}

\begin{proof}
Clearly, the only extension of $\Phi$ to  $\hat \CX$ that respects the derivation and the 
action of the symmetric group is given by
$\Phi(\beta, \d^k \mfs) = \beta \d^k \Phi(\mfs)$.
The claim now follows from Lemma~\ref{lem:derT} and Theorem~\ref{theo:injectiveT}.
\end{proof}


We now introduce a notion which allows to describe a useful basis for the vector space $T_\d(\CX)$.

\begin{definition}
Given a finite set $\CX$ endowed with a map $\CX\ni\mfs\mapsto(o_\mfs,i_\mfs)\in\N^2$, an $\CX$-graph $\mfg=(V_\mfg,\mft,\phi_\mfg)_\low^\upper$ of degree $(\upper,\low)\in\N^2$
consists of a finite vertex set $V_\mfg$, a type map $\mft \colon V_\mfg \to \CX$, and a map $\phi_\mfg \colon \Out(\mfg) \to \In(\mfg)$,
where
\begin{equ}
\Out(\mfg) = [\low] \sqcup \overline{\Out}(\mfg) \;,\qquad \In(\mfg) = [\upper] \sqcup \overline{\In}(\mfg)\sqcup\bigsqcup_{v\in V_\mfg} \{(v,\star)\}\;,
\end{equ}
with $\overline{\Out}(\mfg)$ and $\overline{\In}(\mfg)$ defined as in \eqref{e:defTg2}, such
that every element of $[\upper] \sqcup \overline{\In}(\mfg)$ has exactly one preimage
under $\phi_\mfg$ and $\phi_\mfg^{-1}([\upper])\subset\overline{\Out}(\mfg)$.
(Note that targets of type $\star$ can have any number of preimages, including none.)

As in the case of $T$-graphs, two $\CX$-graphs $\mfg_i = (V_i,\mft_i,\phi_i)_\low^\upper$
are isomorphic if there exists a bijection
$\iota \colon V_1 \to V_2$ intertwining the $\mft_i$ and such that 
the induced bijection $\hat \iota \colon \Out(\mfg_1) \cup \In(\mfg_1) \to \Out(\mfg_2) \cup \In(\mfg_2)$
intertwines the $\phi_i$.
\end{definition}

\begin{wrapfigure}[5]{r}{2.5cm}
\vspace{-.7em}
\tikzset{external/export next=false}
\begin{tikzpicture}[scale=0.2,baseline=-2,draw=symbols,line join=round,decoration={
    markings,
    mark=at position 0.5 with {\arrow{>}}}]
\draw [black,postaction={decorate}] (-2,4.35) to[out=-90,in=90] (0.3,0.55);
\draw [black,postaction={decorate}] (2,4.35) to[out=-90,in=90] (-0.3,0.55);
\draw [black,postaction={decorate}] (0.3,-0.55) to (0.3,-2);
\draw [black,postaction={decorate}] (-2.7,7) to (-2.7,5.6);
\draw [black,postaction={decorate}] (1.3,7) to (1.3,5.6);
\draw [black,postaction={decorate}] (-0.3,-0.55) .. controls (-0.3,-4.7) and (-6.7,8.7) .. (-2.7,5.6);
\draw (0,0) node[rec] {}; 
\draw (-2,4.9) node[cerc] {};
\draw (-2.7,5.6) node[stars] {};
\draw (1.3,5.6) node[stars] {};
\draw (-0.9,0.7) node[stars] {};
\draw (2,4.9) node[cerc] {};
\end{tikzpicture}
\end{wrapfigure}
On the right we show a graphical representation of an $\CX$-graph of degree $(1,2)$, with $\CX$ composed of two elements: $\<CX_cerc>$ of degree $(1,0)$ and $\<CX_rec>$ of degree $(2,2)$ . 

Recall that, by definition, the space $T_\d(\CX)$ is freely generated (as a vector space) by the set of all
$g \otimes \bigotimes_{v \in V} w_{v}$ with 
\begin{enumerate}
\item $g= (V,\mfd, \phi)\in T^u_\ell$ for some $u,\ell\geq 0$, 
\item $w_v=(\beta_v,\partial^{k_v}\mfs_v)\in \hat\CX^{o_v}_{i_v}$ for all $v\in V$, with $\mfd(v) = (o_v, i_v)=
(o_{\mfs_v},k_v+i_{\mfs_v})$ and $\beta_v=(\beta_v^\out,
\beta_v^\inc)\in\Sym(o_{\mfs_v})\times\Sym(k_v+i_{\mfs_v})$, 
\end{enumerate} 
modulo the identifications \eqref{e:ident}, \eqref{e:equivrel} and \eqref{e:identbeta}. 
We claim now that every such $g \otimes \bigotimes_{v \in V} w_{v}$ can
be interpreted as an $\CX$-graph $(V_\mfg,\mft,\varphi_\mfg)_\low^\upper$  as follows.
\begin{enumerate}
\item The vertex set is simply $V_\mfg:=V$.
\item The type $\mft \colon V_\mfg \to \CX$ associated to each $v\in V_\mfg =V$ is $\mft(v)=\mfs_v$.
\item To define $\phi_\mfg$, note first that there is a bijection
$\iota_{\out}\colon \Out(g) \to \Out(\mfg)$ which is the identity on $[\low]$ and maps
elements of type $(v,j)$ to $(v, (\beta_v^\out)^{-1}(j))$.
Similarly, there is a \textit{surjection}
$\iota_\inc\colon \In(g) \to \In(\mfg)$ which is the identity on $[\upper]$ and maps elements of the type $(v,j)$
to $(v,\star)$ if $(\beta_v^\inc)^{-1}(j) \le k_v$ and to $(v,(\beta_v^\inc)^{-1}(j)-k_v)$ otherwise. 
Note that the previous item guarantees that this
definition is consistent.

We then define
$\varphi_\mfg\colon \Out(\mfg) \to \In(\mfg)$ by
\begin{equ}
\varphi_\mfg =  \iota_\inc \circ \varphi_g \circ \iota_\out^{-1}\;.
\end{equ}
\end{enumerate}
The following proposition will be used repeatedly when describing free $T_\d$-algebras.

\begin{proposition}
The map $w = g \otimes \bigotimes_{v \in V} w_{v}\mapsto\mfg = \mfg(w)$ is well-defined 
in the sense that if $w \sim \bar w$ according to \eqref{e:ident}, \eqref{e:equivrel} or \eqref{e:identbeta},
then $\mfg(w)$ and $\mfg(\bar w)$ are isomorphic.
Furthermore, it yields a bijection between $T_\d(\CX)$
and the free vector space generated by all isomorphism classes of $\CX$-graphs.
\end{proposition}

\begin{proof}
We first show that it is well-defined.
Since every identification $\iota \colon V \to V$ of $T$-graphs as in \eqref{e:ident}
induces an isomorphism of $\CX$-graphs for $\mfg$ (via the same map on $V$), 
we only consider \eqref{e:equivrel} and \eqref{e:identbeta}. 
Let us fix 
$g \otimes \bigotimes_{v \in V_g} w_{v}$, $\alpha \in \Sym_g$ and a choice of 
$\gamma_{v}\in \Sym(k_{v})$ for every $v\in V_g$; if $w_v=(\beta_v, \d^{k_v} \mfs_v)$, we set
$\bar w_v=(\beta_v \circ (\gamma_v \cdot \id_{i_{\mfs_v}}^{o_{\mfs_v}}), \d^{k_v} \mfs_v)$.
Let us check that the $\CX$-graphs $\mfg$ and $\bar \mfg$ associated, respectively, with $g \otimes \bigotimes_{v \in V_g} w_{v}$ and with
\[
(\alpha \actint g) \otimes \bigotimes_{v \in V_{\alpha \actint g}} (\alpha_v \bar w_v),
\]
are isomorphic. For this, we recall that $\alpha \actint g = (V,\mfd,\alpha \circ \phi_g \circ \alpha^{-1})$,
so that the vertex set is the same for $g$ and $\alpha \actint g$; in particular we have $\mfg=(V,\mft,\varphi_\mfg)$
and $\bar\mfg=(V,\mft,\varphi_{\bar\mfg})$. Considering the identifications
$\id:V\to V$, $\id \colon \Out(\mfg) \cup \In(\mfg) \to \Out(\bar\mfg) \cup \In(\bar\mfg)$, it remains to show
that $\varphi_\mfg=\varphi_{\bar\mfg}$. We have
$\alpha_v \bar w_v=(\bar\beta_v,\d^{k_v} \mfs_v)$, where $\bar\beta_v=(\bar\beta_v^\out,\bar\beta_v^\inc)$ is given by
\[
\bar\beta_v^\out=\alpha_v^\out\circ \beta_v^\out, 
\qquad \bar\beta_v^\inc=\alpha_v^\inc\circ\beta_v^\inc\circ
(\gamma_v\cdot\id_{i_{\mfs_v}}^{o_{\mfs_v}}).
\]
Then it is easy to see that the map $(\bar\iota_\out)^{-1}\colon \Out(g) \to \Out(\bar\mfg)$ defined in point 3 above is equal to $\alpha\circ(\iota_\out)^{-1}$, where $\alpha$ is as in \eqref{e:alph}. Analogously, 
$\bar\iota_\inc=\iota_\inc\circ\alpha^{-1}$: indeed, $\bar\iota_\inc(v,j)=(v,\star)$ if and only if 
\[
(\bar\beta^\inc_v)^{-1}(j)=(\gamma_v^{-1}\cdot\id_{i_{\mfs_v}}^{o_{\mfs_v}})\circ(\beta_v^\inc)^{-1}\circ(\alpha_v^\inc)^{-1}
(j)\in\{1,\ldots,k_v\},  
\]
and since $\gamma_v$ acts on $\{1,\ldots,k_v\}$ this is equivalent to $(\beta_v^\inc)^{-1}((\alpha_v^\inc)^{-1}(j))
\le k_v$; finally, for the same reason $(\bar\beta^\inc_v)^{-1}(j)>k_v$ if and only if $(\beta_v^\inc)^{-1}((\alpha_v^\inc)^{-1}(j))> k_v$ and in this case 
\[
\bar\iota_\inc(v,j)=(v,(\bar\beta^\inc_v)^{-1}(j)-k_v)=\iota_\inc(\alpha^{-1}(v,j)).
\]
Therefore, since $\varphi_{\alpha \actint g }=\alpha \circ \phi_g \circ \alpha^{-1}$, we finally have
\[
\varphi_{\bar\mfg}=\bar\iota_\inc \circ \varphi_{\alpha \actint g } \circ (\bar\iota_\out)^{-1}=\iota_\inc \circ \varphi_g \circ \iota_\out^{-1}=
\varphi_{\mfg}.
\]
It remains to show that $w \mapsto \mfg$ is invertible. Given an $\CX$-graph $\mfg = (V_\mfg,\mft,\phi_\mfg)$,
we define $g = (V,\mfd,\phi_g)$ by setting $V = V_\mfg$ and $\mfd(v) = (o_v, i_v)$ with
\begin{equ}
o_v = o_{\mft(v)}\;,\qquad i_v = i_{\mft(v)} + k_v\;,\qquad k_v \eqdef |\phi_\mfg^{-1}(v,\star)|\;.
\end{equ}
We then set $w_v = (\id, \d^{k_v} \mft(v))$ for all $v \in V$ and we define $\iota_\inc$ and $\iota_\out$ exactly as above, but with $\beta_v^\inc, \beta_v^\out = \id$, which allows us to choose for $\phi_g$ any
bijection satisfying $\iota_\inc \circ \phi_g = \phi_{\mfg}\circ \iota_\out$. 
Since different choices of $\phi_g$ necessarily differ by post-composition with a permutation of 
the preimages of elements of type $(v,\star)$ under $\iota_\inc$, these give rise to the same
$\CX$-graph by \eqref{e:identbeta}, so that we have specified $g$ uniquely. Verifying that this construction
is the inverse of the previous one is routine.
\end{proof}

\subsection{Hilbert space structure}
\label{sec:Hilbert}

The free $T_\d$-algebra $T_\d(\CX)$ comes with a natural scalar product 
such that $\scal{g,g} = |\CG_g|$, where $\CG_g$ denotes the group of automorphisms of $g$, 
and such that $\scal{g,h} = 0$ if the $\CX$-graphs $g$ and $h$ aren't isomorphic. 
It is then natural to ask what are the adjoints of the 
four defining operations of a $T_\d$-algebra. It is immediate that, for $\alpha \in \Sym(\upper,\low)$,
one has $\alpha^* = \alpha^{-1}$. Regarding the product, we define its adjoint $\Delta \colon 
T_\d(\CX) \to T_\d(\CX) \otimes T_\d(\CX)$ to be the linear map such that 
\begin{equ}[e:charDelta]
\scal{f\mult g,h} = \scal{f \otimes g, \Delta h}\;.
\end{equ}
To characterise $\Delta$, we say that an $\CX$-graph $h$ is irreducible if whenever 
$h = f\mult g$ one must have either $f=\one$ or $g=\one$. 

Irreducible graphs of degree $(0,0)$ play a special role since they commute with every
element of $T_\d(\CX)$. For every $\CX$-graph $h$, we can then find a unique (up to permutation) collection of 
distinct irreducible graphs $(g_i)_{i=1}^n$ of degree $(0,0)$ as well as a collection of exponents $k_i$ and 
irreducible graphs $(f_j)_{j=1}^m$ such that 
\begin{equ}[e:decomph]
h = f_1\mult \ldots \mult f_m \mult g^k\;,\qquad g^k = \prod_{i=1}^n g_i^{k_i}\;.
\end{equ}
With this notation at hand, we claim that one has 
\begin{equ}[e:defDelta]
\Delta h = \sum_{p \le m} \sum_{\ell \le k}\binom{k}{\ell} (f_1\mult \ldots \mult f_p)\mult g^\ell
\otimes (f_{p+1}\mult \ldots \mult f_m)\mult g^{k-\ell}\;,
\end{equ}
again with the conventions that empty products are interpreted as $\one$, that $\ell \le k$ means $\ell_i \le k_i$ for $0\leq i\leq n$ and $\binom{k}{\ell}=\prod_{i=1}^{n}\binom{k_i}{\ell_i}$.

From the uniqueness of the decomposition \eqref{e:decomph}, it follows that $\Delta h$ as characterised by \eqref{e:charDelta} must be
of the type \eqref{e:defDelta}, with the binomial coefficient replaced by some coefficient depending on
$p$, $\ell$ and $h$.

At this stage, we note that one has
\begin{equ}
\scal{h,h} = |\CG_h| = k! \Big(\prod_{i=1}^m |f_i|^2\Big)\Big(\prod_{j=1}^n |g_j|^2\Big)\;,
\end{equ}
since any isomorphism for $h$ is obtained by composing isomorphisms for the irreducible factors with
a permutation of the identical factors $g_i$. (The $f_i$ cannot be permuted even if some of them are identical 
as a consequence of the fact that they are anchored.)
The claim then follows at once by verifying that the binomial coefficient is the only choice that
guarantees that \eqref{e:decomph} holds if we take for $f$ and $g$ the left (resp.\ right) factor in
one of the summands of \eqref{e:defDelta}.

Regarding the trace operation, we claim that for $g = (V_g,\mft,\phi)^\upper_\low$ one has the identity
\begin{equ}[e:defTrace]
\tr^* g = \sum_{e \in \Out(g)} \Cut_e(g)\;,
\end{equ}
where $\Cut_e(g) = (V_g,\mft,\hat \phi)^{\upper+1}_{\low+1}$ is obtained by 
setting $\hat \phi(e) = \upper + 1$, $\hat \phi(\low + 1) = \phi(e)$, and $\hat \phi = \phi$ otherwise.
Similarly to before, is obvious from the definition of $\tr$
(in particular in view of the representation given by Lemma~\ref{lem:representationg})
 that one has
\begin{equ}
 \tr^* g  = \sum_{e \in \Out(g)} \alpha_e \Cut_e(g)\;,
\end{equ}
for some constants $\alpha_e$. Since furthermore $\tr \Cut_e(g) = g$ for every edge $e$, one has
\begin{equs}
|g|^2 &= \scal{\Cut_e(g), \tr^* g} = \sum_{\bar e \in \Out(g)} \alpha_{\bar e} \scal{\Cut_e(g),\Cut_{\bar e}(g)} \\
&= |{\Cut_e g}|^2 \sum_{\bar e\,:\, \Cut_{\hat e}(g) = \Cut_e(g)} \alpha_{\bar e} \;.
\end{equs}
As a consequence, we only
need to verify that, for every $e \in \Out(g)$, one has $|g|^2 = N_e |{\Cut_e g}|^2$, where
$N_e$ denotes the number of edges $\hat e$ such that $\Cut_{\hat e}(g) = \Cut_e(g)$.

Viewing elements of $\CG_g$ as bijections on $\Out(g)$, it is straightforward to verify that $\CG_{\Cut_e(g)}$ is naturally identified with 
the subgroup of $\CG_g$ that fixes $e$, while
$N_e$ is nothing but the size of the orbit of $e$ under the action of $\CG_g$.
The claim \eqref{e:defTrace} now follows from the orbit-stabiliser theorem.

Finally, for $g$ as above, we have $\d^* g =0$ unless $\low > 0$ 
and $\phi(1) = (v,\star)$ for some $v \in V_g$.
In that case, we claim that $\d^*g = \hat g = (V_g,\mft, \hat \phi)^\upper_{\low-1}$
where $\hat \phi(j) = \phi(j+1)$ for $j \in [\low]$ and $\hat \phi = \phi$ otherwise. 
Again, it is clear that $\d^* g$ must be some multiple of $\hat g$. Furthermore, by definition
of the adjoint, we would like to have 
\begin{equ}
\scal{\hat g, \d^* g} = \scal{\d \hat g, g} = \hat N_{g} \scal{g,g}\;,
\end{equ}
where $\hat N_{g}$ is the number of vertices $\hat v$ of $\hat g$ such that $\d_{\hat v} \hat g = g = \d_v \hat g$,
with $\d_{\hat v} \hat g$ defined as in the proof of Lemma \ref{lem:Talgstruc} and $v$ the specific vertex of $V_g$
singled out above.
If we now interpret $\CG_{\hat g}$ as a group of bijections of $V_g$, 
we note again that $\hat N_{g}$ is nothing but the orbit of $v$ under $\CG_{\hat g}$,
while $\CG_g$ is the stabiliser of $v$, so that the claim follows again from the orbit-stabiliser 
theorem.

\section{Application to SPDEs}
\label{sec:spaces}

We now return to the reason for introducing all the algebraic machinery of Section \ref{sec:algebra}, namely the proof
of Proposition~\ref{prop:ItoStratgeo}.
We first use this to provide a characterisation of the space $\CS_\geo$.

\subsection{Characterisation of geometric counterterms}
\label{sec:geo2}

We consider from now on smooth one-parameter families $(\psi_t)_{t\geq 0}$ of diffeomorphisms of $\R^d$ such that
\begin{equ}[psih]
\psi_0=\id:\R^d\to\R^d, \qquad
\d_t \psi_t|_{t=0} = h:\R^d\to\R^d.
\end{equ}

\begin{definition}
We write $T_{\<generic>}$ \label{S generic page ref} for the free $T_\d$-algebra generated by 
$\{\<generic>_i,\<not>\,:\, i =1,\ldots,m\}$ with $\<generic>_i$ of degree $(1,0)$
and $\<not>$ of degree $(1,2)$, quotiented by the 
ideal of $T_\d$-algebras generated by $(\<not> - \swap_{1,1}^{1}\<not>)$. 
We define $ T_{\<generic>,\<diff>} $ \label{S diff page ref} in the 
same way as $T_{\<generic>}$ by adding an additional element
$ \<diff> $ of degree $(1,0)$ to the set of generators.
\end{definition}

We note that $T_{\<generic>}$ can also be constructed as $\Tr(W_{\<generic>})$ with the following choice of $(W_{\<generic>})^\upper_\low$, $\upper,\low\geq 0$: we set $\CX:=\{\<generic>_i,\<not>\,:\, i =1,\ldots,m\}$, with the above degrees, and we define $(W_{\<generic>})^\upper_\low$ as $W\CX^\upper_\low/A^\upper_\low$, 
where $W\CX^\upper_\low$ is as in \eqref{e:defWCX} and $A^\upper_\low$ is the vector space generated by 
$\{(\id^1_{k+2},\partial^k\<not>)-(\id^1_k\cdot\swap_{1,1},\partial^k\<not>), k\geq 0\}$. We need to 
consider these additional quotients in order to take into account the symmetry $\Gamma^\alpha_{\beta\gamma}=\Gamma^\alpha_{\gamma\beta}$.

We will henceforth view the space $\CS_{\<generic>}$ defined on page~\pageref{CS page ref} 
(and therefore also the reduced space $\CS \subset \CS_{\<generic>}$, where we use $\iota:\CS \to \CS_{\<generic>}$ in \eqref{def:iota}
as an inclusion) as a
subspace of $T_{\<generic>}$ in the obvious way, namely each tree $\tau \in \SS_{\<generic>}$ 
is viewed as an $\CX$-graph $g = (V_g,\mft,\phi)$ where the vertex set $V_g$ is given by the inner vertices of 
$\tau$, the type of a vertex is given by $\<generic>_i$ if it is incident to such a noise, while it is given by
$\<not>$ otherwise. The map $\phi$ is determined by the edges of $\tau$, with thin edges representing 
connections to the corresponding instance of $(v,\star)$, while thick edges represent connections to the
`native' incoming slots of $\<not>$.

In this way of representing elements of $\CS_{\<generic>}$, we have for example
\begin{equ}
\<Xi4eabisc1bis> \approx \tr^3\bigl( \d \<not> \mult \<generic>_k \mult \<generic>_j \mult\tr\bigl(\d \<generic>_i \mult \<generic>_\ell\bigr)\bigr)\;.
\end{equ}
(With this way of writing we avoid having to use the action of the symmetric group.)
For any fixed choice of $\Gamma$ and $\sigma$, we then write 
$\Upsilon_{\Gamma,\sigma}:T_{\<generic>}\to\CW[\R^d]$ for the (unique) morphism of $T_\d$-algebras mapping 
$\<generic>_i$ to $\sigma_i$
and $\<not>$ to $2\Gamma$, where the target space $\CW[\R^d]$ is the $T_\d$-algebra constructed in
Remark~\ref{rem:tensorT}. It is straightforward to convince oneself that this is consistent with our
previous notation in the sense that $\Upsilon_{\Gamma,\sigma}$ is an extension of the valuation
given in Section~\ref{sec:BPHZthm} to all of $T_{\<generic>}$. \label{Evaluationmap2 page ref}
The reason why $\<not>$ is mapped to $2\Gamma$ and not just to $\Gamma$ is that in $T_{\<generic>}$
every instance of $\<not>$ always has two incoming thick edges, so that one performs two derivatives in 
the $q$ variables for the corresponding vertex in the formula \eqref{recursive_Upsilon}.

Given furthermore a smooth function $ h :\R^d\to\R^d$, we extend $ \Upsilon_{\Gamma,\sigma} $ to a morphism of 
$T_\d$-algebras $  \Upsilon_{\Gamma,\sigma}^{h} : T_{\<generic>,\<diff>}  \to\CW[\R^d]$ by 
sending $ \<diff> $ to $ h $.
We also define $\phi_\geo:T_{\<generic>}\to T_{\<generic>,\<diff>}$ \label{phi geo page ref}
as the unique infinitesimal morphism of $T_\d$-algebras with respect to the canonical injection $\iota:T_{\<generic>}\to T_{\<generic>,\<diff>}$ (see \eqref{infinit-morph}) such that
\begin{equs}
\phi_\geo(\<generic>_i) &= [\<generic>_i, \<diff>] = \<generic>_i \graftI \<diff> - \<diff> \graftI \<generic>_i \;,\label{e:phigeo} \\
\phi_\geo(\<not>) &= [\<not>, \<diff>] - \tr (\swap^{1,1}_3 (\d \<diff> \mult \<not>))
- \swap^1_{1,1} \tr (\swap^{1,1}_3 (\d \<diff> \mult \<not>))  - 2 \d^2 \<diff>\;,
\end{equs}
see \eqref{e:graft} for the definition of $\graftI$.
Pictorially, $\phi_\geo$ acts by setting
\begin{equ}
\def\offset{.8,1.3}
\tikzset{external/export next=false}
\begin{tikzpicture}[scale=0.2,baseline=-2]
\draw[tinydots] (0,0)  -- (0,-0.8);
\node[xi] (root) at (0,0) {};
\end{tikzpicture}
\;\mapsto\;
\tikzset{external/export next=false}
\begin{tikzpicture}[scale=0.2,baseline=2]
\draw[symbols]  (-.5,2) -- (0,0) ;
\draw[tinydots] (0,0)  -- (0,-0.8);
\node[diff] (root) at (0,-0.1) {};
\node[xi] (diff) at (-0.5,2) {};
\end{tikzpicture}
\;-\;
\tikzset{external/export next=false}
\begin{tikzpicture}[scale=0.2,baseline=2]
\draw[symbols]  (-.5,2) -- (0,0) ;
\draw[tinydots] (0,0)  -- (0,-0.8);
\node[xi] (xi) at (0,0) {};
\node[diff] (root) at (-0.5,2) {};
\end{tikzpicture}\;,\qquad\qquad
\def\offset{.8,1.3}
\def\offsett{1.3,.8}
\tikzset{external/export next=false}
\begin{tikzpicture}[scale=0.2,baseline=-2]
\coordinate (root) at (0,0);
\coordinate (t1) at (-.8,1.5);
\coordinate (t2) at (.8,1.5);
\draw[tinydots] (root)  -- +(0,-0.8);
\draw[kernels2] (t1) -- (root);
\draw[kernels2] (t2) -- (root);
\node[not] (rootnode) at (root) {};
\end{tikzpicture}
\;\mapsto\;
\tikzset{external/export next=false}
\begin{tikzpicture}[scale=0.2,baseline=-.25cm]
\coordinate (root) at (0,0);
\coordinate (tri) at (0,-2);
\coordinate (t1) at (-.8,1.5);
\coordinate (t2) at (.8,1.5);
\draw[tinydots] (tri)  -- +(0,-0.8);
\draw[kernels2] (t1) -- (root);
\draw[kernels2] (t2) -- (root);
\draw[symbols] (root) -- (tri);
\node[not] (rootnode) at (root) {};
\node[diff] (trinode) at (tri) {};
\end{tikzpicture}
\;-\;
\tikzset{external/export next=false}
\begin{tikzpicture}[scale=0.2,baseline=2]
\coordinate (root) at (0,0);
\coordinate (tri) at (1.2,1.2);
\coordinate (t1) at (-1.2,1.2);
\coordinate (t2) at (0,1.7);
\draw[tinydots] (root)  -- +(0,-0.8);
\draw[kernels2] (t1) -- (root);
\draw[kernels2] (t2) -- (root);
\draw[symbols] (root) -- (tri);
\node[not] (rootnode) at (root) {};
\node[diff] (trinode) at (tri) {};
\end{tikzpicture}
\;-\;
\tikzset{external/export next=false}
\begin{tikzpicture}[scale=0.2,baseline=-2]
\coordinate (root) at (0,0);
\coordinate (t1) at (-.5,2.5);
\coordinate (tri) at (1,1);
\coordinate (t2) at (1,2.5);
\draw[tinydots] (root)  -- +(0,-0.8);
\draw[kernels2] (t1) -- (root);
\draw[symbols] (t2) -- (tri);
\draw[kernels2] (tri) -- (root);
\node[not] (rootnode) at (root) {};
\node[diff] (trinode) at (tri) {};
\end{tikzpicture}
\;-\;
\tikzset{external/export next=false}
\begin{tikzpicture}[scale=0.2,baseline=-2]
\coordinate (root) at (0,0);
\coordinate (t1) at (.5,2.5);
\coordinate (tri) at (-1,1);
\coordinate (t2) at (-1,2.5);
\draw[tinydots] (root)  -- +(0,-0.8);
\draw[kernels2] (t1) -- (root);
\draw[symbols] (t2) -- (tri);
\draw[kernels2] (tri) -- (root);
\node[not] (rootnode) at (root) {};
\node[diff] (trinode) at (tri) {};
\end{tikzpicture}
\;-\;2
\tikzset{external/export next=false}
\begin{tikzpicture}[scale=0.2,baseline=-2]
\coordinate (root) at (0,0);
\coordinate (t1) at (-.8,1.5);
\coordinate (t2) at (0.8,1.5);
\draw[tinydots] (root)  -- +(0,-0.8);
\draw[symbols] (t1) -- (root);
\draw[symbols] (t2) -- (root);
\node[diff] (rootnode) at (root) {};
\end{tikzpicture}\;,
\end{equ}
which furthermore has a very natural interpretation in terms of how
vector fields and Christoffel symbols transform under infinitesimal changes of coordinates, see \eqref{e:calculgens} below.
Here, we used the graphical convention that edges that end in one of the two
`original' slots of $\<not>$ are drawn as thick lines, while edges that end in one
of the additional slots created by the operator $\d$ are drawn as thin lines.
We also set \label{hat phi geo page ref}
\begin{equ}
\hat \phi_\geo (\tau) = \phi_\geo(\tau) - [\tau,\<diff>]\;,
\end{equ}
for any $\tau \in (T_{\<generic>})_k^1$, so that one has for example $\hat\phi_\geo(\<generic>_i)=0$,
\begin{equ}
 \phi_\geo\left( \<arbreact> \right) = \<arbreact1> - 2\<arbreact2> + \<arbreact3>- \left(\<diff> \graftI  \<arbreact>\right) \ \text{   and   }\  \hat \phi_\geo\left( \<arbreact> \right) = \<arbreact3>  - 2\<arbreact2>\;.
\end{equ}
With all of these notations at hand, we are ready to give the following characterisation of the space $\CS_\geo$.

\begin{proposition}\label{geoker}
One has $\CS_\geo = \CS \cap \ker \hat \phi_\geo$.
\end{proposition}

\begin{proof}
Recall first the action \eqref{e:actions} of the group of diffeomorphisms on connections and tensors
and Definition~\ref{def:geoIto} of $\CS_\geo$, where the action on $\Upsilon_{\Gamma,\sigma}\tau$ is that on vector fields as
given by \eqref{eq:actionvect}. It follows that if $(\psi_t)_{t\geq 0}$ is as in \eqref{psih} then, for any $\tau \in \CS$
\begin{equs}[e:diffRHS]
\d_t \bigl(\psi_t \act \Upsilon_{\Gamma,\sigma}\tau\bigr)\bigl|_{t=0}
&= \d_t \bigl(\bigl(\d_\alpha \psi_t\cdot \Upsilon^\alpha_{\Gamma,\sigma}\tau\bigr) \circ \psi_t^{-1}\bigr)\big|_{t=0} \\
&= \d_\alpha h\cdot \Upsilon^\alpha_{\Gamma,\sigma}\tau - \d_\alpha \Upsilon_{\Gamma,\sigma}\tau\cdot h^\alpha
= \Upsilon^h_{\Gamma,\sigma}[\tau,\<diff>]\;.
\end{equs}
On the other hand, it is easy to see from the definitions that the derivative of a one-parameter family 
$t \mapsto \phi_t$ of morphisms of $T_\d$-algebras is an infinitesimal morphism of $T_\d$-algebras at $\phi_t$ and
is therefore defined by its actions on any set of generators. For the family 
$t \mapsto \Upsilon_{\psi_t \act\Gamma,\psi_t \act\sigma}$, we conclude from \eqref{e:actions} 
(see also \cite[Eq.~2.16]{LieDer} for example) that
\begin{equs}{}
 & \d_t \Upsilon_{\psi_t \act\Gamma,\psi_t \act\sigma}(\<generic>_i) \big|_{t=0} = [\sigma_i, h]\;,\label{e:calculgens}\\
& \d_t \Upsilon_{\psi_t \act\Gamma,\psi_t \act\sigma}(\<not>)^\alpha_{\beta\gamma} \big|_{t=0}  \\ 
& \qquad = 2\big(\d_\eta h^\alpha \Gamma^\eta_{\beta\gamma}
- \Gamma^\alpha_{\eta\beta} \d_\gamma h^\eta- \Gamma^\alpha_{\eta\gamma} \d_\beta h^\eta - h^\eta \d_\eta \Gamma^\alpha_{\beta\gamma} - \d^2_{\beta\gamma} h^\alpha\big) \;,
\end{equs}
where we recall that the Lie bracket is given by $[\sigma_i, h]^\alpha = \sigma_i^\beta \partial_\beta h^\alpha - h^\beta \partial_\beta \sigma_i^\alpha$.
Hence, for every $\tau \in T_{\<generic>}$,
\begin{equ}[e:diffLHS]
\d_t \Upsilon_{\psi_t \act\Gamma,\psi_t \act\sigma}(\tau)\big|_{t=0} = \Upsilon_{\Gamma,\sigma}^h \phi_\geo(\tau)\;,
\end{equ}
since both $\d_t \Upsilon_{\psi_t \act\Gamma,\psi_t \act\sigma}$ and $\Upsilon_{\Gamma,\sigma}^h$ are infinitesimal morphisms with respect to $\Upsilon_{\Gamma,\sigma}$ and they agree on the generators by 
\eqref{e:calculgens}.

Comparing \eqref{e:diffLHS} and \eqref{e:diffRHS}, we see that if $\tau \in \CS_\geo$, then one necessarily has
\begin{equ}
\Upsilon_{\Gamma,\sigma}^h \phi_\geo(\tau) = \Upsilon_{\Gamma,\sigma}^h [\tau,\<diff>]\;,
\end{equ}
for every choice of $\Gamma$, $\sigma$ and $h$, and we conclude that $\tau \in \ker \hat \phi_\geo$ by
Theorems~\ref{theo:injectiveT} and~\ref{theo:mainInjective}.

Conversely, let $\tau \in \CS \cap \ker \hat \phi_\geo$ and fix any diffeomorphism $\psi$ homotopic to the identity.
We then write $\{\psi_t\}_{t \in [0,1]}$ for any smooth homotopy connecting the identity to $\psi$ and
set $\hat \psi_t = \psi \circ \psi_t^{-1}$. Let furthermore $h_t = \d_t \psi_t\circ \psi_t^{-1}$ and
$\hat h_t = \d_t \hat \psi_t\circ \hat \psi_t^{-1}$ and note that as a consequence of the identity
$\d_t (\hat \psi_t \circ \psi_t) = 0$, one obtains the relation
\begin{equ}[e:relhath]
\hat h_t = - \hat \psi_t \act h_t\;.
\end{equ}
Performing the same calculation as \eqref{e:diffLHS} and \eqref{e:diffRHS} for the expression 
$t \mapsto \Phi_t \eqdef \hat \psi_t \act \bigl(\Upsilon_{\psi_t \act \Gamma,\psi_t \act \sigma}\tau\bigr)$, 
we obtain
\begin{equ}
\d_t \Phi_t = \hat \psi_t \act \bigl(\Upsilon_{\psi_t \act\Gamma,\psi_t \act\sigma}^{h_t} \phi_\geo(\tau)\bigr)
-  \big[\hat h_t , \Phi_t\big]\;.
\end{equ}
Using \eqref{e:relhath}, we obtain
\begin{equ}
\big[\hat h_t , \Phi_t\big] = - \hat \psi_t \act \bigl[h_t,\Upsilon_{\psi_t \act\Gamma,\psi_t \act\sigma}\tau\bigr] = 
 -\hat \psi_t \act \bigl(\Upsilon_{\psi_t \act\Gamma,\psi_t \act\sigma}^{h_t} [\<diff>,\tau]\bigr)\;,
\end{equ}
so that $\d_t \Phi_t = \hat \psi_t \act \bigl(\Upsilon_{\psi_t \act\Gamma,\psi_t \act\sigma}^{h_t} \hat \phi_\geo(\tau)\bigr) = 0$,
since $\tau \in \ker \hat \phi_\geo$.
It follows that $\Phi_0 = \Phi_1$, which is precisely the desired identity.
\end{proof}

\subsection{Characterisation of It\^o counterterms}
\label{sec:Ito}

\begin{definition}
\label{bar CS g page ref}
Let $ T_{\<g>} $ (resp. $ T_{\<g>, \<diff>} $) be the free $T_\d$-algebra generated
by $\{\<g>,\<not>\}$ (resp. $\{\<g>,\<not>, \<diff>\}$), with $\<not>$ (and $ \<diff> $) as before and $\<g>$ of degree $(2,0)$, 
quotiented by the ideal of $T_\d$-algebras generated by $(\<not> - \swap_{1,1}^{1}\<not>)$. 
We then write $\phi_\Ito\colon T_{\<g>} \to T_{\<generic>}$ (resp. $ \phi_{\Ito}^{\<diff>} \colon T_{\<g>,\<diff>} \to T_{\<generic>,\<diff>}$) \label{phi ito page ref} for the morphism 
 mapping $\<not>$ to $\<not>$ (and $ \<diff> $ to $ \<diff> $) 
and $\<g>$ to $\sum_{i=1}^m (\<generic>_i \mult \<generic>_i)$. Throughout this section, $m$
will always be used to denote the number of generators $\<generic>_i$ of $T_{\<generic>}$ and $T_{\<generic>,\<diff>}$.

We also define $ \CS_{\<diff>} \subset T_{\<generic>,\<diff>} $ as the vector space spanned by trees having $ 2 $ or $ 4 $ noises $ \<generic> $ and exactly one $ \<diff> $. \label{cs diff page ref} The subspace $\CS_\Ito^{\<diff>} \subset  \CS_{\<diff>}$ \label{CS ito diff page ref}  consists of the elements $\tau\in \CS_{\<diff>}$ such that $\Upsilon^h_{\Gamma,\sigma}\tau = \Upsilon^h_{\Gamma,\bar \sigma}\tau$
holds for any choice of $d$, $\Gamma$, $\sigma$, $\bar \sigma$ and $h$ such that $\sigma \sigma^\top = \bar \sigma \bar \sigma^\top$.
\end{definition}
\begin{example}\label{ex:phi_ito}
For example, we have 
\[
\phi_\Ito (\<pre_im_I1Xitwos>) = \<I1Xitwos>, \qquad
\phi_\Ito \Big(\,\<pre_im_cI1Xi4abs>\,\Big)=2\, \<cI1Xi4abs>, \qquad 
\phi_\Ito \left(\<pre_im_1s>\right) =2\left( \<2I1Xi4bc1s> + \<Xi4cbc2s>\right),
\]
which is relevant to Proposition~\ref{CB Ito page ref} below.
\end{example}
\begin{proposition}\label{prop:rangeIto}
One has $\CS_\Ito = \CS\cap \range \phi_\Ito$ and $\CS_\Ito^{\<diff>} = \CS_{\<diff>}\cap \range \phi^{\<diff>}_\Ito$.
\end{proposition}

\begin{proof}
We focus on proving $\CS_\Ito = \CS\cap \range \phi_\Ito$ since the proof that
$\CS_\Ito^{\<diff>} = \CS_{\<diff>}\cap \range \phi^{\<diff>}_\Ito$ works in exactly the same way.
The inclusion $\CS\cap \range \phi_\Ito \subset \CS_\Ito$ follows from the fact that 
$\Upsilon_{\Gamma,\sigma} \circ \phi_{\Ito}\colon T_{\<g>} \to \CW[\R^d]$ 
is a  morphism of $T_\d$-algebras which maps $\<not>$ to $2\Gamma$  and $ \<g> $ to $g=\sigma \sigma^\top$,
so that it only depends on $\Gamma$ and $g$, and not on the specific choice of $\sigma$.

For the converse direction, we want to show the following: for any $\tau \in \CS \setminus \range \phi_\Ito$, we can find $d,m,\Gamma,\sigma,\bar\sigma$ such that $\sigma \sigma^\top = \bar \sigma \bar \sigma^\top$ but $\Upsilon_{\Gamma,\sigma}\tau\ne\Upsilon_{\Gamma,\bar\sigma}\tau$, implying that
$\tau\notin\CS_\Ito$. 
The reason why this is not an immediate consequence of the second part of Theorem~\ref{theo:injectiveT} is that
although one has $T_{\<g>} = \Tr(W_{\<g>})$ and $T_{\<generic>} = \Tr(W_{\<generic>})$ 
for suitable choices of $W_{\<g>}$ and $W_{\<generic>}$,
the inclusion $\phi_\Ito\colon T_{\<g>} \hookrightarrow T_{\<generic>}$ is not 
generated by an inclusion $W_{\<g>} \hookrightarrow W_{\<generic>}$ since the generator
$\<g>$ is mapped to an element of order~$2$.
The idea then is to find an intermediate $T_\d$-algebra $T_{\red} = \Tr(W_{\red})$ 
together with inclusions $\psi_\Ito \colon T_{\<g>} \to T_{\red}$ and 
$\iota_{\red} \colon T_{\red} \to T_{\<generic>}$ satisfying the following properties:
\begin{enumerate}
\item One has the factorisation $\phi_\Ito = \iota_{\red} \circ \psi_\Ito$ and 
$\psi_\Ito$ maps $W_{\<g>}$ into $W_{\red}$ (so that it is ``linear'' in the sense 
of Remark~\ref{rem:linear}).
\item The image of $T_{\red}$ in $T_{\<generic>}$ under $\iota_{\red}$ contains $\CS$.
\item For every finite-dimensional vector space $V$ there exist $m$\footnote{Recall that $m$ denotes the number of generators of type $\<generic>_i$ of $T_{\<generic>}$.}  and a set $\CU_+ \subset \Hom(T_{\red}, \CV[V])$ such that, for every $\Phi \in \CU_+$ there exists
 $\Upsilon \in \Hom_\d( T_{\<generic>}, \CW[V])$ such that $\Upsilon \tau$ is a polynomial for every $\tau$ and such that  
$\Phi = \Ev_0 \circ \Upsilon \circ \iota_{\red} $, where $\Ev_0 \in \Hom(\CW[V],\CV[V])$ denotes
the evaluation at the origin.
\item There exist a $T$-algebra $T_{\red}^{(3)} = \Tr(W_{\red}^{(3)})$, an open set 
$\CU_+^{(3)} \in \Hom(T_{\red}^{(3)},\CV[V])$, and a morphism
$\pi_{\red}^{(3)} \in \Hom(T_{\red}, T_{\red}^{(3)})$ such that, for all $\Phi^{(3)} \in \CU_+^{(3)}$, one
has $\Phi \eqdef \Phi^{(3)} \circ \pi_{\red}^{(3)} \in \CU_+$.
Furthermore, $\pi_{\red}^{(3)}$ is injective on $\iota_{\red}^{-1}(\CS)$ and ``linear''
in the sense of Remark~\ref{rem:linear}.
\end{enumerate}
Figure~\ref{fig:diagram} illustrates this situation (the space $\hat T_{\red}$ also shown 
here will appear as an intermediate step in our construction).
\begin{figure}
\begin{center}
\begin{tikzcd}[row sep=5ex, column sep=5ex]
T_{\<g>}  \arrow[rr, "\phi_\Ito"] \arrow[dd, "\psi_{\Ito}"]
&&   T_{\<generic>} \arrow[rdd, bend left = 10, "\Upsilon"]& \\ \\
T_{\red} \arrow[dd, "\pi_{\red}^{(3)}"]\arrow[rrdd, "\Phi"]\arrow[rruu, "\iota_{\red}" description]\arrow[rr,bend left = 10,"\iota"] && \hat T_{\red}\arrow[uu,"\psi"]\arrow[ll,bend left = 10,"\pi"] & \CW[V]\arrow[ldd,bend left=10,"\Ev_0"]\\ \\
T_{\red}^{(3)}\arrow[rr, "\Phi^{(3)}"] && \CV[V]
\end{tikzcd}
\end{center}
\caption{Spaces and morphisms appearing in our construction.}\label{fig:diagram}
\end{figure}
Note that one can not simply consider  $T_{\red}^{(3)} = T_{\red}$ since, unlike $\mathcal{U}_{+}^{(3)}$, the set $\mathcal{U}_{+}$ is not open. The fact that $\mathcal{U}_{+}^{(3)}$ is an open set is crucial in order to apply Proposition~\ref{prop:injectiveLocal}. Before we proceed with this construction, let us show how it yields the claim of the theorem.
Fix $\tau \in \CS \setminus \range \phi_\Ito$. Since $\psi_\Ito$ allows us to view $W_{\<g>}$ as a 
subspace of $W_{\red}$ by the first property and since $\tau = \iota_{\red} \bar \tau$ for some
$\bar \tau \in T_{\red} \setminus \psi_\Ito(T_{\<g>})$ by the second property,
Proposition~\ref{prop:injectiveLocal} yields the existence of $V$ and, since $\CU_+^{(3)}$ is open
and $\pi_{\red}^{(3)}$ is ``linear'' and injective on $\iota_{\red}^{-1}(\CS)$,
of two morphisms $\Phi^{(3)}$ and $\bar \Phi^{(3)}$
in $\CU_+^{(3)}$ such that $\Phi^{(3)} = \bar \Phi^{(3)}$ on the range of $\pi_{\red}^{(3)}\circ \psi_\Ito$, but such that 
$\Phi^{(3)} \pi_{\red}^{(3)}\bar \tau \neq \bar \Phi^{(3)} \pi_{\red}^{(3)}\bar \tau$.
Here, the ``linearity'' of $\pi_{\red}^{(3)} \circ \psi_\Ito$ guarantees that Proposition~\ref{prop:injectiveLocal}
can be applied with $\hat W = (\pi_{\red}^{(3)} \circ \psi_\Ito)(W_{\red})$.

In particular, setting $\Phi = \Phi^{(3)} \circ \pi_{\red}^{(3)} \in \CU_+$ as in point 4
and similarly for $\bar \Phi$, one has $\Phi = \bar \Phi$ on the range
of $\psi_\Ito$ and $\Phi \bar \tau \neq \bar \Phi \bar \tau$.
Writing now $\Upsilon$ and $\bar \Upsilon$ for the
corresponding elements of $\Hom_\d(T_{\<generic>},\CW[V])$ given by point 3 above,
we conclude that $\Ev_0 \circ \Upsilon$ and $\Ev_0 \circ \bar \Upsilon$ coincide on the
range of $\phi_\Ito$. Since $\Upsilon$ is a $T_\d$-morphism mapping elements to polynomials, it follows
that one actually has $\Upsilon = \bar \Upsilon$ on the range of $\phi_\Ito$.

Writing $\sigma_i = \Upsilon( \<generic>_i)$, $\Gamma = \Upsilon(\<not>)$, 
and similarly for $\bar \sigma_i$ and $\bar \Gamma$, we conclude 
that $\sigma \sigma^\top = \bar \sigma \bar \sigma^\top$ and $\Gamma = \bar \Gamma$ 
since $\phi_\Ito (\<g>) = \sum_i (\<generic>_i \mult \<generic>_i)$ and $\phi_\Ito (\<not>) = \<not>$
by definition. On the other hand, since $(\Ev_0 \circ \Upsilon)(\tau) = \Phi(\tau)$, similarly
for $\bar \Upsilon$, and since $\Phi(\tau) \neq \bar \Phi(\tau)$, we conclude that $\sigma \neq \bar \sigma$
as desired.

We now proceed with the construction of the various objects appearing in the argument given above.
 We set $T_{\red} = \hat T_{\red} / \ker \psi$, \label{bar CS red page ref}
where $\hat T_{\red}$ \label{hat CS red page ref}
is the free $T_\d$-algebra with generators $\<not>$ of degree $(1,2)$ and  $\{\gen{k}{\ell}\}_{k,\ell \ge 0}$
of degree $(2,k+\ell)$, 
and $\psi \colon \hat T_{\red} \to T_{\<generic>}$ \label{psi morpishm page ref} is the unique morphism
such that 
\begin{equ}[e:defpsi]
\<not> \mapsto \<not>\;,\qquad \gen{k}{\ell} \mapsto \sum_{i=1}^m (\d^k \<generic>_i\mult \d^\ell \<generic>_i)\;.
\end{equ}
It immediately follows from our construction that $\CS \subset \range \psi$ so that the
second property holds.
We also define $\psi_\Ito\colon T_{\<g>} \to T_{\red}$ as the unique morphism sending 
$\<not>$ to $\<not>$ and $\<g>$ to $\gen 0 0$. 

\begin{remark}
Note that $\CS$ is strictly smaller than the subspace of $T_{\red}$ generated by $\CX$-graphs 
of degree $(1,0)$ with
either one or two generators of type $\gen{k}{\ell}$. This is because the latter does for example
contain the $\CX$-graph obtained by taking the generator $\gen{1}{0}$ and connecting its first output to its
input. The map $\Upsilon^\alpha_{\Gamma,\sigma}\circ \psi$ then maps this to
the expression $\d_\eta\sigma_i^\eta\sigma_i^\alpha$ which is `disconnected' and therefore 
not generated by any of the graphs on Page~\pageref{listPage}. If on the other hand one connects the \textit{second} output
of $\gen{1}{0}$ to its input, then its image under $\psi$ is given by $\<Xi2>$.
\end{remark}

We next identify the kernel of $\psi$.
Denote by $I \subset \hat T_{\red}$ the smallest ideal of $T_\d$-algebras containing all elements of the form
\begin{equ}[e:groupkl]
S^{1,1}_{k,\ell} \gen{k}{\ell} - \gen{\ell}{k}\;,\qquad (\alpha\cdot\beta)\, \gen{k}{\ell} - \gen{k}{\ell}\;, \qquad \swap_{1,1}^{1}\<not> - \<not>\;,
\end{equ}
as well as 
\begin{equ}[e:dergen]
\d\, \gen{k}{\ell} - \gen{k+1}{\ell} - (S^1_{k,1} \cdot \id^1_\ell)  \, \gen{k}{\ell+1}\;,
\end{equ}
for all $k,\ell\geq 0$, $\alpha \in \Sym(1,k)$ and $\beta \in \Sym(1,\ell)$.
It is easy to check that $I \subset \ker \psi$. 
We also have the following crucial result, the proof of which is postponed to the end of the current proof.

\begin{lemma}\label{lem:section}
The canonical projection $\pi \colon \hat T_{\red} \to \hat T_{\red} / I$ admits
a right inverse $\iota\colon \hat T_{\red} / I \to \hat T_{\red}$ which is a morphism of $T_\d$-algebras
and satisfies $(\iota\circ \pi)(\gen 0 0) = \gen 0 0$.
\end{lemma}

We will see momentarily that one actually has $I = \ker \psi$ (at least if $m$ is large enough) so that 
$\hat T_{\red} / I = T_{\red}$. Setting $\iota_{\red} = \psi \circ\iota$,
the identity $\phi_\Ito = \iota_{\red}\circ \psi_\Ito$ is then immediate since
both are morphisms of $T_\d$-algebra and they coincide on the generators $\<not>$ 
and $\<g>$, so that the first property is also satisfied.

Write now $(W_{\red})^\upper_\low$ for the collection of spaces such that
$\hat T_{\red} / I  = \Tr(W_{\red})$. In view of \eqref{e:groupkl} and \eqref{e:dergen},
$(W_{\red})^\upper_2$ is generated by $\alpha \gen \ell n$ with $\ell + n = \upper$
and $\alpha \in \Sym(2,\upper)$, quotiented by the first two relations of 
\eqref{e:groupkl}. Similarly, $(W_{\red})^{\upper+2}_1$ is generated by 
$\alpha \d^\upper\<not>$, quotiented by the last relation of \eqref{e:groupkl}
as well as \eqref{e:idendd}. One has $(W_{\red})^\upper_\low = \{0\}$ for all other values. 

For any given $n > 0$, we then define $J^{(n)} \subset \hat T_{\red} / I$ as the ideal of $T_\d$-algebras generated by 
\begin{equ}[e:genIdeal]
\{\gen k \ell \,:\,k \vee \ell > n \}\cup\{\d^k \<not>\,:\,k > n\}\;.
\end{equ}
(Note that by \eqref{e:dergen} $J^{(n)}$ actually
coincides with the ideal of $T$-algebras generated by the same set.)
We then define $T_{\red}^{(n)} = (\hat T_{\red} / I) / J^{(n)}$ and write 
$\pi^{(n)}_{\red}$ for the corresponding canonical projection. Note that 
one can write
 $T_{\red}^{(n)} \simeq  \Tr(W_{\red}^{(n)})$ with
 \begin{equ}
 (W_{\red}^{(n)})^u_\ell = 
 \left\{\begin{array}{cl}
 	\{0\} & \text{if $\ell \not \in \{1,2\}$ or $\ell = 1$ and $u>n+2$,}\\ 
 	(W_{\red})^u_\ell & \text{if $\ell = 1$ and $u\le n+2$,}
 \end{array}\right.
 \end{equ} 
and $(W_{\red}^{(n)})^u_2 \subset (W_{\red})^u_2$ given by the Sym-invariant subspace  generated by $\{\gen k \ell\,:\, k+\ell = u\;\&\; k \vee \ell \le n\}$.
 (In particular one again has $(W_{\red}^{(n)})^u_2 = \{0\}$ for $u > 2n$.) By Proposition~\ref{prop:freeT}, the inclusion $W_{\red}^{(n)} \subset W_{\red}$ 
yields a canonical injection $\iota_{\red}^{(n)}\in \Hom(T_{\red}^{(n)}, T_{\red})$ which is a right inverse to 
$\pi^{(n)}_{\red}$. Note that $\iota^{(n)}_{\red}$ is \textit{not} a morphism of 
$T_\d$-algebras since, setting for example $\tau = \d^n \<not> \in  T_{\red}^{(n)}$,
one has $\d \tau = 0$ but $\d \iota^{(n)}_{\red} \tau \neq 0$. With these notations, one has 
the following result.

\begin{lemma}\label{lem:reprPhiSigma}
For every $n > 0$ and any finite-dimensional vector space $V$, there exist $m \ge 1$ and
an open set $\CU_+^{(n)} \subset \Hom(T_{\red}^{(n)}, \CV[V])$ such that, for each 
$\Phi^{(n)} \in \CU_+^{(n)}$, there exists $\hat \Upsilon \in \Hom(T_{\<generic>}, \CV[V])$ such that 
\begin{equ}[e:idenUpsilon0]
\hat \Upsilon \circ \psi = \Phi \circ \pi\;,\qquad 
\Phi \eqdef  \Phi^{(n)} \circ \pi_{\red}^{(n)}\;.
\end{equ}
Furthermore, one has $\hat \Upsilon \d^{n+1} \tau = 0$ for every $\tau \in W_{\<generic>}$.
\end{lemma}

We now have the missing ingredient to show that $I$ is actually equal to $\ker \psi$, so that one has 
indeed $T_{\red} = \hat T_{\red} / I$ as claimed above. 
Assume by contradiction that there exists $\tau \in \ker\psi$ with $\pi \tau \neq 0$.
In particular, one must have $\pi_{\red}^{(n)} \pi \tau \neq 0$ for some $n > 0$ which we
now fix. It then follows from Remark~\ref{rem:DiffValue}
that we can find a finite-dimensional vector space $V$, a morphism $\Phi^{(n)} \in \CU_+^{(n)}$ 
(with $\CU_+^{(n)} \subset \Hom(T_{\red}^{(n)}, \CV[V])$ as in Lemma~\ref{lem:reprPhiSigma}) such that 
$\Phi(\pi\tau) \neq 0$ with $\Phi$ as in \eqref{e:idenUpsilon0}. On the other hand, Lemma~\ref{lem:reprPhiSigma}
guarantees the existence of $\hat \Upsilon$ 
such that $\hat \Upsilon \psi(\tau) = \Phi \pi \tau$.
Since $\tau \in \ker \psi$, the left hand side of this identity must vanish, thus creating the 
required contradiction. 

We now go back to the construction of the objects which are needed for property~4. In the 
notations of Lemma~\ref{lem:reprPhiSigma}, we fix from now on $n=3$.
We note that, provided we define $\CU_+$ to be the image of $\CU_+^{(3)}$ under the map
$\Phi^{(3)} \mapsto \Phi$ given by \eqref{e:idenUpsilon0}, 
Lemma~\ref{lem:reprPhiSigma} immediately yields property~4 above. The ``projection''
$\pi_{\red}^{(n)}$ is indeed injective on $\iota_{\red}^{-1}(\CS)$ provided that $n \ge 3$
since no fourth derivative of either $\sigma$ or $\Gamma$ appears in the list of trees constituting it.
It remains to show that property~3 is satisfied. For this, note first that given 
$\hat \Upsilon$ as in the conclusion of Lemma~\ref{lem:reprPhiSigma}, it follows from Theorem~\ref{theo:mainInjective}
that one can find $\Upsilon \in \Hom_\d(T_{\<generic>}, \CW[V])$ such that $\hat \Upsilon = \Ev_0 \circ \Upsilon$.
Since we know at this point that $T_{\red} = \hat T_{\red} / I$
and since $\iota_{\red} = \psi \circ\iota$ and $\pi \circ \iota = \id$, 
the required identity $\Phi = \Ev_0 \circ \Upsilon \circ \iota_{\red}$ now follows at once.
\end{proof}


\begin{proof}[of Lemma~\ref{lem:reprPhiSigma}]
The morphism $\hat \Upsilon$ is uniquely determined by  the quantities
\begin{equ}[e:defDkfun]
D^k \Gamma(0) = {1\over 2}\hat \Upsilon(\d^k \<not>)\;,\qquad
D^k \sigma_i(0) = \hat \Upsilon(\d^k \<generic>_i)\;,
\end{equ}
where our choice of notations is informed by the interpretation of these quantities made above, namely
the fact that one can write $\hat \Upsilon = \Ev_0 \circ \Upsilon$, and \eqref{e:defDkfun} is then consistent
with the notations $\sigma_i = \Upsilon(\<generic>_i)$ and $\Gamma = \Upsilon(\<not>)$.

Our aim then is to find conditions on the morphism $\Phi^{(n)}$ (and therefore on $\Phi$) 
such that one can find
$\hat\Upsilon$ satisfying \eqref{e:idenUpsilon0}.
Regarding $\Gamma$, we simply set $D^k \Gamma(0) = {1\over 2}\Phi \big(\d^k \<not>\big)$ 
(here and below we write just $\Phi$ instead of $\Phi \circ \pi$), so we only need to 
explain how to choose $m$
and the quantities $\{D^k\sigma_i(0)\,:\,i \le m,\,k\ge 0\}$ in 
such a way that \eqref{e:idenUpsilon0} holds. 

Write $S_{\gen{k}{\ell}}$ for the subgroup of 
$\Sym(2,k+\ell)$ given by all elements of the form
$\gamma=(\id^2_0,\alpha\cdot\beta)\in\Sym(2)\times\Sym(k+\ell)$, for some $\alpha \in \Sym(k)$, $\beta \in \Sym(\ell)$.
As a group, $S_{\gen{k}{\ell}}$ is isomorphic to $\Sym(k,\ell)$
and our definition of $I$ guarantees that one has the identity
$\gen k \ell = \gamma \, \gen k \ell$ in $\hat T_{\red} / I$
for all $\gamma\in S_{\gen{k}{\ell}}$.

We then write $W_k = V \otimes (V^*)^{\otimes_s k}$
with $\otimes_s$ denoting the symmetric tensor product and we perform an arbitrary choice of 
scalar product on each of the $W_k$. 
It follows from the definition of the 
group $S_{\gen{k}{\ell}}$ that for any morphism $\Phi \colon \hat T_{\red} / I 
\to \CV[V]$, one can view $\Phi(\gen{k}{\ell})$ as an element of 
$L( W_\ell, W_k)$. This is because it belongs to 
\begin{equs}
\bigl(V^{\otimes 2} \otimes (V^*)^{\otimes (k+\ell)}\bigr)/S_{\gen{k}{\ell}}
&\approx V \otimes V \otimes (V^*)^{\otimes_s k}\otimes (V^*)^{\otimes_s \ell}
\approx W_k\otimes W_\ell \\ &\approx W_k\otimes W_\ell^* \approx L( W_\ell, W_k)\;,
\end{equs}
where the second isomorphism is obtained by exchanging the two middle factors and the third one
uses the choice of scalar product. In the case $k = \ell$, this matrix is furthermore symmetric as a consequence
of the first identity in \eqref{e:groupkl}.

We  then set $W = \bigoplus_{k=0}^n W_k$, write  $m = \dim W$,
and choose an arbitrary orthonormal basis for $W$, so that $W \simeq \R^m$ via that choice. 
In particular, we have 
\begin{equs}
D^k \sigma(0) = (D^k \sigma_i(0))_{i \le m} &\in 
(V \otimes (V^*)^{\otimes_s k})^m \simeq
W^* \otimes V \otimes (V^*)^{\otimes_s k} \\
&\quad\simeq L(W,W_k) \simeq \bigoplus_{j=0}^n L(W_j,W_k)\;,
\end{equs}
which we write as $D^k \sigma(0) = (D^k \sigma^{(j)}(0))_{j \le n}$.

With these identifications in place, 
we then set $(D^k \sigma^{(j)})(0) = 0$ for
$j > k$ (in particular $\sigma^{(j)}(0)$ for $j \neq 0$), $\sigma^{(0)}(0) = \sqrt{\Phi(\gen{0}{0})}$, and then recursively
for $k > \ell$,
\begin{equs}
\bigl(D^k \sigma^{(\ell)}\bigr)(0) &= \Bigl(\Phi(\gen{k}{\ell})
- \sum_{m < \ell} \big(D^k \sigma^{(m)}(0)\big)\big(D^\ell \sigma^{(m)}(0)\big)^* \Bigr) \big(D^\ell \sigma^{(\ell)}(0)\big)^{-1}\;,\\
\bigl(D^k \sigma^{(k)}\bigr)(0) &= \Bigl(\Phi(\gen{k}{k})
- \sum_{m < k} \big(D^k \sigma^{(m)}(0)\big)\big(D^k \sigma^{(m)}(0)\big)^* \Bigr)^{1/2}\;.\label{e:Phikk}
\end{equs}
This is of course only possible in general if $\Phi(\gen{0}{0})$ is 
strictly positive definite and the same is
true for each of the expressions in the large parenthesis of \eqref{e:Phikk}.

Note now that by \eqref{e:genIdeal} (see also the explicit characterisation 
$T_{\red}^{(n)} \simeq  \Tr(W_{\red}^{(n)})$ given earlier)
the morphism $\Phi^{(n)}$ is uniquely determined by 
its action on $\d^k \<not>$ for $k \le n$ and on $\gen k \ell$ for $k \vee \ell \le n$, which in turn take
values in finite-dimensional vector spaces, so we simply define $\CU_+^{(n)}$ as the set of morphisms
$\Phi^{(n)}$ such that the expression in the large parenthesis of \eqref{e:Phikk}
is strictly positive definite for all $k \le n$. 
For $k > n$, one has $\pi^{(n)}(\gen kk) = 0$ and therefore $\Phi(\gen kk) = 0$, but we note that 
since one also has $\Phi(\gen k\ell) = 0$ for $\ell < k$, it follows that 
in this case $D^k \sigma^{(\ell)}(0) = 0$, so that the expression in the parenthesis vanishes
identically. 
This shows that \eqref{e:Phikk} is then consistent with setting $D^k \sigma(0) = 0$ for all $k > n$,
so that the first identity in \eqref{e:idenUpsilon0} holds on $\d^k \<not>$ and $\gen k \ell$ for all values
of $k$ and $\ell$. This also implies that $\hat \Upsilon \d^{n+1} = 0$ on $W_{\<generic>}$, thus completing the proof.
\end{proof}

\begin{proof}[of Lemma~\ref{lem:section}]
We look for a morphism $\bar \pi \colon \hat T_{\red} \to \hat T_{\red}$
such that
$\bar \pi \<not> = {1\over 2} \bigl(\<not> + S^1_{1,1} \<not>\bigr)$ and such that the following properties hold:
\begin{claim}
\item One has $I \subset \ker \bar \pi$ and $\bar \pi \gen00 = \gen00$.
\item For every $k$ and $\ell$, one has $\bar \pi \gen{k}{\ell} = \gen{k}{\ell} + R_{k,\ell}$
where $R_{k,\ell} \in I$ is a linear combination of terms of the form $\alpha \d^a \gen bc$ with $a+b+c = k+\ell$ and
$\alpha \in \Sym(2,k+\ell)$.
When combined with the previous property, this guarantees that $\bar \pi$ is idempotent. 
\end{claim}
If we can build such a map $\bar \pi$, then the desired right inverse $\iota$ is 
the unique morphism such that $\bar \pi = \iota\circ\pi$. Indeed, the existence of such a $\iota$ is guaranteed by
the fact that $I \subset \ker \bar \pi$, while the second property guarantees that $\pi \circ \iota = \id$.

Before we turn to the construction of $\bar \pi$, we define $I_0$ as the smallest ideal of $T_\d$-algebras containing all elements of the form \eqref{e:groupkl} (but without the elements of the form \eqref{e:dergen}), and we consider the canonical projection $\pi_0 \colon \hat T_{\red} \to \hat T_{\red} / I_0$. 
The equivalence class $[\gen k \ell]$ then consists of all elements that are of the form
\begin{equ}
(\alpha \cdot \beta)\gen k \ell\;,\qquad\text{or}\qquad (\alpha \cdot \beta) S_{\ell,k}^{1,1} \gen \ell k\;,
\end{equ}
with $\alpha \in \Sym(1,k)$ and $\beta \in \Sym(1,\ell)$. Furthermore, we have a natural action of the group
\begin{equ}
\hat S_{\gen k \ell} \eqdef \Z_2 \times  \Sym(1,k) \times \Sym(1,\ell)
\end{equ}
onto $[\gen k \ell]$, so that $\pi_0$ admits a section $\iota_0$ by averaging over the equivalence 
classes, namely $\iota_0$ is the unique morphism with
$(\iota_0\circ \pi_0)(\gen k \ell) = |\hat S_{\gen k \ell}|^{-1}\sum_{\alpha \in \hat S_{\gen k \ell}} \alpha \gen k \ell $.

We now denote by $\hat T_{\red}^{(1)}$ the subspace of $\hat T_{\red}$ generated by 
$\CX$-graphs with exactly one vertex, and that vertex is not of type $\<not>$. We also write 
$\CP^\d$ (the reason for this choice of notation will become apparent soon) 
for the image of $\hat T_{\red}^{(1)}$ under $\pi_0$. Since $\bar \pi$ is determined by its values on
the generators, it then suffices to find a map $\bar \pi_0 \colon \CP^\d \to \CP^\d$ which preserves degrees, is equivariant under
the action of the symmetric group and the derivation operation, and which is such that the two properties
above hold. The desired map $\bar \pi$ is then given by $\bar \pi = \iota_0 \circ \bar \pi_0 \circ \pi_0$.
Note that we do not need to consider the original product anymore. On the other hand, $\CP^\d$ can be identified
in a canonical way with the free non-commutative unital algebra generated by three symbols 
$X$, $\bar X$, and $\d$ under the correspondence
\begin{equ}[e:correspondence]
\d^m \gen k \ell \sim \d^m X^k \bar X^\ell \;.
\end{equ}
We write $\CP^\d_n \subset \CP^\d$ for the subspace of homogeneous polynomials of degree $n$. 
Elements $\alpha \in \Sym(2,n) \sim \Z_2 \times \Sym(n)$ act on $\CP^\d_n$ by 
having $\Sym(n)$ permuting the factors of each monomial and the generator of $\Z_2$ swapping $X$ and $\bar X$
(we will also write this as $P \mapsto \bar P$).
For example, one has
\begin{equ}
(S_{2,3}\cdot \id_2^2) \d \gen{2}{4} \sim X \bar X^2 \d X \bar X^2\;,\qquad
(S_{2,3}^{1,1} \cdot \id_2^0) \d \gen{2}{4} \sim \bar X X^2 \d \bar X X^2\;.
\end{equ}
The image under $\pi_0$ of the projection onto  $\hat T_{\red}^{(1)}$ of the ideal (of $T_\d$-algebras) $I$ 
is then equal to the ideal (of algebras!) $\hat I \subset \CP^\d$ generated by $\d -X - \bar X$.

We claim that one possible choice for $\bar \pi_0$ is given by the unique morphism of algebras
such that 
\begin{equ}[e:rightInverse]
\bar \pi_0 \d = \d\;,\quad 
\bar \pi_0 X = {1\over 2}(X-\bar X+\d)\;,\quad 
\bar \pi_0 \bar X = {1\over 2}(\bar X-X+\d)\;.
\end{equ}
This choice clearly satisfies the identity $\overline{\bar \pi_0 P} = \bar \pi_0 \bar P$ so that 
$\bar \pi_0$ is indeed equivariant as required. 
It is also obvious that it satisfies $\bar \pi_0 (\d P) = \d \bar \pi_0 P$ and, since
 $\hat I$ is generated by $\d -X - \bar X$ which is mapped to $0$ by $\bar \pi_0$, that 
 $\hat I \subset \ker \bar \pi_0$. It remains to show that $P - \bar \pi_0 P \in \hat I$,
 which is equivalent to the property that $\bar \pi_0^2 = \bar \pi_0$. This however is
 obvious since $\bar \pi_0$ leaves both $\d$ and $X-\bar X$ invariant. 
 
The fact that $\bar \pi(\gen 0 0) = \gen 0 0$ is an immediate consequence of the fact that 
$\bar \pi 1 = 1$. 
\end{proof}

\begin{proposition}\label{prop:both}
For every $\tau \in \CS_\both$, one has $\hat \phi_\geo(\tau) \in \range \phi_\Ito^{\<diff>}$.
\end{proposition}

\begin{proof}
The proof of this is essentially the same as that of the converse direction for Proposition~\ref{prop:rangeIto}. Indeed, let $ \tau \in \CS_\both $, for any choice of $\Gamma$, any pair $\sigma, \bar \sigma$ such that 
\eqref{e:sigmabar} holds, and any smooth family $(\psi)_{t \geq 0}$ of diffeomorphisms of $\R^d$ as in \eqref{psih}, one has from Definition~\ref{def_S_both}
\begin{equ}
\psi_{t} \act (\Upsilon_{\Gamma,\bar \sigma} - \Upsilon_{\Gamma,\sigma})\, \tau
= 
(\Upsilon_{\psi_{t} \act\Gamma,\psi_{t} \act\bar \sigma} - \Upsilon_{\psi_{t} \act\Gamma,\psi_{t} \act\sigma})\,\tau\;.
\end{equ}
Differentiating this identity at $t=0$ and using \eqref{e:diffRHS} and \eqref{e:diffLHS}, we obtain  
\begin{equs}
\Upsilon_{\Gamma,\sigma}^h \hat \phi_\geo (\tau) = \Upsilon_{\Gamma,\bar \sigma}^h \hat \phi_\geo (\tau).
\end{equs}
Therefore, $ \hat \phi_\geo (\tau) \in \CS_{\Ito}^{\<diff>} $ and we conclude by the fact $\CS_\Ito^{\<diff>} = \CS_{\<diff>}\cap \range \phi^{\<diff>}_\Ito$ by Proposition~\ref{prop:rangeIto}. 
\end{proof}

We conclude this subsection with the missing ingredient for the proof of 
Proposition~\ref{prop:ortho}.

\begin{corollary}\label{cor:charSnice}
With $\hat \CS^\nice\subset \CS$ defined by \eqref{e:charSnice}, one has
$\CS^\nice \subset \hat \CS^\nice$.
\end{corollary}

\begin{proof}
Let $W_{\red}^\nice \subset W_{\red}$ be the smallest linear subspace invariant under the action of the 
symmetric group that furthermore contains $\pi\<not>$ and all elements of the form $\pi\gen 1 k$ with $k \ge 0$.
The description of $\hat \CS^\nice$ given at the start of the proof of Proposition~\ref{prop:ortho} combined with the 
definition of $\iota_{\red}$ shows that $\hat \CS^\nice = \CS \cap \iota_{\red} I_{\red}^\nice$, where
$I_{\red}^\nice$ denotes the ideal of $T$-algebras generated by $W_{\red}^\nice$. 
Note also that if we write $W_{\red}^\perp \subset W_{\red}$ for the smallest linear subspace invariant 
under the action of the symmetric group that furthermore contains $\d^k \<not>$ for all $k > 0$
as well as all $\pi\gen k \ell$ with $k \neq 1$ and $\ell \neq 1$, one has
\begin{equ}[e:decompWred]
W_{\red} = W_{\red}^\nice \oplus W_{\red}^\perp\;,\qquad
\Tr(W_{\red}) = I_{\red}^\nice\oplus \Tr(W_{\red}^\perp) \;.
\end{equ}
Here, the second identity follows from the fact that the analogous identity holds for
regular tensor algebras and that for any $T$-graph $g$ this decomposition is preserved 
under the action of $\Sym_g$. 

Fix furthermore a basis of $W_{\red}$ respecting the direct sum decomposition \eqref{e:decompWred}
and, with $A^\upper_\low$ as after \eqref{e:defV}, let $\hat A^\upper_\low \subset A^\upper_\low$ 
be the subset of those indices corresponding to basis elements in $W_{\red}^\perp$.
Let then $V$ be as constructed in the proof of Theorem~\ref{theo:injectiveT} and fix a morphism 
$\Phi_0 \colon T_{\red} \to \CV[V]$ with the property that $\Phi_0(\pi \<not>) = \Phi_0(\pi\gen 1 k) = 0$ for all
$k \ge 0$ and, for $k \neq 1$, $\Phi_0(\pi\gen k k)$ is such that the terms in the large 
parenthesis in \eqref{e:Phikk} are all strictly positive definite.
We also let $\Phi_1$ be the morphism defined by \eqref{e:defPhi} and define the morphisms $\Phi_t$
as in \eqref{e:interpPhi}.

We then note that our choice for $\Phi_0$ and the sets $\hat A^\upper_\low$ guarantees 
that, for every $t \in \R$, one has $I_{\red}^\nice \subset \ker \Phi_t$.
Furthermore, arguing as in Proposition~\ref{prop:injectiveLocal}, we deduce that, for every 
$\tau \in T_{\red} \setminus I_{\red}^\nice$ one can find $t$ arbitrarily small such
that $\Phi_t \tau \neq 0$. 
We then note that even though $\Phi_t(\gen 11) = 0$, the recursion \eqref{e:Phikk}
is well-defined provided that we set $\big(D^k \sigma^{(1)}\bigr)(0) = 0$. In particular,
it guarantees that $\Gamma(0) = 0$, $(D\sigma)(0) = 0$, and, by \eqref{e:idenUpsilon0},
$\Upsilon_{\Gamma,\sigma} \iota_{\red}(\tau) \neq 0$.
This shows that $\iota_{\red}(\tau) \not \in \CS^\nice$, showing that 
$\CS \setminus \hat \CS^\nice \subset \CS \setminus \CS^\nice$ as required.
\end{proof}

\subsection{Dimension counting}
\label{sec:counting}

\begin{proposition}
The space $\CS$ is of dimension $\dim \CS = 54$.
\end{proposition}

\begin{proof} By definition, $\CS$ is the vector space freely generated by $\SS$. The list of elements of $\SS$ is given in \eqref{e:SS},
and it is indeed of cardinality $54$.
\end{proof}

In order to obtain more information about the space $\CS_\geo$, the following result will be useful. This is simply the
algebraic counterpart of the fact that covariant differentiation transforms in the `correct' way under the action
of the diffeomorphism group.

\begin{lemma} \label{Nabla preserved by phigeo}
Define $\nabla\colon (T_{\<generic>})^1_0 \times (T_{\<generic>})^1_0 \to (T_{\<generic>})^1_0$ by
\begin{equ}[e:nabla]
\Nabla_{A}B = A \graftI B + {1\over 2}\tr^2(\<not> \mult A\mult B)\;,
\end{equ}
where $\graftI$ is defined in \eqref{e:graft}. Then, $\nabla$ preserves $\ker \hat \phi_\geo$.
\end{lemma}

\begin{proof}
Take $A, B \in \ker \hat \phi_\geo$ so that for example $\phi_\geo(A) = [A,\<diff>]$. Combining this with the defining 
property of an infinitesimal morphism, one has the identity 
\begin{equs}[e:bigidentity]
\phi_{\geo}(\Nabla_{A}B) & =[A,\<diff>] \graftI B +  A \graftI [B,\<diff>] +   {1\over 2}\tr^2(\phi_{\geo}(\<not>) \mult A\mult B) \\ &  +   {1\over 2}\tr^2(\<not> \mult [A,\<diff>] \mult B) + {1\over 2}\tr^2(\<not> \mult A \mult [B,\<diff>]) \;.
\end{equs}
Using the diagrammatic representation (in the universal $T_\d$-algebra of $\CX$-graphs generated by 
$A$, $B$ and $\<not>$, with the quotient by $(\<not> - \swap_{1,1}^{1}\<not>)$ meaning that we do not
order the two inputs of $\<not>$), we can rewrite \eqref{e:nabla} as
\begin{equ}[e:covar]
\Nabla_{A}B = \tikzset{external/export next=false}
\begin{tikzpicture}[scale=0.2,baseline=2]
\draw[symbols]  (-.5,2.5) -- (0,0) ;
\draw[tinydots] (0,0)  -- (0,-1.3);
\node[var] (root) at (0,-0.1) {\tiny{$ B $ }};
\node[var] (diff) at (-0.5,2.5) {\tiny{$ A $ }};
\end{tikzpicture} + \frac{1}{2} \; \tikzset{external/export next=false}
\begin{tikzpicture}[scale=0.2,baseline=2]
\coordinate (root) at (0,0);
\coordinate (t1) at (-1,2);
\coordinate (t2) at (1,2);
\draw[tinydots] (root)  -- +(0,-0.8);
\draw[kernels2] (t1) -- (root);
\draw[kernels2] (t2) -- (root);
\node[not] (rootnode) at (root) {};
\node[var] (t1) at (t1) {\tiny{$ A $}};
\node[var] (t1) at (t2) {\tiny{$ B $}};
\end{tikzpicture}.
\end{equ}
With this representation, \eqref{e:bigidentity} yields the identities
\begin{equs}
\tr^2(\phi_{\geo}(\<not>) \mult A\mult B) &= \tikzset{external/export next=false}
\begin{tikzpicture}[scale=0.2,baseline=-5]
\coordinate (root) at (0,0);
\coordinate (tri) at (0,-2);
\coordinate (t1) at (-1,2);
\coordinate (t2) at (1,2);
\draw[tinydots] (tri)  -- +(0,-0.8);
\draw[kernels2] (t1) -- (root);
\draw[kernels2] (t2) -- (root);
\draw[symbols] (root) -- (tri);
\node[not] (rootnode) at (root) {};
\node[diff] (trinode) at (tri) {};
\node[var] (rootnode) at (t1) {\tiny{$ A $}};
\node[var] (trinode) at (t2) {\tiny{$ B $}};
\end{tikzpicture}
\;-\;
\tikzset{external/export next=false}
\begin{tikzpicture}[scale=0.2,baseline=2]
\coordinate (root) at (0,0);
\coordinate (tri) at (1.2,1.2);
\coordinate (t1) at (-1.8,1.6);
\coordinate (t2) at (0,2.5);
\draw[tinydots] (root)  -- +(0,-0.8);
\draw[kernels2] (t1) -- (root);
\draw[kernels2] (t2) -- (root);
\draw[symbols] (root) -- (tri);
\node[not] (rootnode) at (root) {};
\node[diff] (trinode) at (tri) {};
\node[var] (rootnode) at (t1) {\tiny{$ A $}};
\node[var] (trinode) at (t2) {\tiny{$ B $}};
\end{tikzpicture}
\;-\;
\tikzset{external/export next=false}
\begin{tikzpicture}[scale=0.2,baseline=2]
\coordinate (root) at (0,0);
\coordinate (t1) at (-.5,3);
\coordinate (tri) at (1,1);
\coordinate (t2) at (1.5,3);
\draw[tinydots] (root)  -- +(0,-0.8);
\draw[kernels2] (t1) -- (root);
\draw[symbols] (t2) -- (tri);
\draw[kernels2] (tri) -- (root);
\node[not] (rootnode) at (root) {};
\node[diff] (trinode) at (tri) {};
\node[var] (rootnode) at (t1) {\tiny{$ A $}};
\node[var] (trinode) at (t2) {\tiny{$ B $}};
\end{tikzpicture}
\;-\;
\tikzset{external/export next=false}
\begin{tikzpicture}[scale=0.2,baseline=2]
\coordinate (root) at (0,0);
\coordinate (t1) at (.5,3);
\coordinate (tri) at (-1,1);
\coordinate (t2) at (-1.5,3);
\draw[tinydots] (root)  -- +(0,-0.8);
\draw[kernels2] (t1) -- (root);
\draw[symbols] (t2) -- (tri);
\draw[kernels2] (tri) -- (root);
\node[not] (rootnode) at (root) {};
\node[diff] (trinode) at (tri) {};
\node[var] (rootnode) at (t1) {\tiny{$ B $}};
\node[var] (trinode) at (t2) {\tiny{$ A $}};
\end{tikzpicture}
\;-\;2
\tikzset{external/export next=false}
\begin{tikzpicture}[scale=0.2,baseline=2]
\coordinate (root) at (0,0);
\coordinate (t1) at (-1,2);
\coordinate (t2) at (1,2);
\draw[tinydots] (root)  -- +(0,-0.8);
\draw[symbols] (t1) -- (root);
\draw[symbols] (t2) -- (root);
\node[diff] (rootnode) at (root) {};
\node[var] (rootnode) at (t1) {\tiny{$ A $}};
\node[var] (trinode) at (t2) {\tiny{$ B $}};
\end{tikzpicture}\;
\\
{}[A,\<diff>] \graftI B &= \tikzset{external/export next=false}\begin{tikzpicture}[scale=0.2,baseline=2]
\draw[symbols]  (-.5,2) -- (0,0) ;
\draw[symbols]  (0.5,3.5) -- (-0.5,2) ;
\draw[tinydots] (0,0)  -- (0,-1.3);
\node[var] (root) at (0,-0.1) {\tiny{$ B $ }};
\node[diff] (diff) at (-0.5,2) {};
\node[var] (diff) at (0.5,3.5) {\tiny{$ A $  }};
\end{tikzpicture} \; - \; \tikzset{external/export next=false} \begin{tikzpicture}[scale=0.2,baseline=2]
\draw[symbols]  (-.5,2) -- (0,0) ;
\draw[symbols]  (0.5,3.5) -- (-0.5,2) ;
\draw[tinydots] (0,0)  -- (0,-1.3);
\node[var] (root) at (0,-0.1) {\tiny{$ B $ }};
\node[var] (diff) at (-0.5,2) {\tiny{$ A $ }};
\node[diff] (diff) at (0.5,3.5) {};
\end{tikzpicture}, \\
A \graftI [B,\<diff>] &= \tikzset{external/export next=false}
\begin{tikzpicture}[scale=0.2,baseline=-2]
\coordinate (root) at (0,0);
\coordinate (t1) at (-1,2);
\coordinate (t2) at (1,2);
\draw[tinydots] (root)  -- +(0,-0.8);
\draw[symbols] (t1) -- (root);
\draw[symbols] (t2) -- (root);
\node[diff] (rootnode) at (root) {};
\node[var] (rootnode) at (t1) {\tiny{$ A $}};
\node[var] (trinode) at (t2) {\tiny{$ B $}};
\end{tikzpicture} \; +  \;
\tikzset{external/export next=false}
\begin{tikzpicture}[scale=0.2,baseline=2]
\coordinate (root) at (0,0);
\coordinate (t1) at (-1,3.5);
\coordinate (t2) at (1,2);
\draw[tinydots] (root)  -- +(0,-0.8);
\draw[symbols] (t1) -- (t2);
\draw[symbols] (t2) -- (root);
\node[diff] (rootnode) at (root) {};
\node[var] (rootnode) at (t1) {\tiny{$ A $}};
\node[var] (trinode) at (t2) {\tiny{$ B $}};
\end{tikzpicture} \;
 - \; \tikzset{external/export next=false}\begin{tikzpicture}[scale=0.2,baseline=2]
\draw[symbols]  (-.5,2) -- (0,0) ;
\draw[symbols]  (0.5,3.5) -- (-0.5,2) ;
\draw[tinydots] (0,0)  -- (0,-1.3);
\node[var] (root) at (0,-0.1) {\tiny{$ B $ }};
\node[diff] (diff) at (-0.5,2) {};
\node[var] (diff) at (0.5,3.5) {\tiny{$ A $  }}; \end{tikzpicture}  \;
 - \;\tikzset{external/export next=false} \begin{tikzpicture}[scale=0.2,baseline=2]
\draw[symbols]  (-.5,2) -- (0,0) ;
\draw[symbols]  (1,2) -- (0,-0.1) ;
\draw[tinydots] (0,0)  -- (0,-1.3);
\node[var] (root) at (0,-0.1) {\tiny{$ B $ }};
\node[diff] (diff) at (-0.5,2) {};
\node[var] (diff) at (1,2) {\tiny{$ A $  }};
\end{tikzpicture}, \\
\tr^2(\<not> \mult [A,\<diff>] \mult B) &= 
\tikzset{external/export next=false}
\begin{tikzpicture}[scale=0.2,baseline=-2]
\coordinate (root) at (0,0);
\coordinate (t1) at (.5,3);
\coordinate (tri) at (-1,1);
\coordinate (t2) at (-1.5,3);
\draw[tinydots] (root)  -- +(0,-0.8);
\draw[kernels2] (t1) -- (root);
\draw[symbols] (t2) -- (tri);
\draw[kernels2] (tri) -- (root);
\node[not] (rootnode) at (root) {};
\node[diff] (trinode) at (tri) {};
\node[var] (rootnode) at (t1) {\tiny{$ B $}};
\node[var] (trinode) at (t2) {\tiny{$ A $}};
\end{tikzpicture} \; - \;  \tikzset{external/export next=false}
\begin{tikzpicture}[scale=0.2,baseline=-2]
\coordinate (root) at (0,0);
\coordinate (t1) at (.5,3);
\coordinate (tri) at (-1.5,1.5);
\coordinate (t2) at (-1.5,3.5);
\draw[tinydots] (root)  -- +(0,-0.8);
\draw[kernels2] (t1) -- (root);
\draw[symbols] (t2) -- (tri);
\draw[kernels2] (tri) -- (root);
\node[not] (rootnode) at (root) {};
\node[diff] (trinode) at (t2) {};
\node[var] (rootnode) at (t1) {\tiny{$ B $}};
\node[var] (trinode) at (tri) {\tiny{$ A $}};
\end{tikzpicture} , \quad \tr^2(\<not> \mult A \mult [B,\<diff>]) = \tikzset{external/export next=false}
\begin{tikzpicture}[scale=0.2,baseline=-2]
\coordinate (root) at (0,0);
\coordinate (t1) at (-.5,3);
\coordinate (tri) at (1,1);
\coordinate (t2) at (1.5,3);
\draw[tinydots] (root)  -- +(0,-0.8);
\draw[kernels2] (t1) -- (root);
\draw[symbols] (t2) -- (tri);
\draw[kernels2] (tri) -- (root);
\node[not] (rootnode) at (root) {};
\node[diff] (trinode) at (tri) {};
\node[var] (rootnode) at (t1) {\tiny{$ A $}};
\node[var] (trinode) at (t2) {\tiny{$ B $}};
\end{tikzpicture} \; - \; \tikzset{external/export next=false}
\begin{tikzpicture}[scale=0.2,baseline=-2]
\coordinate (root) at (0,0);
\coordinate (t1) at (-.5,3);
\coordinate (tri) at (1.5,1.5);
\coordinate (t2) at (1.5,3.5);
\draw[tinydots] (root)  -- +(0,-0.8);
\draw[kernels2] (t1) -- (root);
\draw[symbols] (t2) -- (tri);
\draw[kernels2] (tri) -- (root);
\node[not] (rootnode) at (root) {};
\node[diff] (trinode) at (t2) {};
\node[var] (rootnode) at (t1) {\tiny{$ A $}};
\node[var] (trinode) at (tri) {\tiny{$ B $}};
\end{tikzpicture}\;.
\end{equs}
Combining these terms, we obtain
\begin{equs}
\phi_{\geo}(\Nabla_{A}B) & =\tikzset{external/export next=false}
\begin{tikzpicture}[scale=0.2,baseline=4]
\coordinate (root) at (0,0);
\coordinate (t1) at (-1,3.5);
\coordinate (t2) at (1,2);
\draw[tinydots] (root)  -- +(0,-0.8);
\draw[symbols] (t1) -- (t2);
\draw[symbols] (t2) -- (root);
\node[diff] (rootnode) at (root) {};
\node[var] (rootnode) at (t1) {\tiny{$ A $}};
\node[var] (trinode) at (t2) {\tiny{$ B $}};
\end{tikzpicture} \; -  \; \tikzset{external/export next=false} \begin{tikzpicture}[scale=0.2,baseline=2]
\draw[symbols]  (-.5,2) -- (0,0) ;
\draw[symbols]  (0.5,3.5) -- (-0.5,2) ;
\draw[tinydots] (0,0)  -- (0,-1.3);
\node[var] (root) at (0,-0.1) {\tiny{$ B $ }};
\node[var] (diff) at (-0.5,2) {\tiny{$ A $ }};
\node[diff] (diff) at (0.5,3.5) {};
\end{tikzpicture} \;
 - \;\tikzset{external/export next=false} \begin{tikzpicture}[scale=0.2,baseline=2]
\draw[symbols]  (-.5,2) -- (0,0) ;
\draw[symbols]  (1,2) -- (0,-0.1) ;
\draw[tinydots] (0,0)  -- (0,-1.3);
\node[var] (root) at (0,-0.1) {\tiny{$ B $ }};
\node[diff] (diff) at (-0.5,2) {};
\node[var] (diff) at (1,2) {\tiny{$ A $  }};
\end{tikzpicture}\;  + \; \frac{1}{2} \tikzset{external/export next=false}
\begin{tikzpicture}[scale=0.2,baseline=-5]
\coordinate (root) at (0,0);
\coordinate (tri) at (0,-2);
\coordinate (t1) at (-1,2);
\coordinate (t2) at (1,2);
\draw[tinydots] (tri)  -- +(0,-0.8);
\draw[kernels2] (t1) -- (root);
\draw[kernels2] (t2) -- (root);
\draw[symbols] (root) -- (tri);
\node[not] (rootnode) at (root) {};
\node[diff] (trinode) at (tri) {};
\node[var] (rootnode) at (t1) {\tiny{$ A $}};
\node[var] (trinode) at (t2) {\tiny{$ B $}};
\end{tikzpicture} \; - \; \frac{1}{2} \tikzset{external/export next=false}
\begin{tikzpicture}[scale=0.2,baseline=2]
\coordinate (root) at (0,0);
\coordinate (t1) at (-.5,3);
\coordinate (tri) at (1.5,1.5);
\coordinate (t2) at (1.5,3.5);
\draw[tinydots] (root)  -- +(0,-0.8);
\draw[kernels2] (t1) -- (root);
\draw[symbols] (t2) -- (tri);
\draw[kernels2] (tri) -- (root);
\node[not] (rootnode) at (root) {};
\node[diff] (trinode) at (t2) {};
\node[var] (rootnode) at (t1) {\tiny{$ A $}};
\node[var] (trinode) at (tri) {\tiny{$ B $}};
\end{tikzpicture}  \; - \; \frac{1}{2} \tikzset{external/export next=false}
\begin{tikzpicture}[scale=0.2,baseline=2]
\coordinate (root) at (0,0);
\coordinate (t1) at (.5,3);
\coordinate (tri) at (-1.5,1.5);
\coordinate (t2) at (-1.5,3.5);
\draw[tinydots] (root)  -- +(0,-0.8);
\draw[kernels2] (t1) -- (root);
\draw[symbols] (t2) -- (tri);
\draw[kernels2] (tri) -- (root);
\node[not] (rootnode) at (root) {};
\node[diff] (trinode) at (t2) {};
\node[var] (rootnode) at (t1) {\tiny{$ B $}};
\node[var] (trinode) at (tri) {\tiny{$ A $}};
\end{tikzpicture} \;-\; \frac{1}{2}
\tikzset{external/export next=false}
\begin{tikzpicture}[scale=0.2,baseline=2]
\coordinate (root) at (0,0);
\coordinate (tri) at (1.2,1.2);
\coordinate (t1) at (-1.8,1.6);
\coordinate (t2) at (0,2.5);
\draw[tinydots] (root)  -- +(0,-0.8);
\draw[kernels2] (t1) -- (root);
\draw[kernels2] (t2) -- (root);
\draw[symbols] (root) -- (tri);
\node[not] (rootnode) at (root) {};
\node[diff] (trinode) at (tri) {};
\node[var] (rootnode) at (t1) {\tiny{$ A $}};
\node[var] (trinode) at (t2) {\tiny{$ B $}};
\end{tikzpicture} \\ & = [\Nabla_{A}B,\<diff>]\;,
\end{equs}
as claimed.
\end{proof}


\begin{proposition}\label{prop:geoDim}
The space $\CS_\geo$ has dimension at least $15$ and 
the set \label{B geo page ref}
\begin{equ}
\CB_\geo = \Big\{\<Xi2> \,, \<Xi4_1s> \,, \<Xi4c1s>\,, \<Xi4_2s>\,,  \<Xi4ec3s> \,,  \<Xi4ec1s> \,,  \<Xi4ec2s> \,, \<Xi41s> \,, \<Xi42s> \,, \<Xi4b1s>\;,    \<Xi4eac2s>, \<Xi4ca2s>,  
\<Xi4eac1s> - \<Xi4eabisc1s>, \<Xi4eac1s> + \<Xi4eabisc1s> ,   \<Xi4ba1bs>   \Big\}\;,
\end{equ}
viewed as linear functionals on $\CS$ is linearly independent over $\CS_\geo$. Besides one has ${1\over 2} \<Xi4b1> - \<Xi4ba2> -  {1\over 2} \<Xi4ba1b> \perp \CS_\geo$. 
Furthermore, $\CS_\geo^\nice= \CS_\geo \cap \CS^\nice$ is of dimension at least $13$ and $\CB_\geo \setminus \{\<Xi4ba1bs>,\<Xi4b1s>\}$ is linearly independent over $\CS_\geo^\nice$. 
\end{proposition}

\begin{proof}
Write $\CS_{\flt}\subset \CS$ for the subspace generated by those trees in $\CS$ that 
contain only noise nodes, namely by $\big\{\<Xi2> \,, \<Xi4_1s> \,, \<Xi4c1s>\,, \<Xi4_2s>\,,  \<Xi4ec3s> \,,  \<Xi4ec1s> \,,  \<Xi4ec2s> \,, \<Xi41s> \,, \<Xi42s> \,, \<Xi4b1s>\big\}$,
and write $\pi_{\flt}\colon \CS \to \CS_{\flt}$ for the corresponding orthogonal projection.
It then follows immediately from \eqref{e:nabla} and the definition \eqref{e:graft} of $\graftI$ that for $\tau,\sigma\in\CS$ we have the identity
\begin{equ}[e:rec]
\pi_{\flt} \bigl(\nabla_{\tau}\sigma\bigr) = \pi_{\flt}\tau  \graftI \pi_{\flt}\sigma\;,
\end{equ}
which is nothing but the algebraic counterpart of the fact that 
$(\nabla_{X}Y)^\alpha = X^\beta D_\beta Y^\alpha$ in flat space.
Let $T_{\flt}$ be the subspace of $T_{\<generic>}$ generated by the elements $\<generic>_i$
and the operation $\graftI$, and
define a linear map 
$\phi_{\flt} \colon T_{\flt} \to T_{\<generic>}$ uniquely by the properties
\begin{equ}
\phi_{\flt}(\<generic>_i) = \<generic>_i\;,\quad \phi_{\flt}(\tau \graftI \sigma) = \nabla_{\phi_{\flt}(\tau)}
\, \phi_{\flt}(\sigma)\;.
\end{equ}
Since $\graftI$ acts freely on $T_\flt$ and generates all (linear combinations of) trees with
only noise nodes by \cite{preLie}, one has
$\CS_{\flt}\subset T_{\flt}$ and therefore $\phi_{\flt}\colon\CS_{\flt}\to\CS$. Moreover by Lemma \ref{Nabla preserved by phigeo} we have that $\phi_{\flt}\colon\CS_{\flt}\to\ker \hat \phi_\geo$, so that $\phi_{\flt}\colon\CS_{\flt}\to\CS_\geo$ by Proposition \ref{geoker}.
By \eqref{e:rec}, $\phi_{\flt}\colon\CS_{\flt}\to\CS_\geo$ is a right inverse for $\pi_{\flt}$ and for $\tau \in \CS_{\flt}$
one has
\begin{equ}
\scal{\tau, \phi_{\flt}(\tau)} = \scal{\pi_{\flt} \tau, \phi_{\flt}(\tau)} = 
\scal{\tau, \pi_{\flt} \phi_{\flt}(\tau)} = \|\tau\|^2\;.
\end{equ}
Therefore if $\tau\in\CS_{\flt}$ is such that $\scal{\tau,\sigma}=0$ for all $\sigma\in\CS_\geo$, we obtain $\|\tau\|^2=\scal{\tau, \phi_{\flt}(\tau)}= 0$ and it follows that
$\CS_{\flt}$ is linearly independent over $\CS_\geo$. 

To complete the proof of the first statement, 
it remains to find a subspace $\CS_R\subset \CS_\geo$ that is annihilated by
$\CS_{\flt}$ and such that the remaining five elements 
\[
\big\{  
 \<Xi4eac2s>, \<Xi4ca2s>,  
\<Xi4eac1s> - \<Xi4eabisc1s>, \<Xi4eac1s> + \<Xi4eabisc1s> ,   \<Xi4ba1bs>   \big\} = \CB_\geo \setminus \CS_{\flt}
\]
are linearly independent over $\CS_R$. For this, we note that formal `vector fields' 
involving the curvature $ R $ defined in \eqref{e:R}
are annihilated by $\CS_{\flt}$. From an analytic perspective, this is because $R = 0$ 
whenever $\Gamma \equiv 0$, but this is also seen easily from
the algebraic representation of $R$ since it only involves trees with at least one $\Gamma$-node.
Then, we first notice that $\<Xi4ca2s>$ and $\<Xi4ba1bs>$ are linearly independent over  $ \Nabla_{\<generic>}(R(\<generic>,\<genericb>)\<genericb>)  $ and $ \Nabla_{R(\<generic>,\<genericb>)\<genericb>}\<generic>  $. Indeed, since one has
\begin{equs} 
R(\<generic>,\<genericb>)\<genericb> = 
 \frac{1}{2}\, \<R1>  - \frac{1}{2}\, \<R2>  + \frac{1}{4} \, \<R3>  - \frac{1}{4} \, \<R4> \;,
\end{equs}
the expression \eqref{e:covar} of the covariant derivative shows that $\<Xi4ca2s>$ annihilates
$\Nabla_{\<generic>}(R(\<generic>,\<genericb>)\<genericb>)$ but not $ \Nabla_{R(\<generic>,\<genericb>)\<genericb>}\<generic>  $, while the opposite is true for $\<Xi4ba1bs>$.

 The remaining three terms   $\<Xi4eac2s>,   
\<Xi4eac1s> - \<Xi4eabisc1s>, \<Xi4eac1s> + \<Xi4eabisc1s> $ are seen to be linearly independent over $R(\<generic>,\Nabla_{\<genericb>}\<genericb>)\<generic>,   R(\<generic>,\Nabla_{\<genericb>}\<generic>)\<genericb> $ and $ R(\<generic>,\Nabla_{\<genericb>}\<generic>- 2\Nabla_{\<generic>}\<genericb>)\<genericb>, $ explicitly given by
\begin{equs}
R(\<generic>,\Nabla_{\<genericb>}\<generic>)\<genericb> & = \frac{1}{2} \, \<Xi4eabisc1> + \frac{1}{4} \, \<Xi4eabbisc1> - \frac{1}{2} \, \<Xi4eac1> - \frac{1}{4} \, \<Xi4eabc1> + \frac{1}{4}  \, \<I1Xi4acc1>  + \frac{1}{8} \, \<I1Xi4abcc1>   - \frac{1}{4} \, \<2I1Xi4cc1>   - \frac{1}{8}  \,  \<2I1Xi4c1>  \\
R(\<generic>,\Nabla_{\<generic>}\<genericb>)\<genericb> & = \frac{1}{2} \,   \<Xi4eabisc2> + \frac{1}{4} \, \<Xi4eabbisc1> - \frac{1}{2} \, \<Xi4eac1> - \frac{1}{4} \, \<Xi4eabc1>   + \frac{1}{4}  \,\<I1Xi4acc2>  + \frac{1}{8} \, \<I1Xi4abcc1>  - \frac{1}{4} \, \<2I1Xi4cc1> 
- \frac{1}{8}  \,  \<2I1Xi4c1>   \\
R(\<generic>,\Nabla_{\<genericb>}\<genericb>)\<generic> & = \frac{1}{2} \, \<Xi4eabisc3> + \frac{1}{4} \,  \<Xi4eabbisc2> -  \frac{1}{2} \, \<Xi4eac2> - \frac{1}{4} \, \<Xi4eabc2>  
 + \frac{1}{4}  \,  \<cI1Xi4ac>  + \frac{1}{8} \, \<I1Xi4abcc2> - \frac{1}{4} \, \<2I1Xi4cc2> - \frac{1}{8} \, \<2I1Xi4c2>
\end{equs}
which are all three annihilated by $\<Xi4ca2s>$ and $\<Xi4ba1bs>$.
The last statement of the proposition follows from Proposition~\ref{prop:ortho}.

It remains to prove that ${1\over 2}\<Xi4b1s> - \<Xi4ba2s> -  {1\over 2} \<Xi4ba1bs> \perp \CS_\geo$. For any $\tau \in \CS_\geo$, we have $\hat{\phi}_\geo(\tau)=0$ by Proposition \ref{geoker}, so that
$\scal{\hat{\phi}_\geo(\tau),\<Xi4ba1bdiffs>}=\scal{\tau,\hat{\phi}_\geo^*(\<Xi4ba1bdiffs>)}=0$. We conclude from the fact that $\hat{\phi}_\geo^*(\<Xi4ba1bdiffs>)= {1\over 2}\<Xi4b1s> - \<Xi4ba2s> -  {1\over 2} \<Xi4ba1bs>$ as a consequence of a simple calculation, using the fact
that these three basis vectors are the only ones that can possibly generate a copy of $\<Xi4ba1bdiffs>$,
as well as the fact that $|\<Xi4b1s>|^2 = |\<Xi4ba2s>|^2 = 2$ while $|\<Xi4ba1bs>|^2 = 4$.
\end{proof}

\begin{proposition} 
\label{prop:CB} \label{CB Ito page ref}
One has $\CB_\Ito \subset \CS_\Ito$ with
\begin{align}\label{Bito}
\CB_\Ito = \left\{\!\!
\begin{array}{c}
\<I1Xitwos> \,,   \<cI1Xi4abs>  \,, \<I1Xi4acc1s> +  \<2I1Xi4cc1s> \,,   \<I1Xi4acc2s> \,,\<I1Xi4abcc2s> \,, \<I1Xi4abcc1s> \,,   \<2I1Xi4c2s>  \,,  \<2I1Xi4c1s>  \,,  \<2I1Xi4bc1s> + \<Xi4cbc2s> \\
\<Xi4eabisc2s>   \,, \<Xi4eabc2s>  \,,  \<Xi4eabc1s>  \,, \<Xi4eabbisc2s>  \,, \<Xi4eabbisc1s>  \,, 
  \<Xi4cabc1s> \,, \<Xi4cabc2s>  \,, \<Xi4ba2s> \,, \<Xi4eac1s> + \<Xi4eabisc1s> \,,   \<Xi4ba1bs>
\end{array}\!\!
\right\}\;.
\end{align} 
In particular, $\CS_\Ito$ is of dimension at least $19$.
\end{proposition}

\begin{proof}
Each of these elements belongs to the image of $\phi_\Ito$ by direct inspection, see Example \ref{ex:phi_ito}.
\end{proof}

\begin{corollary}\label{ge32}
One has $\dim (\CS_\geo + \CS_\Ito) \ge |\CB_\geo| + |\CB_\Ito| - 2 = 32$.
\end{corollary}

\begin{proof}
It suffices to note that all distinct elements of $\CB_\geo \cup \CB_\Ito$ are mutually orthogonal
and that $|\CB_\geo \cap \CB_\Ito| = 2$.
\end{proof}

Let us define the free $T_\d$-algebra $\hat T_{\red,\<diff>}$ \label{CS red diff page ref} in the same way as $\hat T_{\red}$,
except that we add an element  $\<diff>$ of degree $(1,0)$ to the set of generators.
We define the morphism $\psi_{\<diff>} \colon \hat T_{\red,\<diff>} \to T_{\<generic>,\<diff>}$ \label{psi diff page ref}
to be the unique extension of the morphism $\psi$ defined in \eqref{e:defpsi} that sends $\<diff>$ to $\<diff>$
and we set $T_{\red,\<diff>} = \hat T_{\red,\<diff>} / \ker \psi_{\<diff>}$ \label{Quotient space CS red diff page ref}. Note that similarly to before, both $\CS$ and $\CS_{\<diff>}$ can (and will from now on)
be seen as subspaces of $T_{\red,\<diff>}$ and $\phi_\Ito^{\<diff>}$ can be viewed as the map 
$\phi_\Ito^{\<diff>} \colon T_{\<g>,\<diff>} \to T_{\red,\<diff>}$ sending $\<not>$ to $\<not>$, $\<diff>$ to $\<diff>$, and $\<g>$ to $\pi^{\<diff>}  \gen 0 0$, where $\pi^{\<diff>}\colon \hat T_{\red,\<diff>} \to T_{\red,\<diff>}$ denotes the canonical projection.

Let $\hat m_\Ito \colon \hat T_{\red,\<diff>} \to  T_{\<g>, \<diff>}$
be the unique morphism of $T_\d$-algebras such that, for $(k,\ell) \neq (0,0)$ we have
\begin{equ}
\hat m_\Ito(\<diff>) = \<diff>\;,\quad
\hat m_\Ito(\<not>) = \<not>\;,\quad
\hat m_\Ito(\gen 0 0) = \<g>\;,\quad
\hat m_\Ito(\gen k \ell) = 0 \;.
\end{equ}
By Lemma~\ref{lem:section}, $\pi^{\<diff>}$ admits a right inverse 
$\iota^{\<diff>}\colon T_{\red,\<diff>} \to \hat T_{\red,\<diff>}$ which is moreover a morphism of $T_\d$-algebras. 
Using this, we then define
$m_\Ito : T_{\red,\<diff>} \to  T_{\<g>, \<diff>}$  \label{Merging map m ito page ref} by
\begin{equ}
m_\Ito = \hat m_\Ito \circ \iota^{\<diff>}\;.
\end{equ}
Note that since $\iota^{\<diff>}\pi^{\<diff>} \gen 0 0 = \gen 0 0$ by Lemma~\ref{lem:section}, it follows immediately that
\begin{equ}
m_\Ito \circ \phi_{\Ito}^{\<diff>} = \id\colon T_{\<g>, \<diff>} \to T_{\<g>, \<diff>}\;.
\end{equ}
In order to see how $m_\Ito$ (and therefore also $P_\Ito$ defined in \eqref{e:defPIto} below) 
acts in a more concrete way, note that 
by \eqref{e:dergen}, an element in $T_{\red,\<diff>}$ can always be written in terms of the generators
$\gen k \ell$ without having their derivatives appear. Furthermore, by the correspondence \eqref{e:correspondence}
and the explicit
formula \eqref{e:rightInverse}, $\iota^{\<diff>}$ maps $\gen k \ell$ to $2^{-k-\ell} \d^{k+\ell}\gen 00$, which is then
mapped to $2^{-k-\ell} \d^{k+\ell} \<g>$ by $\hat m_\Ito$.
For example, it follows that one has
\begin{equ}
m_\Ito(\<2I1Xi4bc1s>) = {1\over 4} \<pre_im_1>\;,\qquad m_\Ito(\<cI1Xi4as>) = {1\over 4} \<disconnect>\;.
\end{equ}
Note here that there is a major difference between these two examples. In the first case, the 
$\CX$-graph $\<pre_im_1>$, viewed as a directed graph between its vertices, is acyclic, while in the
second case it contains a cycle. As a consequence, the morphism $\phi_\Ito$ maps the first graph back into
$\CS$, while the second graph is mapped to an element of $T_{\<generic>}$ that does not belong
to $\CS$.

The following diagram summarises the relations between the spaces introduced above.
\begin{equ} 
\begin{tikzcd}[row sep=11ex, column sep=11ex]
\hat T_{\red,\<diff>}  \arrow[r, bend left = 10, "\pi^{\<diff>} "] \arrow[d, "\psi_{\<diff>}"]  \arrow[rd, "\hat m_{\Ito}" description]
&   T_{\red,\<diff>} \arrow[l, bend left = 10, "\iota^{\<diff>} "] \arrow[d,  bend left = 10, " m_{\Ito} "] \\
T_{\<generic>,\<diff>} 
& T_{\<g>,\<diff>} \arrow[u,  bend left = 10, " \phi_{\Ito}^{\<diff>} "]
\end{tikzcd}
\end{equ}

\begin{remark} \label{cycle CS red}
Elements in $T_{\red,\<diff>}$ can naturally be represented by objects of the type
$(V_g, \mft,\phi,\CP_g)^\upper_\low$, where 
$(V_g, \mft,\phi)^\upper_\low$ is an $\CX$-graph
for $\CX = \{\<not>,\<diff>,\<generic>\}$ such that there are an even number
of vertices in $V_g$ of type $\<generic>$ and $\CP_g$ is a pairing of these vertices. The notion of
isomorphism is also the same as for $\CX$-graphs except that of course pairings need to be
preserved (but the two elements within a pair can be exchanged since these pairs are unordered). 
This is consistent with our existing graphical notation for elements of 
$\CS$ which are represented as graphs together with a pairing of the vertices of type $\<generic>$.
The existing scalar product on $\CS$ then coincides with the natural scalar product on  $T_{\red,\<diff>}$
for which $\scal{g,g}$ equals the number of its automorphisms.

One also has representations
of the adjoints of multiplication, trace and derivation that are essentially the same as
in Section~\ref{sec:Hilbert} with the obvious modifications. 
(In particular the operator $\Delta$ is not allowed to split pairs.)
\end{remark}

This motivates the introduction of the linear map $P^\acyc\colon  T_{\red,\<diff>} \to T_{\red,\<diff>}$
which maps every acyclic $\CX$-graph in $T_{\red,\<diff>}$ to itself, while it maps those $\CX$-graphs 
containing a cycle to $0$.
A very important remark here is that \textit{in the notion of `acyclic graph', paired
nodes should be identified} so that $\<2I1Xi4bc1s>$ is considered acyclic while $\<cI1Xi4as>$ is not.
Note also that $P^\acyc$ is of course not a morphism of $T_\d$-algebras 
and that it is well-defined since the quotienting
procedure for $\<not>$ does not affect the directed graph structure.

With these definitions in place, we set
\begin{equ}[e:defPIto]
P_\Ito = P^\acyc \circ M_\Ito \colon T_{\red,\<diff>} \to 
T_{\red,\<diff>}\;,
\end{equ}
where $M_\Ito = \phi_{\Ito}^{\<diff>} \circ  m_\Ito$, so that one has for example 
$P_\Ito (\<2I1Xi4bc1s> ) = \frac{1}{2}\<2I1Xi4bc1s> + \frac{1}{2}\<Xi4cbc2s>$, while $P_\Ito(\<cI1Xi4as>) = 0$. Let us remark that, since in the notion of  `acyclic graph', paired
nodes should be identified, one can define a similar projection on $T_{\<g>, \<diff>}$, also denoted by $P^\acyc$, such that $P_{\Ito} = \phi_\Ito^{\<diff>} \circ P^\acyc\circ m_\Ito$: one can disregard non-acyclic graphs either before or after applying $\phi_\Ito^{\<diff>}$. 

Let  $\Vec(\CS,\CS_{\<diff>})$ be the vector space generated by $\CS$ and $\CS_{\<diff>}$. As a consequence of the previous properties, one can deduce the following lemma:

\begin{lemma}\label{lem:proj} \label{P ito page ref}
The linear map $P_\Ito$ is a self-adjoint projection on $\Vec(\CS,\CS_{\<diff>})$ and $\range(P_\Ito)= \Vec(\CS_{\Ito},\CS_{\Ito}^{\<diff>})$.
\end{lemma}

\begin{proof}
For $m \ge 0$, write $E_m \subset T_{\red,\<diff>}$ for the set of elements of the form $\alpha \gen k \ell$
for $k+\ell = m$ and $\alpha \in \Sym(2,m)$. This set is of cardinality $2^m$ since by the correspondence
indicated in Remark~\ref{cycle CS red} we can interpret its elements as precisely those graphs with 
$m$ inputs, two outputs, and two nodes of type $\<generic>$, so that that there are 
$2^m$ ways of connecting the inputs to the two nodes. (There are of course also two ways of connecting these
nodes to the two outputs, but these are isomorphic to the graphs obtained by a suitable permutation of the inputs.)

The map $M_\Ito$ is then  given by 
\begin{equ}
M_\Ito h = 2^{-m} \sum_{g \in E_{m}} g\;,\qquad \forall h \in E_m\;.
\end{equ}
Since all elements of $E_m$ have the same directed graph structure (since for the purpose
of that structure we identify the two vertices of type $\<generic>$), 
it follows that $M_\Ito$ does not change that structure. This immediately shows that 
$M_\Ito$ commutes with $P^\acyc$, so that $P_\Ito^2 = P_\Ito$,
since both $P^\acyc$ and $M_\Ito$ are idempotent. To show that $P_\Ito$ is self-adjoint, 
it therefore suffices to show that both $P^\acyc$ and $M_\Ito$ are self-adjoint.
The fact that this is the case for $P^\acyc$ is obvious since it is diagonal in the 
basis given by $\CX$-graphs with pairings, which is also orthogonal for our scalar product.

Since elements of $E_m$ have no internal symmetries, they are orthonormal, which immediately
implies that $M_\Ito = M_\Ito^*$ on $\bigoplus_{m \ge 0} \scal{E_m}$. (As a matrix, it is given on
each $\scal{E_m}$ by the matrix with all entries identical and equal to $2^{-m}$.)
Since one also has $M_\Ito^*\<not> = M_\Ito\<not>$ and $M_\Ito^*\<diff> = M_\Ito\<diff>$, it remains to show 
that $M_\Ito^*$ is a morphism of $T_\d$-algebras to then conclude that $M_\Ito^* = M_\Ito$.
For this, we need to show that $M_\Ito^*$ commutes with the four defining operations of a $T_\d$-algebra,
namely the action of the symmetric group, multiplication, trace and derivation. 
Equivalently, we need to show that $M_\Ito$ commutes with the adjoints of these three operations,
as described in Section~\ref{sec:Hilbert} and Remark~\ref{cycle CS red}. 

The fact that $(M_\Ito \otimes M_\Ito)\Delta \tau = \Delta M_\Ito\tau$ follows immediately
from the fact that, since it does not change the directed graph structure, 
$M_\Ito$ maps irreducible elements to irreducible elements.
Similarly, the fact that $M_\Ito \alpha^* \tau = \alpha^* M_\Ito \tau$
is immediate from the fact that $\alpha^* = \alpha^{-1}$ and $M_\Ito$ 
itself is a morphism of $T_\d$-algebras.

We now show that $M_\Ito \tr^* g = \tr^* M_\Ito g$ for any
$\CX$-graph $g = (V_g, \mft,\phi,\CP_g)^\upper_\low \in T_{\red,\<diff>}$. 
Note that $M_\Ito g$ is the average over all graphs 
of the type $\hat g = (V_g, \mft,\hat \phi,\CP_g)^\upper_\low$ where, whenever 
$\phi(e) = (v,\star)$ for some $v$ of type $\<generic>$, one has
$\hat \phi(e) \in \{(v,\star),(\bar v,\star)\}$, where $\bar v$ is the unique
vertex such that $\{v,\bar v\} \in \CP_g$. If $\phi(e)$ is not of this type, then
one has $\hat \phi(e) = \phi(e)$.
Note also that for any of the graphs $\hat g$ appearing in the description of $M_\Ito g$,
one has a natural identification of $\Out(\hat g)$ with $\Out(g)$ and, for any $\{v,\bar v\} \in \CP_g$.
it is still the case that if $\cod(e) \in \{v,\bar v\}$ in $g$, then the same is true in $\hat g$.
It is then straightforward to see that
\begin{equ}
\Cut_e M_\Ito g = M_\Ito \Cut_e g\;,
\end{equ}
which immediately implies the claimed commutation.

The fact that $M_\Ito \d^* g = \d^*(M_\Ito g)$ follows in a similar way, thus completing the
proof that $M_\Ito$ is self-adjoint. 

It remains to show that, when restricted to $\Vec(\CS,\CS_{\<diff>})$, $P_\Ito$ is indeed the projection 
onto $\Vec(\CS_\Ito,\CS_\Ito^{\<diff>})$.
Recall that $\Vec(\CS,\CS_{\<diff>})$ is spanned by the irreducible graphs of degree $(1,0)$
with at most two pairs of vertices of type $\<generic>$, which are furthermore connected
when viewed as $\CX$-graphs \textit{without} identifying the paired vertices.
The image of the map $M_\Ito$ is \textit{not} in general contained in this space since,
although it preserves the directed graph structure with pairs of vertices identified, it
does not preserve the structure without this identification. 
We claim however that $M_\Ito$ maps the image of $P^\acyc$ into $\Vec(\CS,\CS_{\<diff>})$.
This is because, since we consider only objects of type $(1,0)$, 
the only way the graph obtained without identifying paired vertices can be disconnected is if it has a connected
component of type $(0,0)$. Since every generator has exactly one output, such a component 
has necessarily a cycle and therefore cannot appear in the image of $P^\acyc$.

The fact that $\range(P_\Ito)= \Vec(\CS_{\Ito},\CS_{\Ito}^{\<diff>})$ then follows from Proposition~\ref{prop:rangeIto},
noting that since $\phi_{\Ito}^{\<diff>}$ is injective and it maps 
cyclic graphs out of $\Vec(\CS,\CS_{\<diff>})$, restricting its domain to acyclic graphs does not 
affect the intersection of its image with $\Vec(\CS_{\Ito},\CS_{\Ito}^{\<diff>})$.
\end{proof}

\begin{theorem}\label{theo:mainSum} \label{CB Ito page ref1}
One has $\CS_\Ito + \CS_\geo = \CS_\both$ and $\CS_\Ito \cap \CS_\geo = \scal{\{\tau_\star,\tau_c\}}$.
Furthermore, $\CS_\Ito = \scal{\CB_\Ito}$.
\end{theorem}

\begin{proof}
Since $\CS_\Ito + \CS_\geo \subset \CS_\both$ and using Corollary \ref{ge32}, it suffices to show 
that $\dim \CS_\both \le 32$ in order to prove that $\CS_\Ito + \CS_\geo = \CS_\both$ and $\dim(\CS_\Ito \cap \CS_\geo)=2$. Since $\dim \CS = 54$, this means that it suffices to find a collection $\mathcal{E}^*$ of $22$ linearly independent elements of $\CS^*$ such that $\CS_\both \subset \bigcap_{e\in {\mathcal E}^*} \ker(e)$. 
A possible choice for ${\mathcal E}^*$ will be obtained using the set $\mathcal{T}_{c}^{\<not>}\subset \CS$ given by
\begin{align*}
\mathcal{T}_{c}^{\<not>} = \left\{\!\!
\begin{array}{c}
 \<cI1Xi4as> \,,   \<I1Xi4ac1s> \,, \<I1Xi4ac2s> \,, \<cI1Xi4bs> \,,  \<I1Xi4abc1s> \,,  \<cI1Xi4bcs> \,, \<I1Xi4bcc1s> \,, \<cI1Xi4acs> \,,     \<2I1Xi4bc3s> \,,  \<2I1Xi4bc2s> \,,  \<2I1Xi4cc2s>  \,,  \<I1Xi4bc1s>  \,,\<cI1Xi4cs> \,,  \<I1Xi4cc1s>\\   \<I1Xi4cc2s> \,,  \<Xi4eac2s>+\<Xi4eabisc3s> \,,  \<Xi4ebc1s> \,, \<Xi4ebc2s> \,,  \<Xi4ca1s>+\<Xi4ca2s>  \,,   \<Xi4cbc1s>\,, \<I1Xi4acc1s> -  \<2I1Xi4cc1s>, \<2I1Xi4bc1s> - \<Xi4cbc2s>
\end{array}
\!\!\right\}\;,
\end{align*}
which has cardinality $22$. For any $\tau \in \CS$, using the correspondence $\CS\approx \CS^*$ given by the scalar product, we define the linear form $\ell_\tau=\hat{\phi}_\geo^*\phi_{\<diff>}(\tau)$ where $\hat{\phi}_\geo^*$ is the adjoint of $\hat{\phi}_\geo$. We show that $\CS_\both \subset \bigcap_{\tau \in \mathcal{T}_{c}^{\<not>}} \ker(\ell_\tau)$, then, we explain why the elements $\{\ell_\tau: \tau\in\mathcal{T}_c^{\<not>}\}\subset\CS^*$ are linearly independent.

For all $\tau \in \mathcal{T}_{c}^{\<not>}$, except the two last elements, 
$m_\Ito (\tau)$ is a linear combination of cyclic graphs, while for the two last elements, $m_\Ito(\tau)=0$, so that
$P_\Ito(\tau) = P^\acyc \circ \phi_{\Ito}^{\<diff>} \circ  m_\Ito (\tau)=0$ for all $\tau\in \mathcal{T}_c^{\<not>}$. Let $\phi_{\<diff>}: T_{\red} \to T_{\red,\<diff>}$ be the infinitesimal morphism of 
$T_\d$-algebras with respect to the canonical injection, mapping $\<not>$ to $\d^2 \<diff>$ and $\gen{k}{l}$ to $0$ 
(this is well-defined since $S_{1,1}^1 \d^2 \<diff> = \d^2 \<diff>$ by \eqref{e:d2} and
since all other generators of $I$ in \eqref{e:groupkl} and \eqref{e:dergen} involve linear combinations of
the generators $\gen k \ell$). Since $\phi_{\<diff>}$ maps any $\CX$-graph to a linear
combination of $\CX$-graphs having the same directed graph structure, we get that $P_{\Ito}\circ \phi_{\<diff>}=\phi_{\<diff>}\circ P_{\Ito}$ and thus, for any $\tau\in \mathcal{T}_c^{\<not>}$, $\phi_{\<diff>}(\tau)\in \ker P_\Ito$. 

Using Proposition~\ref{prop:both} and Lemma~\ref{lem:proj}, for any $v\in \mathcal{S}_\both$, $\hat{\phi}_\geo(v) \in \range{P_\Ito}$, and, since $P_\Ito$ is self-adjoint, $\hat{\phi}_\geo(v)$ is orthogonal to $\ker{P_\Ito}$. Thus for $\tau\in \mathcal{T}_c^{\<not>}$, $\scal{\phi_{\<diff>}(\tau),\hat \phi_\geo(v)}=0$, hence $\scal{\hat{\phi}_\geo^*\phi_{\<diff>}(\tau),v}=0$. 
Since $\ell_\tau=\hat{\phi}_\geo^*\phi_{\<diff>}(\tau)$, we have shown that $\CS_\both \subset \bigcap_{\tau \in \mathcal{T}_{c}^{\<not>}} \ker(\ell_\tau)$. Let us remark that we omitted $\<Xi4eac1s> - \<Xi4eabisc1s>$ in the definition of $\mathcal{T}_c^{\<not>}$ since $\phi_{\<diff>}(\<Xi4eac1s> - \<Xi4eabisc1s>)=0$ hence $\ell_{\<Xi4eac1s> - \<Xi4eabisc1s>} = 0$. 
It remains to prove that $\{\ell_\tau: \tau\in\mathcal{T}_c^{\<not>}\}\subset\CS^*$ is a linearly independent family. 

Instead of working with the usual basis of $\CS$ consisting of trees, we change the family $\{\<Xi4eac2s>,\<Xi4eabisc3s>,\<Xi4ca1s>,\<Xi4ca2s> ,\<I1Xi4acc1s>,  \<2I1Xi4cc1s>, \<2I1Xi4bc1s> , \<Xi4cbc2s>\}$ into the family $\{\<Xi4eac2s>+\<Xi4eabisc3s>,\<Xi4eac2s>-\<Xi4eabisc3s>,\<Xi4ca1s>+\<Xi4ca2s>,\<Xi4ca1s>-\<Xi4ca2s>,\<I1Xi4acc1s> -  \<2I1Xi4cc1s>, \<I1Xi4acc1s> +  \<2I1Xi4cc1s>, \<2I1Xi4bc1s> - \<Xi4cbc2s>, \<2I1Xi4bc1s> + \<Xi4cbc2s>\}$. Moreover, for any tree $\tau$ in $\CS_{\<diff>}$, let ${\sf dg}(\tau)$ denote the total number of
vertices of type $\<not>$ and $\<diff>$ in $\tau$. A non-zero element $\eta \in \CS_{\<diff>}$ is homogeneous of degree $d$ if it is a linear combination of trees of degree $d$.

Let us show that for any $\tau \in \mathcal{T}_{c}^{\<not>}$, any element $\tilde{\tau}$ of this new basis such that ${\sf dg}(\tilde{\tau})\geq {\sf dg}(\tau)$, $\scal{\ell_\tau,\tilde{\tau}}\neq 0$ if and only if $\tau = \tilde{\tau}$. Using \eqref{e:phigeo}, one gets that for any homogeneous $\eta \in \CS$ such that $ \hat{\phi}_\geo(\eta)+ 2\phi_{\<diff>}(\eta) \neq 0 $, ${\sf dg}(\hat{\phi}_\geo(\eta)+ 2\phi_{\<diff>}(\eta))={\sf dg}(\eta)+1$. Since trees of different degrees are orthogonal, for any homogeneous $\eta$, $\tilde{\eta} \in \CS$ such that ${\sf dg}(\tilde{\eta})\geq {\sf dg}(\eta)$, 
\begin{equs}
\label{eq:elltau}
\scal{\ell_\eta,\tilde{\eta}}=\scal{\phi_{\<diff>} (\eta), \hat{\phi}_\geo(\tilde{\eta})}=- 2 \scal{\phi_{\<diff>} (\eta), \phi_{\<diff>}(\tilde{\eta})}.
\end{equs}
In particular, this easily implies that for any tree $\tau$ in $\mathcal{T}_{c}^{\<not>}$, any $\tilde{\tau}$ in the new basis such that ${\sf dg}(\tilde{\tau})\geq {\sf dg}(\tau)$, $\scal{\ell_\tau,\tilde{\tau}} \neq 0$ if and only if  $\tau = \tilde{\tau}$. Let us show that this holds also for the remaining elements in $\{\<Xi4eac2s>+\<Xi4eabisc3s>, \<Xi4ca1s>+\<Xi4ca2s> , \<I1Xi4acc1s> -  \<2I1Xi4cc1s>, \<2I1Xi4bc1s> - \<Xi4cbc2s>\} \subset \mathcal{T}_{c}^{\<not>}$ by showing that:  
\begin{align*}
\ell_{\<Xi4eac2s>+\<Xi4eabisc3s>}(\<Xi4eac2s>-\<Xi4eabisc3s>) = \ell_{\<Xi4ca1s>+\<Xi4ca2s>} (\<Xi4ca1s>-\<Xi4ca2s>) =\ell_{\<I1Xi4acc1s> - \<2I1Xi4cc1s>} (\<I1Xi4acc1s> + \<2I1Xi4cc1s>) = \ell_{\<2I1Xi4bc1s> - \<Xi4cbc2s>}(\<2I1Xi4bc1s> + \<Xi4cbc2s>) =0.\end{align*}
Using  \eqref{eq:elltau}, we have to show that: 
\begin{align*}
\scal{\phi_{\<diff>} (\<Xi4eac2s>+\<Xi4eabisc3s>), \phi_{\<diff>}(\<Xi4eac2s>-\<Xi4eabisc3s>)} &= \scal{\phi_{\<diff>} (\<Xi4ca1s>+\<Xi4ca2s>), \phi_{\<diff>}(\<Xi4ca1s>-\<Xi4ca2s>)} =0, \text{ and }\\
\scal{\phi_{\<diff>} (\<I1Xi4acc1s> - \<2I1Xi4cc1s>), \phi_{\<diff>}(\<I1Xi4acc1s> + \<2I1Xi4cc1s>)} &= \scal{\phi_{\<diff>} (\<2I1Xi4bc1s> - \<Xi4cbc2s>), \phi_{\<diff>}(\<2I1Xi4bc1s> + \<Xi4cbc2s>)} =0. 
\end{align*}
The first assertion is a consequence of the fact that $\phi_{\<diff>}(\<Xi4eac2s>-\<Xi4eabisc3s>) = 0 = \phi_{\<diff>} (\<Xi4ca1s>-\<Xi4ca2s>)$. As for the 
second, it can be deduced from the fact that for any $\tau$ appearing in $\phi_{\<diff>} (\<I1Xi4acc1s> + \<2I1Xi4cc1s>)$, $\|\tau\|=1$ and for any $\tau$ appearing in $\phi_{\<diff>} ( \<2I1Xi4bc1s> + \<Xi4cbc2s>)$, $\|\tau\|^2=2$.

Thus, we have proven that for any $\tau \in \mathcal{T}_{c}^{\<not>}$, any element $\tilde{\tau}$ of this new basis such that ${\sf dg}(\tilde{\tau})\geq {\sf dg}(\tau)$, $\scal{\ell_\tau,\tilde{\tau}}\neq 0$ if and only if $\tau = \tilde{\tau}$. In particular, this proves that $\ell_\tau = \alpha_\tau \tau+\sigma$ with ${\sf dg(\sigma)}<{\sf dg(\tau)}$ and $\alpha_\tau$ a non-zero real: the elements $\{\ell_\tau: \tau\in\mathcal{T}_c^{\<not>}\}\subset\CS^*$ are linearly independent thanks to the triangular structure given by counting the number of $\<not>$ vertices.

Let us study the intersection of $ \CS_\Ito$ and $\CS_\geo$. Using the previous results, we get that $\dim(\CS)=54$ and also $\dim(\CS_\Ito \cap \CS_\geo)=2$. In order to conclude, it remains to prove that $\tau_\star, \tau_c \subset \CS_\Ito \cap \CS_\geo$. Using Lemma \ref{Nabla preserved by phigeo} and \eqref{e:taustar}, \eqref{e:tauc}, we get that $\tau_\star, \tau_c\in \CS_\geo$. Besides, one can check that $\tau_\star, \tau_c\in \CS_\Ito$ using the explicit formulas \eqref{eq:valeurR} and \eqref{eq:valeurC}. 
\end{proof}

\begin{remark}\label{rem:geo}
This theorem implies that $\dim \CS_\geo = 15$ and that the family
 \begin{equ}(\phi_{\flt}(\tau))_{\tau \in \mathcal{B}_\geo\cap \mathcal{S}_{\flt}} \cup \{\Nabla_{\<generic>}(R(\<generic>,\<genericb>)\<genericb>),\ \Nabla_{R(\<generic>,\<genericb>)\<genericb>}\<generic>,  R(\<generic>,\Nabla_{\<genericb>}\<genericb>)\<generic>,\ 
R(\<generic>,\Nabla_{\<genericb>}\<generic>)\<genericb>,\ 
R(\<generic>,\Nabla_{\<genericb>}\<generic>- 2\Nabla_{\<generic>}\<genericb>)\<genericb>\} \end{equ} 
provided by the proof of Proposition \ref{prop:geoDim} is a basis of $\CS_\geo$. After a change of basis, 
one can verify that a simpler one is given by $\VV\cup \left\{\Nabla_{\<generic>}\<generic>\right\}$, where $\VV$ was defined in \eqref{eq:covderiv}. In particular, since $\Upsilon_{ \Gamma, \sigma}(\Nabla_{\tau}\bar \tau) = \Nabla_{\Upsilon_{ \Gamma, \sigma }\tau} \Upsilon_{\Gamma,\sigma} \bar\tau$, this proves that for any $\tau \in \CS_\geo$,
the identity \eqref{e:idendiff} automatically holds for all diffeomorphisms. Let us remark that the only element missing from the list of combinatorial covariant derivatives is $\Nabla_{\<genericb>}\Nabla_{\<generic>}\Nabla_{\<genericb>}\<generic>$,
the reason being that it is equal to
\begin{equ}\label{relationV}
\Nabla_{\<generic>} \Nabla_{\<genericb>} \Nabla_{\<genericb>} \<generic> + \Nabla_{\<genericb>} \Nabla_{\Nabla_{\<genericb>}\<generic>} \<generic> - \Nabla_{\Nabla_{\<genericb>}\Nabla_{\<genericb>}\<generic>}\<generic> - \Nabla_{\Nabla_{\Nabla_{\<genericb>}\<generic>}\<generic>}\<genericb> + 
\Nabla_{\Nabla_{\Nabla_{\<genericb>}\<generic>}\<genericb>}\<generic> -
\Nabla_{\<generic>}\Nabla_{\Nabla_{\<genericb>}\<generic>} \<genericb>+
\Nabla_{\Nabla_{\<generic>}\Nabla_{\<genericb>}\<generic>}\<genericb>. 
\end{equ}
Finally, we conclude from the previous theorem that $\dim (\CS_\Ito) = 19$ and that $\mathcal{B}_\Ito$ is a basis of $\CS_\Ito$. 
\end{remark}

\appendix
\section{Symbolic index}
In this appendix, we collect the most used symbols of the article, together
with their meaning and the page where they were first introduced.
 \begin{center}
\renewcommand{\arraystretch}{1.1}
\begin{longtable}{lll}
\toprule
Symbol & Meaning & Page\\
\midrule
\endfirsthead
\toprule
Symbol & Meaning & Page\\
\midrule
\endhead
\bottomrule
\endfoot
\bottomrule
\endlastfoot
 $[\upper]$ & Set $\{i\in\N: 1\leq i\leq \upper\}$ & \pageref{upper page ref}
\\ $ \CB_\geo $ & Linear independent set over $\CS_\geo$ &  \pageref{B geo page ref}
\\ $ \CB_{\Ito} $ & Collection of trees generating $ \CS_{\Ito} $ &
\pageref{CB Ito page ref}
\\
$\CB_\star^a$ & Space in which the laws of the solutions take values
 &  \pageref{law space page ref} 
 \\
 $C_{\eps,\geo}^\BPHZ $ & Element of $ \CS $ renormalising $ U_\eps^\geo $ &  \pageref{constant geo page ref} 
 \\
 $C_{\eps,\Ito}^\BPHZ$ &  Element of $ \CS $ renormalising $ U_\eps^\Ito $  & \pageref{constant ito page ref}
 \\
 $\hat C_{\eps,\geo}^\BPHZ $ &  Element of $ \CS_{\geo} $ such that $\hat C_{\eps,\geo}^\BPHZ -  C_{\eps,\geo}^\BPHZ $ converges &  \pageref{constant hat geo page ref}
 \\
 $\hat C_{\eps,\Ito}^\BPHZ $ &  Element of $ \CS_{\Ito} $ such that $\hat C_{\eps,\Ito}^\BPHZ -  C_{\eps,\Ito}^\BPHZ $ converges &  \pageref{constant hat geo page ref}
 \\
 $\CC_\star^a$ & Space in which the solutions take values  & \pageref{space solution page ref}
 \\
  $ \d $ &  Abstract derivation from $ \CV_\low^\upper $  to $ \CV_{\low+1}^{\upper}$ & \pageref{derivation page ref}
\\
$\Gamma$ & Christoffel symbols for the Levi-Civita connection & \pageref{gamma page ref}
\\
$g$ & Inverse metric tensor of $\CM$ given by $\sigma \sigma^\top$ & \pageref{g page ref}
\\
$H_{\Gamma,\sigma}$ & Vector field statisfying $ H_{\Gamma,\sigma} =  \Upsilon_{\Gamma,\sigma} \tau_{\star} $ & \pageref{vector field H page ref}
\\
 $ m_\Ito $ & Linear map from $  T_{\red,\<diff>} $ to $   T_{\<g>, \<diff>}$ replacing a pairing by $ \<g> $ & \pageref{Merging map m ito page ref} 
\\
$ \CM $ & Compact Riemannian manifold & \pageref{manifold page ref}
\\
$\Moll$ & Space of mollifiers & \pageref{Moll page ref} \\
$P_\Ito$ & Self-adjoint projection on $\Vec(\CS,\CS_{\<diff>})$, $\range{P_\Ito}=\Vec(\CS_{\Ito},\CS_{\Ito}^{\<diff>})$ & \pageref{P ito page ref}
\\ $\PPi^\BPHZ$ & BPHZ model & 
\pageref{BPHZ Model  page ref}
\\ $ \PPi_\geo^{(\eps)} $ & Model of the geometric regularisation & 
 \pageref{Ito and geo approximations  page ref}
 \\
  $ \PPi_\Ito^{(\eps)} $ & Model of the Itô regularisation & 
 \pageref{Ito and geo approximations  page ref}
\\
$ \hPPi^{(\eps)}_\geo $ & BPHZ renormalisation
of $ \PPi_\geo^{(\eps)} $  &
\pageref{BPHZ renormalisation of ito geo model page ref}
\\
$ \hPPi^{(\eps)}_\Ito $ & BPHZ renormalisation
of $ \PPi_\Ito^{(\eps)} $  &
\pageref{BPHZ renormalisation of ito geo model page ref}
\\
 $\phi_\geo $ & Geometric infinitesimal morphism from $T_{\<generic>}$ to $T_{\<generic>,\<diff>}$  & \pageref{phi geo page ref}
 \\
$ \hat \phi_{\geo} $ & Linear map defined on $(T_{\<generic>})_k^1$ by $  \phi_\geo - [\cdot,\<diff>] $ &  \pageref{hat phi geo page ref}
 \\
  $\phi_\Ito$  & Itô  morphism from $T_{\<g>}$ to $T_{\<generic>}$  & \pageref{phi ito page ref} 
\\   $ \phi_{\Ito}^{\<diff>} $    & Itô  morphism from $T_{\<g>,\<diff>}$ to $T_{\<generic>,\<diff>}$ & \pageref{phi ito page ref}
\\
 $\psi $ & Morphism from $ \hat T_{\red}$ to $ T_{\<generic>}$ & \pageref{psi morpishm page ref} 
 \\
  $\psi_{\<diff>} $ & Morphism from $ \hat T_{\red,\<diff>} $ to $ T_{\<generic>,\<diff>}$ & \pageref{psi diff page ref}
 \\
$\CQ$ & Set of pairs $(\sigma, \bar \sigma)$
such that $\sigma \sigma^\top = \bar \sigma \bar \sigma^\top$ & \pageref{CQ}
\\ 
 $ R $ &  Riemannian curvature tensor  & \pageref{Riemannian curvature tensor page ref}
\\
 $\sigma_i$  & Collection
of vector fields on $\CM$ & \pageref{sigma page ref} \\
$\SS_{0}$ & Symbols with one noise type& \pageref{one noise type  page ref} 
\\
 $\SS_{\<generic>}^{(k)}$ & Collection of trees with $k$ noises  & \pageref{SS2 page ref}\\
 $\SS_{\<generic>}$ & Collection of trees $\SS_{\<generic>}^{(2)} \cup \SS_{\<generic>}^{(4)}$ & \pageref{SS page ref}\\
 $\SS$ & Trees endowed with partition of noises into pairs  & \pageref{SShat page ref}\\
 $\CS_{\<generic>}$ & Linear span of $\SS_{\<generic>}$ & \pageref{CS page ref}\\
  $\CS$ & Linear span of elements of $\SS$  & \pageref{CS page ref2} \\
 $ \CS_{\<diff>}$ & Subspace of $ T_{\<generic>,\<diff>} $ spanned by trees with $ 2$ or $4 $ noises and one $\<diff>$ & \pageref{cs diff page ref}
 \\
 $\CS^\nice$ & Subspace of $ \CS $ consisting of `minimal' counterterms  & \pageref{S nice page ref}
  \\
 $\CS_\geo $ & Set of geometric elements (vector fields) in $\CS$ & \pageref{def:geoIto}\\
 $ \CS_\Ito $ & Set of It\^o elements (depending on $\sigma$ only through $g$) in $\CS$ &\pageref{def:geoIto} \\
  $\CS_\Ito^{\<diff>} $& Subspace of Itô elements in $  \CS_{\<diff>}$ & \pageref{CS ito diff page ref} 
 \\
$ \CS_\geo^\nice $ & Subspace of  $\CS^\nice$ such that $ \CS_\geo^\nice = \CS^\nice \cap \CS_\geo  $  & \pageref{subspace of CS nice} 
 \\
 $ \CS_\Ito^\nice $ & Subspace of  $\CS^\nice$ such that $ \CS_\Ito^\nice = \CS^\nice \cap \CS_\Ito  $  & \pageref{subspace of CS nice}
 \\
 $\CS_\both$ & Subspace of $ \CS $ such that $  \CS_\both = \CS_{\Ito} + \CS_{\geo}$  &  \pageref{S both page ref}
 \\
 $\swap_{\low_1,\low_2}^{\upper_1,\upper_2}$ &
 Element of $\Sym(\upper_1+\upper_2,\low_1+\low_2)$  & \pageref{swapping map page ref} 
 \\
  $\Sym(\upper)$ & Symmetric group on $[\upper]$  & \pageref{sym page ref} \\
   $\Sym(\upper,\low)$ & Direct product of groups $   \Sym(\upper) \times \Sym(\low)$ & \pageref{product sym page ref} 
 \\
$T_{\<generic>}$ & Extension of $\CS_{\<generic>}$ with disconnected graphs and loops & \pageref{S generic page ref}
 \\
 $\hat T_{\red}$ & Free $T_\d$-algebra with generators $\<not>$ and $\gen{k}{\ell}$ & \pageref{hat CS red page ref}
 \\
  $\hat T_{\red,\<diff>}$ &  Extension of $\hat T_{\red} $ with the generator $ \<diff> $ & \pageref{CS red diff page ref}
 \\
 $T_{\red} $ &  Quotiented space given by  $T_{\red} = \hat T_{\red} / \ker \psi$ & \pageref{bar CS red page ref}
 \\
  $T_{\red,\<diff>}$ & Quotiented space given by $T_{\red,\<diff>} = \hat T_{\red,\<diff>} / \ker \psi_{\<diff>}$ & \pageref{Quotient space CS red diff page ref}
 \\
  $ T_{\<generic>,\<diff>} $ & Extension of $ T_{\<generic>} $ that includes the generator $ \<diff> $ & \pageref{S diff page ref}
 \\
 $ T_{\<g>} $  & Same definition as for
 $ T_{\<generic>} $ but with $ \<g> $  instead of the $ \<generic>_i $ &  \pageref{bar CS g page ref}
 \\
  $ T_{\<g>, \<diff>} $ &  Extension of $ T_{\<g>} $ that includes the generator $ \<diff> $ &  \pageref{bar CS g page ref}
 \\
 $T_\d(\CX)$ & Free $T_\d$-algebra generated by the set $\CX$ & \pageref{page:TdCX}\\
 $\tau_\star$ & Degree of freedom of canonical family of solutions & \pageref{def taustar}, \pageref{def taustar2}\\
 $\tau_c$ & Additional degree of freedom for non-canonical family & \pageref{def tauc}\\
 $ \tr $ &  Abstract partial trace from $ \CV_{\low+1}^{\upper+1}$ to $ \CV_{\low}^{\upper}$ & \pageref{trace page ref}
 \\
 $\Tr(W)$ & $T$-algebra generated by the collection of spaces $\{W^u_\ell\}$ & \pageref{page:TrW}\\
 $\Tr_g(W)$ & Subspace of $\Tr(W)$ corresponding to the $T$-graph $g$ & \pageref{page:TrWg}\\
 $U^\family$ & Law of the canonical family of solutions  & \pageref{family solutions page ref}
 \\
  $U^\geo$ & Law of the geometric solution  & 
  \pageref{U geo page ref} 
   \\
 $U^\Ito$ & Law of the solution satisfying the Itô isometry & \pageref{U ito page ref} 
\\
$U^\BPHZ$ & Law of the BPHZ solution &
\pageref{U solutions page ref}
 \\
  $ \VV $ & Vector fields appearing as counterterms & \pageref{VV page ref}\\ 
 $ \CV $ & Linear span of $ \VV $ & \pageref{V page ref}\\ 
 $\CV^\nice$ &  Subspace of $ \CV $ consisting of `minimal' counterterms & \pageref{V nice page ref}
 \\
 $\CV_\star$ & Subspace of $ \CV^{\nice} $ generated by $ \tau_{\star} $  & \pageref{V star page ref}
 \\
 $\CV_\star^\perp$ & An arbitrary complement of $ \CV_{\star} $ in $ \CV^{\nice} $ &  \pageref{V start perp page ref}
 \\
 $\CX$ &  A finite number of types  & \pageref{set of types page ref}
 \\
 $\Nabla_X Y$ & Covariant derivative of $ Y $ in the direction of $ X $ & \pageref{covariant derivative page ref}
\\
 $\Upsilon_{\Gamma,\sigma}$ & Valuation map &  \pageref{Evaluationmap1 page ref}, \pageref{Evaluationmap2 page ref}
 \\
 \end{longtable}
 \end{center}
\endappendix

\bibliographystyle{./Martin}

\bibliography{refs}

\end{document}